\documentclass[11pt]{article}
\usepackage[centertags,intlimits]{amsmath}                     
\usepackage{amsfonts,amssymb}                     
\usepackage{amsthm}
\usepackage{graphics}
\usepackage{color}
\allowdisplaybreaks[2]
\numberwithin{equation}{section}
\newtheorem{theorem}{Theorem}[section]
\newtheorem{proposition}{Proposition}[section]
\newtheorem{lemma}{Lemma}[section]
\newtheorem{corollary}{Corollary}[section]

\newtheorem{condition}{Condition}[section]
\theoremstyle{remark}
\newtheorem{remark}{Remark}[section]

\providecommand{\abs}[1]{\lvert #1\rvert}
\providecommand{\norm}[1]{\lVert #1\rVert}

\newcommand{\nc}{\newcommand}
\nc{\vb}{\mathbf{v}}
\nc{\bx}{\mathbf{x}}
\nc{\by}{\mathbf{y}}
\nc{\bz}{\mathbf{z}}
\nc{\bu}{\mathbf{u}}
\nc{\bv}{\mathbf{v}}
\nc{\ba}{\mathbf{a}}
\nc{\bs}{\mathbf{s}}
\nc{\bq}{\mathbf{q}}
\nc{\bd}{\mathbf{d}}
\nc{\bb}{\mathbf{b}}
\nc{\bc}{\mathbf{c}}
\nc{\bi}{\mathbf{i}}
\nc{\bfr}{\mathbf{r}}
\nc{\bA}{\mathbf{A}}
\nc{\R}{\mathbb R}
\nc{\N}{\mathbb N}
\nc{\bC}{\mathbb{C}} 
\nc{\D}{\mathbb D}
\nc{\Z}{\mathbb Z}
\nc{\F}{\mathbf F}
\nc{\bbS}{\mathbb S}
\nc{\B}{\cal B}
\nc{\br}{\bigr}
\nc{\bl}{\bigl}
\nc{\Bl}{\Bigl}
\nc{\Br}{\Bigr}
\nc{\ind}{\mathbf{1}}
\nc{\bP}{\mathbf{P}}
\DeclareMathOperator*{\esssup}{ess\,sup\;}  
\title{On large deviations
 of coupled   diffusions\\ with time scale separation}
\author{Anatolii A. Puhalskii\\
University of Colorado Denver  and\\
 Institute for Problems in Information
Transmission}
\begin{document}

\maketitle
\sloppy
\begin{abstract}
\item 
We consider two  It\^o equations that evolve on
  different time scales.
The equations are fully coupled in the sense
  that all of the coefficients may depend on both the ``slow'' and the
  ``fast'' motion and the diffusion
  terms may be correlated. The diffusion term in the slow process is small.
A large deviation principle is obtained
  for the joint distribution of the slow process and of the empirical
process of the fast variable. 
By projecting on the slow and fast
  variables, we arrive at new results on large deviations in the
  averaging framework and on large deviations of the empirical
  measures of ergodic diffusions, respectively. The proof of the main
  result
 relies
  on the property that exponential tightness implies large deviation
  relative compactness. The identification of the large deviation rate
  function is accomplished by analysing the large deviation limit of an
  exponential martingale. 
\end{abstract}

\section{Introduction}
\label{sec:derivation}
Consider the coupled diffusions specified by the stochastic differential
equations 
\begin{equation}
        \label{eq:67} 
  \begin{split}
  dX^\epsilon_t&={A}(X^\epsilon_t,x^\epsilon_{t})\,dt
+\sqrt{\epsilon}\,{B}(X^\epsilon_t,x^\epsilon_{t})\,dW^\epsilon_t,\\
dx^\epsilon_t&=\frac{1}{\epsilon}\,{a}(X^\epsilon_t,x^\epsilon_t)\,dt+
\frac{1}{\sqrt{\epsilon}}\,{b}(X^\epsilon_t,x^\epsilon_t)\,  
dW^\epsilon_t\,,
  \end{split}
\end{equation}
where $\epsilon>0$ is a small parameter\,.
Here, 
 ${A}(u,x)$, where $u\in\R^n$ and $x\in \R^l$, is an $n$-vector, 
${B}(u,x)$ is an $n\times k$-matrix, ${a}(u,x)$ is an
$l$-vector,   ${b}(u,x)$ is an $l\times k$-matrix, and 
$W^\epsilon=(W^\epsilon_t,
\,t\in\R_+)$ 
is an $\R^k$-valued standard
Wiener process.
Accordingly, the stochastic process $X^\epsilon=(X^\epsilon_t,\,t\in\R_+)$ takes
values in $\R^n$ and the stochastic process $x^\epsilon=
(x^\epsilon_t,\,t\in\R_+)$ takes
values in $\R^l$.  The processes $X^\epsilon$ and
$x^\epsilon$ are seen to evolve on different time scales in that  time for
$x^\epsilon$ is accelerated by a factor of $1/\epsilon$\,. 
In a number of application areas, one is concerned with finding
 the logarithmic asymptotics of large deviations for
the ``slow'' process $X^\epsilon$ as $\epsilon\to0$\,.
(As a matter of fact, our interest in this setup
has been aroused by an application to optimal portfolio
selection.) 
When no diffusion
term is present in the equation for the slow process, this sort of
result is usually referred to as ``the averaging principle''. For
contributions,
see Freidlin \cite{Fre78}, 
Veretennikov
\cite{Ver13,Ver94,Ver99,Ver00},
 Feng and Kurtz \cite[Section 11]{FenKur06}, and references therein.
A different perspective  has been offered by
Liptser \cite{Lip96} whose insight was to   consider
 the joint distribution of the
slow process and of the empirical process associated with the fast
variable. For the case where  the processes $X^\epsilon$ and $x^\epsilon$ are
one-dimensional, the coefficients
 $a(u,x)$ and $b(u,x)$ do not depend on the first variable and 
 the  Wiener processes driving the diffusions can be taken
independent, 
they derived  a large deviation principle (LDP)
 for the pair $(X^\epsilon,\mu^\epsilon)$ and
identified the associated large deviation rate function, where 
$\mu^\epsilon$ represents the empirical process associated with 
$x^\epsilon_t$\,. 
 The large
 deviation principle  for the slow process then follows by projection.

In  this paper, we  extend the joint LDP in Liptser
\cite{Lip96} to the
 multidimensional case. It is assumed that the process dimensions are
arbitrary and that all
 coefficients may depend on both variables in a continuous fashion, 
on the time variable in a
 measurable fashion, and on $\epsilon$\,.
The diffusions driving the slow and the fast processes do
 not have to be uncorrelated.
We prove the large deviation principle for the distribution of 
$(X^\epsilon,\mu^\epsilon)$ 
and produce  the large
deviation rate function.
Projections on the first and second coordinates yield   LDPs 
 for $X^\epsilon$ and $\mu^\epsilon$\,, respectively.

The  results  in the literature 
that obtain  an LDP 
and identify the large deviation rate function
for $X^\epsilon$\,,  with a
nondegenerate diffusion term  being present
 in the first of  equations (\ref{eq:67}),
concern  the time
homogeneous case where the
diffusion coefficient in the equation for the fast process does not
depend on the slow process,
Veretennikov 
\cite{Ver98a,Ver_letter,Ver00},  Liptser \cite{Lip96}, 
Feng and Kurtz \cite[Section 11]{FenKur06}.
 The latter restriction can be removed in the
setting of the averaging principle provided
 the state space of the
fast process  is compact,  Veretennikov
\cite{Ver13,Ver98b,Ver99}. 
This paper   fills in the  gaps
by tackling a  case 
 of fully coupled diffusions  in a noncompact
 state space.
 In addition, the coefficients may depend on the time variable explicitly.
 The results cover both the setup with a
 nondegenerate diffusion term 
 and the setup with no diffusion term
 in the equation for the slow process. The form of the large deviation
 rate function for the slow process is new.
For the time-homogeneous
case, the continuity and nondegeneracy 
conditions on the coefficients   are similar to those in 
the literature, except that  additional smoothness properties are assumed of
 $b(u,x)$ as a function of $x$\,, as it is done in
 Liptser \cite{Lip96}. In return, we obtain that 
the probability measures for which
the large deviation rate functions are finite 
must have  weakly differentiable densities whose square roots belong
to the Sobolev space $\mathbb{W}^{1,2}(\R^l)$\,.
In particular, additional insight is gained into the LDP 
for the empirical measures of
ergodic diffusion processes. 
On the other hand, the ergodicity requirements on the fast process in
the nongradiental case are   more restrictive than those in some of
 the literature.

As in Liptser \cite{Lip96},  an important part in our approach is played by
 the property that exponential
 tightness implies  large deviation relative  compactness so that
 once exponential tightness has been shown, 
establishing that
 a large deviation limit point is unique   concludes
an  LDP proof.
  Liptser \cite{Lip96} identifies
 the large deviation rate function by evaluating limits of
the  probabilities that the process in question resides in small balls.
We use   a different device.
The general idea is to consider a  characterisation
of  stochastic processes that admits taking the large
 deviation limit.
 Such a characterisation may be the
 property that a certain process be a martingale,
 it may also arise out of the description
of the process dynamic.
The large deviation rate function is  identified by 
the limiting relation,  cf.
Puhalskii \cite{Puh97,Puh01,Puhrgr}, Puhalskii and Vladimirov \cite{PuhVla07}.
In this paper, similarly to Puhalskii \cite{Puh97,Puh01}, the large deviation limit is taken
in an exponential martingale problem that has the 
 distribution of $(X^\epsilon,\mu^\epsilon)$ 
 as a solution. We then undertake a study of
  the limit equation. On the one hand,  
 regularity properties of  solutions
are investigated. That analysis 
has much in common with and
 uses the results and methods of the regularity theory of
 elliptic partial differential equations. 
On the other hand, the domain of the validity of 
 the equation is expanded. Put together, those tools enable us
 to show that  the equation has a unique solution
and to identify that solution. 

The rest of the paper is organised as follows. In
Section~\ref{sec:results}, the main results
are stated, their implications are discussed, and earlier
contributions are given a more detailed consideration. 
 Section \ref{sec:an-outline-approach} outlines the proof
strategy. It is implemented in Sections
\ref{sec:expon-tightn}--\ref{sec:appr-large-devi}. The proof is
completed in Section~\ref{sec:proof-theor-refth}.

We conclude the introduction by giving
 a list of  notation and conventions adopted in the paper. 
The   blackboard bold font is
reserved for  topological spaces, the boldface font is used for
entities associated with probability.
Vectors are treated as column vectors. 
The Euclidean length of a 
vector $x=(x_1,\ldots,x_d)$ from $\R^d$\,, where $d\in\N$\,, 
is denoted by $\abs{x}$\,, 
$^T$ stands for the transpose of a matrix or a vector.
  For a matrix $A$,   $\norm{A}$  denotes
the operator norm and $A^\oplus$ denotes the Moore-Penrose 
pseudoinverse, 
if $A$ is
square then $\text{tr}(A)$  represents the trace of $A$.
Given  a positive semidefinite symmetric matrix
 $\Delta$ and a matrix $z$ of a suitable dimension, which may be a
 vector,
  we define
$  \norm{z}^2_{\Delta}=z^T \Delta\, z$\,.
Derivatives are understood as  weak, or Sobolev, derivatives.
For the definitions and basic properties, the reader is referred
either to 
Adams and Fournier~\cite{MR56:9247} or to 
Gilbarg and Trudinger~\cite{GilTru83}.
For an $\R$-valued  function $f$ on $\R^d$, 
 $D  f$ denotes the gradient and
 $D ^2 f$ 
denotes the Hessian matrix of 
$f$\,.  If $f$ assumes its values in $\R^{d_1}$, then 
$D f$ is the $d\times d_1$-matrix with entries 
$\partial f_i/\partial x_j$ and $\text{div}f$\,
represents the divergence of $f$\,, where $d_1\in\N$\,.
 Subscripts may be added to indicate that
 differentiation is carried out with respect to a
specific variable.
For instance, for an $\R$-valued function $f(t,u,x)$, 
 where
$u=(u_1,\ldots,u_d)$ and $x=(x_1,\ldots,x_{d_1})$\,,
 $D_xf$ and $D_uf$ refer to gradients  in the
third and the second variables, respectively, $D ^2_{uu}f$ is the
matrix with entries $\partial^2 f/\partial u_i\partial u_j$,
$D ^2_{xx}f$ is the
matrix with entries $\partial^2 f/\partial x_i\partial x_j$, and 
$D ^2_{ux}f$ is the
matrix with entries $\partial^2 f/\partial u_i\partial x_j$\,.
The divergence of a matrix is computed rowwise.
If $q>1$, we will denote by $q'$ the conjugate: $q'=q/(q-1)$\,.
  We use  standard notation for
spaces of differentiable functions, e.g., 
$\mathbb{C}^{1,2}(\Upsilon)$ denotes the space
of $\R$-valued 
functions that are  continuously differentiable once in the first variable
and   twice in the second variable over a
domain $\Upsilon$ in $\R^d$,
$\mathbb{C}^{1,2}_{0}(\Upsilon)$ is the subspace of $\mathbb{C}^{1,2}(\Upsilon)$ of functions
of compact support, $\mathbb{C}^{1}_0(\Upsilon)$ is the space of continuously
differentiable functions of compact support, 
and $\mathbb{C}^\infty_0(\Upsilon)$ is the space of infinitely
differentiable functions of compact support. 
  Given a 
measurable function ${c}(x)$
on $\Upsilon$ with  values 
in the set of
 positive definite symmetric $d\times d$-matrices
 and  an $\R_+$-valued measurable function ${m}(x)$   on
$\Upsilon$, we will denote by
$\mathbb{L}^2(\Upsilon,\R^{d},{c}(x),{m}(x)\,dx)$ the Hilbert space of
$\R^d$-valued measurable functions  on $\Upsilon$ with the norm
$\norm{f}_{{c}(\cdot),{m}(\cdot)}
=\bl(\int_{\Upsilon}\norm{f(x)}_{c(x)}^2\,{m}(x)\,dx\br)^{1/2}$\,. 
If  ${c}(x)$ is the identity matrix, the notation
 will be shortened to $\mathbb{L}^2(\Upsilon,\R^{d},m(x)\,dx)$ and to 
 $\mathbb{L}^2(\Upsilon,\R^{d})$ if, in addition,
${m}(x)=1$\,. Spaces
 $\mathbb{L}^2(\Upsilon,{m}(x)\,dx)$
and $\mathbb{L}^2(\Upsilon)$ are defined similarly and consist of
$\R$-valued functions. Space
 $\mathbb{L}^2(\Upsilon,\R^d,\mu(dx))$ is defined via integration with
 respect to measure $\mu$\,.
Also,  standard notation for Sobolev spaces
is adhered to, e.g., $\mathbb{W}^{1,2}(\Upsilon)$ is the Hilbert 
space of  $\R$-valued functions $f$ that
possess the first
Sobolev derivatives with the norm $\norm{f}_{\mathbb{W}^{1,2}(\Upsilon)}
=\norm{f}_{\mathbb L^2(\Upsilon)}+\norm{Df}_{\mathbb
  L^2(\Upsilon,\R^d)}\,.$
The local version of a function space, e.g.,
$\mathbb{W}^{1,2}_{\text{loc}}(\Upsilon)$\,, consists
of functions whose  products 
with  arbitrary $\mathbb{C}_0^\infty$-functions belong
to that space, i.e.,  $\mathbb{W}^{1,2}(\Upsilon)$
in this case, and  is endowed with the weakest topology
under which the mappings that associate with  functions such products
are  continuous.
We let
$\mathbb{W}^{1,2}(\Upsilon, m(x)\,dx)$ denote the set of functions 
$f\in\mathbb{W}^{1,1}_{\text{loc}}(\Upsilon)$ such that
$f\in \mathbb{L}^2(\Upsilon, m(x)\,dx)$ and
$Df\in \mathbb{L}^2(\Upsilon,\R^d, m(x)\,dx)$ equipped with the norm 
$\norm{f}_{\mathbb W^{1,2}(\Upsilon,\, m(x)\,dx)}
=\norm{f}_{\mathbb L^2(\Upsilon, m(x)\,dx)}+\norm{Df}_{\mathbb L^2(\Upsilon,\R^d, m(x)\,dx)}$ and let 
$\mathbb H^{1,2}(\Upsilon, m(x)\,dx)$ denote the closure of the set
$\mathbb C^\infty(\Upsilon)\cap \mathbb{W}^{1,2}(\Upsilon, m(x)\,dx)$ 
in $\mathbb
  W^{1,2}(\Upsilon,\, m(x)\,dx)$\,. 
Spaces $\mathbb W^{1,2}(\Upsilon, c(x),\, m(x)\,dx)$ and $\mathbb
H^{1,2}(\Upsilon, c(x),\, m(x)\,dx)$ are defined similarly.
We let
$\mathbb{L}_0^{1,2}(\Upsilon,\R^d, c(x), m(x)\,dx)$ represent
 the closure  of the set of
the gradients of functions from 
 $\mathbb{C}_0^\infty(\Upsilon)$
 in
$\mathbb{L}^2(\Upsilon,\R^d, c(x), m(x)\,dx)$\,.
The space of  continuous functions on $\R_+$ with values in  metric
space $\mathbb{S}$ is denoted
by $\mathbb{C}(\R_+,\mathbb{S})$\,. 
It is endowed with the compact-open
topology.
If 
function $X=(X_s,\,s\in\R_+)$ from $\bC(\R_+,\R^d)$ is absolutely
continuous w.r.t. Lebesgue measure, $\dot X_s$ denotes its derivative at $s$\,.
We let $\mathbb{M}(\R^d)$ (respectively, $\mathbb{M}_1(\R^d)$)
represent the set of  finite (respectively, probability) measures 
on $ \R^d$ endowed with the weak topology, see, e.g., Tops\oe \cite{Top};
  $\mathbb{P}(\R^d)$ denotes the set of probability densities $m(x)$
on $\R^d$ such that
 $m\in\mathbb{W}^{1,1}_{\text{loc}}(\R^d)$
 and  $\sqrt{m}\in\mathbb{W}^{1,2}(\R^d)$\,.
Topological spaces are equipped with Borel $\sigma$-algebras, 
except for $\R_+$ which is equipped with the Lebesgue $\sigma$-algebra,
products of topological spaces are equipped with product topologies, and
products of measurable spaces are equipped with product $\sigma$-algebras.
The ``overbar'' notation is reserved for the closures of sets,
 $\ind_\Gamma$
denotes the indicator function of a set $\Gamma$\,,
$\lfloor a\rfloor$ stands for the integer part of  real number $a$\,,
 $a\wedge
b=\min(a,b)$\,, 
$a\vee b=\max(a,b)$\,, and $a^+=a\vee 0$\,.
Notation $U\subset\subset V$\,, where $U$ and $V$ are open subsets of
$\R^d$, 
is to signify that the closure of $U$ is a compact subset of $V$\,.
Throughout, the conventions that
$\inf_{\emptyset}=\infty$ and $0/0=0$ are adopted.
The terms ``absolutely continuous'', ``a.e.'', ``almost all''
refer to Lebesgue measure unless specified otherwise.
All suprema in the time variable are understood as essential suprema
with respect to Lebesgue measure.

We say that a net of probability measures $\mathbf{P}^\epsilon$, where
$\epsilon>0$, defined on  metric space $\mathbb{S}$ obeys the large deviation
principle (LDP) 
with a (tight) large deviation (rate) function $\mathbf{I}$ for rate
$1/\epsilon$ as $\epsilon\to0$ if $\mathbf{I}$ is a 
function from $\mathbb{S}$ to $[0,\infty]$ such that the sets 
$\{z\in \mathbb{S}:\,\mathbf{I}(z)\le \delta\}$ 
are compact for all $\delta\in\R_+$,
$  \liminf_{\epsilon\to0}\epsilon\ln \mathbf{P}^\epsilon(G)
\ge -\inf_{z\in G}\mathbf{I}(z)$ for all open sets $G\subset \mathbb{S}$\,,
and $  \limsup_{\epsilon\to0}\epsilon\ln \mathbf{P}^\epsilon(F)
\le -\inf_{z\in F}\mathbf{I}(z)$
 for all closed sets $F\subset \mathbb{S}$.
We say that the net $\mathbf{P}^\epsilon$ is exponentially tight for
rate $1/\epsilon$ if 
$\inf_{K}\limsup_{\epsilon\to0}\bl(\mathbf{P}^\epsilon(\mathbb{S}\setminus
K)\br)^{\epsilon}=0$ where $K$ ranges over the collection 
of compact subsets of
$\mathbb S$\,.

\section{Main results}
\label{sec:results}
We will  consider a    time nonhomogeneous version of \eqref{eq:67} 
in which the coefficients may depend on $\epsilon$ as well: 
\begin{subequations}
  \begin{align}
     \label{eq:1}
 dX^\epsilon_t&=A^\epsilon_t(X^\epsilon_t,x^\epsilon_{t})\,dt
+\sqrt{\epsilon}\,B^\epsilon_t(X^\epsilon_t,x^\epsilon_{t})\,dW^\epsilon_t,\\
  \label{eq:5}
dx^\epsilon_t&=\frac{1}{\epsilon}\,a^\epsilon_t(X^\epsilon_t,x^\epsilon_t)\,dt+
\frac{1}{\sqrt{\epsilon}}\,b^\epsilon_t(X^\epsilon_t,x^\epsilon_t)\,  
dW^\epsilon_t\,.
\end{align}
\end{subequations}
As above,
 ${A^\epsilon_t}(u,x)$ is an $n$-vector, 
$B^\epsilon_t(u,x)$ is an $n\times k$-matrix, $a^\epsilon_t(u,x)$ is an
$l$-vector,   $b^\epsilon_t(u,x)$ is an $l\times k$-matrix, and 
$W^\epsilon=(W^\epsilon_t,
\,t\in\R_+)$ 
is an $\R^k$-valued standard
Wiener process.
The stochastic process $X^\epsilon=(X^\epsilon_t,\,t\in\R_+)$ takes
values in $\R^n$ and the stochastic process $x^\epsilon=
(x^\epsilon_t,\,t\in\R_+)$ takes
values in $\R^l$.
We assume that the functions $A^\epsilon_t(u,x)$, $a^\epsilon_t(u,x)$,
$B^\epsilon_t(u,x)$, and $b^\epsilon_t(u,x)$ are  measurable and
locally bounded in
$(t,u,x)$ and are such that 
  the
equations (\ref{eq:1}) and (\ref{eq:5}) admit 
 a  weak solution $(X^\epsilon,x^\epsilon)$ with trajectories in
 $\bC(\R_+,\R^n\times \R^l)$ for every
initial condition $(X^\epsilon_0,x^\epsilon_0)$\,.
More specifically, we assume that there exists  a complete probability space
$(\Omega^\epsilon,\mathcal{F}^\epsilon,\mathbf{P}^\epsilon)$ with filtration
$\mathbf{F}^\epsilon=(\mathcal{F}^\epsilon_t,\,t\in\R_+)$ such that $(W^\epsilon_t,\,t\in\R_+)$ is
a Wiener process relative to $\mathbf{F}^\epsilon$, the processes
$X^\epsilon=(X^\epsilon_t,\,t\in\R_+)$ and 
$x^\epsilon=(x^\epsilon_t,\,t\in\R_+)$ are $\mathbf{F}^\epsilon$-adapted, have
continuous trajectories, and the relations 
(\ref{eq:1}) and (\ref{eq:5}) hold for all $t\in\R_+$
$\mathbf{P}^\epsilon$-a.s.
(To ensure uniqueness which we do not assume apriori, 
one may require, in addition to the above hypotheses,
 that the coefficients be
Lipschitz continuous.) For background information, see 
  Ethier and Kurtz \cite{EthKur86}, Ikeda and Watanabe \cite{IkeWat}, 
Stroock and Varadhan \cite{StrVar79}. We note that since the
dimensions $n$, $k$, and $l$ are arbitrary, the assumption that
both $X^\epsilon$ and $x^\epsilon$ are driven by the same Wiener
process does not constitute a loss of generality.

Let us denote $C^\epsilon_t(u,x)=B^\epsilon_t(u,x)B^\epsilon_t(u,x)^T$ and 
$c^\epsilon_t(u,x)=b^\epsilon_t(u,x)b^\epsilon_t(u,x)^T$\,.
We  introduce the boundedness and growth conditions that for all 
$N>0$ and $t>0$
\begin{subequations}
  \begin{align}
\label{eq:10}
\limsup_{\epsilon\to0}\sup_{s\in[0,t]}
\sup_{x\in\R^l}\sup_{u\in\R^n:\,\abs{u}\le
  N}\norm{c^\epsilon_s(u,x)}&<\infty\,,\\
\label{eq:10a}
\limsup_{\epsilon\to0}\sup_{s\in[0,t]}
\sup_{x\in\R^l}\sup_{u\in\R^n:\,\abs{u}\le N}\abs{A^\epsilon_s(u,x)}&<\infty\,,\\
    \label{eq:7}
\limsup_{\epsilon\to0}\sup_{s\in[0,t]}
\sup_{x\in\R^l}  \sup_{u\in  \R^n}\frac{u^T
  A^\epsilon_s(u,x)}{1+|u|^2}&<\infty,
\intertext{and}
\limsup_{\epsilon\to0}\sup_{s\in[0,t]}
\sup_{x\in\R^l}\sup_{u\in  \R^n}\frac{\norm{C^\epsilon_s(u,x)}}{1+|u|^2}&<\infty
  \label{eq:8}\,.
\end{align}
\end{subequations}

We also assume as  given ''limit coefficients'':
 ${A_t}(u,x)$ is an $n$-vector, 
$B_t(u,x)$ is an $n\times k$-matrix, $a_t(u,x)$ is an
$l$-vector, and  $b_t(u,x)$ is an $l\times k$-matrix\,.
Let 
$ C_t(u,x)=B_t(u,x)B_t(u,x)^T$ and
$c_t(u,x)=b_t(u,x)b_t(u,x)^T$\,.
The following regularity properties 
will be needed.
\begin{condition}
\label{con:coeff}
The functions  ${A_t}(u,x)$\,, 
$B_t(u,x)$\,, and   $b_t(u,x)$ are measurable and are bounded
locally in  $(t,u)$ and globally in $x$ and are
 continuous in $(u,x)$\,,
the function $a_t(u,x)$ is measurable and 
locally bounded in $(t,u,x)$
and  is Lipschitz continuous in $x$
locally uniformly in $(t,u)$\,, 
the functions $a_t(u,x)$ and
   $c_t(u,x)$ are continuous in $u$ locally uniformly in $t$ and
   uniformly in $x$\,, 
$c_t(u,x)$ is of class $\mathbb{C}^1$ in $x$\,, with the first
partial derivatives  
being bounded and
Lipschitz continuous in $x$ locally uniformly in $(t,u)$\,,
 and $\text{div}_x\,c_t(u,x)$ 
is continuous in
$(u,x)$\,.

\end{condition}
Another set of regularity requirements is furnished by the next condition.
We introduce
\begin{equation}
  \label{eq:40}
  G_t(u,x)=B_t(u,x)b_t(u,x)^T\,.
\end{equation} 
\begin{condition}
  \label{con:positive}
The matrix $c_t(u,x)$   is  positive definite uniformly in $x$ and
locally uniformly in $(t,u)$\,.  
Either  
$C_t(u,x)=0$ for all $(t,u,x)$
and $A_t(u,x)$ is locally Lipschitz continuous in  $u$ locally uniformly in
$t$ and uniformly in $x$\,, or the matrix 
$C_t(u,x)-G_t(u,x)c_t(u,x)^{-1}G_t(u,x)^T$ is positive definite
 uniformly in $x$ and
locally uniformly in $(t,u)$\,.

\end{condition}

Finally, certain  stability properties will be required:
 for all $N>0$ and $t>0$\,,
 \begin{subequations}
   \begin{align}
  \label{eq:9}
\lim_{M\to\infty}
\limsup_{\epsilon\to0}
\sup_{s\in[0,t]}\sup_{x\in\R^l:\,\abs{x}\ge M}\sup_{u\in\R^n:\,\abs{u}\le
  N}a^\epsilon_s(u,x)^T\frac{x}{\abs{x}}
=-\infty\intertext{ and }
  \label{eq:116}
  \lim_{\abs{x}\to\infty}
\sup_{s\in[0,t]}\sup_{u\in\R^n:\,\abs{u}\le
  N}a_s(u,x)^T\frac{x}{\abs{x}}
=-\infty\,.
\end{align}
 \end{subequations}

 Let  $\bC_\uparrow(\R_+,\mathbb{M}(\R^l))$ represent
the subset of $\bC(\R_+,\mathbb{M}(\R^l))$ of 
functions $\mu=(\mu_t,\,t\in\R_+)$ such  that $\mu_t-\mu_s$ is an
element of $\mathbb{M}(\R^l)$ for $t\ge s$ and $\mu_t(\R^l)=t$\,.
It is endowed with the subspace topology and is a complete separable 
metric space, being
 closed in
$\bC(\R_+,\mathbb{M}(\R^l))$\,.
The   stochastic process
  $\mu^\epsilon=(\mu^\epsilon_t\,,t\in\R_+)$\,,
where \begin{equation*}
\mu^\epsilon_t(\Theta)=\int_0^t
\ind_{ \Theta}(x^\epsilon_s) \,ds\,,
\end{equation*}
for $\Theta\in\mathcal{B}(\R^l)$\,,
is a random
element of $\bC_\uparrow(\R_+,\mathbb{M}(\R^l))$\,.
We will regard $(X^\epsilon,\mu^\epsilon)$ as a random element of 
$\mathbb{C}(\R_+,\R^n)\times \bC_\uparrow(\R_+,\mathbb{M}(\R^l))$\,.
It is worth noting that 
the elements of $\bC_\uparrow(\R_+,\mathbb{M}(\R^l))$ can be also
regarded as $\sigma$-finite measures on $\R_+\times \R^l$\,.
We will then use notation $\mu(dt,dx)$ for $\mu$\,.

Let $\Gamma$ represent the  set of $(X,\mu)$ such that
the  function 
$X=(X_s,\,s\in\R_+)$ from $\bC(\R_+,\R^n)$ 
 is absolutely continuous w.r.t.  Lebesgue measure on $\R_+$ and
 function $\mu=(\mu_s,\,s\in\R_+)$ from
$\bC_\uparrow(\R_+,\mathbb{M}(\R^l))$\,,
when considered as a measure on $\R_+\times\R^l$\,, is
absolutely continuous w.r.t.  Lebesgue measure on $\R_+\times\R^l$\,, i.e.,
 $\mu(ds,dx)=m_s(x)\,dx\,ds$\,, where
   $m_s(x)$\,,   as a function of $x$\,, belongs to 
$\mathbb{P}(\R^l)$
  for almost all $s$\,. Given $(X,\mu)\in\Gamma$\,, we define
\begin{multline}
  \label{eq:77}
  \mathbf{I}'(X,\mu)=
\int_0^\infty
\sup_{\lambda\in\R^n}\Bl(
\lambda^T\bl(\dot{X}_s-
\int_{\R^l} A_s(X_s,x)\,m_s(x)\,dx\br)
-\frac{1}{2}\, \norm{\lambda}_{\int_{\R^l}C_s(X_s,x)\,m_s(x)\,dx}^2\\
+\sup_{h\in \mathbb{C}_0^1(\R^l)}\int_{\R^l}\Bl( D  h(x)^T \bl(
\frac{1}{2}\,\text{div}_x\,\bl(c_s(X_s,x)m_s(x)\br)-\bl(
a_s(X_s,x)+  G_s(X_s,x)^T\lambda\br)m_s(x)
\br)
\\-\frac{1}{2}\,
\norm{D
  h(x)}_{c_s(X_s,x)}^2\,m_s(x)\Br)\,dx\Br)\,ds\,.  \end{multline}
We let   $\mathbf{I}'(X,\mu)=\infty$ if $(X,\mu)\not\in\Gamma$\,.
It follows, on letting $\lambda=0$\,, that if 
$\mathbf{I}'(X,\mu)<\infty$ then, for $t\in\R_+$\,,
\begin{multline*}
   \int_0^t  \sup_{h\in \mathbb{C}_0^1(\R^l)}\int_{\R^l}\bl( D  h(x)^T \bl(
\frac{1}{2}\,\text{div}_x\,\bl(c_s(X_s,x)m_s(x)\br)
-a_s(X_s,x)m_s(x)
\br)
\\-\frac{1}{2}\,
\norm{D  h(x)}_{c_s(X_s,x)}^2m_s(x)\br)\,dx\,ds<\infty
\end{multline*}
so that, thanks to $G_s(X_s,x)$ being bounded, for all $\lambda\in\R^n$\,, 
\begin{multline}
  \label{eq:16}
  \int_0^t  \sup_{h\in \mathbb{C}_0^1(\R^l)}\int_{\R^l}\bl( D  h(x)^T \bl(
\frac{1}{2}\,\text{div}_x\,\bl(c_s(X_s,x)m_s(x)\br)
-\bl(a_s(X_s,x)+  G_s(X_s,x)^T\lambda\br)m_s(x)
\br)
\\-\frac{1}{2}\,
\norm{D  h(x)}_{c_s(X_s,x)}^2m_s(x)\br)\,dx\,ds<\infty\,.
\end{multline}
We introduce the following convergence condition.

\begin{condition}
  \label{con:int_conv}
If  $\mathbf I'(X,\mu)<\infty$\,, then there exists 
a nonincreasing 
$[0,1]$-valued $\mathbb C^1_0(\R_+)$-function $\eta(y)$
 that equals $1$ for $y\in[0,1]$ and equals $0$ for 
$y\ge 2$\,,  such that
\begin{equation}
  \label{eq:118}
\int_1^2 \frac{\abs{D\eta(y)}^2}{1-\eta(y)}\,dy<\infty\,,
\end{equation}
and, for  arbitrary $t\in\R_+$ and $\lambda\in\R^n$\,,
\begin{multline}
  \label{eq:125}
\lim_{r\to\infty}
  \int_0^t  \sup_{h\in \mathbb{C}_0^1(\R^l)}\int_{\R^l}\bl( D  h(x)^T \bl(
\frac{1}{2}\,\text{div}_x\,\bl(c_s(X_s,x)m_s(x)
\br)
-\bl(a_s(X_s,x)+  G_s(X_s,x)^T\lambda\br)m_s(x)
\br)\\
-\frac{1}{2}\,
\norm{D
  h(x)}_{c_s(X_s,x)}^2m_s(x)\br)\eta\bl(\frac{\abs{x}}{r}\br)^2\,dx\,ds\\
=  \int_0^t  \sup_{h\in \mathbb{C}_0^1(\R^l)}\int_{\R^l}\bl( D  h(x)^T \bl(
\frac{1}{2}\,\text{div}_x\,\bl(c_s(X_s,x)m_s(x)
\br)
-\bl(a_s(X_s,x)+  G_s(X_s,x)^T\lambda\br)m_s(x)\br)
\\-\frac{1}{2}\,
\norm{D
  h(x)}_{c_s(X_s,x)}^2m_s(x)\br)\,dx\,ds\,,
\end{multline}
where $m_s(x)=\mu(ds,dx)/(ds\,dx)$\,.
\end{condition}
We note that \eqref{eq:118} is satisfied if $\eta(y)=1-e^{-1/(y-1)}$ in a
right neighbourhood of $1$\,.
The next lemma, whose proof is relegated to the appendix,
 furnishes a way of verifying Condition \ref{con:int_conv}.
\begin{lemma}
  \label{le:gradient}
Suppose that    Conditons \ref{con:coeff} and \ref{con:positive} hold
and that  $\mathbf I'(X,\mu)<\infty$\,.
 If 
\begin{equation}
  \label{eq:19}
\int_0^t
\int_{\R^l} \abs{a_s(X_s,x)}^2\,m_s(x)\,dx\,ds<\infty\,,  
\end{equation}
then
Condition \ref{con:int_conv} holds.
 
 If   either 
  \begin{equation}
    \label{eq:193}
  \limsup_{\abs{x}\to\infty}
\sup_{s\in[0,t]}a_s(X_s,x)^T\frac{x}{\abs{x}^2}
<0
  \end{equation}
 or there exists real-valued function $\hat a_s(x)$ which belongs to
 $\mathbb
W^{1,q}_{\text{loc}}(\R^l)$  in $x$\,, 
where $q>2$ and $q\ge l$\,, such that
  \begin{equation}
    \label{eq:30}
    c_s(X_s,x)^{-1}\bl(a_s(X_s,x)-
\frac{1}{2}\,\text{div}_x\,c_s(X_s,x)\br)
=D_x \hat a_s(x)\,,  
\end{equation}
 then \eqref{eq:19} holds.
\end{lemma}
We state the main result.
\begin{theorem}
  \label{the:ldp}
Let  \eqref{eq:10}--\eqref{eq:8}, \eqref{eq:9}, 
 \eqref{eq:116}, and
Conditions \ref{con:coeff}, \ref{con:positive}, and 
\ref{con:int_conv}  hold.
If the net $X^\epsilon_0$ obeys the LDP in
$\R^n$  with large deviation  function
$\mathbf{I}_0$ for rate $1/\epsilon$ as $\epsilon\to0$,
 the net $x^\epsilon_0$ is exponentially tight in
$\R^l$ for rate $1/\epsilon$ as $\epsilon\to0$\,, and,
 for all $t>0$ and $N>0$\,,  the convergences
\begin{multline}
  \label{eq:14}
  \lim_{\epsilon\to0}\sup_{s\in[0,t]}
\sup_{x\in\R^l:\,\abs{x}\le N}\sup_{u\in\R^n:\,\abs{u}\le N}
\bl(\abs{A^\epsilon_s(u,x)-A_s(u,x)}+\abs{a^\epsilon_s(u,x)-a_s(u,x)}\\
+\norm{B^\epsilon_s(u,x)-B_s(u,x)}+\norm{b^\epsilon_s(u,x)-b_s(u,x)}\br)=0
\end{multline}
hold, then
the net $(X^\epsilon,\mu^\epsilon)$ obeys the LDP in $\bC(\R_+,
\R^n)\times\mathbb{C}_\uparrow(\R_+,\mathbb{M}(\R^l))$ for rate
$1/\epsilon$ as $\epsilon\to0$ with large deviation  function
$\mathbf{I}$ defined as follows:
\begin{equation*}
\mathbf{I}(X,\mu)=
\begin{cases}
\mathbf{I}_0(X_0)+
  \mathbf{I}'(X,\mu),&\text{ if }(X,\mu)\in\Gamma,\\
\infty,&\text{ otherwise.}
\end{cases}
\end{equation*}

\end{theorem}
\begin{remark}
  Condition \ref{con:int_conv} may be superfluous 
 as far as  the validity of 
  Theorem \ref{the:ldp} is concerned.
 It is used at the final stage of
the proof only, see  Theorem~\ref{the:appr}.
\end{remark}
\begin{remark}
  By Lemma~\ref{le:zero}, $\mathbf{I}(X,\mu)=0$ provided that
a.e.
\begin{align*}
\dot{X}_s=\int_{\R^l}A_s(X_s,x)\,m_s(x)\,dx\intertext{and }
  \int_{\R^l}\bl(\frac{1}{2}\,\text{tr}\,(c_s(X_s,x)\,  D^2 p(x))
+D p(x)^T a_s(X_s,x)\br)\, m_s(x)\,dx  =0,
\end{align*}
where $X_0$ satisfies the equality  $\mathbf{I}_0(X_0)=0$ and the  
latter equation 
 holds for all  $p\in \mathbb{C}_0^\infty(\R^l)$\,. Consequently,
$m_s(\cdot)$ is the invariant density of the diffusion process with the
infinitesimal 
drift $a_s(X_s,\cdot)$ and diffusion matrix $c_s(X_s,\cdot)$\,.
\end{remark}
\begin{remark}
Conditions \ref{con:coeff}
and \eqref{eq:14}
imply that
  \begin{equation}
      \label{eq:93}\limsup_{\epsilon\to0}\sup_{s\in[0,t]}
\sup_{x\in\R^l:\,\abs{x}\le N}\sup_{u\in\R^n:\,\abs{u}\le
  N}\abs{a^\epsilon_s(u,x)}
<\infty\,.
\end{equation}
Conditions \eqref{eq:10}--\eqref{eq:8} and \eqref{eq:14}
  imply that
  \begin{subequations}
  \begin{align}
      \label{eq:63}
\sup_{s\in[0,t]}
\sup_{x\in\R^l}\sup_{u\in\R^n:\,\abs{u}\le
  N}\norm{c_s(u,x)}&<\infty\,,\\\label{eq:63a}
\sup_{s\in[0,t]}
\sup_{x\in\R^l}\sup_{u\in\R^n:\,\abs{u}\le N}\abs{A_s(u,x)}&<\infty\,,\\
    \label{eq:122a}
\sup_{s\in[0,t]}
\sup_{x\in\R^l}  \sup_{u\in  \R^n}\frac{u^T
  A_s(u,x)}{1+|u|^2}&<\infty,
\intertext{and}
  \label{eq:82}
\sup_{s\in[0,t]}
\sup_{x\in\R^l}\sup_{u\in  \R^n}\frac{\norm{C_s(u,x)}}{1+|u|^2}&<\infty\,.
\end{align}
\end{subequations}
In particular, some of the boundedness requirements in Condition
\ref{con:coeff} are consequences of the
other hypotheses of Theorem \ref{the:ldp}.
\end{remark}
\begin{remark}
If $c_t(u,x)$  and $C_t(u,x)$ are  positive definite uniformly in $x$ and
locally uniformly in $(t,u)$\,, then,
since $b_t(u,x)^T
c_t(u,x)^{-1}b_t(u,x)$ is the orthogonal projection operator onto
the range of $b_t(u,x)^T$\,,  the condition that $C_t(u,x)-
G_t(u,x)c_t(u,x)^{-1}G_t(u,x)^T$ be positive definite
 uniformly in $x$ and
locally uniformly in $(t,u)$ is implied by the following 
 ``angle condition'':
 for any bounded region of $(t,u)$\,,
there exists $\ell\in(0,1)$\,
 such that $\abs{y_1^Ty_2}\le \ell\abs{y_1}\abs{y_2}$ for all $y_1$
 and $y_2$ from the ranges of $B_t(u,x)^T$ and $b_t(u,x)^T$\,,
 respectively, where $x$ is arbitrary and $(t,u)$ belongs to the
 region.
To put it another way, the condition requires that 
the angles between the elements of
the range of $B_t(u,x)^T$, on the one hand, and the elements of the
range of $b_t(u,x)^T$, on the other hand,
 be bounded away from zero 
uniformly in $x$ and
locally uniformly in $(t,u)$\,.
It  ensures that  the processes
$X^\epsilon$ and $x^\epsilon$  are ''sufficiently random''
in relation to each other. Under that condition, the
ranges of $B_t(u,x)^T$ and $b_t(u,x)^T$ do not have common  nontrivial
subspaces and  $k\ge n+l$\,. On the other hand, if $\norm{C_t(u,x)}$
is bounded uniformly in $x$ and locally uniformly in $(t,u)$, as is
the case under the hypotheses of Theorem~\ref{the:ldp} according to
\eqref{eq:82}, the converse is also true: if $C_t(u,x)-
G_t(u,x)c_t(u,x)^{-1}G_t(u,x)^T$ is  positive definite
 uniformly in $x$ and
locally uniformly in $(t,u)$\,, then the angle condition holds.
\end{remark}
 The solution of the variational problem in \eqref{eq:77}
plays an important part in the proof below, so we
proceed with describing it.
Let  $c(x)$ represent a measurable function
defined for $x\in \R^{d}$ and taking values 
in the space of
 positive definite symmetric $d\times d$-matrices,
 let  $m(x)$ represent a probability density on
$\R^{d}$\,, 
 and let  $S_i$ represent an open ball of radius $i$ 
centred at the origin in $\R^d$\,,  where $d\in\N$
and $i\in \N$\,.
For  function $\psi_j\in
\mathbb{L}_0^{1,2}(S_j,\R^d,c(x),m(x)\,dx)$ and $ j\ge i$\,, where $j\in\N$\,,
 we let $\pi_{ji}\psi_{j}$ denote the
orthogonal 
projection of the restriction of $\psi_j$ to $S_i$ onto
$\mathbb{L}_0^{1,2}(S_i,\R^d,c(x),m(x)\,dx)$ in
$\mathbb{L}^2(S_i,\R^d,c(x),m(x)\,dx)$\,.
Thus, the function $\pi_{ji}\psi_{j}$ is the element 
of $\mathbb{L}_0^{1,2}(S_i,\R^d,c(x),m(x)\,dx)$ such that
$
\int_{S_i}D p(x)^T
c(x)\pi_{ji}\psi_{j}(x)\,m(x)\,dx=
\int_{S_i}D p(x)^T
c(x)\psi_j(x)\,m(x)\,dx$ for all $p\in
\mathbb{C}_0^\infty(S_i)$\,.
We note that  if the density $m(x)$ is locally bounded away
from zero, then
$\pi_{ji}\psi_{j}$ is a certain gradient:
$\pi_{ji}\psi_{j}=D \chi_{ji}$\,, where $\chi_{ji}$ is the weak
solution of  the
Dirichlet problem $\text{div}\,(c(x)m(x)\,D
\chi_{ji}(x))=
\text{div}\,(c(x)m(x)\psi_j(x))$ for $x\in S_i$ with a zero boundary
condition
(cf. the proof of Lemma~\ref{le:reg}).
Since, for $i\le j\le k$, $\pi_{ji}\circ\pi_{kj}=\pi_{ki}$\,,
the family  $(\mathbb{L}_0^{1,2}(S_j,\R^d,c(x),m(x)\,dx),\,\pi_{ji})$
is a projective (or inverse)  system  in the category of sets.
Given a function $\phi\in
\mathbb{L}^2_{\text{loc}}(\R^d,\R^d,c(x),m(x)\,dx)$\,, 
the orthogonal projections $\phi_i$
of the restrictions of $\phi$ to $S_i$ onto
$\mathbb{L}_0^{1,2}(S_i,\R^d,c(x),m(x)\,dx)$ in
$\mathbb{L}^2(S_i,\R^d,c(x),m(x)\,dx)$ are such  that
$\pi_{ji}\phi_j=\phi_i$\,, provided $i\le j$\,, so they specify an element of the
projective (or inverse) limit of
$(\mathbb{L}_0^{1,2}(S_j,\R^d,c(x),m(x)\,dx),\,\pi_{ji})$\,, which we
denote  by
$\Pi_{c(\cdot), m(\cdot)}\phi $\,. 
On extending the $\phi_i$ by zero outside of $S_i$, 
one has that, for $i\le j$, 
$\norm{\phi_j}^2_{c(\cdot),m(\cdot)}-
\norm{\phi_i}^2_{c(\cdot),m(\cdot)}=
\norm{\phi_j- \phi_i}^2_{c(\cdot),m(\cdot)}
$\,, where the norms are taken in $\R^d$\,.
Hence, if 
$\lim_{i\to\infty}\norm{\phi_i}^2_{c(\cdot),m(\cdot)}<\infty$,
then  the sequence $\phi_i$ converges in
$\mathbb{L}^2(\R^d,\R^d,c(x),m(x)\,dx)$ as $i\to\infty$ and 
 one can identify
$\Pi_{c(\cdot),m(\cdot)}\phi $ with the limit, so
 $\Pi_{c(\cdot),m(\cdot)}\phi\in
\mathbb{L}^2(\R^d,\R^d,c(x),m(x)\,dx)$\,. 
It is uniquely specified by the requirements that $\Pi_{c(\cdot),m(\cdot)}\phi\in
\mathbb{L}_0^{1,2}(\R^d,\R^d,c(x),m(x)\,dx)$ and that,
for all $p\in \mathbb{C}_0^\infty(\R^d)$\,,
\begin{equation}
  \label{eq:186}
\int_{\R^d} Dp(x)^T c(x)\Pi_{c(\cdot),m(\cdot)}\phi(x)\,m(x)\,dx=
    \int_{\R^d} Dp(x)^T c(x) \phi(x)\,m(x)\,dx\,.
\end{equation}
In particular,  if $\phi$ is an element of
$\mathbb{L}^2(\R^d,\R^d,c(x),m(x)\,dx)$\,, then $\Pi_{c(\cdot),m(\cdot)}\phi$
is the orthogonal projection of $\phi$ onto 
$\mathbb{L}_0^{1,2}(\R^d,\R^d,c(x),m(x)\,dx)$\,.
For results on 
the existence and uniqueness for equation \eqref{eq:186} when 
$\Pi_{c(\cdot),m(\cdot)}\phi$ is a gradient, see
 Pardoux and Veretennikov \cite{ParVer01}. 

In the setting of Theorem \ref{the:ldp}, $d=l$\,.
Since, under the hypotheses of Theorem~\ref{the:ldp}, 
 the matrix functions $c_t(u,\cdot)^{-1}G_t(u,\cdot)^T$
 are 
bounded,  so  the matrix function
$\Pi_{c_t(u,\cdot),m(\cdot)}\bl(c_t(u,\cdot)^{-1}G_t(u,\cdot)^T\br)$\,,
whose  columns are the projections of the $n$ columns
of $c_t(u,\cdot)^{-1}G_t(u,\cdot)^T$ onto the space
$\mathbb{L}_0^{1,2}(\R^l,\R^l,c_t(x),m(x)\,dx)$\,, is a well defined element
of the space  $\mathbb{L}^2(\R^l,\R^{l\times n},c_t(u,x),m(x)\,dx)$
 and we denote it 
 by
$\Psi_{t,m(\cdot),u}$\,.
We also  define
\begin{equation}
  \label{eq:88}
    Q_{t,m(\cdot)}(u,x)=C_t(u,x)-\norm{
\Psi_{t,m(\cdot),u}(x)}^2_{ c_t(u,x)}\,.
\end{equation}
The function $Q_{t,m(\cdot)}(u,x)$ assumes values in the space of
 positive semi-definite $n\times n$-matrices.
If the matrix $C_t(u,x)-G_t(u,x)c_t(u,x)^{-1}G_t(u,x)^T$ 
is positive definite uniformly in $x$
and locally uniformly in $(t,u)$\,, then  the matrix 
$\int_{\R^l}Q_{t,m(\cdot)}(u,x)m(x)\,dx$ is positive definite locally
uniformly in $(t,u)$\,.
We also  introduce $\Phi_{t,m(\cdot),u}=\Pi_{c_t(u,\cdot),m(\cdot)}
\bl(c_t(u,\cdot)^{-1}\bl(
a_t(u,\cdot)-\text{div}_x\,c_t(u,\cdot)/2\br)\br)$\,. 
Since $a_t(u,\cdot)$ is not necessarily  square
integrable with respect to $m(x)\,dx$, the function 
$\Pi_{c_t(u,\cdot),m(\cdot)}
\bl(c_t(u,\cdot)^{-1}\bl(
a_t(u,\cdot)-\text{div}_x\,c_t(u,\cdot)/2\br)\br)$, as a function of 
$x\in\R^l$,    might not be an element of
 $\mathbb{L}^2(\R^l,\R^{l},c_t(u,x),m(x)\,dx)$\,.

    For future reference, we note that, according to \eqref{eq:186}, a.e.,
\begin{subequations}
  \begin{align}
\label{eq:79}
\int_{{\R^l}}  D  p(x)^T
c_s(u,x) \Phi_{s,m(\cdot),u}(x)
m(x)\,dx=
\int_{{\R^l}}  D  p(x)^T \bl(
a_s(u,x)-\frac{1}{2}\,\text{div}_x c_s(u,x)\br)
m(x)\,dx
\intertext{and
}
  \label{eq:151}
\int_{\R^l} D p(x)^Tc_s(u,x)\Psi_{s,m(\cdot),u}(x)
\,m(x)
\,dx=
    \int_{\R^l}D p(x)^TG_s(u,x)^T\,m(x)\,dx\,,
\end{align}
\end{subequations}
for all $p\in\mathbb{C}_0^\infty(\R^l)$\,.
In addition, \eqref{eq:151} extends to  $Dp$ representing an
arbitrary element of $\mathbb{L}_0^{1,2}(\R^l,\R^l,c_s(u,x),m(x)\,dx)$\,.
A similar extension property  holds for \eqref{eq:79}, provided
 $a_s(u,\cdot)
\in \mathbb{L}^2(\R^l,\R^l,c_s(u,x),\,m(x)\,dx)$\,.

\begin{proposition}
  \label{cor:ldp} If, under the hypotheses of Theorem \ref{the:ldp},
$\mathbf I'(X,\mu)<\infty$\,, then 
  $\Phi_{s,m_s(\cdot),X_s}$ belongs to the space
$\mathbb{L}^2(\R^l,\R^l,c_s(X_s,x),m_s(x)\,dx)$ 
for almost all $s$ and
\begin{equation*}
\dot{X}_s-\int_{\R^l}A_s(X_s,x)m_s(x)\,dx-
\int_{\R^l}
G_s(X_s,x)\,\bl(\frac{D_xm_s(x)}{2m_s(x)}
-\Phi_{s,m_s(\cdot),X_s}(x)
\br)\,m_s(x)\,dx  
\end{equation*}
belongs to the
range of $\int_{\R^l}Q_{s,m_s(\cdot)}(X_s,x)m_s(x)\,dx$ 
  for almost all $s$\,. Furthermore,  $\Phi_{s,m_s(\cdot),X_s}(x)$
and $\Psi_{s,m_s(\cdot),X_s}(x)$ are
measurable in $(s,x)$ so that
   in the statement of
 Theorem~\ref{the:ldp},
 \begin{multline}
   \label{eq:22}
     \mathbf{I}(X,\mu)=\mathbf{I}_0(X_0)+\frac{1}{2}\,\int_0^\infty
\Bl(
\int_{\R^l}\norm{\frac{D_xm_s(x)}{2m_s(x)}-
\Phi_{s,m_s(\cdot),X_s}(x)}^2_{c_s(X_s,x)}m_s(x)\,dx\\
+\norm{\dot{X}_s-\int_{\R^l}A_s(X_s,x)m_s(x)\,dx\\-
\int_{\R^l}
G_s(X_s,x)\,\bl(\frac{D_xm_s(x)}{2m_s(x)}-
\Phi_{s,m_s(\cdot),X_s}(x)\br)\,
m_s(x)\,dx}^2_{(\int_{\R^l}Q_{s,m_s(\cdot)}(X_s,x)m_s(x)\,dx)^\oplus}\Br)\,ds\,.
\end{multline}
\end{proposition}
\begin{remark}
If $m(x)$ is an element of
$\mathbb{W}^{1,1}_{\text{loc}}(\R^l)$\,,
then  $Dm(x)=0$ for almost all $x$ on the
set where $m(x)=0$\,, so
we will assume throughout that 
 $Dm(x)/m(x)=0$ a.e. on that set.
\end{remark}

\begin{remark}
  The expression on the righthand side of
 \eqref{eq:22} serves both the case where
  $C_t(u,x)=0$ for all $(t,u,x)$ 
and $A_t(u,x)$ is locally Lipschitz continuous in  $u$ locally uniformly in
$t$ and uniformly in $x$, and  the case 
where $C_t(u,x)-G_t(u,x)c_t(u,x)^{-1}G_t(u,x)^T$
 is positive definite uniformly in $x$ and locally
uniformly in $(t,u)$\,.
In each of the two cases, however,
it simplifies as follows. If $C_t(u,x)=0$ for all $(t,u,x)$
and $A_t(u,x)$ is locally Lipschitz continuous in  $u$ locally uniformly in
$t$ and uniformly in $x$\,, then
$Q_{s,m_s(\cdot)}(u,x)=0$ for all $(s,u,x)$\,, so,
in order for $\mathbf{I}(X,\mu)$ to be finite,
it is necessary  that, a.e.,
\begin{equation*}
\dot{X}_s=\int_{\R^l}A_s(X_s,x)m_s(x)\,dx
\end{equation*}
so that 
\begin{equation}
  \label{eq:20}
  \mathbf{I}(X,\mu)  =\mathbf{I}_0(X_0)+\frac{1}{2}\,
\int_0^\infty\int_{\R^l}\norm{
\frac{D_xm_s(x)}{2m_s(x)}-\Phi_{s,m_s(\cdot),X_s}(x)}^2_{c_s(X_s,x)}m_s(x)\,dx\,ds  \,.
\end{equation}
If $C_t(u,x)-G_t(u,x)c_t(u,x)^{-1}G_t(u,x)^T$ is
positive definite 
uniformly in $x$ and locally
uniformly in $(t,u)$\,,
then the matrix 
$\int_{\R^l}Q_{s,m_s(\cdot)}(X_s,x)m_s(x)\,dx$ is invertible, so
its pseudo-inverse is the same as the inverse and the range condition
in the statement of Proposition \ref{cor:ldp} is supefluous.
\end{remark}
\begin{remark}\label{re:square}
By    Theorem \ref{le:vid}, in order  for 
$\mathbf I(X,\mu)$ 
  to be finite it is necessary that
$\int_0^t  \int_{\R^l}\bl(\abs{D_xm_s(x)}^2/m_s(x)\,dx\,ds+
\abs{\Phi_{s,m_s(\cdot),X_s}(x)}^2\br)\,dx\,ds<\infty$ 
 for all
$t\in\R_+$\,.
\end{remark}
\begin{remark}
  The large deviation function in \eqref{eq:22}
 can  also be written  as
\begin{equation*}
  \mathbf{I}(X,\mu)=\mathbf{I}_0(X_0)+
\frac{1}{2}\,\int_0^\infty\int_{\R^l}
\abs{B_s(X_s,x)^T
\hat\lambda_s+b_s(X_s,x)^T\hat
g_s(x)}^2\,m_s(x)\,dx\,ds,
\end{equation*}
where the pair $(\hat\lambda_s,\hat
g_s(x))$ attains
the supremum in \eqref{eq:77},
with $\hat g_s$ assuming the role of $Dh\,$\,: 
\begin{align*}
\hat \lambda_s&=(\int_{\R^l}Q_{s,m_s(\cdot)}(X_s,x)m_s(x)\,dx)^\oplus\bl(
\dot{X}_s-\int_{\R^l}A_s(X_s,x)m_s(x)\,dx\\&-
\int_{\R^l}
G_s(X_s,x)\,\bl(\frac{D_xm_s(x)}{2m_s(x)}-
\Phi_{s,m_s(\cdot),X_s}(x)\br)\,
m_s(x)\,dx\br)
\intertext{ and }
  \hat g_s(x)&=
\frac{D_xm_s(x)}{2m_s(x)}-\Phi_{s,m_s(\cdot),X_s}(x)
-\Psi_{s,m_s(\cdot),X_s}(x)\hat \lambda_s\,.
\end{align*}


\end{remark}


  In the symmetric case where $c_t(u,x)^{-1}\bl(
2a_t(u,x)-\text{div}_x\,c_t(u,x)\br)\linebreak=D_x\hat m_t(u,x)/\hat
m_t(u,x)$\,, 
for some positive probability density $\hat m_t(u,\cdot)$ from
$\mathbb W^{1,1}_{\text{loc}}(\R^l)$\,, one can identify
$\Phi_{t,m_t(\cdot),u}$
with $D_x\hat m_t(u,\cdot)/(2\hat m_t(u,\cdot))$\,.
(We note that
 the diffusion process with the infinitesimal
drift coefficient $a_t(u,\cdot)$ and diffusion matrix $c_t(u,\cdot)$
has $\hat m_t(u,\cdot)$ as an  invariant density.)
One can then write the  large deviation function in \eqref{eq:20} by using
 a Dirichlet form: 
\begin{equation*}
\mathbf{I}(X,\mu)=\mathbf{I}_0(X_0)+
\frac{1}{2}\,
\int_0^\infty\int_{\R^l}
\norm{D_x\sqrt{\frac{m_s(x)}{\hat m_s(X_s,x)}}}^2_{c_s(X_s,x)}
\,\hat m_s(X_s,x)\,dx\,ds  \,,
\end{equation*}
provided  $D_x\hat m_t(u,\cdot)/\hat m_t(u,\cdot)\in
\mathbb{L}^{2}(\R^l,\R^l,c_t(x),m_t(x)\,dx)$\,.

Let us  look at a one-dimensional example:
\begin{align*}
   dX^\epsilon_t&=A_t(X^\epsilon_t,x^\epsilon_{t})\,dt
+\sqrt{\epsilon}\, 
B_t(X^\epsilon_t,x^\epsilon_{t})\,dW^\epsilon_{1,t}\,,\\
dx^\epsilon_t&=\frac{1}{\epsilon}\,a_t(X^\epsilon_t,x^\epsilon_t)\,dt+
\frac{1}{\sqrt{\epsilon}}\, b_t(X^\epsilon_t,x^\epsilon_t)\,  
dW^\epsilon_{2,t}\,,
\end{align*}
where all  coefficients are scalars and $W^\epsilon_{1,t}$ and 
$W^\epsilon_{2,t}$ are  one-dimensional
standard Wiener processes. Assuming that 
$\mathbf{E}^\epsilon(W^\epsilon_{1,t}W^\epsilon_{2,t})= \rho t$\,, where
 $\abs{\rho}<1$\,,
this setup can be cast as (\ref{eq:1}) and (\ref{eq:5}) with
$W_t^\epsilon=(W^\epsilon_{1,t},W^\epsilon_{3,t})^T$\,,
$B^\epsilon_t(u,x)=( B_t(u,x),0)$\,,
and $b^\epsilon_t(u,x)=(\rho \, b_t(u,x),
\sqrt{1-\rho^2}\, b_t(u,x))$\,, where 
$W^\epsilon_{3,t}$ represents a standard one-dimensional Wiener process that
is independent of $W^\epsilon_{1,t}$\,.
If $ B_t(u,x)$ is bounded away from zero,
the large
deviation function in \eqref{eq:22} takes the form
\begin{multline*}
  \mathbf{I}(X,\mu)
=\mathbf{I}_0(X_0)+\int_0^\infty
\Bl(\frac{1}{8}\,
\int_{\R}
\abs{\frac{D_xm_s(x)}{m_s(x)}-\frac{
D_x\hat{m}_s(X_s,x)}{\hat m_s(X_s,x)}
}^2\, b_s(X_s,x)^2\,m_s(x)\,dx\\
+\frac{1}{2(1-\rho^2)}\,\dfrac{1}{\int_{\R} B_s(X_s,x)^2m_s(x)\,dx}
\abs{\dot{X}_s-\int_{\R}A_s(X_s,x)m_s(x)\,dx\\-
\frac{\rho}{2}\,
\,\int_{\R}
  B_s(X_s,x)\, b_s(X_s,x)\,\bl(
\frac{D_xm_s(x)}{m_s(x)}-
\frac{D_x\hat{m}_s(X_s,x)}{\hat m_s(X_s,x)}\br)\,
m_s(x)\,dx}^2
\Br)\,ds\,.
\end{multline*}
If $ B_t(u,x)=0$\,, then according to \eqref{eq:20},
\begin{align*}
\mathbf{I}(X,\mu)  =\mathbf{I}_0(X_0)+
\frac{1}{8}\,\int_0^\infty
\int_{\R}
\abs{\frac{D_xm_s(x)}{m_s(x)}
-\frac{D_x\hat{m}_s(X_s,x)}{\hat m_s(X_s,x)}
}^2\, b_s(X_s,x)^2\,&m_s(x)\,dx\,ds\,,
\end{align*}
provided  $\dot{X}_s=\int_{\R^l}
A_s(X_s,x)\,m_s(x)\,dx$  a.e. 
For the special case that $A_s(u,x)$ and
$ B_s(u,x)$ do not depend on $s$, $a_s(u,x)$ and $ b_s(u,x)$ do
not depend on either $s$ or $u$, and $\rho=0$\,, 
this large deviation function appears in 
Liptser \cite{Lip96}.

We now project to obtain an LDP for $X^\epsilon$\,.
The device of Lemma \ref{le:sup} and the minimax theorem,
 see, e.g., 
Theorem 7 on p.319  in Aubin and Ekeland
\cite{AubEke84}, yield the following expression for
$\inf_\mu\mathbf{I}(X,\mu)$\,. 
\begin{corollary}
  \label{the:ave}
Under the hypotheses of Theorem 
\ref{the:ldp},
the net $X^\epsilon$ obeys the LDP in $\bC(\R_+,
\R^n)$ for rate
$1/\epsilon$ as $\epsilon\to0$ with large deviation  function
$\mathbf{I}^X$ defined as follows.
If  function 
$X=(X_s,\,s\in\R_+)$ from $\bC(\R_+,\R^n)$ 
 is absolutely continuous w.r.t.  Lebesgue measure on $\R_+$\,,   then
\begin{multline*}
  \mathbf{I}^X(X)=\mathbf{I}_0(X_0)+
\int_0^\infty
\sup_{\lambda\in\R^n}\Bl(
\lambda^T\dot{X}_s-\sup_{m\in \mathbb{P}(\R^l)}\Bl(
\lambda^T\int_{\R^l} A_s(X_s,x)\,m(x)\,dx
+\frac{1}{2}\, \norm{\lambda}_{\int_{\R^l}C_s(X_s,x)\,m(x)\,dx}^2
\\-\sup_{h\in \mathbb{C}_0^1(\R^l)}\int_{\R^l}\Bl( D  h(x)^T \bl(
\frac{1}{2}\,\text{div}_x\,\bl(c_s(X_s,x)m(x)\br)
-\bl(
a_s(X_s,x)+  G_s(X_s,x)^T\lambda\br)m(x)
\br)\\
-\frac{1}{2}\,
\norm{D  h(x)}_{c_s(X_s,x)}^2\,m(x)\Br)\,dx\Br)\Br)\,ds\,.  \end{multline*}
Otherwise, $\mathbf{I}^X(X)=\infty$.
\end{corollary}
If $X^\epsilon$ is decoupled from $x^\epsilon$\,, i.e., $A_t(u,x)$
and $B_t(u,x)$ do not depend on $x$\,, then Corollary \ref{cor:ldp}
yields the  LDP  for  It\^o processes with  small diffusions, cf. Freidlin
and Wentzell \cite{wf2}: 
with
$A_s(u,x)=A_s(u)$\,, $B_s(u,x)=B_s(u)$\,, and $C_s(u)=B_s(u)B_s(u)^T$\,,
\begin{equation*}
  \mathbf{I}^X(X)=\mathbf{I}_0(X_0)+
\int_0^\infty
\frac{1}{2}\,\norm{\dot{X}_s- A_s(X_s)}^2_{C_s(X_s)^\oplus}\,ds\,,
\end{equation*}
provided $\dot{X}_s- A_s(X_s)$ belongs to the range of $C_s(X_s)$
a.e. and $\mathbf{I}^X(X)=\infty$\,, otherwise.

If one projects the LDP of Theorem~\ref{the:ldp} on the second
variable, then an LDP for $\mu^\epsilon$ is obtained. In particular,
if $x^\epsilon$ is decoupled from $X^\epsilon$ so that
 $a_t(u,x)$ and $b_t(u,x)$ do not depend on
$u$\,, we have the following results
on the large deviations of the empirical processes and empirical
measures of 
diffusion processes.
\begin{corollary}
  \label{cor:emp}
Suppose that
\begin{equation*}
  d\tilde x^\epsilon_t=\frac{1}{\epsilon}\,\tilde 
a^\epsilon_t(\tilde x^\epsilon_t)\,dt+
\frac{1}{\sqrt{\epsilon}}\,\tilde b^\epsilon_t(\tilde x^\epsilon_t)\,  
d\tilde W^\epsilon_t\,,
\end{equation*}
where $\tilde x^\epsilon_t\in \R^l$\,,
$\tilde a^\epsilon_t(x)\in \R^l$\,,  $\tilde
b^\epsilon_t(x)\in\R^{l\times k}$\,, and $\tilde
W^\epsilon_t\in\R^k$\,,
with the coefficients being locally bounded.
Assume that, for all $t\in\R_+$\,,
\begin{align*}
  \limsup_{\epsilon\to0}\sup_{s\in[0,t]}
\sup_{x\in\R^l}\norm{\tilde b^\epsilon_s(x)\tilde b^\epsilon_s(x)^T}&<\infty\,,\\
\lim_{M\to\infty}
\limsup_{\epsilon\to0}\sup_{s\in[0,t]}\sup_{x\in\R^l:\,\abs{x}\ge M}
\tilde a^\epsilon_s(x)^T\frac{x}{\abs{x}}
&=-\infty\,.  
\end{align*}
If, for all $t\in\R_+$ and all $N\in\R_+$\,,
\begin{align*}
    \lim_{\epsilon\to0}\sup_{s\in[0,t]}
\sup_{x\in\R^l:\,\abs{x}\le N}\bl(
\abs{\tilde a^\epsilon_s(x)-\tilde a_s(x)}+
\norm{\tilde b^\epsilon_s(x)-\tilde b_s(x)}\br)=0,
\end{align*}
the matrix $\tilde c_t(x)=
\tilde b_t(x)\tilde b_t(x)^T$   is  positive definite uniformly in $x$
and locally uniformly in $t$\,, is of
class $\mathbb C^1$ in $x$\,, with the first partial derivatives being
Lipschitz continuous and bounded in $x$
locally uniformly in $t$\,,
 $\tilde a_t(x)$ is Lipschitz continuous in $x$ locally uniformly in
$t$\,, 
$\sup_{s\in[0,t]}
\sup_{x\in\R^l}\norm{\tilde c_s(x)}<\infty$\,,
$ \limsup_{\abs{x}\to\infty}
\sup_{s\in[0,t]}\tilde a_s(x)^Tx/\abs{x}^2
<0$ for all $t\in\R_+$\,,
and the net $x^\epsilon_0$ is exponentially tight in
$\R^l$ for rate $1/\epsilon$ as $\epsilon\to0$\,, then
the net $\tilde\mu^\epsilon$\,, where
$\tilde\mu^\epsilon_t(dx)=\int_0^t\ind_{  dx}(\tilde x^\epsilon_s)
\,ds$\,,
 obeys the LDP in $\mathbb{C}_\uparrow(\R_+,\mathbb{M}(\R^l))$ for rate
$1/\epsilon$ as $\epsilon\to0$ with large deviation  function
$\mathbf{J}$ defined as follows.

If  function $\mu=(\mu_s,\,s\in\R_+)$ from
$\bC_\uparrow(\R_+,\mathbb{M}(\R^l))$\,,
when considered as a measure on $\R_+\times\R^l$\,, is
absolutely continuous w.r.t.  Lebesgue measure on $\R_+\times\R^l$\,, i.e.,
 $\mu(ds,dx)=m_s(x)\,dx\,ds$, 
   $m_s(x)$\,,   as a function of $x$\,, 
 belongs to $\mathbb{P}(\R^l)$ for almost all $s$\,,  
and $\tilde\Phi_{s,m_s(\cdot)}$\,, which represents
$\Pi_{\tilde c_s(\cdot),m_s(\cdot)}
\bl(\tilde c_s(\cdot)^{-1}\bl(
\tilde a_s(\cdot)-\text{div}_x\,
\tilde c_s(\cdot)/2\br)\br)$\,, 
is an element of $\mathbb{L}^2(\R^l,\R^l,\tilde c_s(x),m_s(x)\,dx)$ for
almost all $s$\,, then
\begin{multline*}
  \mathbf{J}(\mu)=
\int_0^\infty\sup_{h\in \mathbb{C}_0^1(\R^l)}\int_{\R^l}\Bl( D  h(x)^T \bl(
\frac{1}{2}\,\text{div}_x\,\bl(\tilde c_s(x)m_s(x)\br)
-\tilde a_s(x)\,m_s(x)
\br)
-\frac{1}{2}\,
\norm{D  h(x)}_{  \tilde c_s(x)}^2\,m_s(x)\Br)\,dx\Br)\,ds
\\=\frac{1}{2}\,
\int_0^\infty\int_{\R^l}\norm{
\frac{D_xm_s(x)}{2m_s(x)}
-\tilde\Phi_{s,m_s(\cdot)}(x)}^2_{\tilde c_s(x)}m_s(x)\,dx\,ds\,.
\end{multline*}

Otherwise, $\mathbf{J}(\mu)=\infty$.

\end{corollary}

\begin{corollary}
  \label{co:emp_meas}
Suppose that
\begin{equation*} 
    dY_t=\breve a(Y_t)\,dt+\breve b(Y_t)\,d\breve W_t\,,\,Y_0=0,
\end{equation*}
where $ Y_t\in \R^l$\,,
$\breve a(x)\in \R^l$\,,  $
\breve b(x)\in\R^{l\times k}$\,, and $\breve W_t\in\R^k$\,, with the
coefficients being locally bounded.

If
the  matrix $\breve c(x)=
\breve b(x)\breve b(x)^T$   is  uniformly positive definite,
$  \norm{\breve c(x)}$ is bounded,
 $\breve c(\cdot)\in 
\mathbb{C}^1(\R^l,\R^{\l\times l})$\,, with Lipschitz continuous
bounded first partial
derivatives,  $\breve a(\cdot)$ is Lipschitz continuous,
 and $  \limsup_{\abs{x}\to\infty}
\breve a(x)^Tx/\abs{x}^2
<0\,,$
then the empirical measures $(1/t)\int_0^t\ind_{ dx}(Y_s)\,ds$
obey the LDP in $\mathbb{M}_1(\R^l)$ for rate $t$ as $t\to\infty$
 with the large deviation function
\begin{multline*}
    \breve{\mathbf{J}}(\mu)=
\sup_{h\in \mathbb{C}_0^1(\R^l)}\int_{\R^l}\Bl( D  h(x)^T \bl(
\frac{1}{2}\,\text{div}\,\bl(\breve c(x)m(x)\br)
- \breve a(x)\,m(x)\br)
-\frac{1}{2}\,
\norm{D  h(x)}_{\breve c(x)}^2\,m(x)\Br)\,dx\\=\frac{1}{2}\,
\int_{\R^l}\norm{\frac{Dm(x)}{2m(x)}-\breve \Phi_{m(\cdot)}(x)
}^2_{\breve c(x)}m(x)\,dx
\end{multline*}
provided  probability measure  $\mu$ on $\R^l$ has   density
 $m$\,, which 
 is an element of $\mathbb{P}(\R^l)$\,,
and $\breve\Phi_{m(\cdot)}=
\Pi_{\breve c(\cdot),m(\cdot)}
\bl(\breve c(\cdot)^{-1}\bl(
\breve a(\cdot)-\text{div}\,
\breve c(\cdot)/2\br)\br)$
is an element of $\mathbb{L}^2(\R^l,\R^l,\breve c(x),m(x)\,dx)$\,.
Otherwise, $\breve{\mathbf{J}}(\mu)=\infty$\,.
\end{corollary}

In order to derive Corollary~\ref{co:emp_meas} from
Corollary~\ref{cor:emp}, one takes $\epsilon=1/t$ and defines
$\tilde x^\epsilon_s=Y_{st}$\,.

One can thus write the large deviation function of Theorem
\ref{the:ldp} as
\begin{multline}
  \label{eq:94}
      \mathbf{I}(X,\mu)=\mathbf{I}_0(X_0)+
\int_0^\infty
\sup_{\lambda\in\R^n}\Bl(
\lambda^T\bl(\dot{X}_s-
\int_{\R^l} A_s(X_s,x)\,\nu_s(dx)\br)
-\frac{1}{2}\, \norm{\lambda}_{\int_{\R^l}C_s(X_s,x)\,\nu_s(dx)}^2
\\+\mathbf{J}^{s,X_s,\lambda}(\nu_s)\Br)\,ds\,,
\end{multline}
with $\nu_s(dx)=m_s(x)\,dx$\,, and
 the large deviation function of Corollary
\ref{the:ave} as
\begin{multline}
  \label{eq:77a}
    \mathbf{I}^X(X)=\mathbf{I}_0(X_0)+
\int_0^\infty
\sup_{\lambda\in\R^n}\Bl(
\lambda^T\dot{X}_s-\sup_{\nu\in \mathbb{M}_1(\R^l)}\Bl(
\lambda^T\int_{\R^l} A_s(X_s,x)\,\nu(dx)
+\frac{1}{2}\, \norm{\lambda}_{\int_{\R^l}C_s(X_s,x)\,\nu(dx)}^2\\
-\mathbf{J}^{s,X_s,\lambda}(\nu)\Br)\Br)\,ds\,, 
\end{multline}
where $\mathbf{J}^{s,u,\lambda}$ represents the large deviation function for
the empirical measures 
$\nu_t^{s,u,\lambda}(dx)=
(1/t)\int_0^t\ind_{ dx}(y^{s,u,\lambda}_{r})\,dr$ for rate $t$ 
as $t\to\infty$ and
\begin{equation*}
  dy^{s,u,\lambda}_t=\bl(a_s(u,y^{s,u,\lambda}_t)+
G_s(u,y^{s,u,\lambda}_t)^T\lambda\br)\,dt
+b_s(u,y^{s,u,\lambda}_t)\,dw_t\,, \;y^{s,u,\lambda}_0=0\,, 
\end{equation*}
$(w_t)$ being a
$k$-dimensional standard Wiener process. In particular,
if $G_t(u,x)=0$ so that the diffusions driving the slow and the
fast processes are virtually uncorrelated, then 
$\mathbf{J}^{s,u,\lambda}$ does not depend on $\lambda$ and by 
Corollary~\ref{cor:emp}, Corollary~\ref{co:emp_meas}, and \eqref{eq:94} 
the large deviation
function $\mathbf{I}(X,\mu)$ is the sum of the large deviation function
of the slow process, with the coefficients being averaged over the
''current'' empirical
measure of the fast variable, and of the large deviation function 
of the empirical process of the 
fast variable, with the coefficients ''frozen'' at the current value
of  the slow variable.

The first results on  large deviation
asymptotics for the system (\ref{eq:67})
in the setup of the averaging principle 
available in the literature  appear in
Freidlin \cite{Fre78}, see also
the exposition in  Freidlin and Wentzell \cite[Section 9 of Chapter
7]{wf2}.
Freidlin \cite{Fre78} 
considers
the equations
\begin{align*}
  \dot{x}^\epsilon_t&=b(x^\epsilon_t,y^\epsilon_t),\\
\dot{y}^\epsilon_t&=\frac{1}{\epsilon}\bl[B(x^\epsilon_t,y^\epsilon_t)+g(y^\epsilon_t)\br] +\frac{1}{\sqrt{\epsilon}}\,c(y^\epsilon_t)\dot{w}_t\,.
\end{align*}
It is assumed 
that the state space is a compact manifold.
A noncompact setting is considered by 
Veretennikov \cite{Ver94}.
 Veretennikov 
\cite{Ver13,Ver98b,Ver99}
allows the diffusion coefficient in the fast process to
depend on both variables:
\begin{align*}
  dX^\epsilon_t&=f(X^\epsilon_t,Y^\epsilon_t)\,dt,\\
dY^\epsilon_t&=\epsilon^{-2}B(X^\epsilon_t,Y^\epsilon_t)\,dt
+\epsilon^{-1}C(X^\epsilon_t,Y^\epsilon_t)\,dW_t\,.
\end{align*}
The state space of the fast process is a compact manifold.

 Veretennikov 
\cite{Ver_letter,Ver98a,Ver00} tackles the case where
 the slow process has a small diffusion term and
the state space of the fast process may be  noncompact  but
 the diffusion coefficient in the equation for the fast
process does not depend on the slow process so that
\begin{equation}
  \label{eq:178}
  \begin{split}
  dX_t^\epsilon&=f(X^\epsilon_t,Y^\epsilon_t)\,dt
+\epsilon\bl(\sigma_1(X^\epsilon_t,Y^\epsilon_t)\,dW^1_t
+\sigma_3(X^\epsilon_t,Y^\epsilon_t)\,dW^3_t\br),\\
\\
dY^\epsilon_t&=\epsilon^{-2}B(X^\epsilon_t,Y^\epsilon_t)\,dt
+\epsilon^{-1}(C_1(Y^\epsilon_t)\,dW_t^1+C_2(Y^\epsilon_t)\,dW_t^2)\,,
\end{split}
\end{equation}
where the Wiener processes are independent.
The stability condition on the slow process is similar to
\eqref{eq:9} and
\eqref{eq:116}.

In those papers, results on the LDP
 for the slow processes are obtained  in  the space of 
 continuous functions on the $[0,L]$
 interval
 endowed with uniform norm, where $L>0$\,. The large deviation rate functions  are of
 the form
 \begin{equation*}
      \mathbf{I}(X)=\int_0^L \sup_{\lambda}
\bl(\lambda^T\dot{X}_t-H(X_t,\lambda)\br)\,dt\,,
 \end{equation*}
provided $X_t,\,t\in[0,L],$ is an absolutely continuous function
with a suitable initial condition. Otherwise, 
$ \mathbf{I}(X)=\infty$\,.
Here, with the notation of \eqref{eq:178},
\begin{multline}
  \label{eq:32}
    H(u,\lambda)=\lim_{t\to\infty}\frac{1}{t}
\,\ln \mathbf{E}\exp\bl(\int_0^t \bl(\lambda^Tf(u,y^{u,\lambda}_s)+
\frac{1}{2}\,\lambda^T (\sigma_1\sigma_1^T(u,y^{u,\lambda}_s)
+\sigma_3\sigma_3^T(u,y^{u,\lambda}_s))\lambda\br)\,ds\br)\,,
\end{multline}
where 
\begin{multline*}
    dy^{u,\lambda}_t=\bl(B(u,y^{u,\lambda}_t)+C_1(y^{u,\lambda}_t)
\sigma_1(u,y^{u,\lambda}_t)^T
\lambda\br)\,dt
+(C_1(Y^{u,\lambda}_t)\,dW_t^1+C_2(Y^{u,\lambda}_t)\,dW_t^2)\,,
\;y^{u,\lambda}_0=0 \,.
\end{multline*}
 The existence of the limit in
(\ref{eq:32}) is proved by invoking the Frobenius theorem for compact positive
operators. 

Let us note that if one
assumes the LDP at rate $t$   as $t\to\infty$
 of the empirical measures
$\nu^{u,\lambda}_t(dx)=(1/t)\int_0^t\ind_{  dx}(y^{u,\lambda}_{s})
\,ds$
with large deviation rate function 
$\mathbf{J}^{u,\lambda}$\,,
 then, in view of
 Varadhan's lemma and \eqref{eq:32}, under suitable assumptions,
\begin{multline*}
  H(u,\lambda)=\sup_{\nu\in\mathbb{M}_1(\R^l)}
\bl(\int_{\R^l}\bl(\lambda^Tf(u,x)+
\frac{1}{2}\,\lambda^T
(\sigma_1\sigma_1^T(u,x)
+\sigma_3\sigma_3^T(u,x))
\lambda\br)\,\nu(dx)-\mathbf{J}^{u,\lambda}(\nu)\br)\,, 
\end{multline*}
which is consistent with  \eqref{eq:77a}.

Section 11.6 of
Feng and Kurtz \cite{FenKur06} is concerned with
 the process $X^\epsilon$
satisfying   equations \eqref{eq:67}.
Conditions for the LDP to hold are obtained. 
They 
 require the existence of functions with certain
properties and are not easily
translated into conditions on the
coefficients. When the authors give explicit 
conditions on the coefficients, 
they need, in particular,
  $b(u,x)$ not to depend on $u$ 
(see Lemma 11.60 on p.278). 
The large deviation rate function is identified as having the
 form
 \eqref{eq:77a} corresponding to the time-homogeneous setting,
 provided
$B(u,x)b(u,x)^T=0$ and certain
additional hypotheses hold (see Theorem 11.6.5 on p.282).
The authors choose not to pursue the setup of the averaging principle.



The LDP for the empirical measures of continuous-time Markov processes, such as
in Corollary~\ref{co:emp_meas},
is a well explored
subject, see  Donsker and Varadhan \cite{dv1,dv3}, 
Deuschel and Stroock \cite{DeuStr01}.
The canonical form of the large deviation rate function is
$\sup_{f} \int_{\R^l}-\mathcal{L}f/f\,d\mu$\,, where $\mathcal{L}$
represents the infinitesimal generator of the Markov process, see, e.g.,
Theorem 4.2.43 in Deuschel and Stroock \cite{DeuStr01}.
The form
in Corollary~\ref{co:emp_meas} follows by taking  $f(x)=e^{-h(x)}$\,.
G\"artner \cite{Gar77} and Veretennikov \cite{MR1191695}
characterise the large deviation functions via limits similar to
that in \eqref{eq:32}, the latter author allowing discontinuous
coefficients. Theorem 12.7 on p.291 of Feng and Kurtz~\cite{FenKur06} tackles 
associated empirical processes, cf. Corollary \ref{cor:emp}.

\section{Some generalities}
\label{sec:an-outline-approach}
This section contains general results on the LDP
 that underlie
the proof of Theorem 
\ref{the:ldp}, 
cf. Puhalskii \cite{Puh01}.
Let $\Sigma$ represent a directed set, let
 $\mathbf{P}_\sigma$\,, where $\sigma\in \Sigma$\,,
 represent a net of probability measures   on a
metric space $\mathbb{S}$ indexed with the elements of $\Sigma$\, and
let $r_\sigma$ represent an $\R_+$-valued function which tends to infinity
as $\sigma\in\Sigma$\,.
A $[0,\infty]$-valued function $\mathbf I$ on $\mathbb S$ is referred
to as a  large deviation function if the sets
$K_\delta=\{z\in\mathbb{S}:\,\mathbf I(z)\le \delta\}$ are compact for all $\delta\in\R_+$\,.
We say that the net  $\mathbf{P}_\sigma$ obeys the LDP
 with a  large deviation  function $\mathbf{I}$ for rate
$r_\sigma$ as $\sigma\in\Sigma$ if 
$  \liminf_{\sigma\in \Sigma}r^{-1}_\sigma\ln \mathbf{P}_\sigma(G)
\ge -\inf_{z\in G}\mathbf{I}(z)$ for all open sets $G\subset \mathbb{S}$
and $  \limsup_{\sigma\in \Sigma}r^{-1}_\sigma\ln \mathbf{P}_\sigma(F)
\le -\inf_{z\in F}\mathbf{I}(z)$
 for all closed sets $F\subset \mathbb{S}$.
We say that  $\mathbf{I}$  is a large deviation (LD) limit point of
$\mathbf{P}_\sigma$ for rate $r_\sigma$ 
if there exists a subsequence $\sigma_i\,,$ where $ i\in\N\,,$ such that 
$\mathbf{P}_{\sigma_i}$  satisfies the
LDP with $\mathbf I$ for rate $r_{\sigma_i}$ as $i\to\infty$\,.
We say that the net $\mathbf{P}_\sigma$
is sequentially 
 large deviation (LD) relatively  compact for rate $r_\sigma$ as
 $\sigma\in\Sigma$
 if any subsequence 
$\mathbf{P}_{\sigma_i}$ of $\mathbf{P}_\sigma$
 contains a further subsequence $\mathbf{P}_{\sigma_{i_j}}$ which
satisfies the LDP for rate $r_{\sigma_{i_j}}$ with some
 large deviation  function as
$j\to\infty$\,. We say that the net $\mathbf{P}_\sigma$ is
exponentially (or large deviation) tight for rate $r_\sigma$ as
$\sigma\in\Sigma$ if for arbitrary $\kappa>0$ there exists compact
$K\subset \mathbb{S}$ such that
$\limsup_{\sigma\in\Sigma}\mathbf{P}_\sigma(\mathbb S\setminus
K)^{1/r_\sigma}<\kappa$\,. We say that the net $\mathbf{P}_\sigma$ is
sequentially exponentially  tight for rate $r_\sigma$ as
$\sigma\in\Sigma$ if any subsequence $\mathbf{P}_{\sigma_i}$ 
 is
exponentially tight for rate $r_{\sigma_i}$ as $i\to\infty$\,.
We say that a net $Y_\sigma$ of random elements of $\mathbb S$ 
 obeys the
LDP, respectively, is sequentially  LD relatively  compact, respectively, is
exponentially tight, respectively, is sequentially exponentially tight if the
net of their laws has the indicated property.

The cornerstone of our approach is 
the next result (Puhalskii \cite{Puh91,puhwolf,Puh93,Puh01}, see also
 Feng and Kurtz \cite{FenKur06} and references therein).
\begin{theorem}
  \label{the:rel_comp}
If the net $\mathbf P_\sigma$ is sequentially 
exponentially tight for rate $r_\sigma$ as $\sigma\in\Sigma$\,, then
the net $\mathbf P_\sigma$ is sequentially 
LD relatively compact for rate $r_\sigma$ as $\sigma\in\Sigma$\,.
\end{theorem}
The proof of the following theorem is standard.
\begin{theorem}
  \label{the:uniq} 
If the net $\mathbf P_\sigma$ is
sequentially  LD relatively  compact for rate $r_\sigma$ as
$\sigma\in\Sigma$ and
$\mathbf I$ is a unique LD limit point of the  $\mathbf P_\sigma$\,,
then the net $\mathbf P_\sigma$ satisfies the LDP with $\mathbf I$ for
rate $r_\sigma$ as $\sigma\in\Sigma$\,.
\end{theorem}
The next theorem is essentially Varadhan's lemma, see, e.g., Deuschel
and Stroock \cite{DeuStr01}. It  will be used to
obtain  equations for LD limit points.
\begin{theorem}
  \label{the:ld_limit}
Suppose the net $\mathbf P_\sigma$ is sequentially 
exponentially tight  for
rate $r_\sigma$ as $\sigma\in\Sigma$ and let
$\mathbf I$ represent an LD limit point
of $\mathbf P_\sigma$\,.
Let $U_\sigma$ be a net of uniformly 
bounded real valued functions on $\mathbb S$ such
that $\int_{\mathbb S}\exp(r_\sigma U_\sigma(z))\,\mathbf P_\sigma(dz)=1$\,. 
If $U_\sigma\to U$ uniformly on compact sets 
as $\sigma\in \Sigma$\,, where the function $U$ is continuous,
   then
$\sup_{z\in\mathbb S}(U(z)-\mathbf I(z))=0$\,.
\end{theorem}
Identification of LD limit points will be carried out with the aid of
the next result.
\begin{theorem}
  \label{the:id}
Suppose $\mathbf I$ is a large deviation function on $\mathbb S$ and 
$\mathcal{U}$ is a collection of  functions on
$\mathbb S$ such that 
$\sup_{z\in\mathbb S}(U(z)-\mathbf I(z))=0$ for all
$U\in\mathcal{U}$\,. Let $\mathbf
I^{\ast\ast}(z)=\sup_{U\in\mathcal{U}}U(z)$\,
and $K_\delta=\{z\in \mathbb S:\,\mathbf I(z)\le \delta\}$\,, where
$\delta\in\R_+$\,. 
\begin{enumerate}
\item
Let $\tilde{\mathcal{U}}$ represent a set of 
functions $U$ 
such that  
$\sup_{z\in K_\delta}
(U(z)-\mathbf I(z))=0$ for suitable $\delta\in\R_+$\,.
Suppose  $\hat z\in\mathbb S$ is
such that $\mathbf I^{\ast\ast}(\hat z)= \hat U(\hat z)$
for some function $\hat U\in \tilde{\mathcal{U}}$\,.
Suppose there exists sequence $U_i\in\tilde{\mathcal{U}}$
with the following properties:
$\sup_{z\in K_\delta}
(U_i(z)-\mathbf I(z))=0$
for some common $\delta$\,, the functions $U_i$ are continuous when
restricted to $K_\delta$\, and
if $z_i$ is  a convergent sequence of elements of $K_\delta$ such that
$U_i(z_i)=\mathbf I(z_i)$\,,
 then $U_i(z_i)\to \hat{U}(\hat z)$ and
$z_i\to \hat z$ as $i\to\infty$\,. 
Then $\mathbf I(\hat z)=\mathbf I^{\ast\ast}(\hat z)$\,.
\item If for every $z\in\mathbb S$ such that $\mathbf
I^{\ast\ast}(z)<\infty$ there exists a sequence of points
$z_i$ such that $\mathbf I( z_i)=\mathbf I^{\ast\ast}( z_i)$\,,
$z_i\to z$\,, and 
$\mathbf I^{\ast\ast}(z_i)\to \mathbf I^{\ast\ast}(z)$ as $i\to\infty$\,, then 
$\mathbf I(z)=\mathbf I^{\ast\ast}(z)$ for all $z\in \mathbb S$\,.

\end{enumerate}

\end{theorem}
\begin{proof}
Let us first note that $\mathbf I(z)\ge \mathbf{I}^{\ast\ast}(z)$ for
all $z$\,, so, one needs to prove that $\mathbf I(z)\le
\mathbf{I}^{\ast\ast}(z)$ if $\mathbf{I}^{\ast\ast}(z)<\infty$\,.
We prove part 1. Since $\sup_{z\in K_\delta}
(U_i(z)-\mathbf I(z))=0$, $K_\delta$ is compact, and $U_i(z)-\mathbf I(z)$ is
upper semicontinuous when restricted to $K_\delta$\,, 
there exist   $z_i\in K_\delta$ such
that $U_i(z_i)=\mathbf I(z_i)$\,. 
One may assume that the sequence converges. Since
 $U_i(z_i)\to
\hat U(\hat z)$\,, $z_i\to\hat z$ and 
 $\mathbf I$ is lower semicontinuous,
 $\hat U(\hat z)\ge \mathbf
I(\hat z)$\,, so $\mathbf{I}^{\ast\ast}(\hat z)\ge \mathbf I(\hat
z)$\,.
The proof of part 2 is similar.
\end{proof}
In the rest of the paper, the above framework is used to prove
Theorem~\ref{the:ldp}. In Section 
\ref{sec:expon-tightn},  LD relative  compactness 
 is established, see Theorem \ref{the:exp_tigh}. 
In Section~\ref{sec:equat-large-devi}, equations along the
lines of Theorem \ref{the:ld_limit} are derived, see 
Theorem \ref{the:equation}.
Section \ref{sec:regul-prop} is concerned with regularity properties
of $(X,\mu)$ for which the function $\mathbf I^{\ast\ast}$ as defined
in  Theorem
\ref{the:id} assumes finite values. It is also shown to be of the form
given in Proposition \ref{cor:ldp},
 see Theorem~\ref{le:vid}. 
In Theorem~\ref{the:iden_reg} of 
Section \ref{sec:ident},  the large deviation function is identified
for a large class of $(X,\mu)$\,, which implements the recipe of part
1 of Theorem \ref{the:id}. In Theorem~\ref{the:appr} of
Section~\ref{sec:appr-large-devi}, it is proved that that class is dense in
the sense of part 2 of Theorem~\ref{the:id}.
In Section \ref{sec:proof-theor-refth}, the proof of
Theorem~\ref{the:ldp} is completed.

\section{LD  relative  compactness}
\label{sec:expon-tightn}
The main result of this section is the following theorem.
\begin{theorem}
  \label{the:exp_tigh}
Suppose that conditions 
\eqref{eq:10} -- \eqref{eq:8} and  \eqref{eq:9}
 hold and that the net 
$(X^\epsilon_0,x^\epsilon_0)$ is
 exponentially tight for rate $1/\epsilon$ as $\epsilon\to0$\,. 
Then the net $(X^\epsilon,\mu^\epsilon)$ is sequentially    LD relatively
compact  in $\bC(\R_+,\R^n)\times \bC_\uparrow(\R_+,\mathbb{M}(\R^l))$ 
for rate $1/\epsilon$ as $\epsilon\to0$\,.
\end{theorem}
We precede the proof with a criterion of  sequential  LD relative
compactness  in 
$\bC(\R_+,\mathbb{M}(\R^l))$\,. 
Let  $d(\cdot,\cdot)$ represent the  Lipschitz 
metric on $\mathbb{M}(\R^l)$\,:
$  d(\tilde\mu,\hat\mu)=\sup\{\abs{\int_{\R^l}f(x)\,\tilde\mu(dx)-
\int_{\R^l}f(x)\,\hat\mu(dx)}\}\,,$
with the supremum being taken over functions $f:\,\R^l\to\R$ such that
$\sup_{x\in\R^l}\abs{f(x)}\le1$ and
$\sup_{x,y\in\R^l,\,x\not=y}\abs{f(x)-f(y)}/\abs{x-y}\le 1$\,, see,
e.g.,  p.395 in Dudley \cite{Dud02}.
The proof of the next lemma is relegated to the appendix.
\begin{lemma}
  \label{le:exp_tight}
  \begin{enumerate}
  \item 
A net $\{\nu_\epsilon,\,\epsilon>0\}$\,, where
$\nu_\epsilon=(\nu_{\epsilon,t},\,t\in\R_+)$\,,  of random elements of
 $\bC(\R_+,\mathbb{M}(\R^l))$ defined on respective probability spaces
 $(\Omega_\epsilon,
\mathcal{F}_\epsilon,\mathbf{P}_\epsilon)$ is  sequentially
exponentially tight  for rate
$1/\epsilon$ as $\epsilon\to0$ if and only if 
 for all $t\in\R_+$ and all $\eta>0$\,,
\begin{align*}
  \lim_{N\to\infty}\limsup_{\epsilon\to0}
\mathbf{P}_\epsilon\bl(\nu_{\epsilon,t}(x\in\R^l:\,\abs{x}>N)>\eta\br)^{\epsilon}&=0
\intertext{and}
  \lim_{\delta\to0}\limsup_{\epsilon\to0}
\sup_{s_1\in[0,t]}
\mathbf{P}_\epsilon\bl(\sup_{s_2\in [s_1,s_1+\delta]}
d(\nu_{\epsilon,s_1},\nu_{\epsilon,s_2})
>\eta\br)^{\epsilon}&=0\,.
\end{align*}
\item A net $\{Y_\epsilon,\,\epsilon>0\}$\,, where
$Y_\epsilon=(Y_{\epsilon,t},\,t\in\R_+)$\,,  of random elements of
 $\bC(\R_+,\R^n)$ defined on respective probability spaces
 $(\Omega_\epsilon,
\mathcal{F}_\epsilon,\mathbf{P}_\epsilon)$ is  sequentially 
exponentially tight  for rate
$1/\epsilon$ as $\epsilon\to0$ if and only if 
  \begin{align*}
      \lim_{N\to\infty}\limsup_{\epsilon\to0}
\mathbf{P}_\epsilon(\abs{Y_{\epsilon,0}}>N)^\epsilon=0\,
\intertext{and,  for all $t\in\R_+$ and all $\eta>0$\,,
}
   \lim_{\delta\to0}\limsup_{\epsilon\to0}
\sup_{s_1\in[0,t]}
  \mathbf{P}_\epsilon(\sup_{ s_2\in[s_1,s_1+\delta]}
\abs{Y_{\epsilon,s_2}-Y_{\epsilon,s_1}}>\eta)^\epsilon=0\,.
 \end{align*}
  \end{enumerate}

\end{lemma}

\begin{remark}
The form of the conditions is  due to
Feng and Kurtz \cite{FenKur06}.
\end{remark}

\begin{proof}[Proof of Theorem \ref{the:exp_tigh}]
Since  $\bC(\R_+,\R^n)\times \bC_\uparrow(\R_+,\mathbb{M}(\R^l))$
is a closed subset of  $\bC(\R_+,\R^n)\times
\bC(\R_+,\mathbb{M}(\R^l))$ and
$\mathbf{P}^\epsilon((X^\epsilon,\mu^\epsilon)\in 
\bC(\R_+,\R^n)\linebreak\times \bC_\uparrow(\R_+,\mathbb{M}(\R^l)))=1$\,, it is
sufficient to prove that the net
$((X^\epsilon,\mu^\epsilon),\,\epsilon>0)$ is  sequentially 
LD relatively
 compact in $\bC(\R_+,\R^n)\times
\bC(\R_+,\mathbb{M}(\R^l))$\,.
 By Theorem \ref{the:rel_comp}, the latter property holds if
$(X^\epsilon,\mu^\epsilon)$ is sequentially exponentially tight,
 which is the case if
 the nets
 $X^\epsilon$ and  $\mu^\epsilon$ are each sequentially exponentially tight.

We  show that the net $X^\epsilon$ is  sequentially exponentially tight
  first. 
By \eqref{eq:1} and 
It\^o's lemma,
on denoting $g_1(x)=D ^2\ln(1+\abs{x}^2)$,
\begin{multline*}
  \ln(1+\abs{X^\epsilon_t}^2)=\ln(1+\abs{X^\epsilon_0}^2)+
\int_0^t \frac{2(X^\epsilon_s)^T
  A^\epsilon_s(X^\epsilon_s,x^\epsilon_s)}{1+\abs{X^\epsilon_s}^2} \,ds+
\frac{\epsilon}{2}\int_0^t
\text{tr}\,\bl(C^\epsilon_s(X^\epsilon_s,x^\epsilon_s)g_1(X^\epsilon_s)\br)\,ds\\
+\sqrt{\epsilon}\int_0^t \frac{2(X^\epsilon_s)^T}{1+\abs{X^\epsilon_s}^2}\,
  B^\epsilon_s(X^\epsilon_s,x^\epsilon_s) \,dW^\epsilon_s\,.
\end{multline*}
Given $N>0$, let 
$\tau_N^\epsilon=\inf\{s\in\R_+:\,\abs{X^\epsilon_s}\ge N\}\,.$
Since $\tau_N^\epsilon$ is an  $\mathbf{F}^\epsilon$-stopping time and
\begin{equation*}
  \exp\bl(\frac{1}{\sqrt{\epsilon}}
\int_0^t \frac{2(X^\epsilon_s)^T}{1+\abs{X^\epsilon_s}^2}\,
  B^\epsilon_s(X^\epsilon_s,x^\epsilon_s) \,dW^\epsilon_s
-\frac{1}{2\epsilon}
\int_0^t\norm{\frac{2X^\epsilon_s}{1+\abs{X^\epsilon_s}^2}}^2_{C^\epsilon_s(X^\epsilon_s,x^\epsilon_s)}
 \,ds\br)\,,t\in\R_+\,,
\end{equation*}
 is an $\mathbf{F}^\epsilon$-local martingale,
 \begin{multline}
   \label{eq:99}
     \mathbf{E}^\epsilon\exp\Bl(\frac{1}{\epsilon}\ln(1+\abs{X^\epsilon_{t\wedge\tau_N^\epsilon}}^2)-
\frac{1}{\epsilon}\ln(1+\abs{X^\epsilon_0}^2)-\frac{1}{\epsilon}
\int_0^{t\wedge\tau_N^\epsilon} \frac{2(X^\epsilon_s)^T
  A^\epsilon_s(X^\epsilon_s,x^\epsilon_s)}{1+\abs{X^\epsilon_s}^2} \,ds\\-
\frac{1}{2}\int_0^{t\wedge\tau_N^\epsilon}
\text{tr}\,\bl(C^\epsilon_s(X^\epsilon_s,x^\epsilon_s)g_1(X^\epsilon_s)\br)\,ds-
\frac{1}{2\epsilon}\,
\int_0^{t\wedge\tau_N^\epsilon}
\norm{ \frac{2X^\epsilon_s}{1+\abs{X^\epsilon_s}^2}}^2_{C^\epsilon_s(X^\epsilon_s,x^\epsilon_s)}
 \,ds\Br)\le 1\,.
\end{multline}
Since 
\begin{equation}
  \label{eq:78}
  \text{tr}\,\bl(C^\epsilon_s(X^\epsilon_s,x)g_1(X^\epsilon_s)
\br)\le \sqrt{\text{tr}\,g_1(X^\epsilon_s)^2}
\,\sqrt{\text{tr}\,C^\epsilon_s(X^\epsilon_s,x)^2}
\le 2n\sqrt{n}\,\frac{\norm{C^\epsilon_s(X^\epsilon_s,x^\epsilon_s)}}{
1+\abs{X^\epsilon_s}^2} \,,
\end{equation}
on recalling  (\ref{eq:7}) and (\ref{eq:8}), we have that
there exists  $L>0$\,, which does not depend either on $t$ or
on $N$\,, such that for all $\epsilon>0$ small enough,
\begin{equation*}
  \mathbf{E}^\epsilon\exp\Bl(\frac{1}{\epsilon}\ln(1+\abs{X^\epsilon_{t\wedge\tau_N^\epsilon}}^2)-
\frac{1}{\epsilon}\ln(1+\abs{X^\epsilon_0}^2)-\frac{Lt}{\epsilon}\Br)\le 1\,.
\end{equation*}
For  $\tilde{N}>0$,
\begin{multline*}
\mathbf{P}^\epsilon(\sup_{s\in[0,t]}\abs{X^\epsilon_{s}}\ge N)=
  \mathbf{P}^\epsilon(\abs{X^\epsilon_{t\wedge\tau_N^\epsilon}}\ge N)\le
\mathbf{P}^\epsilon(\abs{X^\epsilon_0}> \tilde{N})\\
 + \mathbf{E}^\epsilon\exp\Bl(\frac{1}{\epsilon}\ln(1+\abs{X^\epsilon_{t\wedge\tau_N^\epsilon}}^2)
-\frac{1}{\epsilon}\ln(1+N^2)\Br)\ind_{\{\abs{X^\epsilon_0}\le 
\tilde{N}\}}\\
\le \mathbf{P}^\epsilon(\abs{X^\epsilon_0}>\tilde{N})+
\exp\bl(\frac{1}{\epsilon}\ln(1+\tilde{N}^2)+\frac{Lt}{\epsilon}
-\frac{1}{\epsilon}\ln(1+N^2)\br)\,,
\end{multline*}
so
\begin{equation*}
  \limsup_{N\to\infty}\limsup_{\epsilon\to0}
\mathbf{P}^\epsilon(\sup_{s\in[0, t]}\abs{X^\epsilon_{s}}>N)^\epsilon\le
\limsup_{\epsilon\to0}
\mathbf{P}^\epsilon(\abs{X^\epsilon_0}>\tilde{N})^\epsilon\,.
\end{equation*}
Since $X^\epsilon_0$ is exponentially tight and $\tilde{N}$ is arbitrary, we
conclude that
\begin{equation}
  \label{eq:17}
  \lim_{N\to\infty}\limsup_{\epsilon\to0}
\mathbf{P}^\epsilon(\sup_{s\in[0, t]}\abs{X^\epsilon_s}>N)^\epsilon=0\,.
\end{equation}
By (\ref{eq:1}), for $s\in[0,t]$\,, $\delta>0$\,,
 and $\eta>0$,
\begin{multline*}
  \mathbf{P}^\epsilon(\sup_{\tilde s\in[s,s+\delta]}
\abs{X^\epsilon_{\tilde{s}}-X^\epsilon_s}>\eta)\le
\mathbf{P}^\epsilon(\tau_N^\epsilon\le t)+
\mathbf{P}^\epsilon(\sup_{\abs{u}\le N}\sup_{x\in\R^l}
 \abs{A^\epsilon_s(u,x)}\delta\\+
\sqrt{\epsilon}\sup_{\tilde s\in[s,s+\delta]}
\abs{\int_{s\wedge \tau_N^\epsilon}^{\tilde{s}\wedge \tau_N^\epsilon}
B^\epsilon_r(X^\epsilon_r,x^\epsilon_r)\,dW^\epsilon_r}> \eta)\,.
\end{multline*}
Let $e_i$, for $i=1,2,\ldots,n$, denote the $i$th unit   vector of $\R^n$.
Thanks to \eqref{eq:10a} and \eqref{eq:8},
 for small enough $\delta$
 and arbitrary $\alpha>0$, provided $\epsilon>0$ is small enough, on
 using Doob's inequality,
 \begin{multline*}
   \mathbf{P}^\epsilon(\sup_{\tilde s\in[s,s+\delta]}
\abs{X^\epsilon_{\tilde{s}}-X^\epsilon_s}>\eta)\le
\mathbf{P}^\epsilon(\tau_N^\epsilon\le t)+
\mathbf{P}^\epsilon\bl(\sqrt{\epsilon}\,
\sup_{\tilde s\in[s,s+\delta]}
\abs{\int_{s\wedge \tau_N^\epsilon}^{\tilde{s}\wedge \tau_N^\epsilon}
B^\epsilon_r(X^\epsilon_r,x^\epsilon_r)\,dW^\epsilon_r}>\frac{ \eta}{2}\br)\\
\le \mathbf{P}^\epsilon(\tau_N^\epsilon\le t)+
\sum_{i=1}^n\mathbf{P}^\epsilon\bl(\sqrt{\epsilon}\,
\sup_{\tilde s\in[s,s+\delta]}\bl(e_i^T\int_{s\wedge \tau_N^\epsilon}^{\tilde{s}\wedge \tau_N^\epsilon}
B^\epsilon_r(X^\epsilon_r,x^\epsilon_r)\,dW^\epsilon_r\br)
>\frac{ \eta}{2n}\br)
\le\mathbf{P}^\epsilon(\tau_N^\epsilon\le
t)\\+\sum_{i=1}^n\mathbf{P}^\epsilon
\bl(\sup_{\tilde s\in[s,s+\delta]}\exp\bl(
\frac{\alpha}{\sqrt{\epsilon}}\,e_i^T
\int_{s\wedge \tau_N^\epsilon}^{\tilde{s}\wedge \tau_N^\epsilon}
B^\epsilon_r(X^\epsilon_r,x^\epsilon_r)\,dW^\epsilon_r
-\frac{\alpha^2}{2\epsilon}
\int_{s\wedge \tau_N^\epsilon}^{\tilde{s}\wedge \tau_N^\epsilon}
e_i^T C^\epsilon_r(X^\epsilon_r,x^\epsilon_r)e_i\,dr\br)\\
> e^{\alpha \eta/(2n\epsilon)}\exp\bl(-\frac{\alpha^2\delta}{2\epsilon}
\sup_{r\in[0,t]}\sup_{\abs{u}\le N}\sup_{x\in\R^l}\norm{C^\epsilon_r(u,x)}\br)\br)
\le 
\mathbf{P}^\epsilon(\sup_{\tilde s\in[0,t]}\abs{X^\epsilon_{\tilde s}}\ge N)\\+
 n\,e^{-\alpha \eta/(2n\epsilon)}\exp\bl(\frac{\alpha^2\delta}{2\epsilon}
\sup_{r\in[0,t]}\sup_{\abs{u}\le N}\sup_{x\in\R^l}\norm{C^\epsilon_r(u,x)}\br)\,.
 \end{multline*}
 By \eqref{eq:8}, (\ref{eq:17}) and the fact that $\alpha$ can be chosen
 arbitrarily great,
 \begin{equation*}
   \limsup_{\delta\to0}\limsup_{\epsilon\to0}
\sup_{s\in[0,t]}
  \mathbf{P}^\epsilon(\sup_{\tilde s\in[s,s+\delta]}
\abs{X^\epsilon_{\tilde{s}}-X^\epsilon_s}>\eta)^\epsilon=0\,.
 \end{equation*}
The sequential exponential tightness of $X^\epsilon$
follows from part 2) of Lemma~\ref{le:exp_tight}.

We prove now that $\mu^\epsilon$ is sequentially exponentially tight. 
Let $f$ represent an $\R$-valued
  twice 
continuously differentiable function on $\R^l$.
By \eqref{eq:5} and  It\^o's lemma,
\begin{multline*}
   f(x^\epsilon_{t})=f(x^\epsilon_0)+\frac{1}{\epsilon}\,
\int_0^{t} D  f(x^\epsilon_{s})^T
a^\epsilon_s(X^\epsilon_s,x^\epsilon_{s})\,ds+
\frac{1}{2\epsilon}\,\int_0^t \text{tr}\,\bl(
c^\epsilon_s(X^\epsilon_s,x^\epsilon_{s})D ^2f(x^\epsilon_{s})\br)\,ds\\+
\frac{1}{\sqrt{\epsilon}}\,
\int_0^t D  f(x^\epsilon_{s})^T
b^\epsilon_s(X^\epsilon_s,x^\epsilon_{s})\,
dW^\epsilon_s
\,.
\end{multline*}
Therefore, 
on identifying
$\mu^\epsilon$ with  measure $\mu^\epsilon(dt,dx)$\,, we have that, in
analogy with \eqref{eq:99}, 
\begin{multline*}
      \mathbf{E}^\epsilon
\exp\Bl(  f(x^\epsilon_{t\wedge \tau_N^\epsilon})-f(x^\epsilon_0)
-\frac{1}{\epsilon}\,
\int_0^{t\wedge \tau_N^\epsilon}\int_{\R^l}
 D  f(x)^T a^\epsilon_s(X^\epsilon_s,x)\,\mu^\epsilon(ds,dx)\\-
\frac{1}{2\epsilon}\,\int_0^{t\wedge \tau_N^\epsilon}\int_{\R^l}
\text{tr}\,\bl(c^\epsilon_s(X^\epsilon_s,x)D^2f(x)
\br)\,\mu^\epsilon(ds,dx)
-\frac{1}{2\epsilon}\,\int_0^{t\wedge \tau_N^\epsilon}\int_{\R^l}
\norm{D   f(x)}_{c^\epsilon_s(X^\epsilon_s,x)}^2\,\mu^\epsilon(ds,dx)\Br)
\le1\,.
\end{multline*}
Let $g_2(u)$, where $u\in\R_+$,
be an $\R_+$-valued  nondecreasing
 $\mathbb{C}^2$-function with a bounded second derivative such
that $Dg_2(0)=D^2g_2(0)=0$ and $g_2(u)=u$ for $u\ge 1$\,.
For  given $\breve{N}>0$\,,
we  let  $f(x)=g_2((\abs{x}-\breve{N})^+)$\,, where   $x\in\R^l$\,.
By (\ref{eq:9}), if  
$\breve{N}$ is great enough, then for all
$\epsilon$ small enough,
$(x/\abs{x})^T a^\epsilon_{s\wedge
  \tau^\epsilon_N}(X^\epsilon_{s\wedge \tau^\epsilon_N},x)\le0$
provided  $\abs{x}\ge\breve{N}$\,.
Since $g_2$ is a nondecreasing  function,
\begin{equation*}
D  f(x)^T a^\epsilon_{s\wedge \tau^\epsilon_N}(X^\epsilon_{s\wedge \tau^\epsilon_N},x)=
Dg_2((\abs{x}-\breve{N})^+)(x/\abs{x})^T a^\epsilon_{s\wedge
  \tau^\epsilon_N}(X^\epsilon_{s\wedge \tau^\epsilon_N},x)\le0\,.   
\end{equation*}
In addition, like in  (\ref{eq:78}), 
$\text{tr}\,\bl(c^\epsilon_s(X^\epsilon_s,x)D^2\abs{x}\,
\br)
\le l\sqrt{l-1}\norm{c^\epsilon_s(X^\epsilon_s,x)}/\abs{x}$\,.
We
obtain that
\begin{multline}
  \label{eq:70}
        \mathbf{E}^\epsilon
\exp\Bl(-f(x^\epsilon_0)
-\frac{1}{\epsilon}\,
\int_0^{t\wedge \tau_N^\epsilon}\int_{\abs{x}>\breve{N}+1}
\frac{x^T}{\abs{x}}\, a^\epsilon_s(X^\epsilon_s,x)\,\mu^\epsilon(ds,dx)
\\-\frac{1}{2\epsilon}
\int_0^{t\wedge \tau_N^\epsilon}\int_{\breve{N}\le\abs{x}\le \breve{N}+1}\bl(
 \text{tr}\,\bl(c^\epsilon_s(X^\epsilon_s,x)D^2f(x)\br)
+\norm{D  f(x)}_{ c^\epsilon_s(X^\epsilon_s,x)}^2
\br)\,\mu^\epsilon(ds,dx)
\\-\frac{1}{2\epsilon}
\int_0^{t\wedge \tau_N^\epsilon}\int_{\abs{x}>\breve{N}+1}\bl(
\,\frac{\sqrt{l-1}}{\abs{x}}\,\,l\norm{c^\epsilon_s(X^\epsilon_s,x)}+
\,\norm{\frac{x}{\abs{x}}}^2_{c^\epsilon_s(X^\epsilon_s,x)}
\br)\,\mu^\epsilon(ds,dx)\Br)
\le1\,.
\end{multline}
Since $\norm{c^\epsilon_s(u,x)}$ is asymptotically bounded locally  in
$(s,u)$ and globally in $x$, see \eqref{eq:10},
 there exists  $\tilde{L}>0$ such that 
$\abs{\text{tr}\,\bl(c^\epsilon_{s\wedge \tau^\epsilon_N}(X^\epsilon_{s\wedge \tau^\epsilon_N},x)D^2f(x)\br)
+\norm{D 
  f(x)}^2_{ c^\epsilon_{s\wedge \tau^\epsilon_N}(X^\epsilon_{s\wedge \tau^\epsilon_N},x)}}\le \tilde{L}$ for all $s\le t$, all $\breve{N}$\,,
 and all $x$ such that $\abs{x}\in[\breve{N},\breve{N}+1]$\,,
provided 
 $\epsilon>0$ is small enough.
We can also assume that $\tilde{L}$ is an upper bound on
$\norm{c^\epsilon_{s\wedge \tau^\epsilon_N}(X^\epsilon_{s\wedge \tau^\epsilon_N},x)}$\,. We thus obtain from \eqref{eq:70}, 
on recalling that $\mu^\epsilon([0,t],\R^l)=t$, that provided
$\epsilon$ is small enough and $\breve{N}$ is
great enough,
\begin{equation*}
  \mathbf{E}^\epsilon
\exp\Bl(-f(x^\epsilon_0)
+\frac{1}{\epsilon}\,M^\epsilon\,
 \mu^\epsilon([0,t\wedge\tau_N^\epsilon],\{x:\,\abs{x}>\breve{N}+1\}
-\frac{3\tilde{L}t}{2\epsilon}
)
\Br)
\le1\,,
\end{equation*}
where 
\begin{equation*}
M^\epsilon=-\sup_{s\in[0,t]}\sup_{u\in\R^n:\,\abs{u}\le
N}\sup_{x\in\R^l:\,\abs{x}>\breve{N}+1}
\frac{x^T}{\abs{x}}\, a^\epsilon_s(u,x)>0\,.  
\end{equation*}
It follows that for arbitrary $\delta>0$,   all $\epsilon$ small
enough, and all $\breve{N}$ 
great enough,
\begin{multline*}
  \mathbf{P}^\epsilon\bl(\mu^\epsilon([0,t\wedge\tau_N^\epsilon],\{x\in\R^l:\,\abs{x}>\breve{N}+1\})>\delta\br)\le
  \mathbf{P}^\epsilon(\abs{x^\epsilon_0}>\breve{N})\\
+  \mathbf{E}^\epsilon
\exp\bl(\frac{M^\epsilon}{\epsilon}
  \mu^\epsilon([0,t\wedge\tau_N^\epsilon],\{x:\,\abs{x}>\breve{N}+1\})\br)
\ind_{\{\abs{x^\epsilon_0}\le \breve{N}\}}
\exp\bl(-\frac{M^\epsilon}{\epsilon}\,\delta\br)\\
\le \mathbf{P}^\epsilon(\abs{x^\epsilon_0}>\breve{N})+\exp\bl(\frac{3\tilde{L}t}{2\epsilon}-\frac{M^\epsilon\delta}{\epsilon}+g_2(0)\br)\,,
\end{multline*}
so by the facts that $\liminf_{\epsilon\to0}
M^\epsilon\to\infty$ and 
$\limsup_{\epsilon\to0}\mathbf{P}^\epsilon(\abs{x^\epsilon_0}\ge
\breve{N})^{\epsilon}\to0
$ 
as $\breve{N}\to\infty$,
and that  (\ref{eq:17}) holds, we obtain that
\begin{equation*}
  \lim_{\breve{N}\to\infty}\limsup_{\epsilon\to0}
\mathbf{P}^\epsilon\bl(\mu^\epsilon([0,t],\{x\in\R^l:\,\abs{x}>\breve{N}+1\})>\delta\br)^\epsilon=0.
\end{equation*}
Since $\abs{\mu^\epsilon_t(\Theta)-\mu^\epsilon_s(\Theta)}\le
\abs{t-s}$, for $\Theta\in \mathcal{B}(\R^l)$\,, 
the sequential exponential tightness of $\mu^\epsilon$ follows 
from part 1 of Lemma~\ref{le:exp_tight}.
\end{proof}
\begin{remark}
  Since
  $(X^\epsilon,\mu^\epsilon)$ is continuous in $\epsilon$ in
  distribution,
one can prove that $(X^\epsilon,\mu^\epsilon)$ is exponentially tight.
\end{remark}

\section{The equation for the large deviation function}
\label{sec:equat-large-devi}
In this section,  we derive an equation for  large deviation limit
points of $(X^\epsilon,\mu^\epsilon)$ that is to be used for
identifying the large deviation function.
For $0=t_0<t_1<\ldots<t_i$, let
\begin{equation}
  \label{eq:15}
\lambda(t,X)=\sum_{j=1}^i
\lambda_j(X_{t_{j-1}})\ind_{[t_{j-1},t_j)}(t),  
\end{equation}
 where $X=(X_s,\,s\in\R_+)\in
\mathbb{C}(\R_+,\R^n)$ and the
functions $\lambda_j(u)$\,, for $u\in\R^n$\,,
 are $\R^n$-valued and  continuous.
 We define
\begin{equation}
  \label{eq:84}
\int_0^t\lambda(s,X)\,dX_s=\sum_{j=1}^i\lambda_i(X_{t_{j-1}\wedge
  t})^T
(X_{t\wedge
  t_{j}}-X_{t\wedge t_{j-1}})\,.
\end{equation}
Let $f(t,u,x)$ represent a $\mathbb{C}^{1,2,2}(\R_+\times  \R^n
\times \R^l)$-function with compact support
 in $x$ locally uniformly in $(t,u)$ and
let, with $(X,\mu)\in \bC(\R_+,\R^n)\times
\mathbb{C}_\uparrow(\R_+,\mathbb{M}(\R^l))$\,,
 \begin{multline}
\label{eq:24}
  U_{t}^{\lambda(\cdot),f}(X,\mu) =
\int_0^t\lambda(s,X)\,dX_s-\int_0^t\int_{\R^l} 
 \lambda(s,X)^T A_s(X_s,x)\,\mu(ds,dx)
\\-\int_0^{t}\int_{\R^l} D _x f(s,X_s,x)^T 
a_s(X_s,x)\,\mu(ds,dx)
-\frac{1}{2}\,\int_0^t\int_{\R^l} \text{tr}\,\bl(
c_s(X_s,x)D^2_{xx}f(s,X_s,x)\br)\,\mu(ds,dx)\\-\frac{1}{2}\,\int_0^t\int_{\R^l} 
 \norm{\lambda(s,X)}_{C_s(X_s,x)}^2\,\mu(ds,dx)
-\frac{1}{2}\,\int_0^t\int_{\R^l}
 \norm{ D _xf(s,X_s,x)}_{c_s(X_s,x)}^2\,\mu(ds,dx)
\\-\int_0^t\int_{\R^l}  
 \lambda(s,X)
^T G_s(X_s,x)D _x
f(s,X_s,x)\,\mu(ds,dx)\,.
\end{multline}
Under condition 2.1,  $U_{t}^{\lambda(\cdot),f}(X,\mu)$
is a continuous function of $(X,\mu)$\,.

Let  $\tau(X,\mu)$ represent a continuous function of $(X,\mu)\in
\bC(\R_+,\R^n)\times 
\mathbb{C}_\uparrow(\R_+,\mathbb{M}(\R^l))$ that  is also
 a  stopping time relative to the flow
$\mathbf{G}=(\mathcal{G}_t,\,t\in\R_+)$ on $\bC(\R_+,\R^n)\times
\mathbb{C}_\uparrow(\R_+,\,\mathbb{M}(\R^l))$, 
where the $\sigma$-algebra $\mathcal{G}_t$ is
generated 
by the mappings  $X\to X_s$ and $\mu\to\mu_s$ for $s\le t$\,. (We note
 that the flow  $\mathbf{G}$ is not right continuous, so $\tau$ is
a strict stopping time, see Jacod
and Shiryaev \cite{jacshir}.) Suppose also
that
$X_{t\wedge\tau(X,\mu)}$ is a bounded function of $(X,\mu)$\,.

\begin{theorem}
  \label{the:equation}
Suppose that
 conditions 2.1, \eqref{eq:10}, \eqref{eq:10a}, \eqref{eq:8}, and
    \eqref{eq:14}
 hold. If $\tilde{\mathbf{I}}$ is 
  a large deviation limit point of $(X^\epsilon,\mu^\epsilon)$
  for rate $1/\epsilon$ as $\epsilon\to0$\,, then
\begin{equation}
  \label{eq:21}
  \sup_{(X,\mu)\in \bC(\R_+,\R^n)\times 
\mathbb{C}_\uparrow(\R_+,\mathbb{M}(\R^l))}\bl(U_{t\wedge\tau(X,\mu)}^{\lambda(\cdot),f}(X,\mu)-\tilde{\mathbf{I}}(X,\mu)\br)=0\,.
\end{equation}
\end{theorem}
\begin{proof}
 The process $(\lambda(t,X^\epsilon),\,t\in\R_+)$ is
$\mathbf{F}^\epsilon$-adapted so that by \eqref{eq:1} and
\eqref{eq:84},
\begin{equation}
     \label{eq:43}
\int_0^t \lambda(s,X^\epsilon)\,dX^\epsilon_s=\int_0^t \lambda(s,X^\epsilon)^T A^\epsilon_s(X^\epsilon_s,x^\epsilon_s)\,ds
+\sqrt{\epsilon}\,\int_0^t
\lambda(s,X^\epsilon)^T B^\epsilon_s(X^\epsilon_s,x^\epsilon_s)\,dW^\epsilon_s\,.
\end{equation}

By \eqref{eq:1}, \eqref{eq:5}, and  It\^o's lemma,
\begin{multline}
  \label{eq:2}
  f(t,X^\epsilon_t,x^\epsilon_{t})=f(0,X^\epsilon_0,x^\epsilon_0)
+\int_0^t\frac{\partial
    f(s,X^\epsilon_s,x^\epsilon_s)}{ \partial s}\,ds+
\int_0^{t} D _u f(s,X^\epsilon_s,x^\epsilon_{s})^T 
A^\epsilon_s(X^\epsilon_s,x^\epsilon_{s})\,ds\\
+\sqrt{\epsilon}\,
\int_0^t D _u f(s,X^\epsilon_s,x^\epsilon_{s})^T
B^\epsilon_s(X^\epsilon_s,x^\epsilon_{s})\,dW^\epsilon_s
+\frac{1}{\epsilon}\,
\int_0^{t} D _x f(s,X^\epsilon_s,x^\epsilon_{s})^T 
a^\epsilon_s(X^\epsilon_s,x^\epsilon_{s})\,ds\\+
\frac{1}{\sqrt{\epsilon}}\,
\int_0^t D _x
 f(s,X^\epsilon_s,x^\epsilon_{s})^T
 b^\epsilon_s(X^\epsilon_s,x^\epsilon_{s})\,dW^\epsilon_s+
\frac{\epsilon}{2}\,\int_0^t \text{tr}\,\bl(
C^\epsilon_s(X^\epsilon_s,x^\epsilon_{s})D^2_{uu}f(s,X^\epsilon_s,x^\epsilon_{s})\br)\,ds\\
+\frac{1}{2\epsilon}\,\int_0^t \text{tr}\,\bl(
c^\epsilon_s(X^\epsilon_s,x^\epsilon_{s})D^2_{xx}f(s,X^\epsilon_s,x^\epsilon_{s})\br)\,ds
+
\int_0^t \text{tr}\,\bl(G^\epsilon_s(X^\epsilon_s,x^\epsilon_{s})D^2_{ux}
f(s,X^\epsilon_s,x^\epsilon_{s})\br)\,ds
\,,
\end{multline}
where
$  G_s^\epsilon(u,x)=B_s^\epsilon(u,x)b_s^\epsilon(u,x)^T$\,.
We denote 
\begin{multline*}
  U_{t}^\epsilon(X,\mu)=
\int_0^t\lambda(s,X)\,dX_s-\int_0^t\int_{\R^l} 
 \lambda(s,X)^T A^\epsilon_s(X_s,x)\,\mu(ds,dx)
\\-\int_0^{t}\int_{\R^l} D _x f(s,X_s,x)^T 
a^\epsilon_s(X_s,x)\,\mu(ds,dx)
-\frac{1}{2}\,\int_0^t\int_{\R^l} \text{tr}\,\bl(
c^\epsilon_s(X_s,x)D^2_{xx}f(s,X_s,x)\br)\,\mu(ds,dx)\\-\frac{1}{2}\,\int_0^t\int_{\R^l} 
 \norm{\lambda(s,X)}_{C_s^\epsilon(X_s,x)}^2\,\mu(ds,dx)
-\frac{1}{2}\,\int_0^t\int_{\R^l} \norm{D _x
 f(s,X_s,x)}_{c_s^\epsilon(X_s,x)}^2\,\mu(ds,dx)
\\-\int_0^t\int_{\R^l}  
 \lambda(s,X)
^T G^\epsilon_s(X_s,x)D _x
f(s,X_s,x)\,\mu(ds,dx)\,
\end{multline*}
and
\begin{multline*}
  V_{t}^\epsilon(X,\mu)=
  f(t,X_t,x_{t})-
f(0,X_0,x_0)
-\int_0^t\int_{\R^l}\frac{\partial
    f(s,X_s,x)}{ \partial s}\,\mu(ds,dx)\\-
\int_0^{t}\int_{\R^l} D _u f(s,X_s,x)^T 
A^\epsilon_s(X_s,x)\,\mu(ds,dx)-
\frac{\epsilon}{2}\,\int_0^t\int_{\R^l} \text{tr}\,\bl(
C^\epsilon_s(X_s,x)D^2_{uu}f(s,X_s,x)\br)\,\mu(ds,dx)
\\-
\int_0^t\int_{\R^l} \text{tr}\,\bl(
G^\epsilon_s(X_s,x)D^2_{ux}f(s,X_s,x)\br)\,\mu(ds,dx)
-\frac{\epsilon}{2}\,\int_0^t\int_{\R^l} \norm{D _u
 f(s,X_s,x)}_{C^\epsilon_s(X_s,x)}^2\,\mu(ds,dx)
\\-\int_0^t \int_{\R^l} 
 \lambda(s,X)
^T C^\epsilon_s(X_s,x)D _u
f(s,X_s,x)\,\mu(ds,dx)
-\int_0^t\int_{\R^l} D _u
 f(s,X_s,x)
^T G^\epsilon_s(X_s,x)D _x
f(s,X_s,x)\,\mu(ds,dx)\,.
\end{multline*} 
Since the function $\lambda(s,u)$ is locally bounded,
the function  $f(s,u,x)$ 
and its derivatives  are locally bounded and are
of compact support in $x$, conditions  \eqref{eq:10}, \eqref{eq:10a},
\eqref{eq:8},
and \eqref{eq:93} hold, and $X_{t\wedge
  \tau(X,\mu)}$ is bounded,
we have that there exists number $R(t)>0$ such that 
for  all $\epsilon$ small enough uniformly over $(X,\mu)$,
\begin{equation}
  \label{eq:104}
       \abs{  U_{t\wedge \tau(X,\mu)}^\epsilon(X,\mu)}+
 \abs{  V_{t\wedge \tau(X,\mu)}^\epsilon(X,\mu)}\le R(t)\,.
\end{equation}
Since $X^\epsilon_s$ and $\mu^\epsilon_s$ are
$\mathcal{F}^\epsilon_s$-measurable,    
$\tau(X^\epsilon,\mu^\epsilon)$ is a stopping
time relative to
$\mathbf{F}^\epsilon$\,.
By (\ref{eq:43}), (\ref{eq:2}), and \eqref{eq:104}, the process
$\bl(\exp\bl((1/\epsilon)
U^\epsilon_{{t\wedge\tau(X^\epsilon,\mu^\epsilon)}}(X^\epsilon,\mu^\epsilon)
+
V^\epsilon_{{t\wedge\tau(X^\epsilon,\mu^\epsilon)}}(X^\epsilon,\mu^\epsilon)
\br),t\in\R_+\br)$
is a bounded $\mathbf{F}^\epsilon$-martingale, so
 \begin{equation*}
 \mathbf{E}^\epsilon\exp\Bl(\frac{1}{\epsilon}\,
U^\epsilon_{{t\wedge\tau(X^\epsilon,\mu^\epsilon)}}(X^\epsilon,\mu^\epsilon)
+ V^\epsilon_{{t\wedge\tau(X^\epsilon,\mu^\epsilon)}}(X^\epsilon,\mu^\epsilon)\Br)=1\,.
 \end{equation*}

Since the function $f(s,u,x)$ is of compact
support in $x$, the convergence hypotheses in \eqref{eq:14} 
and the bound in \eqref{eq:104} imply 
 that $U_{{t\wedge\tau(X,\mu)}}^\epsilon(X,\mu)
\to U_{{t\wedge\tau(X,\mu)}}^{\lambda(\cdot),f}(X,\mu)$ as
$\epsilon\to0$ uniformly over compact sets. By Theorem~\ref{the:ld_limit},
$\sup_{(X,\mu)\in \mathbb{C}(\R_+,\R^n)\times
\mathbb{C}_\uparrow(\R_+,\mathbb{M}(\R^l))}
(U^{\lambda(\cdot),f}_{{t\wedge\tau(X,\mu)}}(X,\mu)-\tilde{\mathbf{I}}(X,\mu))=0$\,.
 \end{proof}
 \begin{remark}
   One can see that  there exists compact 
$K\subset \bC(\R_+,\R^n)\times 
\mathbb{C}_\uparrow(\R_+,\mathbb{M}(\R^l))$ such that 
$  \sup_{(X,\mu)\in K}
\bl(U_{t\wedge\tau(X,\mu)}^{\lambda(\cdot),f}(X,\mu)-\tilde{\mathbf{I}}(X,\mu)\br)=0$\,.

 \end{remark}
\section{Regularity  properties}
\label{sec:regul-prop}
Let $\tilde{\mathbf{I}}$ represent
 a large deviation limit point of
$(X^\epsilon,\mu^\epsilon)$ for rate $1/\epsilon$ as $\epsilon\to0$ such that
$\tilde{\mathbf I}(X,\mu)=\infty$  unless
$X_0=\hat
u$\,, where $\hat u$ is a preselected element of $\R^n$\,.
Let, as in Theorem~\ref{the:id}, for $(X,\mu)\in \bC(\R_+,\R^n)\times
\bC_\uparrow(\R_+,\mathbb{M}(\R^l))$\,, 
\begin{equation}
  \label{eq:23}
  \mathbf{I}^{\ast\ast}(X,\mu)=
\sup_{\lambda(\cdot),f,t,\tau}U_{t\wedge \tau(X,\mu)}^{\lambda(\cdot),f}(X,\mu)\,,
\end{equation}
with  the supremum  being taken over $\lambda(t,X)$, 
 $f(t,u,x)$, and  $\tau(X,\mu)$ satisfying the requirements of
 Theorem~\ref{the:equation}
  and over
 $t\ge0$\,. We note that, under Condition 2.1,
 $\mathbf{I}^{\ast\ast}(X,\mu)$ is a lower
 semicontinuous function of $(X,\mu)$ and that
by Theorem~\ref{the:equation}, 
\begin{equation}
  \label{eq:13}
\mathbf{I}^{\ast\ast}(X,\mu)\le
\tilde{\mathbf{I}}(X,\mu)\,.
\end{equation}
The rest of the paper is concerned mostly with proving that equality
prevails in 
\eqref{eq:13},
 provided
$X_0=\hat u$\,. Since  the case where 
$\mathbf{I}^{\ast\ast}(X,\mu)<\infty$ needs to be considered only,  in this section
we undertake a study of the
properties   of $(X,\mu)$
such that $\mathbf{I}^{\ast\ast}(X,\mu)<\infty$\,. We then  prove that
if $\mathbf{I}^{\ast\ast}(X,\mu)<\infty$ and $X_0=\hat u$\,, then
$\mathbf{I}^{\ast\ast}(X,\mu)
=\mathbf I(X,\mu)$\,, where 
 $\mathbf I(X,\mu)$ is given in the statements
of Theorem~\ref{the:ldp} and Proposition \ref{cor:ldp}
with $\mathbf{I}_0(\hat u)=0$\,, see Theorem \ref{le:vid}.
We assume throughout conditions \ref{con:coeff}, 
 \ref{con:positive},
\eqref{eq:116},  \eqref{eq:122a} and \eqref{eq:82}  to hold.
\begin{lemma}
  \label{le:differ}
If $ \mu\in \bC_\uparrow(\R_+,\mathbb{M}(\R^l))$\,, then $\mu$ is of the form
$  \mu(ds,dx)=\nu_s(dx)\,ds$, 
where $\nu_s(dx)$ is a transition probability kernel from $\R_+$ to
$\R^l$\,.
If $ (X,\mu)\in \bC(\R_+,\R^n)\times
\bC_\uparrow(\R_+,\mathbb{M}(\R^l))$ is such that
$\mathbf{I}^{\ast\ast}(X,\mu)<\infty$, then 
$X$ is absolutely continuous with respect to Lebesgue measure.
\end{lemma}
\begin{proof}
As $\mu\in \bC_\uparrow(\R_+,\mathbb{M}(\R^l))$\,, we have that
$\mu(ds,dx)=\nu_s(dx)\mu(ds,\R^l)$, where
$\nu_s(dx)$ is a transition kernel
from $\R_+$ to $\R^l$,
see, e.g., 
Theorem 8.1 on p.502 of Ethier and Kurtz \cite{EthKur86}.

Since $\mu(ds,\R^l)$ is Lebesgue measure on $\R_+$\,, 
$\mu(ds,dx)=\nu_s(dx)\,ds$\,.

 On taking $f=0$ in \eqref{eq:24} and
assuming 
$\lambda(s,X)$ not to depend on $X$, so the piece of notation 
$\lambda(s)$ can be used instead, we have by
  \eqref{eq:24}, (\ref{eq:23}), and 
the part of the lemma just proved  that if
$\mathbf{I}^{\ast\ast}(X,\mu)<\infty$, then
 \begin{multline*}
\int_0^t\lambda(s)\,dX_s\le\int_0^t\int_{\R^l} 
 \lambda(s)^T A_s(X_s,x)\,\nu_s(dx)\,ds+
\frac{1}{2}\,\int_0^t\int_{\R^l} 
 \norm{\lambda(s)}_{ C_s(X_s,x)}^2\,\nu_s(dx)\,ds+\mathbf{I}^{\ast\ast}(X,\mu)\,.
 \end{multline*}
Replacing $\lambda(s)$ with $\delta\lambda(s)$\,, where $\delta>0$\,,
 dividing through
by $\delta$, and minimising the righthand side over $\delta$
obtains that
 \begin{multline*}
\int_0^t\lambda(s)\,dX_s\le\int_0^t\int_{\R^l} 
 \lambda(s)^T A_s(X_s,x)\,\nu_s(dx)\,ds+
\sqrt{2}\sqrt{\mathbf{I}^{\ast\ast}(X,\mu)}\,\sqrt{\int_0^t\int_{\R^l} 
\norm{ \lambda(s)}^2_{ C_s(X_s,x)}\,\nu_s(dx)\,ds\,
}\,.
 \end{multline*}
It follows that $X$ is absolutely continuous with respect to $ds$\,.

\end{proof}
By Lemma~\ref{le:differ}, if $\mathbf{I}^{\ast\ast}(X,\mu)<\infty$, then 
 (\ref{eq:24})  takes the form
\begin{multline}
    \label{eq:25}
  U_{t}^{\lambda(\cdot),f}(X,\mu)=
\int_0^t  \lambda(s,X)^T\dot{X_s}\,ds
-\int_0^t\int_{\R^l} 
 \lambda(s,X)^T A_s(X_s,x)\,\nu_s(dx)\,ds
\\-\int_0^{t}\int_{\R^l} D _x f(s,X_s,x)^T 
a_s(X_s,x)\,\nu_s(dx)\,ds
-\frac{1}{2}\,\int_0^t\int_{\R^l} \text{tr}\,\bl(
c_s(X_s,x)D^2_{xx}f(s,X_s,x)\br)\,\nu_s(dx)\,ds\\
-\frac{1}{2}\,\int_0^t\int_{\R^l} 
 \norm{\lambda(s,X)}_{C_s(X_s,x)}^2\,\nu_s(dx)\,ds
-\frac{1}{2}\,\int_0^t\int_{\R^l} 
 \norm{D _xf(s,X_s,x)}_{c_s(X_s,x)}^2\,\nu_s(dx)\,ds
\\-\int_0^t\int_{\R^l}  
 \lambda(s,X)
^T G_s(X_s,x)D _x
f(s,X_s,x)\,\nu_s(dx)\,ds\,.
\end{multline}
The next step is to show that $\nu_s(dx)$ has to be absolutely
continuous with
respect to $dx$ and establish its integrability properties\,. We need,
however, to lay the groundwork. The proofs of the following two lemmas
are relegated to the appendix.
The first one is essentially due to R\"ockner and Zhang
\cite[pp.204,205]{RocZha92}, \cite{RocZha94}, see also Bogachev, Krylov, and R\"ockner
\cite{BogKryRoc96}. 
The second one 
addresses  regularity of the invariant measures of diffusions and may
be of interest in its own right.
\begin{lemma}
  \label{le:spaces}
Let $d\in\N$ and let   $O$ represent either $\R^d$ or an open ball in
$\R^d$\,.
If $m(x)$ is an $\R_+$-valued measurable function  on $\R^d$ such that
 $m\in \mathbb{W}^{1,1}_{\text{loc}}(\R^d)$ and 
$\sqrt{m}\in \mathbb{W}^{1,2}(O)$\,, then
 $\mathbb H^{1,2}(O,m(x)\,dx)=\mathbb W^{1,2}(O,m(x)\,dx)$\,.
\end{lemma}
\begin{lemma}
  \label{le:lok_kwadr}
For $d\in\N$ and $x\in\R^d$, let $c(x)$ represent a 
 locally Lipschitz continuous function
with values in the set of symmetric positive definite  $d\times
d$-matrices and let $b(x)$ represent an $\R^d$-valued measurable function. 
 Suppose $m(x)$ is a 
probability density on $\R^d$ such that $m(\ln m)^2\in
\mathbb{L}^1_{\text{loc}}(\R^d)$\,, $b\in
\mathbb{L}^2_{\text{loc}}(\R^d,\R^d,m(x)\,dx)$\,, and
\begin{equation*}
  \int_{\R^d}\text{tr}\,(c(x)D^2p(x))m(x)\,dx+
\int_{\R^d} D p(x)^Tb(x)m(x)\,dx=0
\end{equation*}
for all $p\in \mathbb{C}_0^\infty (\R^d)$\,, where we assume that
 $0(\ln0)^2=0$\,.

Then $m\in \mathbb{W}^{1,1}_{\text{loc}}(\R^d)$ and $\sqrt{m}\in
\mathbb{W}^{1,2}_{\text{loc}}(\R^d)$\,. Furthermore, given open ball $S$ from
$\R^d$\,, there exists constant $M$ which depends on $S$, on the Lipschitz
constant of $c(x)$ on $S$, and on $\inf_{x\in S}x^Tc(x)x/\abs{x}^2$
only, such that
\begin{equation}
  \label{eq:108}
       \int_{S}     \frac{\abs{D m(x)}^2}{m(x)}\,dx\le
M\bl(1+\int_{S}(\ln m(x))^2m(x)\,dx
+\int_{S}\abs{ b(x)}^2m(x)\,dx\br)\,.
\end{equation}

\end{lemma}
The latter lemma  is a local version of the result by Bogachev, Krylov, and
R\"ockner \cite{BogKryRoc96}  that 
if $b\in
\mathbb{L}^2(\R^d,m(x)\,dx)$
then $\sqrt{m}\in
\mathbb{W}^{1,2}(\R^d)$\,,  see also 
 Metafune, Pallara, and Rhandi
 \cite{MetPalRha05}\,. 


For the next lemma, we recall that,
according to our conventions,   $q'=q/(q-1)$\,, provided  $q>1$\,.
\begin{lemma}
  \label{le:density}
Suppose that $\mathbf{I}^{\ast\ast}(X,\mu)<\infty$\,, where
 $\mu(ds,dx)=\nu_s(dx)\,ds$\,. Then, for almost all $s$\,,
the
transition kernel
 $\nu_s(dx)$ is absolutely continuous with respect to
Lebesgue measure,  the density $m_s(x)=\nu_s(dx)/dx$  
 is an element
 of   $\mathbb L_{\text{loc}}^\beta(\R^l)$ for all
 $\beta\in[1,l/(l-1))$ and is an element of $\mathbb{W}^{1,\alpha}
 _{\text{loc}}(\R^l)$ for all $\alpha\in[1,2l/(2l-1))$\,, and
 $\sqrt{m_s(\cdot)}\in
\mathbb{W}^{1,2}_{\text{loc}}(\R^l)$\,.
Furthermore, for arbitrary $t>0$ and
 open ball $S\subset \R^l$\,,
 \begin{equation}
   \label{eq:106}
      \int_0^t\int_S\frac{\abs{Dm_s(x)}^2}{m_s(x)}\,dx\,ds<\infty\,.
 \end{equation}
If, in addition, $\sqrt{m_s(\cdot)}\in
\mathbb W^{1,2}(\R^l)$, then 
$Dm_s(\cdot)/m_s(\cdot)\in \mathbb{L}_0^{1,2}(\R^l,\R^l,m_s(x)\,dx)$\,.

If   $\kappa>0$, $q\ge2$\,, and $q> l$\,, then
\begin{equation}
  \label{eq:121}
     \sup_{( X, \mu):\,\mathbf{    I}^{\ast\ast}( X,\mu)\le \kappa}
\, \int_0^t\int_S m_s(x)^{q'}\,dx\,ds<\infty\,.
\end{equation}
\end{lemma}
\begin{proof}
By \eqref{eq:23} and \eqref{eq:25} with  $\lambda(s,X)=0$ and
$f(s,u,x)=\phi(s,x)$, where $\phi\in \mathbb{C}^{1,2}(\R_+\times\R^l)$
and the support of $\phi$ in $x$ is  bounded locally uniformly in $s$\,,
\begin{multline*}
  -\frac{1}{2}\,\int_0^t\int_{\R^l} \text{tr}\,\bl(
c_s(X_s,x)D ^2_{xx}
\phi(s,x)\br)\,\nu_s(dx)\,ds-
\int_0^{t}\int_{\R^l} D _x \phi(s,x)^T 
a_s(X_s,x)\,\nu_s(dx)\,ds\\\le \mathbf{    I}^{\ast\ast}(X,\mu)+\frac{1}{2}\,\int_0^t\int_{\R^l}
\norm{D _x \phi(s,x)}_{c_s(X_s,x)}^2\,\nu_s(dx)\,ds\,.
\end{multline*}
Replacing $\phi(s,x)$ with $\delta\phi(s,x)$, where $\delta>0$,
dividing through by $\delta$\,,
and minimising the righthand side over $\delta$ yields

\begin{multline}
\label{eq:31}
  -\frac{1}{2}\,\int_0^t\int_{\R^l} \text{tr}\,\bl(
c_s(X_s,x)D ^2_{xx}
\phi(s,x)\br)\,\nu_s(dx)\,ds-\int_0^{t}\int_{\R^l} D _x \phi(s,x)^T 
a_s(X_s,x)\,\nu_s(dx)\,ds\\\le\sqrt{2}\,
\mathbf{    I}^{\ast\ast}(X,\mu)^{1/2}
\bl(\int_0^t\int_{\R^l}
\norm{D _x \phi(s,x)}_{c_s(X_s,x)}^2\,\nu_s(dx)\,ds\br)^{1/2}\,.
\end{multline}
Let $\mathbb{L}_0^{1,2}([0,t]\times \R^l,\R^l,
c_s(X_s,x),\nu_s(dx)\,ds)$ denote the closure in 
$\mathbb{L}^2([0,t]\times \R^l,\R^l,c_s(X_s,x),\nu_s(dx)\,ds)$ of the
space of  functions $D_x\phi$\,.
By \eqref{eq:31},
the lefthand side  extends to a continuous functional
$T_t(g)$ on \linebreak $\mathbb{L}_0^{1,2}([0,t]\times \R^l,\R^l,
c_s(X_s,x),\nu_s(dx)\,ds)$\,. By the Riesz representation theorem,
there exists a unique $\psi\in \mathbb{L}_0^{1,2}([0,t]\times \R^l,
\R^l,c_s(X_s,x),\nu_s(dx)\,ds)$ such that 
\begin{equation*}
T_t(g)=\int_0^t\int_{{\R^l}}g(s,x)^T c_s(X_s,x)\psi(s,x)\,
\nu_s(dx)\,ds\,,
\end{equation*}
for all $g\in \mathbb{L}_0^{1,2}([0,t]\times \R^l,\R^l,
c_s(X_s,x),\nu_s(dx)\,ds)$\,,
 and 
\begin{equation}
  \label{eq:34}
  \bl(\int_0^t\int_{{\R^l}}\norm{\psi(s,x)}_{c_s(X_s,x)}^2
\nu_s(dx)\,ds\br)^{1/2}\le\sqrt{2} \mathbf{    I}^{\ast\ast}(X,\mu)^{1/2}\,.
\end{equation}
By uniqueness, $\psi$ can be extended to 
a function on $\R_+\times {\R^l}$ so that for all $t>0$,
\begin{multline}\label{eq:114}
    -\frac{1}{2}\,\int_0^t\int_{\R^l} \text{tr}\,\bl(
c_s(X_s,x)D ^2_{xx}
\phi(s,x)\br)\,\nu_s(dx)\,ds-\int_0^{t}\int_{\R^l} D _x \phi(s,x)^T 
a_s(X_s,x)\,\nu_s(dx)\,ds\\=
\int_0^t\int_{{\R^l}}D _x\phi(s,x)^T c_s(X_s,x)\psi(s,x)\,
\nu_s(dx)\,ds\,.
\end{multline}
It follows that for almost all $s$ and for  all 
 $h\in  \mathbb{C}^2_0({\R^l})$\,,
\begin{multline}
  \label{eq:28}
-\frac{1}{2}\,\int_{\R^l} \text{tr}\,\bl(c_s(X_s,x)D^2
h(x)\br)\,\nu_s(dx)=\int_{\R^l} Dh(x)^T 
a_s(X_s,x)\,\nu_s(dx)\\+\int_{{\R^l}}
 D h(x)^T c_s(X_s,x)\psi(s,x)\,\nu_s(dx)\,.
\end{multline}
Since $\psi\in \mathbb{L}_0^{1,2}([0,t]\times \R^l,
\R^l,c_s(X_s,x),\nu_s(dx)\,ds)$\,,  we have that, for almost all $s$\,,
 $\psi(s,\cdot)$ belongs to the closure
 of the set of the $D _xh$ in
 $\mathbb{L}^2(\R^l,\R^l,c_s(X_s,x),\nu_s(dx))$\,.
In particular, $\int_{\R^l}\abs{\psi(s,x)}^2\,\nu_s(dx)<\infty$\,.
Since $a_s(X_s,\cdot)$ and $\psi(s,\cdot)$ are locally integrable with respect
to $\nu_s(dx)$ and $c_s(X_s,\cdot)$ is uniformly positive definite and
is of class $\mathbb C^1$\,, \eqref{eq:28} and
 Theorem 2.1 in Bogachev, Krylov, and R\"ockner \cite{BogKryRoc01}
 imply that the
measure $\nu_s(dx)$ has  density $m_s(x)$ with respect to  Lebesgue measure
which belongs to  $L_{\text{loc}}^\beta(\R^l)$ 
for all  $\beta<l'$\,. 
It follows , since $a_s(X_s,\cdot)$ and $c_s(X_s,\cdot)$ are
 locally bounded and 
$\int_{\R^l}\abs{\psi(s,x)}^2\,\nu_s(dx)<\infty$, that, for 
arbitrary open ball $S$ in $\R^l$\,,
  there exists  $M>0$ such that
for all  
$h\in \mathbb{C}_0^2(S)$
\begin{equation*}
  \abs{\int_{S} \text{tr}\,\bl(c_s(X_s,x)D^2
h(x)\br)\,m_s(x)\,dx}\le M\norm{D h}_{\mathbb{L}^{2\beta'}(S,\R^l)}\,.
\end{equation*}

\noindent 
Since $c_s(u,\cdot)$ is uniformly positive definite and is of class 
$\mathbb C^1$\,,
by Theorem 6.1 in Agmon \cite{Agm59}, 
the  density $m_s(\cdot)$  belongs to
$\mathbb{W}_{\text{loc}}^{1,\alpha}(S)$ for all $\alpha< 2l/(2l-1)$.
The inclusion $\sqrt{m_s(\cdot)}\in \mathbb{W}^{1,2}_{\text{loc}}(\R^l)$ follows from
 Lemma~\ref{le:lok_kwadr} and  \eqref{eq:28}. 
For the inequality \eqref{eq:106}, we also recall \eqref{eq:108} and
\eqref{eq:34}.

Let $\mathbb
L^{1,2}(\R^l,\R^l,c_s(X_s,x),m_s(x)\,dx)$ 
 represent the closure  in $ \mathbb
L^{2}(\R^l,\R^l,c_s(X_s,x),m_s(x)\,dx)$ of the set of gradients of $\mathbb
C^1(\R^l)$--functions  such that those gradients belong to
$ \mathbb L^2(\R^l,\R^l,c_s(X_s,x),m_s(x)\,dx)$\,.
Let $\varphi_M(y)$ be a smooth function such that $\varphi_M(y)=y$
when $\abs{y}\le M$ and $\varphi_M(y)=(y/\abs{y})(M+1)$ when
$\abs{y}\ge M+1$\,.
Let $\eta(x)$ represent a $[0,1]$-valued  continuously differentiable 
nonincreasing function defined for $x\ge 0$ 
such that $\eta(x)=1$ for $x\in[0,1]$ and $\eta(x)=0$ for $x\ge 2$\,.
Let $\eta_r(x)=\eta(\abs{x}/r)$ where $x\in\R^l $ and $r>0$\,.
 By associating
 with $h\in \mathbb
L^{1,2}(\R^l,\R^l,c_s(X_s,x),m_s(x)\,dx)$ the function 
$\varphi_M(h(x))\eta_r(x)$ and taking limits  as $r\to\infty$ and $M\to\infty$\,,  one can show
that $
\mathbb
L^{1,2}(\R^l,\R^l,c_s(X_s,x),m_s(x)\,dx)=\mathbb
L^{1,2}_0(\R^l,\R^l,c_s(X_s,x),m_s(x)\,dx)$\,.   
The property that 
$Dm_s(\cdot)/m_s(\cdot)\in \mathbb{L}_0^{1,2}(\R^l,\R^l,m_s(x)\,dx)$
when $\sqrt{m_s(\cdot)}\in
\mathbb W^{1,2}(\R^l)$ now follows from Lemma \ref{le:spaces}.

We now adapt the proof of Theorem 2.1 in Bogachev, Krylov, and
R\"ockner \cite{BogKryRoc01} in order to obtain the bound in
\eqref{eq:121}.  Let  $S_1$ represent an open ball
which contains $S$\,.
By \eqref{eq:34}, \eqref{eq:114}, and local boundedness of 
$a_s(u,x)$ and $c_s(u,x)$\,,
 assuming that  $\phi(s,x)$ in \eqref{eq:114} is
supported by  $S_1$ in $x$ for all $s\in[0,t]$, there exists $L_1>0$
such that for all $(X,\mu)$ with $\mathbf{
  I}^{\ast\ast}(X,\mu)\le \delta$,
\begin{equation}
  \label{eq:115}
  \abs{\int_0^t\int_{S_1} \text{tr}\,\bl(
c_s(X_s,x)D ^2_{xx}
\phi(s,x)\br)\,m_s(x)\,dx\,ds}\le 
L_1\Bl(\int_0^{t}\sup_{x\in S_1}\abs{D _x \phi(s,x)}^2\,ds\Br)^{1/2}\,.
\end{equation}
An approximation argument shows that one may assume that $\phi(s,x)$
is measurable in $(s,x)$ and is of class $\mathbb C^2$ in $x$\,.
Let $\zeta(x)$ represent
 a $\mathbb{C}_0^\infty$-function on $\R^l$ with support in
$S_1$ that equals 1 on $S$ and
 let $\varphi(s,x)$ be a measurable function that is of class 
 $\mathbb{C}^\infty$ in $x$\,. On letting
 $\phi(s,x)=\zeta(x)\varphi(s,x)$
in  \eqref{eq:115}, 
we have  that there exists $L_2>0$ such that for all
$\varphi(s,x)$ and all  $(X,\mu)$ that satisfy the inequality $\mathbf{
  I}^{\ast\ast}(X,\mu)\le \delta$\,,
\begin{multline*}
      \abs{\int_0^t\int_{S_1} \text{tr}\,\bl(
c_s(X_s,x)D ^2_{xx}
\varphi(s,x)\br)\,\zeta(x)m_s(x)\,dx\,ds}\le 
L_2\Bl(\int_0^{t}\bl(\sup_{x\in S_1}\abs{\varphi(s,x)}^2+
\sup_{x\in S_1}\abs{D _x \varphi(s,x)}^2\br)\,ds\Br)^{1/2}\,.
\end{multline*}
By  Sobolev's imbedding, $\mathbb{W}^{2,q}(S_1)$ is continuously imbedded into
$\mathbb{W}^{1,\infty}(S_1)$ provided $q>l$ (see, e.g., Theorem 4.12
on p.85 in Adams and Fournier \cite{MR56:9247}), hence,
\begin{equation}
  \label{eq:117}
      \abs{\int_0^t\int_{S_1} \text{tr}\,\bl(
c_s(X_s,x)D ^2_{xx}
\varphi(s,x)\br)\,\zeta(x)m_s(x)\,dx\,ds}\le 
L_3\Bl(\int_0^{t}\norm{\varphi(s,\cdot)}^2_{\mathbb{W}^{2,q}(S_1)}\,ds\Br)^{1/2}\,,
\end{equation}
where $L_3>0$\,.
The latter inequality extends to $\varphi(s,\cdot)\in \mathbb{C}^2(\overline{S_1})$\,. 
Given a bounded continuous function $f(s,x)$ such that
 $f(s,\cdot)\in \mathbb{C}^\infty_0(S_1)$\,, let $\varphi(s,\cdot)
\in \mathbb{C}^2(\overline{S_1})$ be such
that $\text{tr}\,\bl(
c_s(X_s,x)D ^2_{xx}
\varphi(s,x)\br)=f(s,x)$ and $\varphi(s,x)=0$ on the boundary of
$S_1$\,, see Theorem 6.14 on p.107 of Gilbarg and Trudinger \cite{GilTru83}.
By Theorem 9.13 on p.239 in  Gilbarg and Trudinger
\cite{GilTru83}, where we take $\Omega'=\Omega=S_1$, and on recalling
that the
norms $\norm{c_s(u,\cdot)}_{\mathbb{W}^{2,q}(S_1)}$ are bounded
locally in $(s,u)$\,, we have that 
\begin{equation*}
  \norm{\varphi(s,\cdot)}_{\mathbb{W}^{2,q}(S_1)}\le L_4(
\norm{\varphi(s,\cdot)}_{\mathbb{L}^q(S_1)}+
\norm{f(s,\cdot)}_{\mathbb{L}^q(S_1)})
\end{equation*}
   locally uniformly in $s$\,.
By Theorem 9.1 on p.220 in  Gilbarg and Trudinger
\cite{GilTru83}, $\sup_{x\in S_1}\abs{\varphi(s,x)}\le
L_5\norm{f(s,\cdot)}_{\mathbb{L}^{q}(S_1)}$ locally uniformly in
$s$\,. 
We obtain that there exists $L_6>0$ such that 
$  \norm{\varphi(s,\cdot)}_{\mathbb{W}^{2,q}(S_1)}\le L_6
\norm{f(s,\cdot)}_{\mathbb{L}^q(S_1)}\,.$
By \eqref{eq:117}, if $q\ge2$\,, then, for
some $L_7>0$\,,
\begin{equation*}
        \abs{\int_0^t\int_{S_1}f(s,x)\zeta(x)m_s(x)\,dx\,ds}\le 
L_7\bl(\int_0^{t}\int_{S_1}\abs{f(s,x)}^q\,dx\,ds\br)^{1/q}\,.  
\end{equation*}
Since the functions $f(s,x)$ are dense 
in $\mathbb{L}^q([0,t]\times S_1)$\,, 
\begin{equation*}
  \bl(\int_0^{t}\int_{S_1}\abs{\zeta(x)m_s(x)}^{q'}\,dx\,ds\br)^{1/q'}\le L_7\,,
\end{equation*}
   which yields the required bound \eqref{eq:121} if one recalls that $\zeta(x)=1$ on
$S$\,.

\end{proof}
\begin{remark}
\label{re:psi}
  As a byproduct of the proof, the function $\psi(s,\cdot)$ is an element
  of $\mathbb{L}_0^{1,2}(\R^l,\R^l,c_s(X_s,x),m_s(x)\,dx)$ for almost all
$s$\,.
\end{remark}

We now work toward proving that $\mathbf{I}^{\ast\ast}$ is 
 the same as
$\mathbf I$ in Theorem~\ref{the:ldp} and Proposition~\ref{cor:ldp}.
The following  lemma will be useful for calculating $\mathbf
I^{\ast\ast}$\,, cf.
 Lemma A.2 on p.460 in Puhalskii \cite{Puh01}.
 \begin{lemma}
   \label{le:sup}
Let 
$V$ represent a complete separable metric space,
 let $U$ represent a dense subspace,
and let $\R$-valued function 
$f(s,y)$ be defined on $\R_+\times V$\,,
be measurable in $s$ 
 and  continuous in 
 $y$\,. 
Suppose also that
 $f(s,\lambda(s))$ is locally integrable  with respect to Lebesgue measure
for all measurable functions $\lambda(s)$ that assume values in 
$U$\,.
Then, for all $t\in\R_+$\,,
\begin{equation*}
  \sup_{\lambda(\cdot)\in \Lambda}\int_0^t f(s,\lambda(s))\,ds=
\int_0^t \sup_{y\in U}f(s,y)\,ds\,,
\end{equation*}
where  
$\Lambda$ represents the set of measurable   functions
assuming values in $U$\,.  
 \end{lemma}
In the rest of the paper we denote $D _x$ by $D $\,, divergencies are
understood with respect to $x$\,.
The next lemma is the key to proving that $\sqrt{m_s(\cdot)}\in
\mathbb{W}^{1,2}(\R^l)$ in the statement of Theorem \ref{the:ldp}.

\begin{lemma}
  \label{le:kwadr_int}
Let $m_s(x)$\,, where $x\in\R^l$ and
$s\in\R_+$\,, represent an $\R_+$-valued 
 measurable function which is a probability density
on $\R^l$ and an element of $\mathbb{W}^{1,1}_{\text{loc}}(\R^l)$
for almost all $s$\,.
If, for some $t>0$ and $L_1>0$\,, we have that
$\int_0^t\int_S\abs{Dm_s(x)}^2/m_s(x)\,dx\,ds<\infty$\,, for  all open
balls $S$, and
\begin{multline}
  \label{eq:156}
  \int_0^t  \sup_{h\in \mathbb{C}_0^1(\R^l)}\int_{\R^l} \Bl(D  h(x)^T \bl(
\frac{1}{2}\,\text{div}\,(c_s(X_s,x)m_s(x))-a_s(X_s,x)m_s(x) \br)
\\-\frac{1}{2}\,
\norm{D  h(x)}_{c_s(X_s,x)}^2\,m_s(x)\Br)\,dx\,ds\le L_1\,,
\end{multline}
then there exists  $L_2>0$\,, which depends on $L_1$ and $t$ only,
such that
\begin{equation*}
  \int_0^t\int_{\R^l}\frac{\abs{Dm_s(x)}^2}{m_s(x)}\,dx\,ds\le L_2\,.
\end{equation*}
\end{lemma}
\begin{proof}
Let $\eta(x)$ represent a $[0,1]$-valued twice continuously differentiable 
nonincreasing function defined for $x\ge 0$ 
such that $\eta(x)=1$ for $x\in[0,1]$ and $\eta(x)=0$ for $x\ge 2$\,.
Let $\eta_r(x)=\eta(\abs{x}/r)$ where $x\in\R^l $ and $r>0$\,.
 We note that the bound in \eqref{eq:156}
   extends to functions $h(x)$ from the closure $\mathbb
   H^{1,2}_0(S_{2r+1},\,m_s(x)\,dx)$ of $\mathbb C_0^\infty(S_{2r+1})$ in 
$\mathbb W^{1,2}(S_{2r+1},\,m_s(x)\,dx)$\,, where $S_{2r+1}$ represents the open ball of
radius $2r+1$ centred at the origin in $\R^l$\,.  
Let $\delta>1$\,.
Since (the restriction of)
$\ln (m_s(\cdot)\wedge \delta\vee\delta^{-1})
$ to $S_{2r+1}$ is an element of 
$\mathbb W^{1,2}(S_{2r+1},\,m_s(x)\,dx)$
a.e. and
since 
by Lemma~\ref{le:spaces}, 
$\mathbb W^{1,2}(S_{2r+1},\,m_s(x)\,dx)=
\mathbb H^{1,2}(S_{2r+1},\,m_s(x)\,dx)$\,,  
we have that $\ln (m_s(\cdot)\wedge \delta\vee\delta^{-1})\in
\mathbb H^{1,2}(S_{2r+1},\,m_s(x)\,dx)$\,, so,
$\ln (m_s(\cdot)\wedge \delta\vee\delta^{-1})
\eta_r(\cdot)^2$ is an element of $ \mathbb
H^{1,2}_0(S_{2r+1},\,m_s(x)\,dx)$\,.
Hence, one can take
 $h(x)=(1/4)\ln (m_s(x)\wedge \delta\vee\delta^{-1})
\eta_r(x)^2$ in \eqref{eq:156} to obtain
\begin{multline*}
  \int_0^t
\int_{\R^l} \Bl(\bl(
\frac{1}{4}\,\frac{Dm_s(x)}{m_s(x)}\ind_{\{\delta^{-1}\le
  m_s(x)\le \delta\}}(x)\eta_r(x)^2+\frac{1}{2}\,
\ln (m_s(x)\wedge \delta\vee\delta^{-1})\,\eta_r(x)D\eta_r(x)\br)^T\\
 \bl(\frac{1}{2}\,c_s(X_s,x)\,\frac{Dm_s(x)}{m_s(x)}+
\frac{1}{2}\,\text{div}\,c_s(X_s,x)-a_s(X_s,x) \br)m_s(x)\\
-\frac{1}{16}\,
\norm{\frac{Dm_s(x)}{m_s(x)}\ind_{\{\delta^{-1}\le
  m_s(x)\le \delta\}}(x)\eta_r(x)^2}^2_{c_s(X_s,x)}m_s(x)\\
-\frac{1}{4}\,
\norm{\ln (m_s(x)\wedge \delta\vee\delta^{-1})\eta_r(x)\,D\eta_r(x)
}_{c_s(X_s,x)}^2m_s(x)
\Br)\,dx\,ds\le L_1\,.
  \end{multline*}
Therefore,
\begin{multline}
  \label{eq:159}
  \int_0^t
\int_{\R^l} \Bl(\bl(
\frac{1}{8}\,\frac{\norm{Dm_s(x)}^2_{c_s(X_s,x)}}{m_s(x)}\ind_{\{\delta^{-1}\le
  m_s(x)\le \delta\}}(x)\eta_r(x)^2\\
-\frac{1}{16}\,
\frac{\norm{Dm_s(x)}^2_{c_s(X_s,x)}}{m_s(x)}\ind_{\{\delta^{-1}\le
  m_s(x)\le \delta\}}(x)\eta_r(x)^4\br)\,dx\,ds\\-\int_0^t
\int_{\R^l} \Bl(
\frac{1}{4}\,Dm_s(x)^T\,\ind_{\{\delta^{-1}\le
  m_s(x)\le
  \delta\}}(x)\eta_r(x)^2
\bl(\frac{1}{2}\,\text{div}\,c_s(X_s,x)-a_s(X_s,x)
\br)\\+
 \frac{1}{4}\,\ln (m_s(x)\wedge \delta\vee\delta^{-1})
\eta_r(x)\,D\eta_r(x)^T
c_s(X_s,x)\,Dm_s(x)
 \\+\frac{1}{2}\,\ln (m_s(x)\wedge \delta\vee\delta^{-1})\eta_r(x)\,D\eta_r(x)^T
 \bl(
\frac{1}{2}\,\text{div}\,c_s(X_s,x)-a_s(X_s,x) \br)m_s(x)\\
-\frac{1}{4}\,
(\ln (m_s(x)\wedge \delta\vee\delta^{-1}))^2\eta_r(x)^2\,\norm{D\eta_r(x)
}_{c_s(X_s,x)}^2m_s(x)
\Br)\,dx\,ds\le
L_1\,.
\end{multline}
We bound the terms on the righthand side. Integration by parts yields
\begin{multline*}
  \int_{\R^l} 
Dm_s(x)^T\,\ind_{\{\delta^{-1}\le
  m_s(x)\le
  \delta\}}(x)\eta_r(x)^2
\bl(\frac{1}{2}\,\text{div}\,c_s(X_s,x)-a_s(X_s,x)
\br)\,dx
\\=-
  \int_{\R^l} 
(m_s(x)\wedge \delta\vee\delta^{-1})\,
\text{div}\,\bl(
\eta_r(x)^2\bl(\frac{1}{2}\,\text{div}\,c_s(X_s,x)-a_s(X_s,x)
\br)\br)\,dx\\=-
  \int_{\R^l} 
(m_s(x)\wedge \delta\vee\delta^{-1})\,\bl(
\eta_r(x)^2\,
\text{div}\,\bl(\frac{1}{2}\,\text{div}\,c_s(X_s,x)-a_s(X_s,x)\br)\\
+2\eta_r(x)\,D\eta_r(x)^T\,\bl(\frac{1}{2}\,
\text{div}\,c_s(X_s,x)-a_s(X_s,x)\br)\br)
\,dx\,.
\end{multline*}
By Condition \ref{con:coeff}, there exists $M_1>0$ such that
$\abs{\text{div}\,\bl((1/2)\,\text{div}\,c_s(X_s,x)-a_s(X_s,x)\br)}\le
M_1$ and
$\abs{(1/2)\,\text{div}\,c_s(X_s,x)-a_s(X_s,x)}\le M_1(1+\abs{x})$\,, 
for all $0\le s\le t$ and $x\in\R^l$\,, so 
\begin{align*}
  \abs{\int_{\R^l} 
(m_s(x)\wedge \delta\vee\delta^{-1})\,
\eta_r(x)^2\,
\text{div}\,\bl(\frac{1}{2}\,\text{div}\,c_s(X_s,x)-a_s(X_s,x)\br)\,dx}\le
M_1&
\intertext{and, letting $M_2$ represent an upper bound on the
  absolute values of the first  derivative of $\eta(x)$\,,}
\abs{\int_{\R^l} 
(m_s(x)\wedge \delta\vee\delta^{-1})\,2\eta_r(x)
  D\eta_r(x)^T\,\bl(\frac{1}{2}\,
\text{div}\,c_s(X_s,x)-a_s(X_s,x)\br)
\br)\,dx}&
\\=
\abs{\int_{\R^l} 
(m_s(x)\wedge \delta\vee\delta^{-1})\,2\eta_r(x)\,
\frac{1}{r}\,\frac{x^T}{\abs{x}}  D\eta\bl(\frac{\abs x}{r}\br)\,
\bl(\frac{1}{2}\,
\text{div}\,c_s(X_s,x)-a_s(X_s,x)\br)
\br)\,dx}&\\
\le 2M_1M_2\,
\frac{1}{r}\,\int_{x\in\R^l:\,r\le \abs{x}\le 2r} 
m_s(x)(1+\abs{x})\,dx
\le 2M_1M_2
\frac{2r+1}{r}&
\end{align*}
Therefore,
\begin{multline}
  \label{eq:158}
\abs{  \int_{\R^l} 
Dm_s(x)^T\,\ind_{\{\delta^{-1}\le
  m_s(x)\le
  \delta\}}(x)\eta_r(x)^2
\bl(\frac{1}{2}\,\text{div}\,c_s(X_s,x)-a_s(X_s,x)
\br)\,dx}\le
 M_1+  2M_1M_2
\frac{2r+1}{r}\,.
\end{multline}
Similarly,
\begin{multline*}
    \int_{\R^l}\ln (m_s(x)\wedge \delta\vee\delta^{-1})
\eta_r(x)\,D\eta_r(x)^T
c_s(X_s,x)\,Dm_s(x)\,dx\\=
-  \int_{\R^l}\text{div}\bl(
\ln (m_s(x)\wedge \delta\vee\delta^{-1})
\eta_r(x)\,
c_s(X_s,x)D\eta_r(x)\br)\,m_s(x)\,dx\\
=-  \int_{\R^l}
\eta_r(x)\,\,D\eta_r(x)^T
c_s(X_s,x)Dm_s(x)\,\ind_{\{\delta^{-1}\le m_s(x)
\le \delta\}}(x)\,dx\\
-\int_{\R^l}\ln (m_s(x)\wedge \delta\vee\delta^{-1})\,
\text{div}\bl(\eta_r(x)c_s(X_s,x)\,D\eta_r(x)
\br)\,m_s(x)\,dx\,.
\end{multline*}
For the terms on the righthand side, we have that,
for suitable $M_3>0$\,,
\begin{equation*}
    \abs{\int_{\R^l}\ln (m_s(x)\wedge \delta\vee\delta^{-1})
\text{div}\bl(\eta_r(x) c_s(X_s,x)\,D\eta_r(x)
\br)\,m_s(x)\,dx}\le M_3\,\frac{\abs{\ln\delta}}{r}
\end{equation*}
and, for arbitrary $\kappa>0$\,, 
\begin{multline*}
    \abs{
\int_{\R^l} 
\eta_r(x)\,D\eta_r(x)^T
c_s(X_s,x)\,
Dm_s(x)\ind_{\{\delta^{-1}\le m_s(x)
\le \delta\}}(x)\,dx}\\\le
\frac{1}{2\kappa r^2}\,
\int_{\R^l} 
\norm{D\eta\bl(\frac{\abs{x}}{r}\br)}_{c_s(X_s,x)}^2m_s(x)\,dx
+\frac{\kappa}{2}\,\int_{\R^l} 
\,\frac{\norm{Dm_s(x)}_{c_s(X_s,x)}^2}{m_s(x)}
\,\ind_{\{\delta^{-1}\le m_s(x)
\le \delta\}}(x)\eta_r(x)^2
\,dx\,.
\end{multline*}
Hence,
\begin{multline}
  \label{eq:168}
      \abs{\int_{\R^l}\ln (m_s(x)\wedge \delta\vee\delta^{-1})
\eta_r(x)\,D\eta_r(x)^T
c_s(X_s,x)\,Dm_s(x)\,dx}\\\le M_3\,\frac{\abs{\ln\delta}}{r}
+\frac{M_2^2}{2\kappa r^2}\,
\int_{\R^l} \norm{c_s(X_s,x)}
m_s(x)\,dx
+\frac{\kappa}{2}\,\int_{\R^l} 
\,\frac{\norm{Dm_s(x)}^2_{c_s(X_s,x)}}{m_s(x)}
\,\ind_{\{\delta^{-1}\le m_s(x)
\le \delta\}}(x)
\,dx\,.
\end{multline}
The remaining two terms on the righthand side of \eqref{eq:159} are
bounded as follows:
  \begin{align}
  \notag  \label{eq:170}
\abs{\int_{\R^l}  \ln (m_s(x)\wedge \delta\vee\delta^{-1})
\eta_r(x)\,D\eta_r(x)^T
 \bl(
\frac{1}{2}\,\text{div}\,c_s(X_s,x)-a_s(X_s,x) \br)m_s(x)\,dx}
\\\le M_1M_2\,\abs{\ln\delta}\,\frac{2r+1}{r}\,
\int_{x\in\R^l:\, \abs{x}\ge r}m_s(x)\,dx
\intertext{and}\int_{\R^l}
(\ln (m_s(x)\wedge \delta\vee\delta^{-1}))^2\eta_r(x)^2\,\norm{D\eta_r(x)
}_{c_s(X_s,x)}^2m_s(x)\,dx\le \frac{M_2^2(\ln\delta)^2}{r^2}
\int_{\R^l}\norm{c_s(X_s,x)}m_s(x)\,dx\,.\label{eq:170a}
\end{align}
 We  obtain by \eqref{eq:159} --
\eqref{eq:170a} that
there exists $M_4>0$ such that, given arbitrary $\delta>1$\,, 
 for all $r$ great enough (depending on $\delta$),
\begin{equation*}
\frac{1}{16}    \int_0^t
\int_{\R^l} 
\,\frac{\norm{Dm_s(x)}^2_{c_s(X_s,x)}}{m_s(x)}\,\eta_r(x)^2\ind_{\{\delta^{-1}\le
  m_s(x)\le \delta\}}(x)
\bl(2\,-\eta_r(x)^2-2\kappa\br)\,dx\,ds
\le L_1+M_4t
\end{equation*}
so that assuming $\kappa<1/2$\,,
\begin{equation*}
  \int_0^t
\int_{\R^l} 
\,\frac{\norm{Dm_s(x)}^2_{c_s(X_s,x)}}{m_s(x)}\,\eta_r(x)^2\ind_{\{\delta^{-1}\le
  m_s(x)\le \delta\}}(x)\,dx\,ds
\le\frac{16}{1-2\kappa}\,( L_1+M_4t)\,,
\end{equation*} which implies the assertion of the lemma by letting
$r\to\infty$ and $\delta\to\infty$\,. 
\end{proof}
The next theorem establishes the equality $\mathbf
I^{\ast\ast}(X,\mu)=\mathbf I(X,\mu)$ provided $\mathbf
I^{\ast\ast}(X,\mu)<\infty$\,, $X_0=\hat u$\,, and $\mathbf I_0(\hat u)=0$.
\begin{theorem}
  \label{le:vid}
Suppose that conditions \ref{con:coeff}, 
 \ref{con:positive},
\eqref{eq:116},  \eqref{eq:122a} and \eqref{eq:82} hold and that
 $\mathbf{I}^{\ast\ast}(X,\mu)<\infty$\,. Then 
$\mu(ds,dx)=m_s(x)\,dx\,ds$\,, where $m_s(\cdot)\in\mathbb{P}(\R^l)$
a.e. and 
$\int_0^t  \int_{\R^l}\abs{Dm_s(x)}^2/m_s(x)\,dx\,ds<\infty$ for all
$t\in\R_+$\,.
The projection
$\Phi_{s,m_s(\cdot),X_s}(x)$ belongs to 
$\mathbb{L}^2(\R^l,\R^l,c_s(X_s,x),m_s(x)\,dx)$ 
as a function of $x$ for almost every $s$\,,
  $\Phi_{s,m_s(\cdot),X_s}(x)$
and $\Psi_{s,m_s(\cdot),X_s}(x)$ are  measurable in $(s,x)$\,, and
$\int_0^t
\int_{\R^l}\norm{\Phi_{s,m_s(\cdot),X_s}(x)}^2m_s(x)\,dx\,ds<\infty$
for all $t\in\R_+$\,.
We also have that
\begin{align}
   \label{eq:85} \mathbf{ I}^{\ast\ast}(X,\mu)=
\int_0^\infty
\sup_{\lambda\in\R^n}\Bl(
\lambda^T\bl(\dot{X}_s-
\int_{\R^l} A_s(X_s,x)\,m_s(x)\,dx\br)&&\\\notag
-\frac{1}{2}\, \norm{\lambda}_{\int_{\R^l}C_s(X_s,x)\,m_s(x)\,dx}^2
+\sup_{h\in \mathbb{C}_0^1(\R^l)}\int_{\R^l}\Bl( D  h(x)^T \bl(
\frac{1}{2}\,\text{div}\,(c_s(X_s,x)m_s(x))
\notag&&\\-(a_s(X_s,x)+G_s(X_s,x)^T\lambda)
 m_s(x)\br)
-\frac{1}{2}\,\notag
\norm{D  h(x)}_{c_s(X_s,x)}^2m_s(x)\Br)\,dx\Br)\,ds     
&&\\ \label{eq:150}
=\int_0^\infty
\sup_{\lambda\in\R^n}\Bl(
\lambda^T\bl(\dot{X}_s-
\int_{\R^l} A_s(X_s,x)\,m_s(x)\,dx\br)
-\frac{1}{2}\, \norm{\lambda}_{\int_{\R^l}C_s(X_s,x)\,m_s(x)\,dx}^2
&&\\+
    \sup_{g\in \mathbb{L}_0^{1,2}(\R^l,\R^l,c_s(X_s,x),m_s(x)\,dx)}
\int_{\R^l}\Bl(  g(x)^T\notag
\,c_s(X_s,x)\bl(\frac{Dm_s(x)}{2m_s(x)}&&\\
-\Phi_{s,m_s(\cdot),X_s}(x)-
\Psi_{s,m_s(\cdot),X_s}(x)
\lambda\br)
-\frac{1}{2}\,\notag
\norm{g(x)}_{c_s(X_s,x)}^2\Br)m_s(x)\,dx\Bl)\,ds\,.&&
 \end{align}
The vector    $\dot{X}_s-
\int_{\R^l} A_s(X_s,x)\,m_s(x)\,dx-
\int_{\R^l}G_s(X_s,x)
\bl(Dm_s(x)/(2m_s(x))-\Phi_{s,m_s(\cdot),X_s}(x)\br)
m_s(x)\,dx$ is in the range of 
$\int_{\R^l}Q_{s,m_s(\cdot)}(X_s,x) m_s(x)\,dx$ a.e. and the supremum
in \eqref{eq:150}
  is attained at
\begin{multline}
  \label{eq:89}
  \hat\lambda_s=\bl(\int_{\R^l}Q_{s,m_s(\cdot)}(X_s,x) m_s(x)\,dx\br)^\oplus
\bl(\dot{X}_s-\int_{\R^l}A_s(X_s,x)m_s(x)\,dx\\-
\int_{\R^l}G_s(X_s,x)
\bl(\frac{Dm_s(x)}{2m_s(x)}-\Phi_{s,m_s(\cdot),X_s}(x)\br)m_s(x)\,dx\,\br)
\end{multline}
and
\begin{equation}
  \label{eq:111}
  \hat{g}_s(x)=
\frac{Dm_s(x)}{2m_s(x)}-
\Phi_{s,m_s(\cdot),X_s}(x)-\Psi_{s,m_s(\cdot),X_s}(x)\hat\lambda_s
\end{equation}
so  that

\begin{multline}
     \label{eq:55}
 \mathbf{I}^{\ast\ast}(X,\mu)=
\int_0^\infty
\Bl(\frac{1}{2}\,\int_{\R^l}\norm{
\frac{Dm_s(x)}{2m_s(x)}-\Phi_{s,m_s(\cdot),X_s}(x)}^2_{
c_s(X_s,x)}m_s(x)\,dx
\\+\frac{1}{2}\,\norm{\dot{X}_s-\int_{\R^l}A_s(X_s,x)m_s(x)\,dx-
\int_{\R^l}
 G_s(X_s,x)\bl(\frac{Dm_s(x)}{2m_s(x)}\\-\Phi_{s,m_s(\cdot),X_s}(x)\br)
m_s(x)\,dx}^2_{(\int_{\R^l}
Q_{s,m_s(\cdot)}(X_s,x)m_s(x)\,dx)^\oplus}\Br)\,ds\,.
\end{multline}
\end{theorem}

\begin{proof}
We recall  the expression (\ref{eq:23}) for
$\mathbf{I}^{\ast\ast}(X,\mu)$\,, where
 the supremum is taken over $t\in\R_+$, 
 functions
 $\lambda(s,X)$ given by \eqref{eq:15}, 
and $\mathbb{C}^{1,2,2}(\R_+\times  \R^n
\times \R^l)$-functions
 $f(s,u,x)$ that are  compactly supported in $x$ locally uniformly in $(t,u)$\,.
According to 
 Lemma~\ref{le:density},
 if $\mathbf{I}^{\ast\ast}(X,\mu)<\infty$,
then $\nu_s(dx)=m_s(x)\,dx$\,, where $m_s(\cdot)\in \mathbb
W^{1,1}_{\text{loc}}(\R^l)$\,, so
one can integrate by parts in (\ref{eq:25}) to obtain
\begin{multline}
  \label{eq:46}
    U_{t}^{\lambda(\cdot),f}(X,\mu)=
\int_0^t\Bl(  \lambda(s,X)^T\bl(\dot{X_s}
-\int_{\R^l} A_s(X_s,x)\,m_s(x)\,dx\,\br)
-\frac{1}{2}\,\int_{\R^l} 
 \norm{\lambda(s,X)}_{C_s(X_s,x)}^2\,m_s(x)\,dx
\\+\int_{\R^l} D f(s,X_s,x)^T \bl(
\frac{1}{2}\,
\,\text{div}\,(c_s(X_s,x)\,m_s(x))-a_s(X_s,x)\,m_s(x)\br)\,dx
\\
-\frac{1}{2}\,\int_{\R^l} 
\norm{ Df(s,X_s,x)}_{
 c_s(X_s,x)}^2\,m_s(x)\,dx
-\int_{\R^l}  
 \lambda(s,X)
^T G_s(X_s,x)D
f(s,X_s,x)\,m_s(x)\,dx\Br)\,ds\,.
\end{multline}

An 
approximation argument using mollifiers 
 implies that the supremum will not change if $\lambda(s,X)$
is assumed bounded and measurable in $s$
 and if $f(s,u,x)$ is assumed  measurable,
continuously  differentiable in $x$ with bounded first partial
derivatives
 and
 compactly supported in $x$
 locally uniformly in $(s,u)$\,.
 Therefore,
on noting that $X$ is kept fixed, 
 \begin{multline*}
    \mathbf{ I}^{\ast\ast}(X,\mu)=\sup
\int_0^t\Bl(
\lambda(s)^T\bl(\dot{X}_s-
\int_{\R^l} A_s(X_s,x)\,m_s(x)\,dx\br)
-\frac{1}{2}\,
 \norm{\lambda(s)}_{\int_{\R^l} 
 C_s(X_s,x)\,m_s(x)\,dx}^2
\\+\int_{\R^l} D \phi(s,x)^T \bl(
\frac{1}{2}\,
\,\text{div}_x\,(c_s(X_s,x)\,m_s(x))-
a_s(X_s,x)\,m_s(x)\br)\,dx
-\frac{1}{2}\,\int_{\R^l} 
\norm{D \phi(s,x)}_{c_s(X_s,x)}^2
\,m_s(x)\,dx
\\-\int_{\R^l}  
 \lambda(s)
^T G_s(X_s,x)D
\phi(s,x)\,m_s(x)\,dx\,\Br)\,ds\,,
 \end{multline*}
where the supremum is taken over $t\in\R_+$, bounded measurable
 functions $\lambda(s)$, and measurable
 functions $\phi(s,x)$ that are
continuously
 differentiable in $x$ with bounded first partial derivatives
and are compactly supported in $x$ locally uniformly in $s$\,.
By Lemma~\ref{le:sup}, one can optimise with respect to $\lambda(s)$ and 
$D\phi(s,x)$ inside the $ds$-integral which yields   (\ref{eq:85}).
In some more detail, we apply Lemma~\ref{le:sup}  with
$U$ being the Cartesian product of the closed ball of radius $i$ in $\R^n$
 and  of the set $U_i=
\{Dh:\, h\in\mathbb{C}_0^1(\R^l)\,,
\sup_{x\in\R^l}\abs{Dh(x)}\le i\text{ and }h(x)=0\text{ if }
\abs{x}\ge i\}$ and with $V$  being the Cartesian product of the
closed ball of radius $i$ and of the closure of $U_i$ in 
the space of continuous functions
with support in  the open ball  of radius $i$ centred at the
origin in $\R^l$ that are bounded above by $i$ in absolute value, 
the latter space being endowed with the $\sup$-norm topology, 
 where $i\in\N$\,,
and let $i\to\infty$\,.

Integration by parts in 
 \eqref{eq:28}, with $\nu_s(dx)=m_s(x)\,dx$\,,
 yields
\begin{multline*}
\int_{\R^l} Dh(x)^T  \bl(
\frac{1}{2}\,
\text{div}\,(c_s(X_s,x)m_s(x))-a_s(X_s,x)m_s(x)\br)
\,dx=\int_{{\R^l}}
 D h(x)^T c_s(X_s,x) \psi(s,x)\,m_s(x)\,dx\,.
\end{multline*}
On recalling that $\psi(s,\cdot)\in 
\mathbb{L}_0^{1,2}(\R^l,\R^l,c_s(X_s,x),m_s(x)\,dx)$ for almost all
$s$ by Remark \ref{re:psi}, we have that
the function $-\psi(s,x)$  represents the
orthogonal projection of 
$c_s(X_s,x)^{-1}\bl(
a_s(X_s,x)-(1/2)\,\text{div}\,(c_s(X_s,x)m_s(x))/m_s(x)\br)$ onto 
 $\mathbb{L}_0^{1,2}(\R^l,\R^l,c_s(X_s,x),m_s(x)\,dx)$\,.
Since  by \eqref{eq:85}, Lemma \ref{le:density}, and
Lemma \ref{le:kwadr_int}, $D m_s(x)/m_s(x)$ is a member of $
\mathbb{L}_0^{1,2}(\R^l,\R^l,c_s(X_s,x),m_s(x)\,dx)$ for almost all
$s$, we have that the function
$-\psi(s,x)+(1/2)
D m_s(x)/m_s(x)$ belongs to
$\mathbb{L}_0^{1,2}(\R^l,\R^l,c_s(X_s,x),m_s(x)\,dx)$ for almost all
$s$\,, so, by \eqref{eq:79},
it equals $\Phi_{s,m_s(\cdot),X_s}(x)$\,.

We show that $\Phi_{s,m_s(\cdot),X_s}(x)$ and
$\Psi_{s,m_s(\cdot),X_s}(x)$ are properly measurable.
Let $\mathcal{U}_s$  represent the closure 
of the set $\{c_s(X_s,\cdot)^{1/2}\sqrt{m_s(\cdot)}Dp(\cdot):\,
p\in\mathbb{C}^\infty_0(\R^l,\R^n)\}$ in
$\mathbb{L}^2(\R^l, \R^{l\times n})$\,. 
Introducing  $\varphi_s(x)=c_s(X_s,x)^{-1/2}
G_s(X_s,x)^T\sqrt{m_s(x)}$ 
and $\hat
\varphi_s(x)=c_s(X_s,x)^{1/2}\Psi_{s,m_s(\cdot),X_s}(x)\sqrt{m_s(x)}$\,,
 we have
that $\hat
\varphi_s$
is the orthogonal projection of $\varphi_s$ onto $\mathcal{U}_s$
 (see \eqref{eq:151} and \eqref{eq:40}).
By Corollary 8.2.13 on p.317 in Aubin and Frankowska \cite{AubFra09},
$\hat\varphi_s$
 is a  measurable function from $\R_+$ to
$\mathbb{L}^2(\R^l,\R^{l\times n})$\,. (We note that $s\to
\mathcal{U}_s$ is a measurable set-valued map by part vi) of Theorem
8.1.4 on p.310 in Aubin and Frankowska \cite{AubFra09}.)
This implies that the mapping $(s,x)\to \Psi_{s,m_s(\cdot),X_s}(x)$ is measurable.
The reasoning for $\Phi_{s,m_s(\cdot),X_s}$ is similar.


The representation in
  \eqref{eq:150} follows from \eqref{eq:151}, \eqref{eq:79},
  \eqref{eq:40}, and \eqref{eq:85}.
Since the function 
\begin{equation*}
\tilde    g_s(x)=\frac{Dm_s(x)}{2m_s(x)}-\Phi_{s,m_s(\cdot),X_s}(x)-
\Psi_{s,m_s(\cdot),X_s}(x)\lambda
\end{equation*}
 is a member of
$\mathbb{L}_0^{1,2}(\R^l,\R^l,c_s(X_s,x),m_s(x)\,dx)$\,,
it attains the supremum  in \eqref{eq:150}, 
which yields
\begin{multline}
     \label{eq:39}
      \mathbf{ I}^{\ast\ast}(X,\mu)=
\int_0^\infty
\sup_{\lambda\in\R^n}\Bl(
\lambda^T\bl(\dot{X}_s-
\int_{\R^l} A_s(X_s,x)\,m_s(x)\,dx\br)
-\frac{1}{2}\, \norm{\lambda}_{\int_{\R^l}C_s(X_s,x)\,m_s(x)\,dx}^2
\\+\frac{1}{2}\int_{\R^l}\norm{
\Phi_{s,m_s(\cdot),X_s}(x)-\frac{Dm_s(x)}{2m_s(x)}-\Psi_{s,m_s(\cdot),X_s}(x)
\lambda}^2_{c_s(X_s,x)}m_s(x)\,dx
\Br)\,ds\,.
\end{multline}

Since the matrix $Q_{s,m_s(\cdot)}(u,x)=C_s(u,x)-\norm{
\Psi_{s,m_s(\cdot),u}(x)}^2_{ c_s(u,x)}$ (see
\eqref{eq:88})  is positive semidefinite, 
the supremum over $\lambda$ in \eqref{eq:39} is attained at 
\begin{multline*}
  \tilde\lambda=\bl(\int_{\R^l}Q_{s,m_s(\cdot)}(X_s,x) m_s(x)\,dx\br)^\oplus
\bl(\dot{X}_s-\int_{\R^l}A_s(X_s,x)m_s(x)\,dx\\-
\int_{\R^l}\Psi_{s,m_s(\cdot),X_s}(x)^T
 c_s(X_s,x)
\br(\Phi_{s,m_s(\cdot),X_s}(x)-\frac{Dm_s(x)}{2m_s(x)}\br)m_s(x)\,dx\,\br)
\end{multline*}
and equals
\begin{equation*}
 \frac{1}{2}\,\int_{\R^l}\norm{\frac{Dm_s(x)}{2m_s(x)}
-\Phi_{s,m_s(\cdot),X_s}}^2_{c_s(X_s,x)}
m_s(x)\,dx
+\frac{1}{2}\,\norm{\tilde\lambda}^2_{
\bl(\int_{\R^l}Q_{s,m_s(\cdot)}(X_s,x) m_s(x)\,dx\br)^\oplus}\,,
\end{equation*}
provided 
\begin{multline*}
  \dot{X}_s-\int_{\R^l}A_s(X_s,x)m_s(x)\,dx-
\int_{\R^l}\Psi_{s,m_s(\cdot),X_s}(x)^T
 c_s(X_s,x)\bl(\frac{Dm_s(x)}{2m_s(x)}-
\Phi_{s,m_s(\cdot),X_s}(x)
\br)\,m_s(x)\,dx
\end{multline*}
 is in the range of 
$\int_{\R^l}Q_{s,m_s(\cdot)}(X_s,x) m_s(x)\,dx$ a.e. Otherwise, the
supremum equals infinity. The
fact that $\tilde\lambda=\hat\lambda_s$ and
the expression in \eqref{eq:55}  follow from \eqref{eq:151}
 and
\eqref{eq:39}.
The properties that
$\int_0^t  \int_{\R^l}\abs{Dm_s(x)}^2/m_s(x)\,dx\,ds$ 
and $\int_0^t
\int_{\R^l}\norm{\Phi_{s,m_s(\cdot),X_s}(x)}^2\,dx\,ds$ are finite
follow from
Lemma \ref{le:kwadr_int}, \eqref{eq:85}, and \eqref{eq:55}.
\end{proof}
Motivated by  \eqref{eq:55} in Theorem \ref{le:vid}, 
let us introduce, provided $\mathbf{I}^{\ast\ast}(X,\mu)<\infty$ so
that
$\mu(ds,dx)=m_s(x)\,dx\,ds$ and 
$\dot{X}_s-\int_{\R^l}A_s(X_s,x)m_s(x)\,dx-
\int_{\R^l}
 G_s(X_s,x)
\bl(Dm_s(x)/(2m_s(x))-\Phi_{s,m_s(\cdot),X_s}\br)m_s(x)\,dx$ is in the range of
$\int_{\R^l}
Q_{s,m_s(\cdot)}(X_s,x)m_s(x)\,dx$ a.e.,
\begin{multline}
  \label{eq:37}
   \mathbf{I}^{\ast\ast}_t(X,\mu)=
\int_0^t
\Bl(\frac{1}{2}\,
\int_{\R^l}\norm{\frac{Dm_s(x)}{2m_s(x)}-\Phi_{s,m_s(\cdot),X_s}(x)}^2_{
c_s(X_s,x)}m_s(x)\,dx\\
+\frac{1}{2}\,\norm{\dot{X}_s-\int_{\R^l}A_s(X_s,x)m_s(x)\,dx-
\int_{\R^l}
 G_s(X_s,x)
\bl(\frac{Dm_s(x)}{2m_s(x)}\\-\Phi_{s,m_s(\cdot),X_s}(x)
\br)m_s(x)\,dx}^2_{(\int_{\R^l}
Q_{s,m_s(\cdot)}(X_s,x)m_s(x)\,dx)^\oplus}\Br)\,ds\,.
\end{multline}
Similarly to the proof of Theorem~\ref{le:vid}, we also have that
\begin{multline}
  \label{eq:133}
      \mathbf{ I}^{\ast\ast}_t(X,\mu)=
\int_0^t
\sup_{\lambda\in\R^n}\Bl(
\lambda^T\bl(\dot{X}_s-
\int_{\R^l} A_s(X_s,x)\,m_s(x)\,dx\br)
-\frac{1}{2}\, \norm{\lambda}_{\int_{\R^l}C_s(X_s,x)\,m_s(x)\,dx}^2
\\+\sup_{h\in \mathbb{C}_0^1(\R^l)}\int_{\R^l}\Bl( D  h(x)^T \bl(
\frac{1}{2}\,\text{div}\,(c_s(X_s,x)m_s(x))
-(a_s(X_s,x)+G_s(X_s,x)^T\lambda)
 m_s(x)\br)\\
-\frac{1}{2}\,
\norm{D  h(x)}_{c_s(X_s,x)}^2m_s(x)\Br)\,dx\Br)\,ds\,.
\end{multline}
For the proof of Theorem \ref{the:appr},
 it will be  needed to extend $(X,\mu)$ defined on $[0,t]$ past
$t$ in such a way  that $\mathbf I^{\ast\ast}_t(X,\mu)=\mathbf
I^{\ast\ast}(X,\mu)$\,. That is done in the following lemma which also
concerns the zeros of $\mathbf I^{\ast\ast}(X,\mu)$\,.
\begin{lemma}
  \label{le:zero}
For $t\in\R_+$ and $z\in\R^n$\,, the system of equations
\begin{align}
    \label{eq:146a}
    \dot{X}_s=\int_{\R^l}A_{s+t}(X_s,x)\,m_s(x)\,dx,\;X_0=z,\\
    \label{eq:146}
  \int_{\R^l}\bl(\frac{1}{2}\,
\text{tr}\,(c_{s+t}(X_s,x)\,  D^2 p(x))
+a_{s+t}(X_s,x)^TD p(x)\br)\, m_s(x)\,dx  =0
\end{align}
where $p\in \mathbb{C}_0^\infty(\R^l)$ is otherwise arbitrary,
has a solution $(X^\dag,(m^\dag_s(x)))$ such that
$X^\dag$ is locally Lipschitz continuous,
 $m^\dag_s(x)$ is measurable, and
$m^\dag_s(\cdot)\in \mathbb{P}(\R^l)$\,. 
If,  given $(X,\mu)$ such that $\mathbf{I}^{\ast\ast}(X,\mu)<\infty$\,, 
one defines $(\hat X,\hat \mu)$ by the
relations
$\hat X_s=X_s$ and $\hat \mu_s=\mu_s$ for $s\le t$\,, and
$\hat X_s=X^\dag_{s-t}$ and $\hat \mu_s(dx)=\mu_t(dx)+
\int_0^{s-t}m^\dag_{r}(dx)\,dr$ for $s> t$\,, where
$z=X_t$\,, then $\mathbf{I}^{\ast\ast}(\hat X,\hat \mu)=
\mathbf{I}_t^{\ast\ast}(X,\mu)$\,.
 In particular, if $t=0$\,, then
 $\mathbf{I}^{\ast\ast}(X^\dag,\mu^\dag)=0$\,.
\end{lemma}
\begin{proof}
  Since $a_s(u,x)$ is locally  bounded, since $c_s(u,x)$ is bounded, is positive
definite,  and is of class $\mathbb{C}^1$ in $x$,
 and since 
$a_s(u,x)^T x/\abs{x}\to-\infty$ as
$\abs{x}\to\infty$ by \eqref{eq:116},  applications of
 Theorem 1.4.1 in
Bogachev, Krylov, and R\"ockner \cite{BogKryRoc} (with $V(x)=
\sqrt{1+\abs{x}^2}$) and of Theorem 2.2 and
Proposition 2.4 in  Metafune, Pallara, and Rhandi \cite{MetPalRha05},
show that for every $s,t\in \R_+$ and $u\in\R^n$
there exists  a unique probability density
$m_s(x)$
satisfying the equation
\begin{equation}
  \label{eq:81}
      \int_{\R^l}\bl(\frac{1}{2}\,
\text{tr}\,(c_{s+t}(u,x)\,  D^2 p(x))+
D p(x)^T a_{s+t}(u,x)\br)\, m_s(x)\,dx  =0\,.
\end{equation}

We apply the method of successive approximations: let
$X^0_s=z$ and, for $i\in\N$\,,
\begin{equation}
  \label{eq:152}
      \int_{\R^l}\bl(\frac{1}{2}\,
\text{tr}\,(c_{s+t}(X_s^i,x)\,  D^2 p(x))+
D p(x)^T a_{s+t}(X_s^i,x)\br)\, m^i_s(x)\,dx  =0\,,
\end{equation}
\begin{equation}
  \label{eq:6}
    \dot{X}^{i+1}_s=\int_{\R^l}A_{s+t}(X^{i+1}_s,x)\,m^i_s(x)\,dx\,,\;X^{i+1}_0=z.
\end{equation}
We note that $m^i_s(x)$ is a measurable function of $(s,x)$ (one can
use,  e.g., Theorem 8.2.9 on p.315 in Aubin and Frankowska \cite{AubFra09}).
By (\ref{eq:122a}), we have that given $L>0$, there exists 
 $M>0$ such that a.e. in $s\in[0,L]$\,,
$  d\abs{X^{i+1}_s}^2/ds\le M(1+\abs{X^{i+1}_s}^2)$\,.
Gronwall's inequality implies that
$\sup_{i\in\N}\sup_{s\in[0,L]}\abs{X^i_s}<\infty$\,. By \eqref{eq:6}
and \eqref{eq:63a},
the derivatives $\dot{X}^{i+1}_s$ are 
bounded uniformly in $i\in\N$ and $s\in[0,L]$\,,
so the sequence $(X^i_s,\,s\in[0,L])$ is relatively compact for the
 uniform norm on $[0,L]$\,. Let $ X^\dag_s$ represent a limit point.
It is a locally Lipschitz continuous function.

As in  Metafune, Pallara, and Rhandi
\cite[Proposition 2.4]{MetPalRha05},
we have that, for arbitrary  $\delta>0$ and
 $L>0$\,,
\begin{equation}
  \label{eq:123}
\sup_{s\in[0,L]}\sup_{i\in \N}\int_{\R^l}
e^{\delta\abs{x}}\,m_s^i(x)\,dx
<\infty\,.
\end{equation}
In some more detail, let
for a function $p$ which is twice differentiable at $x$\,,
\begin{equation*}
\mathcal{L}^i_sp(x)=\frac{1}{2}\,\text{tr}\,(c_{s+t}(X^i_s,x)
 D^2p(x))+ D p(x)^Ta_{s+t}(X_s^i,x)\,.
  \end{equation*}
Since, for  $\abs{x}>0$\,, 
\begin{equation*}
  \mathcal{L}^i_se^{\delta\abs{x}}=
\Bl(\frac{1}{2}\,\text{tr}\,\bl(c_{s+t}(X^i_s,x)
(\frac{\delta}{\abs{x}}(I-\frac{xx^T}{\abs{x}^2}) 
+\delta^2\,\frac{xx^T}{\abs{x}^2})\br)+\delta\, a_{s+t}^i(X^i_s,x)^T 
\frac{x}{\abs{x}}\Br)e^{\delta \abs{x}}\,,
\end{equation*}
where $I$ represents the $l\times l$ identity matrix,
and $\sup_{i\in\N}a_{s+t}(X^i_s,x)^T x/\abs{x}\to-\infty$ as
$\abs{x}\to\infty$\,,  there exists  $R>1$ such that 
$\mathcal{L}^i_se^{\delta\abs{x}}\le0$ and
$e^{\delta \abs{x}}\le 
\abs{\mathcal{L}^i_se^{\delta\abs{x}}}$ for all $s\in[0,t]$ and all 
$i\in \N$ provided $\abs{x}>R$\,.
 Let $F$ be a $\mathbb{C}^\infty(\R^l)$-function
such that $F(x)=e^{\delta\abs{x}}$ if $\abs{ x}\ge 1$\,. 
Arguing as in  the proof of 
Proposition 2.3 in Metafune, Pallara, and Rhandi \cite{MetPalRha05},
one can see that 
\begin{equation*}
  \int_{x\in\R^l:\,\abs{x}>
  R}\abs{\mathcal{L}^i_se^{\delta\abs{x}}}\,m^i_s(x)\,dx\le 
\int_{x\in\R^l:\,\abs{x}\le
  R}\mathcal{L}^i_sF(x)\,m^i_s(x)\,dx\,
\end{equation*}
so that
\begin{equation}
  \label{eq:95}
    \int_{x\in\R^l:\,\abs{x}>
  R}e^{\delta\abs{x}}\,m^i_s(x)\,dx\le 
\int_{x\in\R^l:\,\abs{x}\le
  R}\mathcal{L}^i_sF(x)\,m^i_s(x)\,dx\,,
\end{equation}
which implies \eqref{eq:123}.

Hence, given $s\in[0,L]$\,, 
the sequence of probability measures 
$m_s^i(x)\,dx$ is tight. 
 Proposition 2.16 in Bogachev,
Krylov, and R\"ockner 
\cite{BogKryRoc01} implies that 
the   $m_s^i(x)$  converge in the variation norm along a
subsequence to  density
 $ m_s^\dag(x)$\,. Since the local 
$\mathbb L^q$-norms of the $m_s^i(x)$  are uniformly bounded for all $q>1$
(see (2.26) in Bogachev,
Krylov, and R\"ockner 
\cite{BogKryRoc01}),
$\sup_{i\in\N}\abs{a_{s+t}(X^i_s,x)}$ grows at most
linearly  with $\abs{x}$ by Lipschitz continuity 
and the fact that $\sup_{x\in\R^l}\sup_{i\in\N}\norm{c_{s+t}^i(X^i_s,x)}<\infty$ 
(see Condition 2.1), and
 $\sup_{i\in
  \N}\int_{\R^l}e^{\delta\abs{x}}\,m_s^i(x)\,dx<\infty$\,, 
   on taking a
limit in \eqref{eq:152}, we have by dominated convergence that
\eqref{eq:146} holds. Since density $m_s^\dag(x)$ is specified uniquely by 
\eqref{eq:146}  $m_s^i(x)\to  m_s^\dag(x)$ as $i\to\infty$
along a subsequence such that the $X^i$ converge to $X^\dag$\,.
Since $\sup_{i\in\N}\sup_{x\in\R^l}\abs{A_{s+t}(X^i_s,x)}<\infty$ 
by \eqref{eq:63a}, a similar reasoning shows that
 taking the above subsequential
limit in \eqref{eq:6} obtains  \eqref{eq:146a}.
Since \eqref{eq:123} implies that
$\int_{\R^l}\abs{a_s(X_s,x)}^2\,m_s^\dag(x)\,dx <\infty$\,,
by Theorem 1.1 in Bogachev, Krylov, and R\"ockner \cite{BogKryRoc96},
$\sqrt{m_s^\dag(\cdot)}\in \mathbb{W}^{1,2}(\R^l)$\,.

On noting that  \eqref{eq:81} can be written as
\begin{equation*}
      \int_{\R^l}D p(x)^T \bl(a_{s+t}(u,x)-\frac{1}{2}\,
\text{div}\,c_{s+t}(u,x)\br) m_s(x)\,dx  =
\frac{1}{2}\int_{\R^l}D p(x)^T c_{s+t}(u,x)Dm_s(x)\,dx\,,
\end{equation*}
we have that $\Phi_{s+t,m_s^\dag(\cdot),X^\dag_s}(x)=Dm_s^\dag(x)/(2m_s^\dag(x))$ which
implies, by \eqref{eq:55} and \eqref{eq:37},
that $\mathbf{I}^{\ast\ast}(\hat X,\hat \mu)=
\mathbf{I}_t^{\ast\ast}(X,\mu)$\,.
\end{proof}
\begin{remark}
  By Proposition 2.4 and Theorem 6.1  (with $\beta=1$)
 in Metafune,
Pallari, and Randi \cite{MetPalRha05}, $m^\dag_s(\cdot)$
decays exponentially at infinity. It is also positive and H\"older
continuous, see
Bogachev, Krylov, and R\"ockner \cite[Theorem 2.8, Corollary 2.10,
Corollary 2.11]{BogKryRoc01} and
Bogachev, Krylov, and R\"ockner \cite{BogKryRoc}.
\end{remark}
\section{Identifying the large deviation function}
\label{sec:ident}
The purpose of this section is to show that 
$\tilde{\mathbf{I}}=\mathbf I^{\ast\ast}$ for sufficiently regular
functions $(X,\mu)$\,.
More specifically, we will prove the following theorem.

\begin{theorem}
  \label{the:iden_reg}
Suppose that
 conditions \ref{con:coeff}, 
 \ref{con:positive},
\eqref{eq:116}, 
and \eqref{eq:82} hold.
Suppose that $\tilde{ \mathbf I}$ is a large deviation function that
satisfies the assertion of Theorem \ref{the:equation} and is such that
$\tilde{\mathbf{I}}(X,\mu)=\infty$ unless $X_0=\hat{u}$\,.
Suppose that $(\hat{X},\hat{\mu})$ is such that 
$\hat{X}_0=\hat{u}$\,,
$\mathbf{I}^{\ast\ast}(\hat{X},\hat{\mu})<\infty$\,,
 $\hat X$ is locally Lipschitz continuous and
that  
$\hat{m}_s(x)=\hat\mu(ds,dx)/(ds\,dx)$  
is of the form
\begin{equation*}
  \hat m_s(x)=M_s\bl(\tilde m_s(x)\hat\eta^2\bl(\frac{\abs{x}}{r}\br)+
e^{-\alpha\abs{x}}\br(1-\hat\eta^2\bl(\frac{\abs{x}}{r}\br)\br)\br)
\end{equation*}
where $\tilde m_s(x)$ is a probability density in $x$ which is
 locally bounded away from
zero and  belongs to $\mathbb{C}^1(\R^l)$ as a
function of $x$\,, with $\abs{Dm_s(x)}$ being locally bounded in $(s,x)$\,,
$\hat\eta(y)$ is
a nonincreasing 
$[0,1]$-valued $\mathbb C^1_0(\R_+)$-function,  where 
 $y\in\R_+$\,, that equals $1$ for $y\in[0,1]$ and equals $0$ for 
$y\ge 2$\,,
$r>0$\,,
 $\alpha>0$\,, 
and $M_s$ is the normalising constant.
Then, for  given  $\tilde m_s(x)$, $\hat\eta(y)$, and $r$,\,
 there exists $\alpha_0>0$ such that $\tilde{\mathbf I}(\hat X,\hat
 \mu)= \mathbf{I}^{\ast\ast}(\hat X,\hat\mu)$ for all $\alpha>\alpha_0$\,.
\end{theorem}
We assume throughout the section the hypotheses of Theorem
\ref{the:iden_reg} to hold.
We start by  extending the assertion of Theorem~\ref{the:equation} to
a larger set of functions $(\lambda(\cdot),f)$\,. 
For economy of notation, we denote  $\gamma=(X,\mu)$ and
recall that $\Gamma$ represents the  set of $\gamma$ such that
$X$ is absolutely continuous and
 $\mu$ admits 
density $m_s(x)$  that is an element of 
$\mathbb P(\R^l)$ in $x$\,, for almost all $s$\,.
Let 
 $\lambda(s,X)$, where $s\in\R_+$ and $X\in\mathbb{C}(\R_+,\R^n),$ 
represent an $\R^n$-valued  measurable
 function  and let
 $ h_s(u,x)$, where $s\in\R_+,\,u\in\R^n$ and $x\in\R^l,$ represent an 
$\R$-valued measurable function, which
 is an element of $\mathbb{W}^{1,1}_{\text{loc}}(\R^l)$
  in $x$ and is of bounded support in $x$ locally uniformly over
  $(s,u)$\,.
If,  
for all $t\in\R_+$ and all  $\gamma\in\mathbb C(\R_+,\R^n)\times
\mathbb C_\uparrow(\R_+,\mathbb M(\R^l))$\,,
$\;\int_0^t\int_{\R^l}\bl(
\abs{\lambda(s,X)}^2+
\abs{D h_s(X_s,x)}^2\br)\mu(dx,ds)<\infty$\,, 
we define, given  $N\in\N$,
 \begin{subequations}
   \begin{align}
  \label{eq:62}
    \tau^N(\gamma)&=\inf\{t\in\R_+:\;\int_0^t\int_{\R^l}\Bl(
\norm{\lambda(s,X)}^2_{C_s(X_s,x)}+
\norm{D h_s(X_s,x)}^2_{c_s(X_s,x)}\Br)\mu(dx,ds)+
X_t^\ast+t\ge N\}
\intertext{and, provided $\gamma\in\Gamma$\,,
     }
\label{eq:27}
  \theta^N(\gamma)&=
 \int_0^{\tau^N(\gamma)}\Bl(
 \lambda(s,X)^T\,\bl(\dot{X}_s-
\int_{\R^l} A_s(X_s,x)m_s(x)\,dx\,\br)
-\frac{1}{2}\,
 \norm{\lambda(s,X)}_{\int_{\R^l}  C_s(X_s,x)\,m_s(x)\,dx}^2\notag
\\&+\int_{\R^l}\bl(D  h_s(X_s,x)^T\bl( \notag
\frac{1}{2}\,  \,\text{div}\,(c_s(X_s,x)\,m_s(x))
-a_s(X_s,x)m_s(x)\br)
\notag\\&
-\frac{1}{2}\, \norm{D h_s(X_s,x)}_{c_s(X_s,x)}^2
\,m_s(x) \br)
\,dx-\int_{\R^l}  
 \lambda(s,X)
^T G_s(X_s,x)
Dh_s(X_s,x)\,m_s(x)\,dx\Br)\,ds\,.
\end{align}
 \end{subequations}
For the latter definition, we assume that, in addition,
\begin{equation}
  \label{eq:41}
    \int_0^t (\abs{\dot
  X_s}^2+\int_{\R^l}\frac{\abs{Dm_s(x)}^2}{m_s(x)}\,dx)\,ds <\infty\,,    
\end{equation}
 for
all $t\in\R_+$\,, and use the piece of notation 
 $X_t^\ast=\sup_{s\in[0,t]}\abs{X_s}$\,.
(The definition of $\theta^N(\gamma)$ is modelled on the expression
for $U^{\lambda(\cdot),f}_t(X,\mu)$ in \eqref{eq:46}.)
We note that $\tau^N(\gamma)\le N$\,.
Furthermore, we have the following lemma, for which we reuse the
piece of notation of Theorem \ref{the:id} that, for $\delta\in\R_+$,
\begin{equation*}
K_\delta=\{\gamma:\,\tilde{\mathbf{I}}(\gamma)\le \delta\}
\end{equation*}
and recall that $K_\delta$ is a compact in $\mathbb C(\R_+,\R^n)\times
\mathbb C_\uparrow(\R_+,\mathbb M(\R^l))$ and that $K_\delta\subset
\Gamma$\,. Theorem \ref{le:vid} implies that \eqref{eq:41} holds
 on $K_\delta$\,.
For the definition of the essential supremum of a family of measurable
functions used in the next lemma, see, e.g., Proposition II.4.1 on p.44
of Neveu \cite{neveng}.
\begin{lemma}
  \label{le:conv_stop}
Let $\lambda^i(s,X)$ and $h^i_s(u,x)$ be sequences of functions
satisfying the same hypotheses as $\lambda(s,X)$ and $h_s(u,x)$\,,
respectively, and
let $\tau^{N,i}(\gamma)$ and $\theta^{N,i}(\gamma)$
be defined by the respective equations  \eqref{eq:62} and
\eqref{eq:27}, with $\lambda^i(s,X)$ and $h^i_s(u,x)$ being
substituted for $\lambda(s,X)$ and $h_s(u,x)$\,, respectively. 
If, in addition, the functions 
$h^i_s(u,x)$ are of bounded support in $x$ uniformly over $i$ and
locally uniformly over $(s,u)$\,, then
\begin{equation}
\label{eq:153d}
\int_0^N\esssup_{\gamma\in K_\delta}
\abs{\lambda(s,X)}^2\,ds
+\int_0^N\esssup_{\gamma\in K_\delta}\int_{\R^l}
\abs{Dh_s(X_s,x)}^2\,m_s(x)\,dx\,ds
<\infty\,,
\end{equation}
\begin{subequations}
  \begin{align}
\label{eq:153a}
    \lim_{i\to\infty}\sup_{\gamma\in K_\delta}
\int_0^N\abs{\lambda(s,X)-\lambda^i(s,X)}^2\,ds=0\,,
\intertext{ and }
     \label{eq:153}
        \lim_{i\to\infty}\sup_{\gamma\in K_\delta}
\int_0^N\int_{\R^l}\abs{Dh_s(X_s,x)-Dh^i_s(X_s,x)}^2m_s(x)\,dx\,ds=0\,,
\end{align}  
\end{subequations}
then 
\begin{subequations}
  \begin{align}
    \label{eq:165}
    \lim_{i\to\infty}\sup_{\gamma\in
      K_\delta}\abs{\tau^N(\gamma)-\tau^{N,i}(\gamma)}=0
\intertext{ and }    \label{eq:165a}
    \lim_{i\to\infty}\sup_{\gamma\in
      K_\delta}\abs{\theta^N(\gamma)-\theta^{N,i}(\gamma)}=0\,.
  \end{align}
\end{subequations}

\end{lemma}
\begin{proof} 
Let us note that under the hypotheses,
\begin{subequations}
  \begin{align}
    \label{eq:164}
    \lim_{i\to\infty}\sup_{\gamma\in
    K_\delta}\int_0^N
\int_{\R^l}\abs{  \norm{\lambda^i(s,X)}_{ C_s(X_s,x)}^2-
 \norm{\lambda(s,X)}_{ C_s(X_s,x)}^2}\,m_s(x)\,dx\,ds=0\,,\\
    \label{eq:164a}
  \lim_{i\to\infty}\sup_{\gamma\in
    K_\delta}\int_0^N\int_{\R^l}\abs{
 \norm{D h^i_s(X_s,x)}_{c_s(X_s,x)}^2
- \norm{D h_s(X_s,x)}_{c_s(X_s,x)}^2}
\,m_s(x) \,dx\,ds=0\,,\\    \label{eq:164b}
  \lim_{i\to\infty}\sup_{\gamma\in K_\delta}\int_0^N\abs{
\bl( \lambda^{i}(s,X)-\lambda(s,X)\br)^T\,\bl(\dot{X}_s-
\int_{\R^l} A_s(X_s,x)m_s(x)\br)\,dx}\,ds=0,
\intertext{and}\label{eq:164'}
  \lim_{i\to\infty}\sup_{\gamma\in K_\delta}
\int_0^N\abs{
\int_{\R^l} \bl(D h^{i}_s(X_s,x)-D h_s(X_s,x)\br)^T 
\bl( \frac{1}{2}\,  \,\text{div}\,(c_s(X_s,x)\,m_s(x))
-a_s(X_s,x)m_s(x)\br)\,dx}\,ds
=0\,.
  \end{align}  
\end{subequations}
The first two convergences are implied by \eqref{eq:153a}, \eqref{eq:82},
 and
 \eqref{eq:153}, \eqref{eq:63},
respectively, and \eqref{eq:153d}.
The convergence in \eqref{eq:164b} follows 
via Cauchy's inequality
from
\eqref{eq:153a} and
the fact that, according to \eqref{eq:55}
in Theorem~\ref{le:vid},
\begin{equation}
  \label{eq:124}
  \sup_{(X,\mu):\,\mathbf I^{\ast\ast}(X,\mu)\le\delta}\int_0^N
\abs{\dot{X}_s-
\int_{\R^l} A_s(X_s,x)m_s(x)\,dx}^2\,ds<\infty\,.
\end{equation}
Similarly, 
  \eqref{eq:164'} is a consequence of (\ref{eq:153}),
if one recalls that the functions involved are of uniformly 
bounded support in $x$ and takes into account part \eqref{eq:106} 
of Lemma \ref{le:density}.

The convergence in \eqref{eq:165}  follows from \eqref{eq:164},
\eqref{eq:164a} and the
observation  that by \eqref{eq:62}
\begin{multline*}
    \abs{\tau^N(\gamma)-\tau^{N,i}(\gamma)}\le
\int_0^N\abs{\int_{\R^l}
\bl(\norm{\lambda(s,X)}^2_{C_s(X_s,x)}-\norm{\lambda^i(s,X)}^2_{C_s(X_s,x)}\\+
\norm{D h_s(X_s,x)}^2_{c_s(X_s,x)}-
\norm{D h^i_s(X_s,x)}^2_{c_s(X_s,x)}\br)\,m_s(x)\,dx}\,ds\,.
\end{multline*}
The convergence in \eqref{eq:165a} follows by \eqref{eq:27}, 
  \eqref{eq:165}, 
\eqref{eq:164}--\eqref{eq:164'}, and \eqref{eq:124}, if one notes that,
thanks to  \eqref{eq:153d},
\begin{align*}
\sup_{\gamma\in K_\delta}\int_0^t
\abs{\lambda(s,X)}^2\,ds\,,\quad
\sup_{\gamma\in K_\delta}\int_0^t\int_{\R^l}
\abs{Dh_s(X_s,x)}^2\,m_s(x)\,dx\,ds\,,
\intertext{
and }
  \sup_{\gamma\in K_\delta}\int_0^t
\int_{\R^l} D h_s(X_s,x)^T 
\bl(\frac{1}{2}\,\text{div}\,(c_s(X_s,x)\,m_s(x))
- a_s(X_s,x)m_s(x)\br)\,dx\,ds
\end{align*}
are continuous functions of
$t\in[0,N]$\,.
\end{proof}
\begin{lemma}
  \label{le:sup_ld_function}
Let   $\lambda_s(u)$ represent an $\R^n$-valued function of
$(s,u)\in\R_+\times\R^n$\,, which is measurable in $s$\,, is
continuous in $u$ for almost all $s$ and is such that
$ \int_0^N\sup_{\abs{u}\le L}
\abs{\lambda_s(u)}^2\,ds<\infty$ for all $L>0$\,.
Suppose that 
the function $h_s(u,x)$\,, in addition to being measurable and being
of class $\mathbb W^{1,1}_{\text{loc}}$ in $x$\,, vanishes when $x$ is outside of some
open ball in $\R^l$ locally uniformly in $(s,u)$\,,
that   the function 
$D h_s(u,x)$ is continuous in $(u,x)$ for almost all $s\in\R_+$\,, 
  and that $\int_0^N\sup_{u\in\R^n:\,\abs{u}\le L}
\int_{\R^l}\abs{Dh_s(u,x)}^q\,dx\,ds<\infty$ for all $q>1$ and $L>0$\,.
   Then, under the hypotheses of Theorem
\ref{the:iden_reg},
the function $\theta^N(\gamma)$\,, where
$\lambda(s,X)=\lambda_s(X_s)$\,,
 is continuous in $\gamma$ when restricted to
 $ K_\delta$\,,
  \begin{equation*}
    \sup_{\gamma\in \Gamma}\bl(\theta^N(\gamma)-\tilde{\mathbf{I}}(\gamma)\br)=0
  \end{equation*}
and the latter supremum is attained\,.
 Furthermore,
  \begin{equation*}
    \sup_{\gamma\in K_{2N+2}}\bl(\theta^N(\gamma)-\tilde{\mathbf{I}}(\gamma)\br)=0\,.
  \end{equation*}
\end{lemma}
\begin{proof}
The functions $\abs{\lambda_s(u)}\ind_{\{\abs{\lambda_s(u)}\ge r\}}(s,u)$
are upper semicontinuous in $u$ and monotonically decreasing in $r$\,, so by Dini's theorem 
$\abs{\lambda_s(u)}^2\ind_{\{\abs{\lambda_s(u)}\ge r\}}(s,u)\to0$ as
$r\to\infty$ uniformly on $\{u\in\R^n:\,\abs{u}\le
  L\}$\,. Let $r_i$ be such that $\int_0^N
\sup_{u\in\R^n:\,\abs{u}\le
  L}\abs{\lambda_s(u)}^2\ind_{\{\abs{\lambda_s(u)}\ge
  r_i\}}(s,u)\,ds<1/i$\,, where $L=\sup_{\gamma\in K_\delta}
\sup_{s\in[0,t]}\abs{X_s}$ and $i\in\N$\,.
Since $\lambda_s(u)$ is a Carath\'eodory function,
as a consequence of the Scorza-Dragoni theorem,
see, e.g., p.235 in Ekeland and Temam \cite{EkeTem76},
 there exists a measurable function
$\breve{\lambda}^i_s(u)$ that is continuous in $(s,u)$\,,
 is bounded above in absolute value by $r_i$\,, and is such that
$\int_0^N
\ind_{\{\lambda_s(\cdot)\not=\breve{\lambda}^i_s(\cdot)\}}(s)\,ds<2/(ir_i^2)$\,. 
Letting $\lambda^i(s,X)=\breve \lambda^i_{\lfloor
  j(i)s\rfloor/j(i)}(X_{\lfloor j(i)s\rfloor/j(i)})$\,, where $j(i)$
is great enough and $j(i)\to\infty$ as $i\to\infty$\,, we
have that \eqref{eq:153a} holds.

Similarly,
let 
\begin{align*}
  h^{i}_s(u,x)&
=\int_{\R\times \R^l} 
\rho_{1/i}(\tilde s,y)
h_{s-\tilde s}(u,x-y)\,d\tilde s\,dy,
  \end{align*}
where 
$\rho_\kappa(\tilde s, y)=
(\tilde\rho_1(\tilde s/\kappa)/\kappa)
(\tilde \rho_2(y/\kappa)/\kappa^{l})$,
$\tilde\rho_1(\tilde s)$ is a mollifier on $\R$
such
that $\tilde\rho_1(\tilde s)=0$ if $\abs{\tilde s}>1$\,,
$\tilde\rho_2(y)$ is a mollifier on $\R^l$
such
that $\tilde\rho_2(y)=0$ if $\abs{y}>1$\,,
 and 
$h_s(u,x)=0$ if $s\le0$\,.
 The function $h^i_s(u,x)$ is
an element of $\mathbb{C}^\infty(\R_+\times\R^l)$  in $(s,x)$ for 
all $u$ and $  Dh^{i}_s(u,x)
=\int_{\R\times \R^l}
 \rho_{1/i}(\tilde s,y)
Dh_{s-\tilde s}(u,x-y)\,d\tilde s\,dy$\,,
cf. Theorem 2.29 on p.36 in Adams and Fournier \cite{MR56:9247}.
In addition, $Dh^i_s(u,x)$  is a
continuous function for every $i$\,.
We also have that, for  all open balls $S$\,, all $L>0$ and all $q>1$\,,
 \begin{equation}
  \label{eq:124a}
    \lim_{i\to\infty}
\int_0^N\sup_{u\in\R^n:\,\abs{u}\le L}
\int_S\abs{Dh_s(u,x)-Dh^i_s(u,x)}^q\,dx\,ds=0\,,
 \end{equation}
which can be shown as follows.
If, in addition, $Dh_s(u,x)$ is continuous in all variables, then
$Dh^i_s(u,x)$ converges to $Dh_s(u,x)$ locally uniformly in
$(s,u,x)$\,, cf. Theorem 2.29 on p.36 in Adams and Fournier
\cite{MR56:9247},  so,
\eqref{eq:124a} holds. In the general case,  in analogy with the above
reasoning, 
there exist $r_j$  such that $\int_0^{N+1}
\sup_{u\in\R^n:\,\abs{u}\le
  L}\int_{\tilde S}\abs{Dh_s(u,x)}^{q}\ind_{\{\abs{Dh_s(u,x)}\ge
  r_j\}}(s,u,x)\,dx\,ds<1/j$\, where $\tilde S$ represents the open
ball in $\R^l$ centred at the origin
of radius one greater than that of $S$\,,
 and there exists a  continuous function 
$\breve h^j_s(u,x)$\,, which is bounded above in absolute value by 
$r_j$\,, such that $\int_0^{N+1}
\ind_{\{Dh_s(\cdot,\cdot)\not=\breve
  h^j_s(\cdot,\cdot)\}}(s)\,ds<2/(jr_j^{q})$\,. Calculations show that
\begin{align*}
  \int_0^N \sup_{u\in\R^n:\,\abs{u}\le L} \int_S\abs{Dh_s(u,x)-\breve
    h^j_s(u,x)}^q\,dx\,ds\le
\frac{2^{q-1}}{j}+ \frac{2^{q+1} }{j}
\intertext{ and }
  \int_0^N \sup_{u\in\R^n:\,\abs{u}\le L}
\int_S\abs{Dh^i_{s}(u,x)-
\int_{\R\times \R^l} 
\rho_{1/i}(\tilde  s,y)
\breve h^j_{s-\tilde s}(u,x-y)\,d\tilde  s\,dy}^q\,dx\,ds
\le 
\frac{2^{q-1}}{j}+\frac{2^{q+1}V(S)}{j}\,,
  \end{align*}
where $V(S)$ represents the volume of the ball $S$\,.
Hence, \eqref{eq:124a} holds.

By an application of H\"older's inequality,
it  follows from \eqref{eq:124a}, 
\eqref{eq:121} in Lemma \ref{le:density}
and $h_s(u,x)$ having
compact support in $x$ locally uniformly over $(s,u)$ 
that  \eqref{eq:153} holds.
 Also,  \eqref{eq:153d}  holds.

Let  $\tau^{N,i}$ and $\theta^{N,i}$ be defined as in Lemma \ref{le:conv_stop}.
The functions $h^{i}_s(u,x)$, $\lambda^{i}(s,X)$, and
$\tau^{N,i}(X,\mu)$ 
satisfy the requirements imposed on the respective functions
$f(s,u,x)$, $\lambda(s,X)$, and $\tau(X,\mu)$ when deriving
(\ref{eq:21}). Furthermore, integration by parts on the righthand side
of \eqref{eq:24} with $\mu(ds,dx)=m_s(x)\,dx\,ds$\,,
 implies
that $\theta^{N,i}(\gamma)=
U_{N\wedge \tau^{N,i}(\gamma)}^{\lambda^{i}(\cdot),h^{i}}(\gamma)$
provided $\gamma\in\Gamma$\,. In addition, by \eqref{eq:62},
$\abs{X_{\tau^{N,i}(\gamma)}}\le N$ and
 $\tau^{N,i}(\gamma)$ is a continuous function of $\gamma\in
 \mathbb{C}(\R_+,\R^n)\times \mathbb{C}(\R_+,\mathbb{M}(\R^l))$\,,
cf. Theorem 2 on p.510 and Theorem 3 on p.511 in
Liptser and Shiryayev \cite{lipshir}.
We obtain by equation
\eqref{eq:21} of Theorem \ref{the:equation} and the fact that 
$\tilde{\mathbf{I}}(\gamma)=\infty$ unless $\gamma\in\Gamma$
(see 
Theorem \ref{le:vid})
 that 
\begin{equation}
  \label{eq:49}
  \sup_{\gamma\in \Gamma}(\theta^{N,i}(\gamma)-\tilde{\mathbf{I}}(\gamma))=0\,.
\end{equation}
Let us show that, for all $\delta>2N+1$\,,
\begin{equation}
  \label{eq:71}
  \sup_{\gamma\in K_\delta}(\theta^{N,i}(\gamma)-\tilde{\mathbf{I}}(\gamma))=0\,.
\end{equation}

Let, for $\gamma\in \Gamma$\,,
\begin{multline*}
   \tilde\theta^{N,i}(\gamma)=
 \int_0^{\tau^{N,i}(\gamma)}\Bl(
2 \lambda^{i,j}(s,X)^T\,\bl(\dot{X}_s-
\int_{\R^l} A_s(X_s,x)m_s(x)\,dx\,\br)
-\frac{1}{2}\,\int_{\R^l} 
 \norm{2\lambda^{i,j}(s,X)}_{ C_s(X_s,x)}^2\,m_s(x)\,dx
\\+\int_{\R^l}\bl(2D  h^i_s(X_s,x)^T\bl( 
\frac{1}{2}\,  \,\text{div}\,(c_s(X_s,x)\,m_s(x))
-a_s(X_s,x)m_s(x)\br)\\
-\frac{1}{2}\, \norm{2D h^i_s(X_s,x)}_{c_s(X_s,x)}^2
\,m_s(x) \br)
\,dx-\int_{\R^l}  4
 \lambda^{i,j}(s,X)
^T G_s(X_s,x)
Dh^i_s(X_s,x)\,m_s(x)\,dx\Br)\,ds\,.
\end{multline*}
By \eqref{eq:24}, $\tilde\theta^{N,i}(\gamma)=
U_{N\wedge \tau^{N,i}(\gamma)}^{2\lambda^{i}(\cdot),2h^{i}}(\gamma)$\,,
provided $\gamma\in\Gamma$\,, so in analogy with \eqref{eq:49},
\begin{equation*}
  \sup_{\gamma\in \Gamma}(\tilde{\theta}^{N,i}(\gamma)
-\tilde{\mathbf{I}}(\gamma))=0\,.
\end{equation*}
On noting that
$  \tilde{\theta}^{N,i}(\gamma)
\ge 2\theta^{N,i}(\gamma)-2N\,$,  we have that, for $M>0$, 
\begin{multline*}
  \sup_{\gamma:\,
\theta^{N,i}(\gamma)\ge M}(\theta^{N,i}(\gamma)
-\tilde{\mathbf{I}}(\gamma))\le
  \sup_{\gamma:\,\theta^{N,i}(\gamma)\ge
    M}(2\theta^{N,i}(\gamma)
-\tilde{\mathbf{I}}(\gamma))-M\\\le
\sup_{\gamma:\,\theta^{N,i}(\gamma)\ge
  M}(\tilde{\theta}^{N,i}(\gamma)-\tilde{\mathbf{I}}(\gamma))+2N-M
\le 2N-M\,.
\end{multline*}
Since,  by (\ref{eq:49}),
\begin{equation*}
0=    \sup_{\gamma\in
      \Gamma}(\theta^{N,i}(\gamma)
-\tilde{\mathbf{I}}(\gamma))\le 
  \sup_{\gamma\in K_\delta}(\theta^{N,i}(\gamma)-\tilde{\mathbf{I}}(\gamma))\vee
\sup_{\gamma:\,\theta^{N,i}(\gamma)\ge
  M}(\theta^{N,i}(\gamma)-\tilde{\mathbf{I}}(\gamma))\vee (M-\delta),
\end{equation*}
we conclude, on choosing $M=2N+1$, that \eqref{eq:71} holds for
$\delta> 2N+1$\,.

Since by Lemma \ref{le:conv_stop},  for arbitrary $\delta\in\R_+$,
\begin{equation}
  \label{eq:113}
    \lim_{i\to\infty}\sup_{\gamma\in K_\delta}
\abs{\theta^{N}(\gamma)-
\theta^{N,i}(\gamma)}=0\,,
\end{equation}
we obtain by \eqref{eq:71} that
\begin{equation}
  \label{eq:139}
    \sup_{\gamma\in K_\delta
  }(\theta^{N}(\gamma)-\tilde{\mathbf{I}}(\gamma))=0\,.  
\end{equation}
Since $\theta^{N,i}(\gamma)=
U_{ \tau^{N,i}(\gamma)}^{\lambda^{i}(\cdot),h^{i}}(\gamma)$\,, the
latter function is continuous in $\gamma$\,, and
 $K_\delta$ is compact,  (\ref{eq:113}) 
implies that
$\theta^{N}(\gamma)$ is continuous on $K_\delta$\,.
Since  $\tilde{\mathbf{I}}(\gamma)$ is a lower semicontinuous function 
of $\gamma$\,,  the  supremum in \eqref{eq:139} is attained.
On the other hand, if $\tilde{\mathbf{I}}(\gamma)<\infty$\,, then by
 \eqref{eq:49} and (\ref{eq:113}),
$  \sup_{\gamma\in \Gamma}(\theta^{N}(\gamma)-\tilde{\mathbf{I}}(\gamma))\le0\,.$

\end{proof}
In order to prove that $\tilde{\mathbf I}=\mathbf{I}^{\ast\ast}$\,,
we will  use $\hat \lambda$ and $\hat g$ defined in
\eqref{eq:89} and \eqref{eq:111}, respectively,
as $\lambda_s(u)$ and $Dh_s(u,x)$ in the preceding lemma. We therefore need
$\Phi_{s,m_s(\cdot),u}$ and $\Psi_{s,m_s(\cdot),u}$ to be sufficiently regular.
The next lemma addresses both regularity and growth-rate 
properties.
\begin{lemma}
  \label{le:reg}
Suppose that
 conditions \ref{con:coeff}, 
 \ref{con:positive},
\eqref{eq:116}, 
 and \eqref{eq:82} hold.
Let $m_s(x)$ represent an $\R_+$-valued measurable function 
that is a probability density in $x$ for
 almost every $s$\,. Suppose
$m_s(x)$ is bounded away from zero on bounded sets of $(s,x)$\,,
$m_s(\cdot)\in \mathbb{C}^1(\R^l)$\,, with $\abs{Dm_s(x)}$ being
locally bounded in $(s,x)$\,,
and
$m_s(x)=M_s e^{-\alpha\abs{x}}$ for all $\abs{x}$ great enough
locally uniformly in $s$\,, where  $\alpha>0$\,.
 Then
there exist $\R$-valued measurable
function $w_s(u,x)$ and  $\R^n$-valued measurable
function $v_s(u,x)$ 
  such that $w_s(u,\cdot)\in
\mathbb{W}_{\text{loc}}^{2,q}(\R^l)$ and
  $v_s(u,\cdot)\in \mathbb{W}_{\text{loc}}^{2,q}(\R^l,\R^n)$\,,
where $q>1$ is otherwise arbitrary,
 $\Phi_{s,m_s(\cdot),u}(\cdot)=D w_s(u,\cdot)$ and
$\Psi_{s,m_s(\cdot),u}(\cdot)=D v_s(u,\cdot)$ for almost all $s\in\R_+$ and
all $u\in\R^n$\,, i.e.,
\begin{subequations}
  \begin{multline}
  \label{eq:96}
    \int_{\R^l}D p(x)^T
\bl(
a_s(u,x)-\frac{1}{2}\,
\text{div}\,c_s(u,x)\br)
{m}_s(x)
\,dx=\int_{\R^l}D p(x)^T c_s(u,x)
\,D w_s(u,x)\,{m}_s(x)
\,dx
 \end{multline}
and \begin{equation}
  \label{eq:97a}
    \int_{\R^l}    D p(x)^TG_s(u,x)^T\,{m}_s(x)\,dx=
\int_{\R^l}D p(x)^Tc_s(u,x)D {v}_s(u,x)^T  {m}_s(x)
\,dx\,
\end{equation}
\end{subequations}
for all $p\in\mathcal{\mathbb{C}}^\infty_0(\R^l)$\,.
Furthermore,  $w_s(u,x)$\,,
 $D w_s(u,x)$\,, $v_s(u,x)$\,, and $D v_s(u,x)$ are
 continuous in $(u,x)$ for  almost all $s\in \R_+$\,, and,
for all open balls $S\subset \R^l$\,, all $L>0$ and all $t>0$\,, 
\begin{multline*}
  \sup_{s\in[0,t]}\sup_{u:\,\abs{u}\le
  L}\bl(\norm{w_s(u,\cdot)}_{\mathbb{W}^{2,q}(S)}+
\norm{v_s(u,\cdot)}_{\mathbb{W}^{2,q}(S,\R^n)}+
\norm{Dw_s(u,\cdot)}_{\mathbb{L}^{2}(\R^l,\R^l,m_s(x)\,dx)}\\+
\norm{Dv_s(u,\cdot)}_{\mathbb{L}^{2}(\R^l,\R^{l\times
    n},m_s(x)\,dx)}\br)<\infty\,.
\end{multline*}


Also, 
there exists $\alpha_0$ which depends on the functions $a_t(u,x)$ and
$c_t(u,x)$ only such that, if $\alpha>\alpha_0$, then for all $L>0$ and 
all $t>0$\,,
\begin{subequations}
  \begin{align}
    \label{eq:91a}
&\sup_{ s\in[0, t]}\sup_{x\in\R^l}\sup_{u\in\R^n:\,\abs{u}\le L}
 w_s(u,x) <\infty\\
\intertext{ and}
  \label{eq:91c}
&\sup_{ s\in[0, t]}\sup_{x\in\R^l}\sup_{u\in\R^n:\,\abs{u}\le L}
\bl(\frac{\abs{ w_s(u,x)}+\abs{D w_s(u,x)}}{1+\abs{x}^2}+
\frac{\abs{ v_s(u,x)}+\norm{D
    v_s(u,x)}}{1+\abs{x}}\br)<\infty\,,
\end{align}
\end{subequations}
and, for all $\abs{x}$ great enough locally uniformly in $s$\,,
\begin{subequations}
  \begin{align}
  \label{eq:91}
&  x^T \bl(a_s(u,x)-\frac{1}{2}\,
\text{div}\,c_s(u,x)-c_s(u,x)D
w_s(u,x)\br)=0\intertext{and}
  \label{eq:91b}
 & \bl(G_s(u,x)-D {v}_s(u,x)c_s(u,x)\br)x=0\,.
\end{align}
\end{subequations}


\end{lemma}
\begin{proof}
Since
$a_s(u,\cdot)\in \mathbb{L}^2(\R^l,\R^l,m_s(x)\,dx)$ 
by the fact that $a_s(u,x)$ grows at most linearly in $x$
 and $ m_s(x)$
decays exponentially, and
$\text{div}\,c_s(u,\cdot)\in \mathbb{L}^2(\R^l,\R^l,m_s(x)\,dx)$ for a
similar reason,
 $\Phi_{s,m_s(\cdot),u}$ as defined by \eqref{eq:79},
is an element
of $\mathbb{L}_0^{1,2}(\R^l,\R^l,c_s(u,x),m_s(x)\,dx)$\,,
being  a projection in the
 Hilbert space $\mathbb{L}^2(\R^l,\R^l,c_s(X_s,x),m_s(x)\,dx)$\,.
In addition,
\begin{multline}
  \label{eq:120}
 \int_{{\R^l}}  \norm{\Phi_{s,m_s(\cdot),u}(x)}_{c_s(u,x) }^2
m_s(x)\,dx\le
\int_{{\R^l}}  \norm{c_s(u,x)^{-1}\bl(
a_s(u,x)-\frac{1}{2}\,\text{div}\,c_s(u,x)\br)}_{c_s(u,x)}^2
m_s(x)\,dx\,.
\end{multline}
By Conditions 2.1 and 2.2 and $m_s(\cdot)$ decaying exponentially,  
\begin{equation}
  \label{eq:29}
\sup_{s\in[0,t]}\sup_{u\in\R^n:\,\abs{u}\le L}
\int_{\R^l}\abs{\Phi_{s,m_s(\cdot),u}(x)}^2\,m_s(x)\,dx<\infty\,.
\end{equation}

We prove that $\Phi_{s,m_s(\cdot),u}(\cdot)$  is a gradient.
    Let $D w_i\to \Phi_{s,m_s(\cdot),u}$ in
$\mathbb{L}^{2}(\R^l,\R^l,c_s(u,x),m_s(x)\,dx)$ as $i\to\infty$\,,
where
$w_i\in\mathbb C^\infty_0(\R^l)$\,. Then for every
$f\in \mathbb{C}_0^\infty(\R^l,\R^l)$ 
such that $\text{div}\, f(x)=0$, we have that
$\int_{\R^l}D w_i(x)^T f(x)\,dx=0$\,. Since $m_s(\cdot)$ is  bounded
away from zero locally and $c_s(u,x)$ is positive definite, convergence in
$\mathbb{L}^{2}(\R^l,\R^l,c_s(u,x),m_s(x)\,dx)$ implies convergence in
$\mathbb{L}^{2}_{\text{loc}}(\R^l,\R^l)$\,, 
so $D w_i\to \Phi_{s,m_s(\cdot),u}$ in $\mathbb{L}^{2}_{\text{loc}}(\R^l,\R^l)$\,. 
Therefore, $\int_{\R^l}\Phi_{s,m_s(\cdot),u}(x)^T\linebreak f(x)\,dx=0$\,.
It follows that $\Phi_{s,m_s(\cdot),u}(x)=D \tilde 
w_s(u,x)$ in the sense of distributions, where $\tilde w_s(u,\cdot)\in
\mathbb{W}^{1,2}_{\text{loc}}(\R^l)$\,, see, e.g., Lemma 2.2.1 on 
p.73 in Sohr \cite{Soh01}. (One could also invoke the Helmholtz
decomposition, see, e.g.,
 Farwig, Kozono, and Sohr \cite{FarKozSoh07}.) Consequently,
$\int_{\R^l}\chi(x)^T\Phi_{s,m_s(\cdot),u}(x)\,dx=-\int_{\R^l}
\text{div}\,\chi(x)\,\tilde w_s(u,x)\,dx$\,, for all $\chi\in \mathbb
C_0^1(\R^l,\R^l)$\,. 

By \eqref{eq:79} and condition 2.1,  for $p\in \mathbb{C}_0^\infty(\R^l)$\,,
\begin{multline*}
-\int_{\R^l}\,\text{div}\,\bl( c_s(u,x)\,m_s(x)D p(x)^T\br)
 \tilde w_s(u,x)\,dx
=        \int_{\R^l}D p(x)^T\bl(
a_s(u,x) -\frac{1}{2}\,
\text{div}\,c_s(u,x)\br) m_s(x)
\,dx\,.
\end{multline*}
Thus, $\tilde w_s(u,\cdot)$ is a weak solution to the
equation
\begin{equation}
  \label{eq:98}
  \text{div}\,\bl( c_s(u,x)
\,D \tilde w_s(u,x)\,m_s(x)\br)=
      \text{div}\,\bl(
\bl(
a_s(u,x)-
\frac{1}{2}\,\text{div}\,c_s(u,x)\br) m_s(x)\br)
\end{equation}
in that
\begin{multline}\label{eq:66}
\int_{\R^l}\,D p(x)^T c_s(u,x)
 D\tilde w_s(u,x)\,m_s(x)\,dx
=        \int_{\R^l}D p(x)^T\bl(
a_s(u,x) -\frac{1}{2}\,
\text{div}\,c_s(u,x)\br) m_s(x)
\,dx\,.
\end{multline}
We note that \eqref{eq:66} uniquely specifies $D\tilde{w}_s(u,\cdot)$ as
an element of\linebreak
 $\mathbb{L}_0^{1,2}( \R^l,c_s(u,x),m_s(x)\,dx)$\,.

Let $S$ and $\tilde{S}$  represent  open balls in
$\R^l$
such that 
$S\subset\subset \tilde S$\,,
let 
$\zeta(x)$ represent a $\mathbb{C}_0^\infty$-function 
with support in $\tilde S$
such that  $\zeta(x)=1$ for $x\in S$\,, and
let $\varphi(x)$ represent a $\mathbb{C}^\infty_0(\tilde S)$ function.
On letting $p(x)=\varphi(x)\zeta(x)$ in \eqref{eq:66}
and integrating by parts, we obtain that
$\zeta(x) \tilde w_s(u,x)$ is a weak solution $f$ to the Dirichlet
problem 
\begin{multline}
  \label{eq:154}
    \text{div}\,(c_s(u,x)m_s(x)\,Df(x))=
\text{div}\,( c_s(u,x)
\,D\zeta(x) \tilde w_s(u,x)\,m_s(x))
+D\zeta(x)^T c_s(u,x)
\,D \tilde w_s(u,x)\,m_s(x)\\
+\text{div}\,\bl(\bl(a_s(u,x)-\frac{1}{2}\,
\text{div}\,c_s(u,x) 
\br)\zeta(x)m_s(x)\br)
-D\zeta(x)^T\bl(
a_s(u,x)-\frac{1}{2}\,
\text{div}\,c_s(u,x) 
\br)m_s(x)
\end{multline}
on $\tilde S$ with a zero boundary condition. 
By Theorem 8.3 on p.181 
and Theorem 8.8 on p.183 in Gilbarg and Trudinger \cite{GilTru83},
$\zeta(x) \tilde w_s(u,x)$ is an element of $\mathbb{W}^{2,2}(S)$ and 
is a strong solution of \eqref{eq:154}.
Therefore,
 $\tilde w_s(u,\cdot)\in \mathbb{W}^{2,2}_{\text{loc}}(\R^l)$ and
 \eqref{eq:98} holds
a.e. in $x$\,.

Differentiation in \eqref{eq:98} 
and division by $ m_s(x)$ yield
\begin{multline}
  \label{eq:105}
  \text{tr}\,\bl( c_s(u,x)
\,D^2 \tilde w_s(u,x)\br)\,
+\bl(c_s(u,x)
\,\frac{Dm_s(x)}{m_s(x)}
+\text{div}\,c_s(u,x)\br)^T \,D \tilde w_s(u,x)\\=
      \text{div}\,
\bl(
a_s(u,x)-\frac{1}{2}\,
\text{div}\,c_s(u,x)\br)
+\bl(
a_s(u,x)-\frac{1}{2}\,
\text{div}\,c_s(u,x)\br)
\frac{Dm_s(x)}{m_s(x)}
\,.
\end{multline}

On writing the lefthand side as 
$\mathcal{L}_{s,u}(x)\tilde w_s(u,x)$  and letting
$f_s(u,x)$ represent the righthand side, we have that 
$\mathcal{L}_{s,u}(x)\tilde w_s(u,x)=f_s(u,x)$\,. 
Let  $Y_{s,u}^{y}(t)$ represent the diffusion process in $t$ with the
infinitesimal generator $\mathcal{L}_{s,u}(\cdot)$ and initial condition
$y$\,,
defined on a probability space $(\Omega,\mathcal{F},\mathbf{P})$ with
expectation denoted by $\mathbf{E}$\,. It  is a
strong Markov process by Conditions 2.1 and 2.2 and  the
hypotheses of the lemma.
 One can also choose $Y_{s,u}^{y}(t)$ to
be measurable in all variables. (A possible line of reasoning invokes
 continuous dependence of solutions of stochastic
differential equations on parameters, see, e.g., Gikhman and Skorokhod
\cite{GikSko82}, or Krylov \cite{Kry77}, and the Scorza-Dragoni theorem.)
If  $\abs{x}$ is great enough so that
$m_s(x)=M_s e^{-\alpha\abs{x}}$\,, then
\begin{equation}
\label{eq:95a}  \frac{Dm_s(x)}{m_s(x)}=
-
\alpha \frac{x}{\abs{x}}\,.
\end{equation}
Hence, on recalling Condition 2.1, in particular that
$\abs{\text{div}\,c_s(u,x)}$ is bounded in  $x$ locally uniformly in
$(s,u)$\,,
 and \eqref{eq:63}, we have that
 there exists $\alpha_0$ which depends on $a_t(u,x)$ and $c_t(u,x)$
only such that if $\alpha>\alpha_0$\,, then
 $\limsup_{\abs{x}\to\infty}
(x/\abs{x})^T\bl(c_s(u,x)
\,Dm_s(x)/m_s(x)
+\text{div}\,c_s(u,x)\br)<0$\,, so
 $Y_{s,u}^{y}(t)$ is an ergodic process, see, e.g.,
Has'minskii \cite{Has80},
Veretennikov \cite{Ver97}, and Malyshkin \cite{Mal00}. Since, by the
divergence theorem,
\begin{equation*}
\int_{\R^l}\mathcal{L}_{s,u}(x) p(x)\,m_s(x)\,dx=
    \int_{\R^l}
  \text{div}\,\bl( c_s(u,x)
\,D p(x)\,m_s(x)\br)\,dx=0  
\end{equation*}
 for all
$p\in\mathbb{C}_0^\infty(\R^l)$\,, $m_s(x)\,dx$ is the unique
invariant measure. Similarly,
\begin{equation*}
  \int_{\R^l}f_s(u,x)\,m_s(x)\,dx=  \int_{\R^l}\text{div}\,\bl(
\bl(
a_s(u,x)-
\frac{1}{2}\,\text{div}\,c_s(u,x)\br) m_s(x)\br)\,dx=0\,,
\end{equation*}
the latter equality being a consequence of $m_s(x)$
decaying exponentially as $\abs{x}\to\infty$\,.
  By  \eqref{eq:105}, \eqref{eq:95a}, Lipschitz continuity of 
 $a_s(u,\cdot)$ and of $\text{div}\,c_s(u,\cdot)$\,, the
 boundedness property of  $\text{div}\,c_s(u,\cdot)$\,, 
 and  by \eqref{eq:116}, we may assume that
$\alpha_0$ is such that if
$\alpha>\alpha_0$\,, then
$f_s(u,x)>0$ for all $\abs{x}$ great enough locally uniformly in
$(s,u)$\,.
Also,  
\begin{equation*}
\sup_{s\in[0,t]}\sup_{x\in\R^l}\sup_{u\in\R^n:\,\abs{u}\le L}
\frac{\abs{f_s(u,x)}}{1+\abs{x}}<\infty\,.  
\end{equation*}
By Theorem 1 in  Pardoux and
Veretennikov \cite{ParVer01}, the function 
\begin{equation}
  \label{eq:187}
     \breve w_s(u,x)=-\int_0^\infty \mathbf{E}f_s(u,Y^x_{s,u}(t))\,dt
\end{equation}
is well defined, belongs to $\mathbb{W}^{2,q}_{\text{loc}}(\R^l)$\,,
for all $q>1$, 
as a
function of $x$,   $D \breve w_s(u,x)$ is of polynomial growth in
$x$\,, in particular,
$D \breve w_s(u,\cdot)\in\linebreak \mathbb{L}^2(\R^l,\R^l,m_s(x)\,dx)$\,,
and 
$\mathcal{L}_{s,u}(x)  \breve w_s(u,x)=f_s(u,x)$\,. 
Since $D \breve w_s(u,x)$ also satisfies \eqref{eq:66},
we have that $D\breve w_s(u,x)=D\tilde w_s(u,x)$\,. In addition, $\breve w_s(u,x)$
is measurable in $(s,u,x)$\,.

As in Pardoux and
Veretennikov \cite{ParVer01}, by \eqref{eq:187} and the strong Markov
property, for $R>0$\,, 
\begin{equation}
  \label{eq:130}
   \breve w_s(u,x)=\mathbf{E} \breve w_s(u,Y^x_{s,u}(\tau^R))
-\mathbf{E}\int_0^{\tau^R} f_s(u,Y^x_{s,u}(t))\,dt\,,
\end{equation}
where $\tau^R=\inf\{t\in\R_+:\,\abs{Y^x_{s,u}(t)}\le R\}<\infty$\,.
Since
$\abs{Y^x_{s,u}(\tau^R)}=R$ if $\abs{x}>R$\,, by $f_s(u,x)$ being
positive for all $\abs{x}$ great enough, we have that if $R$ is great
enough then
$   \breve w_s(u,x)\le  \breve w_s(u,R)$\,, provided $\abs{x}>R$\,.
One can see that the bounds in the calculation of part (a) of the proof of 
Theorem 1 in  Pardoux and
Veretennikov \cite{ParVer01} hold uniformly over 
$u\in[0,L]$ and $s\in[0,t]$\,, which
  shows that $\sup_{x:\,\abs{x}\le R}\sup_{\substack{s\in[0,t],\\u\in\R^n:\,\abs{u}\le L}}\abs{\breve w_s(u,x)}<\infty$\,.
Since 
the righthand side of \eqref{eq:105} grows at most linearly in
$\abs{x}$ locally
uniformly in $(s,u)$,  the arguments of part (b)
of
 the proof of
Theorem 2 (with $\beta=2$ and $\alpha=0$)
and of part (e) of the proof of Theorem 1 in 
Pardoux and Veretennikov \cite{ParVer01}, along with \eqref{eq:130}, show that
 the functions $\abs{
   \breve w_s(u,x)}$ and $\abs{D
   \breve w_s(u,x)}$ grow at most quadratically in $\abs{x}$ locally uniformly in
$(s,u)$\,.

 We define 
$ w_s(u,x)=\breve w_s(u,x)-V_1^{-1}\int_{S_1}\breve w_s(u,y)\,dy$\,,
where
$S_1$ represents the unit open ball centred at the origin
in $\R^l$ and $V_1$
represents the volume of that ball.
 Obviously, the bounds on
$\breve w_s(u,x)$ we have found are also valid for  $w_s(u,x)$\,.
It also satisfies \eqref{eq:96}.
We prove that, for all $q>1$\,,
\begin{equation}
  \label{eq:76}
\sup_{s\in[0,t]}\sup_{u\in\R^n:\,\abs{u}\le
  L}\norm{w_s(u,\cdot)}_{\mathbb{W}^{2,q}(S)}<\infty\,.
\end{equation}
 Since $Dw_s(u,x)=\Phi_{s,m_s(\cdot),u}$ and \eqref{eq:29} holds,
$\int_{S_1}w_s(u,x)\,dx=0$\,,
and $m_s(x)$ is locally bounded away from zero, an application of 
Poincar\'e's inequality yields 
$\sup_{s\in[0,t]}\sup_{u\in\R^n:\,\abs{u}\le L}
\norm{ w_s(u,\cdot)}_{\mathbb{L}^2(S_1)}<\infty$
\,.
If $S_2$ is a ball containing $S_1$, then, for some $L_{S_1,S_2}>0$\,,
\begin{equation*}
\norm{ w_s(u,\cdot)}^2_{\mathbb{W}^{1,2}(S_2)}\le
L_{S_1,S_2}(\norm{D
w_s(u,\cdot)}^2_{\mathbb{L}^2(S_2)}+
\norm{ w_s(u,\cdot)}^2_{\mathbb{L}^2(S_1)})\,,  
\end{equation*}
 see p.299 in Kufner, John, and Fu\^cik 
\cite{KufJohFuc77}, also Theorem 7.4 on p.109 in
Ne{\v{c}}as \cite{Nec12}.  Thus, on recalling that 
$S\subset\subset\tilde S$ and
 letting  $\breve{S}$ represent an open ball in
$\R^l$
such that 
$\tilde S\subset\subset \breve S$\,, we have that
\begin{equation}
  \label{eq:18}
\sup_{s\in[0,t]}\sup_{u\in\R^n:\,\abs{u}\le L}
\norm{ w_s(u,\cdot)}_{\mathbb{W}^{1,2}(\breve S)}<\infty\,.
\end{equation}
By \eqref{eq:96},
Theorem 5.5.5'(a) on p.156 in Morrey \cite{Mor2008}, the discussion on
p.12 of Bogachev, Krylov, and R\"ockner \cite{BogKryRoc},
Shaposhnikov \cite{Sha08}, 
and the fact that
 $\sup_{x\in\breve S}\abs{a_s(u,x)}$ and
$\norm{c_s(u,\cdot)}_{\mathbb{W}^{1,\infty}(\breve S,\R^{l\times l})}$ are bounded
locally uniformly in $(s,u)$\,, we have that
$\norm{ w_s(u,\cdot)}_{\mathbb{W}^{1,q}(\tilde S)}\le 
M_{\tilde S,\breve{S},q}
(1+\norm{ w_s(u,\cdot)}_{\mathbb{L}^{1}(\breve  S)})$\, 
locally uniformly
in $(s,u)$\,. By \eqref{eq:18},
$\sup_{s\in[0,t]}\sup_{\abs{u}\le L}\norm{
  w_s(u,\cdot)}_{\mathbb{W}^{1,q}(\tilde S)}<\infty$\,.
By  \eqref{eq:96}, via a similar argument to
the one used for $\zeta(\cdot) \tilde w_s(u,\cdot)$
above,
$\zeta(\cdot) w_s(u,\cdot)$ is a strong solution to \eqref{eq:154}.
By Theorem 9.15 on p.241 in Gilbarg and Trudinger \cite{GilTru83},
$\zeta(\cdot) w_s(u,\cdot)\in \mathbb{W}^{2,q}(\tilde{S})$\,. 
By Theorem 9.11 on p.235
in Gilbarg and Trudinger \cite{GilTru83}, locally uniformly in
$(s,u)$\,, for some $\tilde M_{S,\tilde S,q}>0$\,,
\begin{equation*}
    \norm{ w_s(u,\cdot)}_{\mathbb{W}^{2,q}(S)}\le
\tilde M_{S,\tilde S,q}(1+\norm{ w_s(u,\cdot)}_{\mathbb{L}^{q}(\tilde
  S)})
\end{equation*}
which implies \eqref{eq:76}.

We now address the continuity of $w_s(u,x)$\,.
Let $u_i\to u$\,. By \eqref{eq:76} and Sobolev's imbedding, 
the sequences $ w_s(u_i,\cdot)$ and
$D w_s(u_i,\cdot)$ are  equicontinuous in $x\in S$\,,
so they are   relatively compact in $\mathbb{C}(S,\R^l)$\,. 
A similar property holds for
$\bl(a_s(u_i,\cdot)-(1/2)
\text{div}\,c_s(u_i,\cdot)\br) m_s(\cdot)$\,.
Taking a subsequential  limit in
 \eqref{eq:96}  implies that $D w_s(u_i,\cdot)\to D w_s(u,\cdot)$ 
 in $\mathbb{C}(S,\R^l)$\,. By Poincar\'e's inequality for $S=S_1$
and the fact that $\int_{S_1}w_s(u,x)\,dx=0$\,,
$ w_s(u_i,\cdot)\to  w_s(u,\cdot)$ 
 in $\mathbb{L}^2(S_1)$\,. The bound 
 \begin{multline*}
\norm{ w_s(u_i,\cdot)- w_s(u,\cdot)}^2_{\mathbb{W}^{1,2}(S_2)}\le
L_{S_1,S_2}(\norm{D w_s(u_i,\cdot)-D w_s(u,\cdot)}^2_{\mathbb{L}^2(S_2)}+
\norm{ w_s(u_i,\cdot)-
  w_s(u,\cdot)}^2_{\mathbb{L}^2(S_1)})   
 \end{multline*}
 shows 
that $ w_s(u_i,\cdot)\to  w_s(u,\cdot)$ 
in $\mathbb{L}^2(S_2)$\,. Since $S_2$ is an arbitrary ball that
contains $S_1$\,,
$ w_s(u_i,\cdot)\to  w_s(u,\cdot)$ 
 in $\mathbb{C}(S,\R^l)$\,.
  Hence, $ w_s(u,x)$ and $D w_s(u,x)$ are continuous in $(u,x)$
 for almost all $s$\,. 

We prove \eqref{eq:91}.
Since $a_s(u,\cdot)\in \mathbb{L}^2(\R^l,\R^l,m_s(x)\,dx)$
and $Dw_s(u,\cdot)\in \mathbb{L}^2(\R^l,\R^l,m_s(x)\,dx)$\,,
(\ref{eq:96}) extends   to    $\mathbb 
C^1(\R^l)$-functions 
$p(x)$ such that
$\int_{\R^l}\bl(p(x)^2+
\abs{Dp(x)}^2\br)\, m_s(x)\,dx<\infty$\,.
For given $\kappa>0$\,, $\delta>0$\,, and $x_0\in \R^l$\,, we let
$  p(x)=\abs{x}^2e^{-\delta[(\abs{x-x_0}^2/\kappa-1)^+]^2}$\,.
By dominated convergence,
\begin{align*}
  \lim_{\delta\to\infty}
\int_{\R^l}D p(x)^T c_s(u,x)\,D  w_s(u,x)m_s(x)\,dx&=
2\int_{x\in\R^l:\,\abs{x-x_0}<\kappa}x
^T c_s(u,x)\,D  w_s(u,x)m_s(x)\,dx\intertext{and}
\lim_{\delta\to\infty}    \int_{\R^l}D p(x)^T
\bl(
a_s(u,x)-\frac{1}{2}\,
\text{div}\,c_s(u,x)\br)m_s(x)
\,dx&=
2\int_{x\in\R^l:\,\abs{x-x_0}<\kappa}x^T
\bl(
a_s(u,x)-\frac{1}{2}\,
\text{div}\,c_s(u,x)\br)m_s(x)
\,dx\,.
\end{align*}
Dividing the righthand sides by the volume of the ball of radius 
$\kappa$ centred at $x_0$\,, letting $\kappa\to0$\, and
accounting for  \eqref{eq:96} and for 
$D w_s(u,x)$\,, $ m_s(x)$\,, $c_s(u,x)$ and
$a_s(u,x)$ being continuous in $x$\,,
yields \eqref{eq:91}.

The part that concerns $v_s(u,x)$ is dealt with similarly, except that
 one uses
 Theorem 2 of Pardoux and Veretennikov \cite{ParVer01}  with
 $\beta=1$
to bound the growth rate of the second term  of the sum in \eqref{eq:91c}.
\end{proof}

We now take on  the proof of Theorem~\ref{the:iden_reg}.
Let  $\hat w_s(u,x)$ and   $\hat v_s(u,x)$ represent 
$w_s(u,x)$ and $v_s(u,x)$\,, respectively, in the statement of
 Lemma \ref{le:reg} for $m_s(x)=\hat m_s(x)$\,.
We define, guided by \eqref{eq:89} and \eqref{eq:111},
on recalling  \eqref{eq:151} and \eqref{eq:40},
\begin{multline}
  \label{eq:33}
     \hat{\lambda}_s(u)=
\bl(\int_{\R^l}Q_{s,\hat m_s(\cdot)}(u,x) \hat{m}_s(x)\,dx\br)^{-1}
\bl(\dot{\hat{X}}_s-\int_{\R^l}A_s(u,x)\hat{m}_s(x)\,dx\\-
\int_{\R^l}G_s(u,x)
\bl(\frac{D\hat m_s(x)}{2\hat m_s(x)}
-D \hat{w}_s(u,x)\br)\hat{m}_s(x)\,dx\br)
\end{multline}
if $C_t(u,x)-G_t(u,x)c_t(u,x)^{-1}G_t(u,x)^T$
 is positive definite uniformly in $x$ and
locally uniformly in $(t,u)$ and
$   \hat{\lambda}_s(u)=0$ if $C_t(u,x)=0$ for all $(t,u,x)$\,,
and 
\begin{equation}
  \label{eq:74}
\hat{h}_s(u,x)=\frac{1}{2}\,\ln
\,\hat m_s(x)
-\hat{w}_s(u,x)-\hat{v}_s(u,x)^T\hat{\lambda}_s(u)
 \end{equation}
so that
\begin{equation}
  \label{eq:87}
 D\hat{h}_s(u,x)=
\frac{D\hat m_s(x)}{2\hat m_s(x)}-D\hat{w}_s(u,x)
-D\hat{v}_s(u,x)^T\hat{\lambda}_s(u)\,.
\end{equation}
We note that by \eqref{eq:88},
\begin{equation}
  \label{eq:157}
Q_{s,\hat m_s(\cdot)}(u,x)=C_s(u,x)-\norm{D\hat v_s(u,x)}^2_{c_s(u,x)}\,.
\end{equation}
The continuity properties of $w_s(u,x)$ and $v_s(u,x)$ established in
Lemma \ref{le:reg} imply that   
  $\hat\lambda_s(u)$ is continuous in $u$ and that
 $\hat h_s(u,x)$
 and $D\hat h_s(u,x)$ are continuous in
 $(u,x)$\,,
 for almost all $s\in\R_+$\,. 

If $C_t(u,x)-G_t(u,x)
c_t(u,x)^{-1}G_t(u,x)^T$
 is positive definite uniformly in $x$ and
locally uniformly in $(t,u)$, then,  the analogue of \eqref{eq:120}
for $D\hat w_s(u,x)$\,,
\eqref{eq:33} and Condition \ref{con:coeff} imply that, for some $\vartheta_1>0$\,,
\begin{multline}
  \label{eq:86}
     \abs{\hat{\lambda}_s(u)}\le
\vartheta_1
\bl(\abs{\dot{\hat{X}}_s}+\sup_{x\in \R^l}\abs{A_s(u,x)}+
\bl(\int_{\R^l}\abs{a_s(u,x)}^2\hat{m}_s(x)\,dx\br)^{1/2}
+\sup_{x\in\R^l}\frac{\abs {D\hat m_s(x)}}{\hat m_s(x)}\br)\,.
\end{multline}
Since  $\hat m_s(x)=M_s e^{-\alpha\abs{x}}$ for
$\abs{x}>2r$ and $\abs{a_s(u,x)}$ grows at most linearly in $x$\,,
we conclude that $\abs{\hat{\lambda}_s(u)}$ is locally bounded in
$(s,u)$\,. 

Therefore, 
by Lemma \ref{le:reg},
for all $L>0$\,, all open balls $S$ in $\R^l$\,, and all $q>1$\,,
\begin{equation}
  \label{eq:64}
  \sup_{s\in[0,t]}\sup_{u\in\R^n:\,\abs{u}\le L}\abs{\hat\lambda_s(u)}
+  
\sup_{s\in[0,t]}\sup_{u\in\R^n:\,\abs{u}\le L}
\int_{S}\abs{D\hat h_s(u,x)}^q\,dx<\infty\,.
\end{equation}

By Theorem \ref{le:vid},
 the supremum in  \eqref{eq:150} is attained at
$\lambda=\hat{\lambda}_s(u)$ and $ g=D\hat{h}_s(u,x)$\,,
however, the function $\hat h_s(u,x)$ might not be of compact support in $x$\,
so in order to use it in Lemma \ref{le:sup_ld_function}, we need to
restrict it to a compact set.
Let $\eta(y)$ represent an $\R_+$-valued
nonincreasing $\mathbb{C}^\infty_0(\R_+)$-function such that $\eta(y)=1$ for $0\le y\le 1$ and
$\eta(y)=0$ for $y\ge 2$\,.
Let
$\hat{w}^{i}_s(u,x)=\hat w_s(u,x)\eta(\abs{x}/i)$
and $\hat{v}^{i}_s(u,x)=\hat v_s(u,x)\eta(\abs{x}/i)$\,.
We note that
\begin{subequations}
  \begin{align}
    \label{eq:169}
    D   \hat{w}^{i}_s(u,x)&=
\eta(\frac{\abs{x}}{i})D   \hat{w}_s(u,x)
+\frac{x}{i\abs{x}}\,\,D\eta(\frac{\abs{x}}{i})\,\hat w_s(u,x)
\intertext{and}
    \label{eq:169a}
    D   \hat{v}^{i}_s(u,x)&=D   \hat{v}_s(u,x)\,
\eta(\frac{\abs{x}}{i})
+\frac{x}{i\abs{x}}\,\hat v_s(u,x)\,D\eta(\frac{\abs{x}}{i})^T\,\,\,.
  \end{align}
\end{subequations}
 We define, in analogy with \eqref{eq:33},
  \begin{multline}
      \label{eq:36}
           \hat{\lambda}^i_s(u) =
\bl(\int_{\R^l}Q_{s,\hat m_s(\cdot)}(u,x) \hat{m}_s(x)\,dx\br)^{-1}
\bl(\dot{\hat{X}}_s-\int_{\R^l}A_s(u,x)\hat{m}_s(x)\,dx\\-
\int_{\R^l}G_s(u,x)\bl(
\frac{1}{2}\,D(\eta(\frac{\abs{x}}{i})\ln \hat m_s(x))-
D \hat{w}^i_s(u,x)\br)
\hat{m}_s(x)\,dx\br)
\end{multline}
 if $C_t(u,x)-G_t(u,x)
c_t(u,x)^{-1}G_t(u,x)^T$
 is positive definite uniformly in $x$ and
locally uniformly in $(t,u)$\,, and
$           \hat{\lambda}^i_s(u) =0$ if $C_t(u,x)=0$\,.
We let, in analogy with \eqref{eq:74},
\begin{equation}
  \label{eq:44}
   \hat{h}^i_s(u,x)=\frac{1}{2}\,\eta(\frac{\abs{x}}{i})\ln \hat m_s(x)
-\hat{w}^i_s(u,x)
-\hat{v}^i_s(u,x)^T\hat{\lambda}^i_s(u)\,.
\end{equation}
In analogy with \eqref{eq:86} and in view of \eqref{eq:169} and
\eqref{eq:91c} in Lemma
\ref{le:reg}, one can see that
the $\abs{\hat\lambda^i_s(u)}$ are bounded uniformly in $i$ and
locally uniformly in $(s,u)$\,,
where the bound may depend on $\alpha$\,. Also, $\hat\lambda^i_s(u)$
is continuous in $u$\,, so it satisfies the hypotheses of Lemma
\ref{le:sup_ld_function}. 

If $C_t(u,x)-G_t(u,x)
c_t(u,x)^{-1}G_t(u,x)^T$
 is positive definite uniformly in $x$ and
locally uniformly in $(t,u)$\,, then by
 \eqref{eq:169}, Lemma \ref{le:reg},
(\ref{eq:33}), and (\ref{eq:36}),
\begin{equation}
  \label{eq:136}
  \lim_{i\to\infty} \sup_{s\in[0,t]}\sup_{u\in\R^n:\,\abs{u}\le L}
\abs{\hat \lambda^i_s(u)-\hat \lambda_s(u)}=0\,.
\end{equation}
The latter convergence also holds if $C_t(u,x)=0$ in that
$\hat \lambda^i_s(u)=\hat \lambda_s(u)=0$\,.

Similarly, since  by \eqref{eq:169}, \eqref{eq:169a}, and \eqref{eq:44},
\begin{multline}
  \label{eq:138}
  D   \hat{h}^{i}_s(u,x)=
\eta\bl(\frac{\abs{x}}{i}\br)
  D   \hat{h}_s(u,x)
+\frac{1}{i}\,\frac{x}{\abs{x}}\,D\eta\bl(\frac{\abs{x}}{i}\br)\bl(
\frac{1}{2}\,\ln \hat m_s(x)-
\hat
w_s(u,x)-\hat v_s(u,x)^T\hat{\lambda}^{i}_s(u)
\br),
\end{multline}
we have that
\begin{equation}
  \label{eq:35}
\lim_{i\to\infty} 
\sup_{s\in[0,t]}\sup_{u\in\R^n:\,\abs{u}\le L}
\int_{S}\abs{D\hat h^i_s(u,x)-D\hat h_s(u,x)}^q\,\,dx=0\,,
\end{equation}
for all $L>0$, all open balls $S$ in $\R^l$\,, and all $q>1$\,.
The functions $\hat h^i_s(u,x)$ also satisfy the hypotheses of Lemma 
\ref{le:sup_ld_function}.

Another auxiliary lemma is in order.
\begin{lemma}
  \label{le:convergence}
Suppose, for $i\in\N$\,,
 $X^i\in \mathbb{C}(\R_+,\R^n)$, $X\in \mathbb{C}(\R_+,\R^n)$\,,
 and  $m^i_s(x)$ and $m_s(x)$ are measurable functions which are 
probability densities in $x$ on $\R^l$ for almost all $s$
such that
\begin{multline*}
    \int_{\R^l}\bl(\frac{1}{2}\,\text{tr}\,(c_s(X^i_s,x)
 D^2p(x))+D p(x)^T
\bl(a_s(X^i_s,x)+G_s(X^i_s,x)^T\hat{\lambda}_s^{i}(X^i_s)
\\-\frac{1}{2}\,\text{div}\,c_s(X^i_s,x)
+c_s(X^i_s,x)D \hat{h}^{i}_s(X^i_s,x)\br) \br) 
m^{i}_s(x) \,dx=0
\end{multline*}
and 
\begin{multline}
  \label{eq:131a}
    \int_{\R^l}\bl(\frac{1}{2}\,\text{tr}\,(c_s(X_s,x)
 D^2p(x))+ D p(x)^T
\bl(a_s(X_s,x)+G_s(X_s,x)^T\hat{\lambda}_s(X_s)\\-
\frac{1}{2}\,\text{div}\,c_s(X_s,x)
+c_s(X_s,x)D \hat{h}_s(X_s,x)\br)\br) 
m_s(x) \,dx=0\,,
\end{multline}
for all $p\in\mathbb{C}^\infty_0(\R^l)$\,.
If $X^i\to X$ as $i\to\infty$\,, then,
for all $\alpha$ great enough and for all $t>0$\,,
\begin{subequations}
  \begin{align}
    \label{eq:90}
\lim_{i\to\infty} \int_0^t\int_{\R^l}\abs{m^i_s(x)-m_s(x)}\,dx\,ds=0,\\
  \label{eq:128}
  \lim_{i\to\infty}\int_0^t\int_{\R^l}
\norm{D \hat h^i_s(X^i_s,x)}^2_{c_s(X^i_s,x)}\,m^i_s(x)\,dx\,ds
=\int_0^t\int_{\R^l}
\norm{D \hat h_s(X_s,x)}^2_{c_s(X_s,x)}\, m_s(x)\,dx\,ds\,,
\intertext{and}
  \label{eq:128a}
\lim_{i\to\infty}\int_0^t\norm{\hat\lambda^i_s(X^i_s)}^2_{\int_{\R^l}C_s(X^i_s,x)\,m^i_s(x)\,dx}\,ds=
\int_0^t\norm{\hat\lambda_s(X_s)}^2_{\int_{\R^l}C_s(X_s,x)\, m_s(x)\,dx}\,ds\,.
\end{align}
\end{subequations}
\end{lemma}
\begin{proof}
Let us first address  existence and uniqueness of $m_s^i(x)$ and
$m_s(x)$\,.
Since $\sup_{\abs{u}\le L}(x/\abs{x})^T a_s(u,x)\to-\infty$ as
$\abs{x}\to\infty$\,, the function $G_s(u,x)$ is bounded, 
the function $\hat{\lambda}_s^{i}(u)$ is bounded locally  in $(s,u)$\,
and the function $\hat h^i_s(u,x)$ is of compact support in $x$ locally
uniformly in $(s,u)$\,, we have that
\begin{multline}
  \label{eq:135}
    \lim_{\abs{x}\to\infty}\frac{x^T}{\abs{x}}
\bl(a_s(X^i_s,x)+G_s(X^i_s,x)^T\hat{\lambda}_s^{i}(X^i_s)
-\frac{1}{2}\,\text{div}\,c_s(X^i_s,x)
+c_s(X^i_s,x)D \hat{h}^{i}_s(X^i_s,x)\br)=-\infty\,,
\end{multline}
which implies that $m^i_s(x)$ is well defined and is specified
uniquely, see, e.g., Metafune, Pallara, and Rhandi \cite[Theorem 2.2,
Proposition 2.4]{MetPalRha05}.

By \eqref{eq:87}, relations
\eqref{eq:91} and \eqref{eq:91b} of Lemma \ref{le:reg} imply that
\begin{multline*}
  x^T\,
\bl(a_s(X_s,x)+G_s(X_s,x)^T\hat{\lambda}_s(X_s)
-\frac{1}{2}\,\text{div}\,c_s(X_s,x)
+c_s(X_s,x)D\hat{h}_s(X_s,x)\br)\\=\frac{x^T}{2}
 c_s(X_s,x)\frac{D\hat m_s(x)}{\hat m_s(x)}\,.
\end{multline*}
If $\abs{x}>2r$, then $D\hat m_s(x)/\hat m_s(x)=-\alpha x/\abs{x}$\,,
so, 
locally uniformly in $s$,
\begin{multline*}
  \limsup_{\abs{x}\to\infty}\frac{x^T}{\abs{x}}\,
\bl(a_s(X_s,x)+G_s(X_s,x)^T\hat{\lambda}_s(X_s)
-\frac{1}{2}\,\text{div}\,c_s(X_s,x)
+c_s(X_s,x)D \hat{h}_s(X_s,x)\br)<0\,,
\end{multline*}
which ensures the existence and uniqueness of $m_s(x)$\,.

As in the proof of Lemma \ref{le:zero},
we will show  that, for arbitrary  $\delta>0$, there exists $\alpha>0$
such that for all $t>0$ 
\begin{equation}
  \label{eq:144}
\sup_{s\in[0,t]}\sup_{i\in \N}\int_{\R^l}e^{\delta\abs{x}}\,m_s^i(x)\,dx<\infty\,.
\end{equation}
We begin by establishing a
uniform version of \eqref{eq:135}:  
\begin{multline}
  \label{eq:137}
\lim_{\alpha\to\infty}
      \limsup_{\abs{x}\to\infty}\sup_{s\in[0,t]}
\sup_{i\in \N}\frac{x^T}{\abs{x}}\,
\bl(a_s(X^i_s,x)+G_s(X^i_s,x)^T\hat{\lambda}_s^{i}(X^i_s)
-\frac{1}{2}\,\text{div}\,c_s(X^i_s,x)\\
+c_s(X^i_s,x)D \hat{h}^{i}_s(X^i_s,x)\br)=-\infty\,.
\end{multline}
By \eqref{eq:138}, \eqref{eq:87}, \eqref{eq:91}, and \eqref{eq:91a},
 for $\abs{x}>2r$\,, on recalling that
$\hat m_s(x)=M_se^{-\alpha\abs{x}}$\,, 
\begin{multline}\label{eq:140}
  \frac{x^T}{\abs{x}}\,
\bl(a_s(X^i_s,x)+G_s(X^i_s,x)^T\hat{\lambda}_s^{i}(X^i_s)
-\frac{1}{2}\,\text{div}\,c_s(X^i_s,x)
+c_s(X^i_s,x)D \hat{h}^{i}_s(X^i_s,x)\br)\\=
\frac{x^T}{\abs{x}}\,
\bl(a_s(X^i_s,x)+G_s(X^i_s,x)^T\hat{\lambda}_s^{i}(X^i_s)
-\frac{1}{2}\,\text{div}\,c_s(X^i_s,x)\br)
\bl(1-\eta\bl(\frac{\abs{x}}{i}\br)\br)
\\-\frac{\alpha}{2}\,\norm{\frac{x}{\abs{x}}}_{c_s(X^i_s,x)}^2\bl(
\eta\bl(\frac{\abs{x}}{i}\br)+\frac{\abs{x}}{i}\,
D\eta\bl(\frac{\abs{x}}{i}\br)\br)\br)
\\-\frac{1}{i}\,\norm{\frac{x}{\abs{x}}}^2_{c_s(X^i_s,x)}
\,D\eta\bl(\frac{\abs{x}}{i}\br)\bl(\frac{1}{2}\,\ln M_s
-\hat
w_s(X^i_s,x)-\hat v_s(X^i_s,x)^T
\hat{\lambda}^{i}_s(X^i_s)\br)\,.
\end{multline}
Let $\kappa\in(0,1)$ be such that $\eta(y)+y\,D\eta(y)>0$ if 
$\eta(y)\in
[1-\kappa,1]$\,. 
If
$\eta(\abs{x}/i)\ge 1-\kappa$\,, then $\eta(\abs{x}/i)\in[
1-\kappa,1]$\,, so substituting $1-\kappa$ and $1$ for $\eta(\abs{x}/i)$\,,
we have that
\begin{multline}\label{eq:59}
  \frac{x^T}{\abs{x}}\,
\bl(a_s(X^i_s,x)+G_s(X^i_s,x)^T\hat{\lambda}_s^{i}(X^i_s)
-\frac{1}{2}\,\text{div}\,c_s(X^i_s,x)\br)\bl(1-
\eta\bl(\frac{\abs{x}}{i}\br)\br)\\\le 
\Bl(\kappa\,\frac{x^T}{\abs{x}}\,
\bl(a_s(X^i_s,x)+G_s(X^i_s,x)^T\hat{\lambda}_s^{i}(X^i_s)
-\frac{1}{2}\,\text{div}\,c_s(X^i_s,x)\br)\Br)^+\,.
\end{multline}
In addition, for some $\vartheta>0$\,,
\begin{equation*}
  \norm{\frac{x}{\abs{x}}}_{c_s(X^i_s,x)}^2\bl(
\eta\bl(\frac{\abs{x}}{i}\br)+\frac{\abs{x}}{i}\,
D\eta\bl(\frac{\abs{x}}{i}\br)\br)\ge  \vartheta
\inf_{y:\,\eta(y)\ge 1-\kappa}(\eta(y)+y\,D\eta(y))>0\,.
\end{equation*}
The righthand side of \eqref{eq:59} being 
 equal to $0$
if $\abs{x}$ is great enough, uniformly in $i$
and locally uniformly in $s$\,,
it follows that there exists $\hat M_1>0$ such that for all $\abs{x}$ great enough,
depending on  $\alpha>0$ and $t>0$\,,
uniformly in $i\in\N$ and $s\in[0,t]$\,,
\begin{multline*}
  \Bl(\frac{x^T}{\abs{x}}\,
\bl(a_s(X^i_s,x)+G_s(X^i_s,x)^T\hat{\lambda}_s^{i}(X^i_s)
-\frac{1}{2}\,\text{div}\,c_s(X^i_s,x)\br)\bl(1-
\eta\bl(\frac{\abs{x}}{i}\br)\br)
\\-\frac{\alpha}{2}\,\norm{\frac{x}{\abs{x}}}_{c_s(X^i_s,x)}^2\bl(
\eta\bl(\frac{\abs{x}}{i}\br)+\frac{\abs{x}}{i}\,
D\eta\bl(\frac{\abs{x}}{i}\br)\br)\Br)\ind_{\{\eta(\abs{x}/i)\ge
  1-\kappa\}}(x)
\le 
-\hat M_1\alpha\ind_{\{\eta(\abs{x}/i)\ge
  1-\kappa\}}(x)\,.
\end{multline*}
If
$\eta(\abs{x}/i)<
  1-\kappa$, then, analogously,
given arbitrary $\hat M_2>0$, we have that, for all $\abs{x}$
great enough, depending on $\alpha$ and $t>0$,   all $i\in\N$\,, and
all $s\in[0,t]$\,,
\begin{equation*}
    \frac{x^T}{\abs{x}}\,
\bl(a_s(X^i_s,x)+G_s(X^i_s,x)^T\hat{\lambda}_s^{i}(X^i_s)
-\frac{1}{2}\,\text{div}\,c_s(X^i_s,x)\br)\bl(1-
\eta\bl(\frac{\abs{x}}{i}\br)\br)\le- \hat M_2.
\end{equation*}
Also, 
\begin{equation*}
  \norm{\frac{x}{\abs{x}}}_{c_s(X^i_s,x)}^2\bl(
\eta\bl(\frac{\abs{x}}{i}\br)+\frac{\abs{x}}{i}\,
D\eta\bl(\frac{\abs{x}}{i}\br)\br)\ge \vartheta
\inf_{y :\,\eta(y)\le 1-\kappa}(\eta(y)+y\,D\eta(y))\,.
\end{equation*}
Hence, given arbitrary 
$\hat M_3>0$,  for all
$\abs{x}$ great enough, depending on $\alpha$ and $t$\,, 
uniformly in $i\in\N$ and $s\in[0,t]$\,,
  \begin{multline*}
  \Bl(\frac{x^T}{\abs{x}}\,
\bl(a_s(X^i_s,x)+G_s(X^i_s,x)^T\hat{\lambda}_s^{i}(X^i_s)
-\frac{1}{2}\,\text{div}\,c_s(X^i_s,x)\br)\bl(1-\eta\bl(\frac{\abs{x}}{i}\br)\br)
\\-\frac{\alpha}{2}\,\norm{\frac{x}{\abs{x}}}_{c_s(X^i_s,x)}^2\bl(
\eta\bl(\frac{\abs{x}}{i}\br)+\frac{\abs{x}}{i}\,
D\eta\bl(\frac{\abs{x}}{i}\br)\br)\Br)\ind_{\{\eta(\abs{x}/i)\le
  1-\kappa\}}(x)
\le -\hat M_3\ind_{\{\eta(\abs{x}/i)\le
  1-\kappa\}}(x)\,.
\end{multline*}
We conclude that 
\begin{multline}
  \label{eq:61}
\lim_{\alpha\to\infty}    \limsup_{\abs{x}\to\infty}\sup_{s\in[0,t]}\sup_{i\in\N}
\Bl(\frac{x^T}{\abs{x}}\,
\bl(a_s(X^i_s,x)+G_s(X^i_s,x)^T\hat{\lambda}_s^{i}(X^i_s)\\
-\frac{1}{2}\,\text{div}\,c_s(X^i_s,x)\br)\bl(1-\eta\bl(\frac{\abs{x}}{i}\br)\br)
-\frac{\alpha}{2}\,\norm{\frac{x}{\abs{x}}}_{c_s(X^i_s,x)}^2\bl(
\eta\bl(\frac{\abs{x}}{i}\br)+\frac{\abs{x}}{i}\,
D\eta\bl(\frac{\abs{x}}{i}\br)\br)\Br)=-\infty.
\end{multline}
We now work with   line 3 of \eqref{eq:140}. Since $D\eta(y)\le 0$\,, 
 and since, by \eqref{eq:91a},
 the $\hat w_s(X^i_s,x)$ are bounded from
above in $x$ and $i$ locally uniformly in $s$\,, 
\begin{equation}
  \label{eq:142}
\limsup_{\alpha\to\infty}
\limsup_{\abs{x}\to\infty}\sup_{s\in[0,t]}\sup_{i\in \N}\frac{1}{i}\,
\bl(-\norm{  \frac{x}{\abs{x}}}_{c_s(X^i_s,x)}^2\br)\,
\,D \eta
\bl(\frac{\abs{x}}{i}\br)\hat w_s(X^i_s,x)<\infty\,.
\end{equation}
Since $D\eta(y)=0$ unless 
$y\in[1,2]$\,, we have that
$  (1/i)\abs{ D\eta(\abs{x}/i)\hat v_s(X^i_s,x)}\le 
4\abs{ D\eta(\abs{x}/i)}\abs{\hat v_s(X^i_s,x)}/(1+\abs{x})$\,, which
is bounded in $i$ and $x$ locally uniformly in $s$  for
$\alpha$ great enough by \eqref{eq:91c}
of Lemma \ref{le:reg}.
It follows that
\begin{equation}
  \label{eq:143}
\limsup_{\alpha\to\infty}
\limsup_{\abs{x}\to\infty}\sup_{s\in[0,t]}\sup_{i\in \N}\frac{1}{i}\,
\bl(-\norm{  \frac{x}{\abs{x}}}_{c_s(X^i_s,x)}^2\br)\,
\, D\eta\bl(\frac{\abs{x}}{i}\br)\hat v_s(X^i_s,x)^T
\hat{\lambda}^{i}_s(X^i_s)<\infty\,.
\end{equation}
In addition,
\begin{equation}
  \label{eq:110}
    \lim_{\abs{x}\to\infty}\sup_{s\in[0,t]}\sup_{i\in \N}
\frac{1}{i}\,
\norm{  \frac{x}{\abs{x}}}_{c_s(X^i_s,x)}^2\,
 D\eta\bl(\frac{\abs{x}}{i}\br)
\frac{1}{2}\,\ln M_s=0\,.
\end{equation}

Combining  (\ref{eq:140}), \eqref{eq:61},
 (\ref{eq:142}), 
(\ref{eq:143}), and \eqref{eq:110} proves (\ref{eq:137}).

The rest of the proof is carried out similarly to the proof of Lemma
\ref{le:zero},
with \eqref{eq:137} assuming the role of the condition that
$\sup_{i\in\N}a_{s+t}(X^i_s,x)^T x/\abs{x}\to-\infty$\,.
Firstly,
letting  
$f_s^i(x)=
a_s(X^i_s,x)+G_s(X^i_s,x)^T\hat{\lambda}_s^{i}(X^i_s)
-\text{div}\,c_s(X^i_s,x)/2
+c_s(X^i_s,x)D \hat{h}^{i}_s(X^i_s,x)  
$ and $\mathcal{L}^i_sp(x)=(1/2)\,\text{tr}\,(c_s(X^i_s,x)
 D^2p(x))+f_s^i(x)^T D p(x)$ one derives the inequality in \eqref{eq:95},
 where  $F$ represents a $\mathbb{C}^\infty(\R^l)$-function
such that $F(x)=e^{\delta\abs{x}}$ if $\abs{ x}\ge 1$\,.
 The inequality in  \eqref{eq:144} follows if one recalls 
 \eqref{eq:138}, that the $\hat\lambda^i_s(u)$ are bounded uniformly in
 $i$ and locally uniformly in $(s,u)$\,, and 
that according to
Lemma~\ref{le:reg}, 
\begin{equation*}
\sup_{s\in[0,t]}\sup_{u\in\R^n:\,\abs{u}\le
  L}\norm{\hat w_s(u,\cdot)}_{\mathbb{W}^{2,q}(S)}+
\sup_{s\in[0,t]}\sup_{u\in\R^n:\,\abs{u}\le
  L}\norm{\hat v_s(u,\cdot)}_{\mathbb{W}^{2,q}(S,\R^n)}<\infty
\end{equation*}
 for all
$q>1$, all $L>0$\,,
and all open balls $S\subset \R^l$ so that by  Sobolev's imbedding
\begin{equation*}
\sup_{s\in[0,t]}\sup_{u\in\R^n:\,\abs{u}\le
  L}\sup_{x\in\R^l:\abs{x}\le R}
(\abs{\hat w_s(u,x)}+\abs{D\hat w_s(u,x)}+
\abs{\hat v_s(u,\cdot)}+\norm{D\hat v_s(u,x)})<\infty\,.
\end{equation*}
Since the $\lambda^i_s(u)$ are bounded uniformly in $i$ and locally
uniformly in $(s,u)$ and \eqref{eq:64} and \eqref{eq:144}  hold,
 by  Proposition 2.16 in Bogachev,
Krylov, and R\"ockner 
\cite{BogKryRoc01}, for almost all $s$
 the functions $m_s^i(\cdot)$ converge 
in the variation norm along a subsequence 
to  probability density $\tilde m_s(\cdot)$\,.
By \eqref{eq:87}, \eqref{eq:138}, the bounds \eqref{eq:91c}, and
by \eqref{eq:144}, we have that
\begin{equation}
  \label{eq:155}
\sup_{s\in[0,t]}\sup_{i\in\N}\int_{\R^l}
\norm{D \hat h^i_s(X^{i}_s,x)}^3_{c_s(X^i_s,x)}\,m^i_s(x)\,dx<\infty
\,.
\end{equation}
Since  $\sup_{i\in\N}\abs{\hat
\lambda^i_s}<\infty$\,, the convergences in \eqref{eq:136} and
\eqref{eq:35} imply that $\tilde m_s(x)$ must satisfy
\eqref{eq:131a}, so $\tilde m_s(x)=m_s(x)$ and $m_s^i(\cdot)\to
m_s(\cdot)$ in the variation norm. 
The limit in \eqref{eq:90} follows by dominated convergence.
The convergence in \eqref{eq:128a} follows from 
\eqref{eq:136}, \eqref{eq:64}, and \eqref{eq:82}.
For \eqref{eq:128}, we also take into account \eqref{eq:155}.
\end{proof}

We finish  the proof of Theorem \ref{the:iden_reg}.
Let, given $N\in\N$\,,  $\hat\tau^{N,i}$ and
$\hat\theta^{N,i}$ be defined by the respective equations 
\eqref{eq:62} and (\ref{eq:27}) 
with $\hat{\lambda}^{i}_s(u)$ and
$\hat{h}^{i}_s(u,x)$ as $\hat{\lambda}_s(u)$ and
$\hat{h}_s(u,x)$\,, respectively. Since 
 the functions $\hat\lambda^i_s(u)$ and $\hat
h^i_s(u,x)$ 
satisfy the hypotheses of Lemma \ref{le:sup_ld_function},
 there exist
$\gamma^{N,i}=
(X^{N,i},\mu^{N,i})\in \Gamma$  such that
$\hat\theta^{N,i}(\gamma^{N,i})
=\tilde{\mathbf{I}}(\gamma^{N,i})$
and $\gamma^{N,i}\in K_{2N+2}$ for all $i$\,.
In particular, $X^{N,i}_0=\hat{u}$\,,
   $\mu^{N,i}(ds,dx)=m^{N,i}_s(x)\,dx\,ds$\,, where 
$m^{N,i}_s(\cdot)\in \mathbb{P}(\R^l)$ (see Theorem \ref{le:vid}), and the set
$\{\gamma^{N,i},\,i=1,2,\ldots\}$ is relatively compact.
Since $\tilde{\mathbf{I}}(\gamma^{N,i})\ge 
\mathbf{I}^{\ast\ast}(\gamma^{N,i})$, on the one hand, and
$\hat\theta^{N,i}(\gamma^{N,i})
\le \mathbf{I}^{\ast\ast}(\gamma^{N,i})$ by
(\ref{eq:85}) and (\ref{eq:27}), on the other hand, we have that
\begin{equation}
  \label{eq:11}
  \hat\theta^{N,i}(\gamma^{N,i})= \mathbf{I}^{\ast\ast}(\gamma^{N,i})=
\tilde{\mathbf{I}}(\gamma^{N,i})\,.
\end{equation}
Let $\mu^{N,i}\to \mu^N$ in $\bC(\R_+,\mathbb{M}(\R^l))$
and $X^{N,i}\to
X^{N}$ in $\mathbb{C}(\R_+,\R^n)$ along a
subsequence of $i$\,, which we still denote by $i$\,.

By  (\ref{eq:27}) and \eqref{eq:11}, the suprema in
 (\ref{eq:85}) for $(X,\mu)=(X^{N,i},\mu^{N,i})$
are attained   at
$(\hat{\lambda}_s^{i}(X^{N,i}_s),\hat{h}^{i}_s(X^{N,i}_s,x))$ when
$s\le
{\hat\tau^{N,i}(\gamma^{N,i})}$\,.
In particular, since the supremum  over $h$ for
$\lambda=\hat{\lambda}_s^{i}(X^{N,i}_s)$ is attained at
$h(x)=\hat{h}^{i}_s(X^{N,i}_s,x)$\,, we have that
\begin{multline}\label{eq:42}
D\hat{h}^{i}_s(X^{N,i}_s,x)=
\Pi_{c_s(X^{N,i}_s,\cdot),m_s^{N,i}(\cdot)}
\bl(\frac{Dm_s^{N,i}(x)}{2m^{N,i}_s(x)}\\+c_s(X^{N,i}_s,\cdot)^{-1}\bl(
\frac{1}{2}\,\text{div}\,c_s(X^{N,i}_s,\cdot)-
a_s(X^{N,i}_s,\cdot)-G_s(X^{N,i}_s,\cdot)^T\hat\lambda^i_s(X^{N,i}_s)\br)\br)(x)\,.
\end{multline}
Hence, for $p\in \mathbb{C}^\infty_0(\R^l)$\,,
\begin{multline*}
  \int_{\R^l}D p(x)^T
\bl(c_s(X^{N,i}_s,x)\,
\frac{Dm_s^{N,i}(x)}{2m^{N,i}_s(x)}
+\frac{1}{2}\,\text{div}\,c_s(X^{N,i}_s,x)
-a_s(X^{N,i}_s,x)\\-G_s(X^{N,i}_s,x)^T\hat{\lambda}_s^{i}(X^{N,i}_s)
-c_s(X^{N,i}_s,x)D \hat{h}^{i}_s(X^{N,i}_s,x)
\br) m^{N,i}_s(x) \,dx=0\,.
\end{multline*}
Integration by parts yields
\begin{multline}
  \label{eq:92}
  \int_{\R^l}\bl(\frac{1}{2}\,\text{tr}\,(c_s(X^{N,i}_s,x)
 D^2p(x))+D p(x)^T
\bl(a_s(X^{N,i}_s,x)-\frac{1}{2}\,\text{div}\,c_s(X^{N,i}_s,x)
\\+G_s(X^{N,i}_s,x)^T\hat{\lambda}_s^{i}(X^{N,i}_s)
+c_s(X^{N,i}_s,x)D \hat{h}^{i}_s(X^{N,i}_s,x)\br)\br) m^{N,i}_s(x) \,dx=0\,.
\end{multline}
Thus, $m^{N,i}_s(x)\,dx$ is an invariant probability for a diffusion. 
By  \eqref{eq:96},  \eqref{eq:97a} (for $\hat w_s(u,x)$ and 
$\hat v_s(u,x)$), and \eqref{eq:87}, via a similar
manipulation, 
\begin{multline*}
    \int_{\R^l}\bl(\frac{1}{2}\,\text{tr}\,(c_s(X^{N}_s,x)D^2p(x))+
D p(x)^T\bl(a_s(X^{N}_s,x)-\frac{1}{2}\,\text{div}\,c_s(X^{N}_s,x)
\\+G_s(X^{N}_s,x)^T\hat{\lambda}_s(X^{N}_s)
+c_s(X^{N}_s,x)D \hat{h}_s(X^{N}_s,x)\br)\br) \hat{m}_s(x) \,dx=0\,.
\end{multline*}
Let $\tilde{m}^{N,i}_s(x)$ represent a probability density that solves
\eqref{eq:92} for all $s\in\R_+$ rather than for $s\le
\hat\tau^{N,i}(\gamma^{N,i})$\,.  The existence of
$\tilde{m}^{N,i}_s(x)$ is established as in the proof of Lemma
\ref{le:convergence}, more specifically, see \eqref{eq:135}.
 Lemma \ref{le:convergence} implies that $\tilde{m}^{N,i}_s(x)\to
\hat m_s(x)$ in $\mathbb{L}^1([0,t]\times \R^l)$ as $i\to\infty$\,, that
\begin{subequations}
      \begin{align}
\label{eq:52}
  \lim_{i\to\infty}\int_0^t\int_{\R^l}
\norm{D \hat h^i_s(X^{N,i}_s,x)}^2_{c_s(X^{N,i}_s,x)}\,
\tilde m^{N,i}_s(x)\,dx\,ds
=\int_0^t\int_{\R^l}
\norm{D \hat h_s(X^N_s,x)}^2_{c_s(X^N_s,x)}\, \hat m_s(x)\,dx\,ds
\intertext{and that}
\label{eq:52a}
\lim_{i\to\infty}\int_0^t\norm{\hat\lambda^i_s(X^{N,i}_s)}^2_{\int_{\R^l}C_s(X^{N,i}_s,x)\,\tilde
  m^{N,i}_s(x)\,dx}\,ds=
\int_0^t\norm{\hat\lambda_s(X^N_s)}^2_{\int_{\R^l}C_s(X^N_s,x)\, \hat m_s(x)\,dx}\,ds\,.
\end{align}
\end{subequations}
By Lemma \ref{le:conv_stop},
$\hat\tau^{N,i}(\tilde\gamma^{N,i})\to
\hat\tau^{N}(\gamma^{N})$ as $i\to\infty$\,, where 
$\tilde\gamma^{N,i}=(X^{N,i},\tilde\mu^{N,i})$ and
$\tilde\mu^{N,i}(dx,ds)=\tilde m^{N,i}_s(x)\,dx\,ds$\,.
Since $\hat\tau^{N,i}(\tilde\gamma^{N,i})=\hat\tau^{N,i}(\gamma^{N,i})$, we
obtain that $\hat\tau^{N,i}(\gamma^{N,i})\to 
\tau^{N}(\gamma^{N})$ and that $m^{N,i}_s(x)\to
\hat m_s(x)$ in 
$\mathbb{L}^1([0,\tau^{N}(\gamma^{N})]\times \R^l)$\,,
so $\mu^N_s(dx)=\hat m_s(x)\,dx$ for almost 
all $s\le\tau^{N}(\gamma^{N})$\,.

We now use the fact that  the supremum 
in \eqref{eq:85} over $\lambda$ for $h(x)=\hat{h}^{i}_s(X^{N,i}_s,x)$
is attained at 
$\lambda=\hat{\lambda}_s^{i}(X^{N,i}_s)$\,.
If $C_t(u,x)=0$ 
and $A_t(u,x)$ is locally Lipschitz continuous in  $u$ locally uniformly in
$t$ and uniformly in $x$\,, then $\hat{\lambda}_s^{i}(X^{N,i}_s)=0$\,, so
$\dot{X}^{N,i}_s=\int_{\R^l}A_s(X^{N,i}_s,x)m^{N,i}_s(x)\,dx$\,,
which, as in the proof of Lemma \ref{le:zero},
implies since $X^{N,i}\to X^N$ in $\mathbb{C}(\R_+,\R^n)$ and 
$(m^{N,i}_s(x))\to (\hat m_s(x))$ in $\mathbb{L}^1([0,\tau^N(\gamma^N)]\times \R^l)$
as $i\to\infty$\, that
$\dot{X}^{N}_s
=\int_{\R^l}A_s(X^{N}_s,x)\hat m_s(x)\,dx$ a.e. for
$s\le \tau^N(\gamma^N)$\,. By uniqueness, $X^N_s=\hat X_s$
for
$s\le \tau^N(\gamma^N)$\,. As a byproduct, $\dot X^{N,i}_s\to \dot{ \hat
  X}_s$ as $i\to\infty$ a.e. on $[0,\tau^N(\gamma^N)]$\,.

Suppose that  $C_t(u,x)-G_t(u,x)c_t(u,x)^{-1}G_t(u,x)^T$
 is positive definite locally uniformly in $(t,u)$ and uniformly in
 $x$\,.
Then the maximisation condition is 
\begin{multline*}
    \dot{X}^{N,i}_s=\int_{\R^l}A_s(X^{N,i}_s,x)m^{N,i}_s(x)\,dx
+\int_{\R^l}
 G_s(X^{N,i}_s,x)D \hat{h}^{i}_s(X^{N,i}_s,x)\,m^{N,i}_s(x)\,dx
\\+\int_{\R^l}
 C_s(X^{N,i}_s,x)\hat{\lambda}_s^{i}(X^{N,i}_s) m^{N,i}_s(x)\,dx\,.
\end{multline*}
On integrating both sides from $0$ to $t$ and 
letting $i\to\infty$\,,
we have by the facts that $\gamma^{N,i}\to\gamma^N$\,, 
that $X^{N,i}\to X^N$\,,
 that
 $m^{N,i}_s(x)\to \hat m_s(x)$ 
 in $\mathbb{L}^1([0,\tau^N(\gamma^N)]\times \R^l)$\,, and that 
$\hat\lambda^i_s(u)\to
\hat\lambda_s(u)$  locally uniformly in $(s,u)$
as $i\to\infty$\,(see \eqref{eq:136}),
by \eqref{eq:35},  by \eqref{eq:155}, and by \eqref{eq:87}
  that,
 for almost all $s\le \tau^N(\gamma^N)$\,,
   \begin{multline}
\label{eq:163}
      \dot{X}^N_s=
\int_{\R^l} A_s(X^N_s,x)\hat{m}_s(x)\,dx
+\int_{\R^l}G_s(X^N_s,x)\,\bl(\frac{D\hat m_s(x)}{2\hat m_s(x)}
-D \hat{w}_s(X^N_s,x)\br)
\hat{m}_s(x)\,dx\\
+\int_{\R^l}
\bl( C_s(X^N_s,x)-G_s(X^N_s,x)\,
 D \hat{v}_s(X^{N}_s,x)\br)
  \hat{m}_s(x)\,dx\,\hat{\lambda}_s(X^N_s)\,.  
\end{multline}
Since $D\hat v_s(u,\cdot)\in
\mathbb{L}_0^{1,2}(\R^l,\R^{n\times l},c_s(x),\hat m_s(x)\,dx)$
and $G_s(u,\cdot)$ is bounded, \eqref{eq:97a}
extends to $Dp$ representing an arbitrary element of
$\mathbb{L}_0^{1,2}(\R^l,\R^l,c_t(x),\hat m_s(x)\,dx)$, so by \eqref{eq:157},
\begin{equation*}
\int_{\R^l}Q_{s,\hat m_s(\cdot)}(u,x)\hat m_s(x)\,dx=
\int_{\R^l}\bl( C_s(u,x)- G_s(u,x)
D \hat{v}_s(u,x)
\br)\,\hat m_s(x)\,dx\,.
\end{equation*}
Substitution of the latter expression in \eqref{eq:33} and
of \eqref{eq:33} into \eqref{eq:163}
obtains that   $\dot{X}^N_s=\dot{\hat{X}}_s$ a.e. on 
$[0,\tau^{N}(\gamma^{N})]$, so on recalling
that $X^N_0=\hat{X}_0=\hat{u}$ we  conclude that
$X^N_s=\hat{X}_s$ for $s\le \tau^N(\gamma^N)$\,.
In addition, $\dot X^{N,i}_s\to \dot{\hat X}_s$ as $i\to\infty$ 
a.e. on $[0,\tau^N(\gamma^N)]$\,.

Hence, in either case,
 $\tau^N(\gamma^N)=\tau^N(\hat{\gamma})$ and 
$\gamma^N_s=\hat \gamma_s$ for $s\le \tau^N(\hat\gamma)$
so that
$\theta^N(\gamma^N)=\theta^N(\hat{\gamma})$\,, where
$\hat\gamma=(\hat X,\hat \mu)$\,.
We show that
\begin{equation}
  \label{eq:51}
 \theta^N(\gamma^N)=\lim_{i\to\infty} 
\hat\theta^{N,i}(\gamma^{N,i})\,.  
\end{equation}
By \eqref{eq:27} and \eqref{eq:42},
\begin{multline*}
    \hat\theta^{N,i}(\gamma^{N,i})=
 \int_0^{\hat\tau^{N,i}(\gamma^{N,i})}\Bl(
 \hat\lambda^i_s(X^{N,i}_s)^T\,\bl(\dot{X}^{N,i}_s-
\int_{\R^l} A_s(X^{N,i}_s,x)m^{N,i}_s(x)\,dx\,\br)\\
-\frac{1}{2}\,
 \norm{\hat\lambda^i_s(X^{N,i}_s)}_{\int_{\R^l}  C_s(X^{N,i}_s,x)\,m^{N,i}_s(x)\,dx}^2
+\frac{1}{2}\,\int_{\R^l} \norm{D\hat h^i_s(X^{N,i}_s,x)}_{c_s(X^{N,i}_s,x)}^2
\,m^{N,i}_s(x) 
\,dx\br)\,ds\,.
\end{multline*}
Similarly,
\begin{multline*}
      \theta^N(\gamma^N)= \int_0^{\tau^N(\gamma^N)}\Bl(
\hat \lambda_s(X^N_s)^T\,\bl(\dot{X}^N_s-
\int_{\R^l} A_s(X^N_s,x)\hat m_s(x)\,dx\,\br)\\
-\frac{1}{2}\,
 \norm{\hat\lambda_s(X^N_s)}_{\int_{\R^l}  C_s(X^N_s,x)\,\hat m_s(x)\,dx}^2
+\frac{1}{2}\,\int_{\R^l} \norm{D \hat h_s(X^N_s,x)}_{c_s(X^N_s,x)}^2
\,\hat m_s(x) 
\,dx\br)\,ds\,.
\end{multline*}
On recalling convergences
\eqref{eq:52} and \eqref{eq:52a} which are locally uniform in $t$\,, 
the fact that $\tilde m_s^{N,i}(x)=m^{N,i}_s(x)$ for $s\le
\tau^{N,i}(\gamma^{N,i})$\,, and the convergences
$\hat\tau^{N,i}(\gamma^{N,i})\to
\tau^{N}(\gamma^N)$\,, $\gamma^{N,i}\to \gamma^N$\,,
  for \eqref{eq:51}, it remains to check
 that
  \begin{align}
    \label{eq:53}
\lim_{i\to\infty}     \int_0^{\hat\tau^N(\gamma^{N,i})}
 \hat\lambda^i_s(X^{N,i}_s)^T\,\bl(\dot{X}^{N,i}_s-\notag
\int_{\R^l} A_s(X^{N,i}_s,x)m^{N,i}_s(x)\,dx\br)\,ds\\=
\int_0^{\tau^N(\gamma^N)}
 \hat\lambda_s(X^N_s)^T\,\bl(\dot{X}^N_s-
\int_{\R^l} A_s(X^N_s,x)\hat m_s(x)\,dx\br)\,ds\,.
  \end{align}
The convergences $\hat\tau^{N,i}(\gamma^{N,i})\to
\tau^{N}(\gamma^N)$\,, $\gamma^{N,i}\to \gamma^N$\,, and
$\dot X^{N,i}_s\to \dot{ 
  X}_s^N$ for almost all $s< \tau^N(\gamma^N)$\,,
 imply that the
$\ind_{\{[0, \hat\tau^{N,i}(\gamma^{N,i})]}(s) \hat
\lambda^i_s(X^{N,i}_s)^T
\,\bl(\dot{X}^{N,i}_s-
\int_{\R^l} A_s(X_s^{N,i},x)m^{N,i}_s(x)\,dx\,\br)$ converge
to
$\ind_{\{[0, \tau^{N}(\gamma^N)]}(s) \hat
\lambda_s(X^N_s)^T
\,\bl(\dot{X}^N_s-\linebreak
\int_{\R^l} A_s(X^N_s,x)\hat m_s(x)\,dx\,\br)$ as $i\to\infty$
for almost all $s$\,. Since the $\hat\lambda^i_s(u)$ are bounded uniformly in
$i$ and locally uniformly in $(s,u)$\,,
the  uniform integrability needed to derive \eqref{eq:53} follows by the bound
$\sup_{\gamma\in K_\delta}\int_0^N\abs{\dot{ X}_s-
\int_{\R^l} A_s( X_s,x)m_s(x)\,dx}^2\,ds<\infty$\,, which is a consequence
of \eqref{eq:55}.

By \eqref{eq:11}, \eqref{eq:51}, and part 1 of Theorem \ref{the:id}, 
$\mathbf{I}^{\ast\ast}(\gamma^N)=
\theta^{N}(\gamma^N)
 = \tilde{\mathbf{I}}(\gamma^{N})$\,. (Alternatively, one can follow
 the proof of part 1 of Theorem \ref{the:id} by letting $i\to\infty$
 in  \eqref{eq:11} to obtain that
$\mathbf{I}^{\ast\ast}(\gamma^N)\ge
\theta^{N}(\gamma^N)\ge
\tilde{\mathbf{I}}(\gamma^{N})$\,.)
Therefore, $\mathbf{I}^{\ast\ast}(\hat{\gamma})\ge\theta^N(\hat{\gamma})=
 \theta^{N}(\gamma^{N})= \tilde{\mathbf{I}}(\gamma^{N})$\,.
Let $\pi_t(\gamma)$\,, where $\gamma=(X,\mu)$\,,
denote  the projection $((X_{s\wedge t},\mu_{s\wedge t}(\cdot)),s\in\R_+) $\,.
We have that
\begin{equation*}
    \tilde{\mathbf{I}}(\gamma^{N})
\ge \inf_{\gamma:\,\pi_{\tau^N(\gamma^N)}( \gamma)
=\pi_{\tau^N(\gamma^N)}(\gamma^N)}
\tilde{\mathbf{I}}(\gamma)\\=\inf_{\gamma:\,\pi_{\tau^N(\hat{\gamma})}
 (\gamma )
=\pi_{\tau^N(\hat\gamma)}(\hat\gamma)}
\tilde{\mathbf{I}}(\gamma)\,.
\end{equation*}
The sets $\pi_{\tau^N(\hat\gamma)}^{-1}
\bl(\pi_{\tau^N(\hat \gamma)}(\hat \gamma)\br)$ are closed and decrease
to $\hat{\gamma}$ 
as $N\to\infty$\,, so the rightmost side converges to $\tilde{\mathbf
I}(\hat \gamma)$\,, by $\tilde{\mathbf{I}}$ being lower compact.
We conclude that
$\mathbf{I}^{\ast\ast}(\hat{\gamma})\ge
\tilde{\mathbf{I}}(\hat{\gamma})$\,,
so $\mathbf{I}^{\ast\ast}(\hat{\gamma})=
\tilde{\mathbf{I}}(\hat{\gamma})$\,.

\section{Approximating  the large deviation function}
\label{sec:appr-large-devi}
By Theorem \ref{the:id}, in order to complete the proof of Theorem
\ref{the:ldp}, it remains to  establish  an approximation theorem for
$\mathbf{I}^{\ast\ast}$ along the lines of
 part 2
of Theorem \ref{the:id}. We   state it next. 
\begin{theorem}
  \label{the:appr}
Suppose that
 conditions \ref{con:coeff} --
  \ref{con:int_conv},
\eqref{eq:116}, 
 and \eqref{eq:82} hold.
If $ \mathbf{I}^{\ast\ast}(X,\mu)<\infty\,$\,, then
there exists  sequence $(X^{(j)},\mu^{(j)})$ 
whose members  satisfy the
requirements on $(\hat X,\hat \mu)$ in the statement of 
Theorem \ref{the:iden_reg} such
that $(X^{(j)},\mu^{(j)})\to (X,\mu)$ and
  $\mathbf{I}^{\ast\ast}(X^{(j)},\mu^{(j)})\to
\mathbf{I}^{\ast\ast}(X,\mu)$   as $j\to\infty$\,.
\end{theorem}

\begin{proof}
Let $\mu(ds,dx)=m_s(x)\,dx\,ds$ and 
\begin{equation}
  \label{eq:166}
  k_s(x)=\frac{1}{2m_s(x)}\,
\,\text{div}\,\bl(c_s(X_s,x)m_s(x)\br)-a_s(X_s,x)\,.
\end{equation}
Since, by Theorem~\ref{le:vid},
$\int_0^t\int_{\R^l}\abs{Dm_s(x)}^2/m_s(x)\,dx\,ds<\infty$\,, for all
$t\in\R_+$\,, we have that $k_s(\cdot)\in \mathbb L^{2}_{\text{loc}}
(\R^l,\R^l,m_s(x)\,dx)$ a.e.

Let function $\eta$ be as in
Condition \ref{con:int_conv}.
 We introduce
 $\eta_r(x)=\eta(\abs{x}/r)$ and 
\begin{equation}
  \label{eq:167}
    k^{r}_s(x)=\frac{1}{2\,\eta_r^2(x)m_s(x)}\,
\,\text{div}\,\bl(c_s(X_s,x)\eta_r^2(x)m_s(x)\br)-a_s(X_s,x)\,,
\end{equation}
 where $x\in \R^l$ and $r>0$\,. 
We also let $S_r$ represent the open ball in $\R^l$ of radius $r$ centred at
the origin.

We first prove that one can 
choose $(X^{(j)},\mu^{(j)})$ of the required form that converge to
$(X,\mu)$ as $j\to\infty$ and are
such that
  $\mathbf{I}^{\ast\ast}_t(X^{(j)},\mu^{(j)})\to
\mathbf{I}^{\ast\ast}_t(X,\mu)$
 for all  $t$\,, where $\mathbf I^{\ast\ast}_t$ is
defined by \eqref{eq:37}.

Let us begin with  the case where $C_t(u,x)=0$ for all $(t,u,x)$ and
$A_t(u,x)$ is Lipschitz continuous in $u$ locally uniformly in $t$ and
uniformly in $x$\,.
By Theorem~\ref{le:vid}, 
$\dot{X}_s=\int_{\R^l}A_s(X_s,x)\,m_s(x)\,dx$ a.e., the latter
equation having a unique solution.
Let  $\rho_\kappa(x)=(1/\kappa^{l})\rho(x/\kappa)$\, for $\kappa>0$\,,  where
$ \rho(x)$ is a mollifier on $\R^l$\,. 
 We define, for $i,j,j'\in\N$ and
 $\alpha>0$\,,
 \begin{subequations}
   \begin{align}
          \label{eq:126}
    m^{i,j,j'}_s(x)= M_s^{i,j,j'}\bl(
\hat m^{i,j'}_s(x)
\eta_j^2(x)+e^{-\alpha\abs{x}}(1-\eta_j^2(x))\br)
\intertext{and}
   M_s^{i,j,j'}=\Bl(\int_{\R^l}\Bl(
\hat m^{i,j'}_{s}(x)
\eta_j^2(x)+e^{-\alpha\abs{x}}(1-\eta_j^2(x))\Br)\,dx\Br)^{-1}\,,
\label{eq:161}
   \end{align}
 \end{subequations}
where
  \begin{align}
  \label{eq:188}
  \hat m^{i,j'}_s(x)=\int_{\R^l} \rho_{1/i}(\tilde x)
\,\hat m^{j'}_{s}(x-\tilde x)\,d\tilde x\,,\qquad
\hat m^{j'}_s(x)=
 m_s(x)\wedge j'\vee\frac{1}{j'}\,.
\end{align}
We note that, thanks to Theorem \ref{le:vid}, 
$\hat m^{j'}_s\in\mathbb W_{\text{loc}}^{1,2}(\R^l)$\,.

We use Lemma \ref{le:zero} to 
define $X^{i,j,j'}$ as the solution of the equation
\begin{equation*}
\dot{X}^{i,j,j'}_s=\int_{\R^l}A_s(X^{i,j,j'}_s,x)\,m^{i,j,j'}_s(x)\,dx\,,
\end{equation*}
with $X^{i,j,j'}_0=X_0$\,. 
The densities $m_s^{i,j,j'}(x)$ are of class $\mathbb C^1$ in $x$\,,
with  bounded derivatives,
and are locally bounded away from
zero, and the $X^{i,j,j'}$ are locally Lipschitz continuous 
by Lemma \ref{le:zero}.

We introduce further
\begin{subequations}
    \begin{align}
    \label{eq:47}
   M_s^{j,j'}&=\Bl(\int_{\R^l}\bl(\hat m^{j'}_s(x)
\eta_j^2(x)+e^{-\alpha\abs{x}}(1-\eta_j^2(x))\br)\,dx\Br)^{-1},\\
    \label{eq:47a}
  m^{j,j'}_s(x)&= M_s^{j,j'}\bl(\hat m^{j'}_s(x)
\eta_j^2(x)+e^{-\alpha\abs{x}}(1-\eta_j^2(x))\br)\,,
\intertext{ and }
    \label{eq:47b}
\dot{X}^{j,j'}_s&=\int_{\R^l}A_s(X^{j,j'}_s,x)\,m^{j,j'}_s(x)\,dx\,,\;X^{j,j'}_0
=X_0\,.
\end{align}
\end{subequations}

 Let also
 \begin{subequations}
   \begin{align}
     \label{eq:175}
        M_s^j&=\Bl(\int_{\R^l}\bl(
m_s(x)
\eta_j^2(x)+e^{-\alpha\abs{x}}(1-\eta_j^2(x))\br)\,dx\Br)^{-1},\\
\label{eq:175a}
  m^{j}_s(x)&= M_s^{j}\bl(
m_s(x)
\eta_j^2(x)+e^{-\alpha\abs{x}}(1-\eta_j^2(x))\br)\,,
\intertext{ and }\notag
\dot{X}^{j}_s&=\int_{\R^l}A_s(X^{j}_s,x)\,m^{j}_s(x)\,dx\,,\;X^j_0=X_0\,.
\end{align}
 \end{subequations}

We have that
\begin{subequations}
\begin{align}
        \label{eq:72}
  \lim_{i\to\infty} M_s^{i,j,j'}= M_s^{j,j'}\,,\;
  \lim_{i\to\infty}\int_{\R^l}\abs{ m_s^{i,j,j'}(x)- m^{j,j'}_s(x)}\,dx=0\,,\;
\lim_{i\to\infty} X^{i,j,j'}_s= X^{j,j'}_s\,,\\
  \label{eq:184}
\lim_{j'\to\infty}   M_s^{j,j'}= M_s^j\,,\;
\lim_{j'\to\infty}
\int_{\R^l}\abs{ m_s^{j,j'}(x)- m^j_s(x)}\,dx=0\,,\;
\lim_{j'\to\infty} X^{j,j'}_s= X^j_s\,,
\intertext{and}
  \label{eq:184a}
  \lim_{j\to\infty} M_s^{j}= 1\,,\;
  \lim_{j\to\infty}\int_{\R^l}\abs{ m_s^{j}(x)- m_s(x)}\,dx=0\,,
\lim_{j\to\infty} X^{j}_s= X_s\,.
\end{align}
\end{subequations}
The third convergence on each line is proved by a similar compactness
argument to the one used in the proof of Lemma \ref{le:zero}.

By \eqref{eq:133}, 
  \begin{align}
  \label{eq:65}
&    \mathbf{ I}_t^{\ast\ast}(X,\mu)=
\int_0^t\sup_{h\in \mathbb{C}_0^1(\R^l)}\Bl(\int_{\R^l}\bl( D  h(x)^T \bl(
\frac{1}{2}\,\text{div}\,(c_s(X_s,x)m_s(x))\notag\\&
-a_s(X_s,x)m_s(x) \br)
-\frac{1}{2}\,
\norm{D  h(x)}_{c_s(X_s,x)}^2\,m_s(x)\br)\,dx\Br)\,ds      \intertext{and}
 &   \mathbf{ I}_t^{\ast\ast}(X^{i,j,j'},\mu^{i,j,j'})=
\int_0^t\sup_{h\in \mathbb{C}_0^1(\R^l)}\Bl(\int_{\R^l}\bl( D  h(x)^T \bl(
\frac{1}{2}\,\text{div}\,(c_s(X^{i,j,j'}_s,x)m^{i,j,j'}_s(x))\notag
\\&-a_s(X^{i,j,j'}_s,x)m^{i,j,j'}_s(x)
 \br)\notag
-\frac{1}{2}\,
\norm{D  h(x)}_{c_s(X^{i,j,j'}_s,x)}^2\,m^{i,j,j'}_s(x)\br)\,dx\Br)\,ds\,.
\end{align}

By \eqref{eq:126},
\begin{equation}
  \label{eq:162}
    \mathbf{ I}_t^{\ast\ast}(X^{i,j,j'},\mu^{i,j,j'})\le 
 M_s^{i,j,j'}\bl(
\int_0^tI_1^j(X^{i,j,j'}_s,\hat m^{i,j'}_s,s)\,ds+
\int_0^tI_2^j(X^{i,j,j'}_s,s)\,ds\br)\,,
\end{equation}
where, for generic $\tilde X_s$ and $\tilde m_s$\,,
\begin{subequations}
  \begin{align}
    \label{eq:75}
    I_1^j(\tilde X_s,\tilde m_s,s)=\sup_{h\in \mathbb{C}_0^1(\R^l)}
\int_{\R^l}\bl( D  h(x)^T \bl(
\frac{1}{2}\,\text{div}\,(c_s(\tilde X_s,x)\eta_j^2(x)\tilde m_s(x))\\
-a_s(\tilde X_s,x)\eta_j^2(x)\tilde m_s(x)
 \br)-\frac{1}{2}\,\notag
\norm{D  h(x)}_{c_s(\tilde X_s,x)}^2\,\eta_j^2(x)\tilde m_s(x)\br)\,dx
\intertext{ and}
 I_2^j(\tilde X_s,s)=
\sup_{h\in \mathbb{C}_0^1(\R^l)}\int_{\R^l}\bl( D  h(x)^T \bl(
\frac{1}{2}\,\text{div}\,(c_s(\tilde X_s,x)
e^{-\alpha\abs{x}}(1-\eta_j^2(x))\\
-a_s(\tilde X_s,x)e^{-\alpha\abs{x}}(1-\eta_j^2(x))
 \br)
-\frac{1}{2}\,
\norm{D  h(x)}_{c_s(\tilde X_s,x)}^2\,
e^{-\alpha\abs{x}}(1-\eta_j^2(x))
\br)\notag
\,dx\,. \end{align}
\end{subequations}
We prove that
\begin{equation}
  \label{eq:134}
  \lim_{i\to\infty}
  I_1^{j}(X_s^{i,j,j'},\hat m^{i,j'}_s,s)
=  I_1^{j}(X^{j,j'}_s,\hat m^{j'}_s,s)
\,.
\end{equation}

Let, in analogy with \eqref{eq:167},
\begin{equation}
  \label{eq:192}
        k^{i,j,j'}_s(x)=\frac{1}{2\,\eta_j^2(x)\hat m^{i,j'}_s(x)}\,
\,\text{div}\,\bl(c_s(X^{i,j,j'}_s,x)
\eta_j^2(x)\hat m^{i,j'}_s(x)\br)-a_s(X^{i,j,j'}_s,x)\,.
\end{equation}
This function is an element of $\mathbb
L^2(\R^l,\R^l,\eta_j^2(x)
\hat m^{i,j'}_s(x)\,dx)$\,.

The supremum in $I_1^j(X_s^{i,j,j'},\hat m^{i,j'}_s,s)$ 
is attained at a unique element 
$g^{i,j,j'}_s$ of $\mathbb L^{1,2}_0(\R^l,\R^l,\eta_j^2(x)\hat m^{i,j'}_s(x)\,dx)$
such that
\begin{multline}
  \label{eq:97}
  \int_{\R^l}Dp(x)^Tk^{i,j,j'}_s(x) \eta_j^2(x)\hat m^{i,j'}_s(x)\,dx
= \int_{\R^l}Dp(x)^Tc_s(X^{i,j,j'}_s,x)
g^{i,j,j'}_s(x)\eta_j^2(x)\hat m^{i,j'}_s(x)\,dx
\end{multline}
for all $p\in\mathbb C^{1}_0(\R^l)$ and
\begin{equation}
  \label{eq:107}
  I_1^j(X^{i,j,j'}_s,\hat m^{i,j'}_s,s)=\int_{\R^l}
\frac{1}{2}\,\norm{g^{i,j,j'}_s(x)}_{c_s(X^{i,j,j'}_s,x)}^2
\hat m^{i,j'}_s(x)\eta_j^2(x)
\,dx\,.
\end{equation}
Similarly, the supremum in $I_1^j(X_s^{j,j'},\hat m^{j'}_s,s)$ 
is attained at a unique element 
$g^{j,j'}_s$ of $\mathbb L^{1,2}_0(\R^l,\R^l,\eta_j^2(x)\hat m^{j'}_s(x)\,dx)$
such that
\begin{equation}
  \label{eq:176}
    \int_{\R^l}Dp(x)^Tk^{j,j'}_s(x)m^{j'}_s(x) \eta_j^2(x)\,dx
= \int_{\R^l}Dp(x)^Tc_s(X^{j}_s,x)
 g^{j,j'}_s(x) m^{j'}_s(x)\eta^2_j(x)\,dx
\end{equation}
and
\begin{equation}
  \label{eq:160}
    I_1^j(X^{j,j'}_s,\hat m^{j'}_s,s)=\int_{\R^l}
\frac{1}{2}\,\norm{g^{j,j'}_s(x)}_{c_s(X^{j,j'}_s,x)}^2
\hat m^{j'}_s(x)\eta_j^2(x)
\,dx\,,
\end{equation}
where
\begin{equation}
  \label{eq:127}
          k^{j,j'}_s(x)=\frac{1}{2\,\eta_j^2(x)\hat m^{j'}_s(x)}\,
\,\text{div}\,\bl(c_s(X^{j,j'}_s,x)
\eta_j^2(x)\hat m^{j'}_s(x)\br)-a_s(X^{j,j'}_s,x)\,.
\end{equation}
Let
\begin{align*}
  Q_1&=\int_{\R^l}\norm{Dp(x)}^2_{c_s(X_{s}^{i,j,j'},x)}
\hat m^{i,j'}_s(x)\eta_j^2(x)\,dx\,,&\\
Q_2&=\int_{\R^l}\norm{k_s^{i,j,j'}(x)\,\hat
    m^{i,j'}_s(x) -k_s^{j,j'}(x)
\hat m^{j'}_s(x)}^2_{c_s(X^{i,j,j'}_{s},x)^{-1}}
\frac{\eta_j^2(x)}{\hat m^{i,j'}_s(x)}\,dx\,,\intertext{ and }
Q_3&=\int_{\R^l}\norm{c_s(X_{s}^{j,j'},x)
g^{j,j'}_s(x)}^2_{c_s(X_{s}^{i,j,j'},x)^{-1}}\,\frac{\hat
    m^{j'}_s(x)^2\eta_j^2(x)}{\hat m^{i,j'}_s(x)}\,dx\,.&
\end{align*}
By \eqref{eq:97} and \eqref{eq:176}, we have that
\begin{multline*}
  \int_{\R^l}Dp(x)^Tc_s(X_{s}^{i,j,j'},x)g^{i,j,j'}_s(x)
\hat m^{i,j'}_{s}(x)\,\eta_j^2(x)\,dx
\\=\int_{\R^l}Dp(x)^T\bl(
k_s^{i,j,j'}(x)\,\hat m^{i,j'}_s(x)-
 k_s^{j,j'}(x)\,
\hat m^{j'}_{s}(x)\br)\eta_j^2(x)\,dx\\
+\int_{\R^l}Dp(x)^T
c_s(X_{s}^{j,j'},x)g_s^{j,j'}(x)\hat m^{j'}_s(x)\eta_j^2(x)\,dx\le
\sqrt{Q_1}\,\sqrt{Q_2}+\sqrt{Q_1}\,\sqrt{Q_3}\,.
\end{multline*}
Hence,
\begin{multline*}
  \sqrt{\int_{\R^l}\norm{g^{i,j,j'}_s(x)}^2_{
c_s(X_{s}^{i,j,j'},x)}\hat m^{i,j'}_{s}(x)\,\eta_j^2(x)\,dx}\\=
\sup_{p\in\mathbb C_0^1(\R^l):\,Q_1\le 1}
\int_{\R^l}Dp(x)c_s(X_{s}^{i,j,j'},x)
g^{i,j,j'}_s(x)\hat m^{i,j'}_{s}(x)\,\eta_j^2(x)\,dx
\le\sqrt{Q_2}+\sqrt{Q_3}\,.
\end{multline*}
By \eqref{eq:107},  for arbitrary $\kappa>0$\,,
\begin{equation*}
I_1^j(X^{i,j,j'}_s,\hat m^{i,j'}_s,s)
\le \frac{1}{2}\,\bl(1+\frac{1}{\kappa}\br)Q_2
+\frac{1}{2}\,(1+\kappa)Q_3\,.
\end{equation*}

By  \eqref{eq:188},
$\norm{\hat{m}^{i,j'}_s-\hat m^{j'}_s}_{\mathbb
  W^{1,2}(S_{2j})}\to 0$ as $i\to\infty$, see, e.g.,
 Lemma 3.16 on p.66 in Adams
  and Fournier \cite{MR56:9247},  so, on recalling \eqref{eq:192},
  \eqref{eq:127}, and Condition 
\ref{con:positive}, we have that
 $Q_2\to0$ as $i\to\infty$\,.
The integrand in $Q_3$ tends to 
$\norm{g^{j,j'}_s(x)}^2_{c_s(X_{s}^{j,j'},x)}\hat
    m^{j'}_s(x)\eta_j^2(x)$ in Lebesgue measure, see \eqref{eq:72}. Since 
the function $\hat m^{j'}_s(x)/\hat m_s^{i,j'}(x)$ 
is  bounded  in  $x$ and  $i$\,, by dominated
convergence, 
$  Q_3$ converges to $\int_{\R^l}\norm{g^{j,j'}_s(x)}^2_{c_s(X_{s}^{j,j'},x)}\hat
    m^{j'}_s(x)\eta_j^2(x)\,dx\,$ as $i\to\infty$ so that, on recalling
  \eqref{eq:160},
\begin{equation*}
     \limsup_{i\to\infty}
  I_1^{j}(X_s^{i,j,j'},\hat m^{i,j'}_s,s)
\le  I_1^{j}(X^{j,j'}_s,\hat m^{j'}_s,s)
\,.
\end{equation*}  
On the other hand, by \eqref{eq:75} and integration by parts,
\begin{multline*}
  I_1^j(X^{i,j,j'}_s,\hat m^{i,j'}_s,s)=
\sup_{h\in \mathbb{C}_0^2(\R^l)}\int_{\R^l}
\bl(-\frac{1}{2}\, \text{tr}\,(
c_s(X^{i,j,j'}_s,x)D^2  h(x))
-D  h(x)^Ta_s(X^{i,j,j'}_s,x)
\\-\frac{1}{2}\,
\norm{D  h(x)}_{c_s(X^{i,j,j'}_s,x)}^2\br)
\hat m^{i,j'}_s(x)\,\eta_j^2(x)
\,dx
\end{multline*}
and a similar representation holds for $I_1^j(X^{j,j'}_s,\hat m^{j'}_s,s)$\,,
which facts imply, in view of       \eqref{eq:72} and the continuity
properties in Condition \ref{con:coeff},
  that
  \begin{equation}
    \label{eq:190}
         \liminf_{i\to\infty}
  I_1^{j}(X_s^{i,j,j'},\hat m^{i,j'}_s,s)
\ge  I_1^{j}(X^{j,j'}_s,\hat m^{j'}_s,s)
\,.
\end{equation}
We have  proved \eqref{eq:134}.
We now show that integrals with respect to $s$ converge too.
Let us note that, by \eqref{eq:97} and \eqref{eq:107}, 
\begin{equation*}
I_1^{j}(X^{i,j,j'}_s,\hat m^{i,j'}_s,s)\le
\int_{\R^l}
\frac{1}{2}\,\norm{k_s^{i,j,j'}(x)}_{c_s(X^{i,j,j'}_s,x)^{-1}}^2
\hat m^{i,j'}_s(x)\eta_j^2(x)
\,dx\,,
\end{equation*}
so, by  \eqref{eq:192},  and conditions \ref{con:coeff} and \ref{con:positive}
there exists $M>0$ such that
\begin{equation*}
  I_1^{i,j,j'}(X^{i,j,j'}_s,\hat m^{i,j'}_s,s)\le M\bl(1+
\int_{\R^l}\frac{\abs{D\hat m^{i,j'}_s(x)}^2}{\hat m^{i,j'}_s(x)}\,\eta_j^2(x)\,dx\br)\,.
\end{equation*}
Accounting for \eqref{eq:188} we have that
\begin{multline*}
\frac{1}{2}\,  \frac{\abs{D\hat m^{i,j'}_s(x)}^2}{\hat m^{i,j'}_s(x)}=
\sup_{y\in\R^l}\bl(y^TD\hat m^{i,j'}_s(x)-\frac{1}{2}\,\abs{y}^2\hat
m^{i,j'}_s(x)\br)\\ \le \int_{\R^l} \rho_{1/i}(\tilde x)
\sup_{y\in\R^l}\bl(y^TD\hat m^{j'}_s(x-\tilde x)-\frac{1}{2}\,\abs{y}^2\hat
m^{j'}_s(x-\tilde x)\br)\,d\tilde x=
\frac{1}{2}\,\int_{\R^l} \rho_{1/i}(\tilde x)
\frac{\abs{D\hat m^{j'}_s(x-\tilde x)}^2}{\hat m^{j'}_s(x-\tilde x)}\,d\tilde x\,.
\end{multline*}
Therefore, recalling the definition of $\hat m_s^{j'}(x)$ in
\eqref{eq:188}, 
\begin{equation*}
    \int_{\R^l}\frac{\abs{D\hat m^{i,j'}_s(x)}^2}{\hat
    m^{i,j'}_s(x)}\,\eta_j^2(x)\,dx
\le \int_{\R^l}\frac{\abs{D\hat m^{j'}_s(x)}^2}{\hat m^{j'}_s(x)}\,dx
\le\int_{\R^l}\frac{\abs{D m_s(x)}^2}{ m_s(x)}\,dx\,.
\end{equation*}
Since $\int_0^t \int_{\R^l}\abs{D m_s(x)}^2/
  m_s(x)\,dx\,ds<\infty$ by Theorem \ref{le:vid},
\eqref{eq:134} and
  Fatou's lemma yield the convergence
  \begin{equation}
    \label{eq:177}
     \lim_{i\to\infty}
 \int_0^t I_1^{j}(X^{i,j,j'}_s,\hat m^{i,j'}_s,s)
\,ds
= \int_0^t I_1^{j}(X^{j,j'}_s,\hat m^{j'}_s,s)
\,ds
\,.
\end{equation}

Let us show that
\begin{equation}
  \label{eq:189}
      \lim_{j'\to\infty}\int_0^t
I_1^j(X^{j,j'}_s,\hat m^{j'}_s,s)\,ds
= \int_0^t
I_1^j(X^{j}_s, m_s,s)\,ds\,.
\end{equation}

We have that
\begin{multline*}
  \abs{I_1^{j}(X^{j,j'}_s,\hat m^{j'}_s,s)-
 I_1^{j}(X^{j,j'}_s, m_s,s)}\\
\le\frac{1}{2}\, \int_{\R^l}
\bl(
\norm{\frac{1}{2\eta_j^2(x) m_s(x)}\,
\text{div}\,(c_s(X^{j,j'}_s,x)\eta_j^2(x) m_s(x))
-a_s(X^{j,j'}_s,x)}^2_{c_s(X^{j,j'}_s,x)^{-1}}\\
+
\norm{\frac{1}{2\eta_j^2(x) }\,
\text{div}\,(c_s(X^{j,j'}_s,x)\eta_j^2(x) )
-a_s(X^{j,j'}_s,x)}^2_{c_s(X^{j,j'}_s,x)^{-1}}\br)
\eta_j^2(x) m_s(x) (1-\ind_{[1/j',j']}(m_s(x)))\,dx\,,
\end{multline*}
so,  by dominated convergence,
\begin{equation}
  \label{eq:183}
  \lim_{j'\to\infty}\int_0^t\abs{I_1^{j}(X^{j,j'}_s,\hat m^{j'}_s,s)-
 I_1^{j}(X^{j,j'}_s, m_s,s)}\,ds=0\,.
\end{equation}
Let 
 $\vartheta>0$ be such that
 $\norm{y}_{  c_s(u,x)}^2\ge \vartheta \abs{y}^2$\,, for all
 $s\in[0,t]$\,,
 for all $u$ from a large enough ball,  all $x$\,, and all $y$\,.
By the convergence of $X^{j,j'}$ to $X^j$ as $j'\to\infty$\,, the continuity
of $c_s(u,x)$ in $u$ locally uniformly in $s$ and uniformly in $x$\,, 
and by $c_s(u,x)$ being positive definite uniformly in $x$ and locally
uniformly in $(s,u)$,
given arbitrary
$\delta\in(0,1)$ and $\kappa\in(0,1)$\,, for all $j'$ great enough, 
locally uniformly in $s$\,,
\begin{multline}
    \label{eq:185}
I_1^j(X^{j,j'}_s,m_s,s)
\le   \sup_{h\in \mathbb{C}_0^1(\R^l)}\int_{\R^l}\bl( D  h(x)^T 
k^{j}_s(x)
-\frac{1}{2}\,
(1-\delta)(1-\kappa)\norm{D  h(x)}_{c_s(X^j_s,x)}^2\br)m_s(x)\eta_j^2(x)
\,dx\\
+
\sup_{h\in \mathbb{C}_0^1(\R^l)}\Bl(\int_{\R^l}\bl( D  h(x)^T \bl(
\frac{1}{2}\,\text{div}\,((c_s(X^{j,j'}_s,x)-c_s(X^j_s,x))m_s(x)\eta_j^2(x))\\
-\bl(a_s(X^{j,j'}_s,x)-a_s(X^j_s,x)\br)m_s(x)\eta_j^2(x) \br)
-\frac{1}{2}\,
\delta(1-\kappa)\norm{D  h(x)}_{c_s(X^j_s,x)}^2\,m_s(x)\eta_j^2(x)\br)
\,dx\Br)\\
\le (1-\delta)^{-1}
(1-\kappa)^{-1} \sup_{h\in \mathbb{C}_0^1(\R^l)}\int_{\R^l}
\bl(Dh(x)^T k^{j}_s(x)
-\frac{1}{2}\,
\norm{Dh(x)}_{ c_s(X^j_s,x)}^2\br)
m_s(x)\,\eta_j^2(x)
\,dx\\
+\delta^{-1}(1-\kappa)^{-1}\frac{\vartheta^{-1}}{2}
\int_{\R^l}
\abs{
\frac{1}{2}\,
\frac{\text{div}\,\bl((c_s(X^{j,j'}_s,x)-c_s(X^j_s,x))m_s(x)\eta_j^2(x)\br)}
{m_s(x)\eta_j^2(x)}\\ -\bl(a_s(X^{j,j'}_s,x)-a_s(X^j_s,x)\br)}^2
m_s(x)\eta_j^2(x)
\,dx\,.
\end{multline}
By the convergence of $X^{j,j'}_s$ to $X^j_s$ as $j'\to\infty$\,, 
Condition 2.1 and the convergence of
$\int_0^t\int_{\R^l}\abs{Dm_s(x)}^2/m_s(x)\,dx\,ds$\,,  
the integral  from $0$ to $t$ of the second integral
on the rightmost side of \eqref{eq:185}
 tends to zero as $j'\to\infty$\,.
Therefore, by 
\eqref{eq:183}, 
\eqref{eq:167},   \eqref{eq:65}, and 
\eqref{eq:75},
\begin{equation*}
       \limsup_{j'\to\infty}\int_0^t
I_1^j(X^{j,j'}_s,\hat m^{j'}_s,s)\,ds
\le \int_0^t
I_1^j(X^{j}_s, m_s,s)\,ds
\end{equation*}
and by an analogue of \eqref{eq:190}, we obtain \eqref{eq:189}.

We now take a limit as $j\to\infty$\,.
By a similar reasoning to the one used in \eqref{eq:185},
given arbitrary
$\delta\in(0,1)$ and $\kappa\in(0,1)$\,, for all $j$ great enough, 
locally uniformly in $s$\,,
\begin{multline}
    \label{eq:69}
I_1^j(X^{j}_s,m_s,s)
\le (1-\delta)^{-2}
(1-\kappa)^{-1} \sup_{h\in \mathbb{C}_0^1(\R^l)}\int_{\R^l}
\bl(Dh(x)^T k_s(x)
-\frac{1}{2}\,
\norm{Dh(x)}_{ c_s(X_s,x)}^2\br)
m_s(x)\,\eta_j^2(x)
\,dx\\+\delta^{-1}(1-\delta)^{-1}(1-\kappa)^{-1}\frac{1}{2}\int_{\R^l}
\norm{D\eta_j(x)}^2_{c_s(X_s,x)}\,m_s(x)\,dx\\+
\delta^{-1}(1-\kappa)^{-1}\frac{\vartheta^{-1}}{2}
\int_{\R^l}
\abs{
\frac{1}{2}\,
\frac{\text{div}\,\bl((c_s(X^{j}_s,x)-c_s(X_s,x))m_s(x)\eta_j^2(x)\br)}
{m_s(x)\eta_j^2(x)}\\ -\bl(a_s(X^{j}_s,x)-a_s(X_s,x)\br)}^2
m_s(x)\eta_j^2(x)
\,dx
\end{multline}
so that, by Condition \ref{con:int_conv} (with $\lambda=0$) and  Condition
\ref{con:coeff},
we have, on recalling
 \eqref{eq:65}, that
\begin{equation}
  \label{eq:38}
    \lim_{j\to\infty}\int_0^t
I_1^j(X^{j}_s,m_s,s)\,ds
= \mathbf{ I}_t^{\ast\ast}(X,\mu)\,.
\end{equation}
Putting together \eqref{eq:177}, \eqref{eq:189}, and \eqref{eq:38}
yields the convergence
\begin{equation}
  \label{eq:191}
       \lim_{j\to\infty}\lim_{j'\to\infty}\lim_{i\to\infty}
 \int_0^t I_1^{j}(X^{i,j,j'}_s,\hat m^{i,j'}_s,s)
\,ds= \mathbf{ I}_t^{\ast\ast}(X,\mu)\,.
\end{equation}
We now show that the term $I_2^j$ is inconsequential.
On recalling that
$\abs{a_s(u,x)}$  grows at most linearly in $\abs{x}$
and   $\abs{\text{div}\,c_s(u,x)}$ and $\norm{c_s(u,x)}$ are bounded
in $x$ locally uniformly in $(s,u)$\,, we have that, for some $L>0$\,,
all $(i,j)$\,, and all $s\le t$\,, 
\begin{multline*}
I_2^j(X_s^{i,j,j'},s)
\le 
\int_{x\in\R^l:\,\abs{x}\ge j}\sup_{y\in\R^l}\Bl( y^T \bl(
\frac{1}{2}\,
\,\frac{\text{div}\,(c_s(X^{i,j,j'}_s,x)\,(1-\eta^2_j(x)))}{1-\eta^2_j(x)}
\\-\alpha c_s(X^{i,j,j'}_s,x)\,\frac{x}{2\abs{x}}-a_s(X^{i,j,j'}_s,x)
 \br)-\frac{1}{2}\,
\norm{y}_{c_s(X^{i,j,j'}_s,x)}^2\Br)\,(1-\eta^2_j(x))e^{-\alpha \abs{x}}\,dx\\
\le 
\int_{x\in\R^l:\,\abs{x}\ge j}L\bl(
1+\alpha^2+\abs{x}^2
+\frac{1}{j^2}\frac{\abs{D\eta(\abs{x}/j)}^2}{1-\eta^2(\abs{x}/j)}
 \br)e^{-\alpha \abs{x}}\,dx\,.
\end{multline*}
Since  $\eta(y)=0$ for $y\ge
2$ and \eqref{eq:118} holds, 
the latter integral
tends to 0 as $j\to\infty$\,,
 so,
\begin{equation}\label{eq:26}
\lim_{j\to\infty}\limsup_{j'\to\infty}\limsup_{i\to\infty}\int_0^t
I_2^j(X_s^{i,j,j'},s)\,ds=0\,.
\end{equation}
 By \eqref{eq:72}, \eqref{eq:184}, \eqref{eq:184a}, \eqref{eq:162},
\eqref{eq:191}, and \eqref{eq:26},
\begin{equation*}
     \limsup_{j\to\infty}\limsup_{j'\to\infty}\limsup_{i\to\infty}
  \mathbf{ I}_t^{\ast\ast}(X^{i,j,j'},\mu^{i,j,j'})\le 
\mathbf{ I}_t^{\ast\ast}(X,\mu)\,.
\end{equation*}
Thus, there exist  sequences $j'(j)\to\infty$ and 
$i(j)\to\infty$ as $j\to\infty$
such that $(X^{i(j),j,j'(j)},\mu^{i(j),j,j'(j)})\to (X,\mu)$ and
\begin{equation*}
    \limsup_{j\to\infty}  \mathbf{ I}_t^{\ast\ast}(X^{i(j),j,j'(j)},\mu^{i(j),j,j'(j)})\le 
\mathbf{ I}_t^{\ast\ast}(X,\mu)\,.   
\end{equation*}
The reverse inequality follows from the lower semicontinuity of 
$\mathbf{ I}_t^{\ast\ast}(X,\mu)$ (see \eqref{eq:133}, where we let 
$\mathbf{ I}_t^{\ast\ast}(X,\mu)=\infty$ if 
$\mathbf{ I}^{\ast\ast}(X,\mu)=\infty$), so 
\begin{equation}
  \label{eq:56}
\lim_{j\to\infty}      \mathbf{ I}_t^{\ast\ast}(X^{i(j),j,j'(j)},\mu^{i(j),j,j'(j)})=
\mathbf{ I}_t^{\ast\ast}(X,\mu)\,,
\end{equation} and one can take
$(X^{(j)},\mu^{(j)})=(X^{i(j),j,j'(j)},\mu^{i(j),j,j'(j)})$\,.

Suppose now that
$C_t(u,x)-G_t(u,x)
c_t(u,x)^{-1}G_t(u,x)^T$
 is positive definite uniformly in $x$ and
locally uniformly in $(t,u)$\,.
We proceed similarly to  the case where $C_t(u,x)=0$
and define $m^{i,j,j'}_s(x)$\,, $ M_s^{i,j,j'}$\,, $ M_s^{j,j'}$\,,
$M_s^j$\,, and $m_s^j(x)$ by the respective relations
\eqref{eq:126}, \eqref{eq:161}, \eqref{eq:188},
\eqref{eq:47}--\eqref{eq:47b}, \eqref{eq:175}, and \eqref{eq:175a}.
We let

\begin{equation}
  \label{eq:54}
\dot X^{i,j,j'}_s=\dot X^{j,j'}_s=
\dot X^{j}_s=\dot X_s\ind_{\{\abs{\dot X_s}\le j\}}(s)\,,\,
 X^{i,j,j'}_0= X^{j,j'}_0=
 X^{j}_0=X_0\,.
\end{equation}
The convergences
in  \eqref{eq:72}, \eqref{eq:184}, and \eqref{eq:184a} still hold.

Replacing $a_s(X_s,x)$ with $a_s(X_s,x)+G_s(X_s,x)^T\lambda$ 
in the proof above,
 we conclude that in analogy with \eqref{eq:56},
 there exist  sequences $i(j)\to\infty$ and $j'(j)\to\infty$
 as $j\to\infty$ such that, for all $\lambda\in\R^n$
with rational components,
  \begin{multline*}
  \lim_{j\to\infty}
\int_0^t    \sup_{h\in \mathbb{C}_0^1(\R^l)}
\int_{\R^l}\Bl( D  h(x)^T \bl(
\frac{1}{2}\,
\,\text{div}\,(c_s(X^{j}_s,x) m^{i(j),j,j'(j)}_s(x))
\\-(a_s(X^{(j)}_s,x)+G_s(X_s^{(j)},x)^T\lambda)m^{i(j),j,j'(j)}_s(x) \br)
-\frac{1}{2}\,
\norm{D  h(x)}_{c_s(X^{j}_s,x)}^2m^{i(j),j,j'(j)}_s(x)\Br)\,dx\,ds\\=
\int_0^t    \sup_{h\in \mathbb{C}_0^1(\R^l)}
\int_{\R^l}\Bl( D  h(x)^T \bl(
\frac{1}{2}\,
\,\text{div}\,(c_s(X_s,x) m_s(x))
-(a_s(X_s,x)+G_s(X_s,x)^T\lambda)m_s(x) \br)
\\-\frac{1}{2}\,
\norm{D  h(x)}_{c_s(X_s,x)}^2m_s(x)\Br)\,dx\,ds\,,
  \end{multline*}
which, in particular, implies that
  \begin{multline*}
  \lim_{j\to\infty}
   \sup_{h\in \mathbb{C}_0^1(\R^l)}
\int_{\R^l}\Bl( D  h(x)^T \bl(
\frac{1}{2}\,
\,\text{div}\,(c_s(X^{j}_s,x) m^{i(j),j,j'(j)}_s(x))\\-
(a_s(X^{j}_s,x)+G_s(X_s^{(j)},x)^T\lambda)m^{i(j),j,j'(j)}_s(x) \br)
-\frac{1}{2}\,
\norm{D  h(x)}_{c_s(X^{j}_s,x)}^2m^{i(j),j,j'(j)}_s(x)\Br)\,dx\\=
   \sup_{h\in \mathbb{C}_0^1(\R^l)}
\int_{\R^l}\Bl( D  h(x)^T \bl(
\frac{1}{2}\,
\,\text{div}\,(c_s(X_s,x) m_s(x))-
(a_s(X_s,x)+G_s(X_s,x)^T\lambda)m_s(x) \br)\\
-\frac{1}{2}\,
\norm{D  h(x)}_{c_s(X_s,x)}^2m_s(x)\Br)\,dx\,,
\end{multline*}
for almost all $s\in[0,t]$\,.

We write these convergences, relabelling $m^{i(j),j,j'(j)}$ as $m^{(j)}$
and $X^j_s$ as $X^{(j)}_s$\,, as
\begin{subequations}
  \begin{multline}
\label{eq:58}
\lim_{j\to\infty}
\int_0^t    \int_{\R^l}\norm{\frac{Dm^{(j)}_s(x)}{2m^{(j)}_s(x)}
-\Phi_{s,m^{(j)}_s(\cdot),X^{(j)}_s}(x)
- \Psi_{s,m^{(j)}_s(\cdot),X^{(j)}_s}(x)\lambda}^2_{
c_s(X^{(j)}_{s},x)}m^{(j)}_{s}(x)\,dx\,ds\\
=\int_0^t
\int_{\R^l}\norm{\frac{Dm_s(x)}{2m_s(x)}-\Phi_{s,m_s(\cdot),X_s}(x)
- \Psi_{s,m_s(\cdot),X_s}(x)\lambda}^2_{
c_s(X_{s},x)}m_{s}(x)\,dx\,ds
\end{multline}
and
\begin{multline}
\label{eq:58a}
\lim_{j\to\infty}
    \int_{\R^l}\norm{\frac{Dm^{(j)}_s(x)}{2m^{(j)}_s(x)}
-\Phi_{s,m^{(j)}_s(\cdot),X^{(j)}_s}(x)
- \Psi_{s,m^{(j)}_s(\cdot),X^{(j)}_s}(x)\lambda}^2_{
c_s(X^{(j)}_{s},x)}m^{(j)}_{s}(x)\,dx\\
=
\int_{\R^l}\norm{\frac{Dm_s(x)}{2m_s(x)}-\Phi_{s,m_s(\cdot),X_s}(x)
- \Psi_{s,m_s(\cdot),X_s}(x)\lambda}^2_{
c_s(X_{s},x)}m_{s}(x)\,dx\,,
\end{multline}
\end{subequations}
respectively.

By \eqref{eq:37} and \eqref{eq:58} with $\lambda=0$\,, it remains to prove that
\begin{multline}
  \label{eq:182}
  \lim_{j\to\infty}\int_0^t
  \norm{\dot{X}^{(j)}_{s}-\int_{\R^l}A_s(X^{(j)}_{s},x)m^{(j)}_{s}(x)\,dx
-\int_{\R^l}
 G_s(X^{(j)}_{s},x)\bl(\frac{Dm^{(j)}_s(x)}{2m^{(j)}_s(x)}\\-\Phi_{s,m^{(j)}_s(\cdot),X^{(j)}_s}(x)\br)
m^{(j)}_s(x)\,dx}^2_{(\int_{\R^l}
 Q_{s,m^{(j)}_s(\cdot)}(X^{(j)}_{s},x)m^{(j)}_s(x)\,dx)^{-1}}\,ds\\
=\int_0^t
  \norm{\dot{X}_{s}-\int_{\R^l}A_s(X_{s},x)m_{s}(x)\,dx-
\int_{\R^l}
 G_s(X_{s},x)\bl(\frac{Dm_s(x)}{2m_s(x)}\\-\Phi_{s,m_s(\cdot),X_s}(x)\br)
m_s(x)\,dx}^2_{(\int_{\R^l}
Q_{s,m_s(\cdot)}(X_{s},x)m_s(x)\,dx)^{-1}}\,ds\,.
\end{multline}

 On subtracting two versions of \eqref{eq:58a} for $\lambda$
 differing by a sign,
\begin{multline}
    \label{eq:60}
  \lim_{j\to\infty}\int_{\R^l}\Psi_{s,m^{(j)}_s(\cdot),X^{(j)}_s}(x)^T
c_s(X^{(j)}_{s},x)
\bl(\frac{Dm^{(j)}_s(x)}{2m^{(j)}_s(x)}-\Phi_{s,m^{(j)}_s(\cdot),X^{(j)}_s}(x)\br)
m^{(j)}_s(x)\,dx\\=\int_{\R^l}\Psi_{s,m_s(\cdot),X_s}(x)^T
c_s(X_{s},x)
\bl(\frac{Dm_s(x)}{2m_s(x)}-\Phi_{s,m_s(\cdot),X_s}(x)\br)
m_s(x)\,dx\,,
\end{multline}
so, on using \eqref{eq:151},
\begin{multline}
  \label{eq:4}
      \lim_{j\to\infty}\int_{\R^l}
 G_s(X^{(j)}_{s},x)\bl(\frac{Dm^{(j)}_s(x)}{2m^{(j)}_s(x)}-\Phi_{s,m^{(j)}_s(\cdot),X^{(j)}_s}(x)\br)
m^{(j)}_s(x)\,dx\\=\int_{\R^l}
 G_s(X_{s},x)\bl(\frac{Dm_s(x)}{2m_s(x)}-\Phi_{s,m_s(\cdot),X_s}(x)\br)
m_s(x)\,dx\,.
\end{multline}

By \eqref{eq:58a} and \eqref{eq:60},
\begin{equation*}
  \lim_{j\to\infty}\int_{\R^l} 
\norm{
\Psi_{s,m^{(j)}(\cdot),X^{(j)}_s}(x)}_{ c_s(X^{(j)}_s,x)}^2m^{(j)}_s(x)\,dx=
\int_{\R^l} 
\norm{
\Psi_{s,m(\cdot),X_s}(x)}_{ c_s(X_s,x)}^2m_s(x)\,dx\,,
\end{equation*}
so, by \eqref{eq:88},
\begin{equation}
  \label{eq:141}
       \lim_{j\to\infty}\int_{\R^l} Q_{s,m^{(j)}_s(\cdot)}(X^{(j)}_{s},x)m^{(j)}_s(x)\,dx=
\int_{\R^l}Q_{s,m_s(\cdot)}(X_{s},x)m_s(x)\,dx\,.
\end{equation}

By  \eqref{eq:54}, \eqref{eq:72}, \eqref{eq:184},  \eqref{eq:184a},
 \eqref{eq:4},
and \eqref{eq:141}, one has  pointwise convergence of the
integrands with respect to $ds$ on the lefthand side of \eqref{eq:182}
to the integrand on the righthand side.
Since the matrices $\int_{\R^l} Q_{s,m^{(j)}_s(\cdot)}(X^{(j)}_{s},x)\\m^{(j)}_s(x)\,dx$
are uniformly positive definite, 
the $j$th integrand is  bounded above
 by
 \begin{equation*}
M\bl(\abs{\dot{X}_{s}}^2+1+\int_{\R^l}\abs{
\frac{Dm^{(j)}_s(x)}{2m^{(j)}_s(x)}-\Phi_{s,m^{(j)}_s(\cdot),X^{(j)}_s}(x)}^2
m^{(j)}_s(x)\,dx\br)
\,,   
 \end{equation*}
 for some $M>0$\,. By \eqref{eq:58} and \eqref{eq:58a}, by the fact
that the right hand side of \eqref{eq:55} is finite, so
$\int_0^t\abs{\dot X_s}^2\,ds<\infty$\,,
 and by dominated convergence, we conclude that
\eqref{eq:182} holds.

We have thus proved that
in both cases there exist $(X^{(j)},\mu^{(j)})$ with needed regularity
properties that converge to
$(X,\mu)$ and are such that
 $\mathbf{I}^{\ast\ast}_t(X^{(j)},\mu^{(j)})\to
\mathbf{I}^{\ast\ast}_t(X,\mu)$ for all $t\in\R_+$\,. 
Picking a suitable subsequence, we can assume that
  $\mathbf{I}^{\ast\ast}_{j}(X^{(j)},\mu^{(j)})\le 
\mathbf{I}^{\ast\ast}_j(X,\mu)+1/j$\,.
We redefine the subsequence $(X^{(j)}_t,\mu^{(j)}_t)$ for $t\ge j$
such that $\mathbf{I}^{\ast\ast}_j(X^{(j)},\mu^{(j)})=
\mathbf{I}^{\ast\ast}(X^{(j)},\mu^{(j)})$\,,
thanks to 
  Lemma \ref{le:zero}.
The obtained sequence
 will still converge to $(X,\mu)$\,. In addition, 
$\limsup_{j\to\infty}$
$\mathbf{I}^{\ast\ast}(X^{(j)},\mu^{(j)})\le
\mathbf{I}^{\ast\ast}(X,\mu)$\,, 
which yields the assertion of Theorem~\ref{the:appr} by the lower
semicontinuity of $\mathbf{I}^{\ast\ast}$\,.
\end{proof}

\section{Proof of Theorem \ref{the:ldp}}
\label{sec:proof-theor-refth}

Suppose that 
$  \lim_{\epsilon\to0}\mathbf{P}^\epsilon(\abs{X^\epsilon_0-\hat{u}}>\kappa)^\epsilon=0
$, for arbitrary $\kappa>0$\,. Then any large deviation limit
point $\tilde{\mathbf{ I}}$ of $\mathbf P^\epsilon$ is such that 
$\tilde{\mathbf{ I}}(X,\mu)=\infty$ unless $X_0=\hat u$\,.
If  $({X},{\mu})$ is such that 
${X}_0=\hat{u}$ and
$\mathbf{I}^{\ast\ast}(X,\mu)<\infty$\,, by
Theorems
 \ref{the:equation},
 \ref{the:iden_reg},
 and \ref{the:appr}, there exist $(X^i,\,\mu^i)$\,, which
satisfy the hypotheses on $(\hat X,\hat \mu)$ in Theorem~\ref{the:iden_reg},
such that $\mathbf{I}^{\ast\ast}(X^i,\mu^i)=
\tilde{\mathbf{I}}(X^i,\mu^i)$,
$(X^i,\mu^i)\to(X,\mu)$ as $i\to\infty$\,,
and $\mathbf{I}^{\ast\ast}(X^i,\mu^i) \to
\mathbf{I}^{\ast\ast}(X,\mu)$ as $i\to\infty$\,.
By  Theorem \ref{the:id} (with the role of 
$\mathcal{U}$ being played by the set of functions
$U_{t\wedge\tau}^{\lambda(\cdot),f}$ in Theorem
\ref{the:equation} and with the role of $\tilde{\mathcal{U}}$ being 
played by the
set of functions $\theta^N$ in Lemma \ref{le:sup_ld_function})
and Theorem~\ref{le:vid},
$
\tilde{\mathbf{I}}(X,\mu)=\mathbf{I}^{\ast\ast}(X,\mu)=
\mathbf{I}(X,\mu)
$ for
all $(X,\mu)$\,.

 In the general setting of
Theorem~\ref{the:ldp}, let $\mathcal{L}^\epsilon_{\hat u}$ denote the regular
conditional distribution of $(X^\epsilon,\mu^\epsilon)$ given that 
$X^\epsilon_0=\hat u$\,, where $\hat u\in\R^n$ and is otherwise
arbitrary. 
By what has been proved,
 if $u^\epsilon\to\hat
u$ as $\epsilon\to0$, then the $\mathcal{L}^\epsilon_{u^\epsilon}$ obey
the LDP in
$\bC(\R_+,\R^n)\times\mathbb{C}_\uparrow(\R_+,\mathbb{M}(\R^l))$ 
with the large deviation function $\breve{\mathbf{I}}_{\hat u}$ as
 defined in the statement of Theorem~\ref{the:ldp}, where 
  $\mathbf I_0(\hat u)=0$ and $\mathbf I_0(u)=\infty$ if
 $u\not=\hat u$\,.
Since by the hypotheses of Theorem~\ref{the:ldp} the distributions of 
$X^\epsilon_0$ obey the LDP with a large deviation function
$\mathbf{I}_0$, it follows that the distributions of
$(X^\epsilon,\mu^\epsilon)$ obey the LDP with 
$\mathbf{I}(X,\mu)=\mathbf{I}_0(X_0)+\breve{\mathbf{I}}_{X_0}(X,\mu)$\,,
see, e.g., Chaganty \cite{Cha97}, Puhalskii \cite{Puh95}.
 Theorem~\ref{the:ldp} has been
proved.

\appendix

\section{Appendix}
\label{sec:appendix}

\begin{proof}[Proof of Lemma \ref{le:gradient}]
Let $\eta_r(x)=\eta(\abs{x}/r)$ and
\begin{equation}
  \label{eq:200}
   k_s(x,\lambda)= c_s(X_s,x)^{-1}\bl(\frac{\text{div}\,\bl(c_s(X_s,x)m_s(x)
\br)}{2m_s(x)}\,
-\bl(a_s(X_s,x)+  G_s(X_s,x)^T\lambda\br)\br)\,.
\end{equation}
By \eqref{eq:19} and Theorem \ref{le:vid}, $k_s(\cdot,\lambda)\in
\mathbb L^2(\R^l,\R^l,c_s(x,X_s),m_s(x)\,dx)$\,, for almost all $s$\,.
We  prove that, for those $s$\,,
\begin{multline}
  \label{eq:4}
  \lim_{r\to\infty}
    \sup_{h\in \mathbb{C}_0^1(\R^l)}\int_{\R^l}\bl( D  h(x)^T
c_s(X_s,x)k_s(x,\lambda)
-\frac{1}{2}\,\norm{D
  h(x)}_{c_s(X_s,x)}^2\,
\br)\eta_r^2(x)\,m_s(x) dx 
\\=    \sup_{h\in \mathbb{C}_0^1(\R^l)}\int_{\R^l}\bl( D  h(x)^T
c_s(X_s,x)k_s(x,\lambda)
-\frac{1}{2}\,\norm{D
  h(x)}_{c_s(X_s,x)}^2
\br) m_s(x) dx\,.
\end{multline}
Since 
\begin{multline*}
  \lim_{r\to\infty}
   \int_{\R^l}\bl( D  h(x)^T
c_s(X_s,x)k_s(x,\lambda)
-\frac{1}{2}\,\norm{D
  h(x)}_{c_s(X_s,x)}^2\,
\br)\eta_r^2(x)\,m_s(x) dx 
\\=    \int_{\R^l}\bl( D  h(x)^T
c_s(X_s,x)k_s(x,\lambda)
-\frac{1}{2}\,\norm{D
  h(x)}_{c_s(X_s,x)}^2
\br) m_s(x) dx\,,
\end{multline*}
we have that
\begin{multline*}
  \liminf_{r\to\infty}
    \sup_{h\in \mathbb{C}_0^1(\R^l)}\int_{\R^l}\bl( D  h(x)^T
c_s(X_s,x)k_s(x,\lambda)
-\frac{1}{2}\,\norm{D
  h(x)}_{c_s(X_s,x)}^2\,
\br)\eta_r^2(x)\,m_s(x) dx 
\\\ge   \sup_{h\in \mathbb{C}_0^1(\R^l)}\int_{\R^l}\bl( D  h(x)^T
c_s(X_s,x)k_s(x,\lambda)
-\frac{1}{2}\,\norm{D
  h(x)}_{c_s(X_s,x)}^2
\br) m_s(x) dx\,,
\end{multline*}
so, we need to prove that
\begin{multline}
  \label{eq:204}
  \limsup_{r\to\infty}
    \sup_{h\in \mathbb{C}_0^1(\R^l)}\int_{\R^l}\bl( D  h(x)^T
c_s(X_s,x)k_s(x,\lambda)
-\frac{1}{2}\,\norm{D
  h(x)}_{c_s(X_s,x)}^2\,
\br)\eta_r^2(x)\,m_s(x) dx 
\\\le   \sup_{h\in \mathbb{C}_0^1(\R^l)}\int_{\R^l}\bl( D  h(x)^T
c_s(X_s,x)k_s(x,\lambda)
-\frac{1}{2}\,\norm{D
  h(x)}_{c_s(X_s,x)}^2
\br) m_s(x) dx\,.
\end{multline}
On denoting 
 by $H_r$  the Hilbert space
$\mathbb   L^{1,2}_0(\R^l,\R^l,c_s(x,X_s),\eta_r^2(x)m_s(x)\,dx)$\,,
one can  write \eqref{eq:204} as
\begin{multline}
  \label{eq:180}
    \limsup_{r\to\infty}
    \sup_{Dh\in H_r}\int_{\R^l}\bl( D  h(x)^T
c_s(X_s,x)k_s(x,\lambda)
-\frac{1}{2}\,\norm{D
  h(x)}_{c_s(X_s,x)}^2\,
\br)\eta_r^2(x)\,m_s(x) dx 
\\\le   \sup_{Dh\in \mathbb{L}_0^{1,2}(\R^l,\R^l,c_s(x,X_s),m_s(x)\,dx)}\int_{\R^l}\bl( D  h(x)^T
c_s(X_s,x)k_s(x,\lambda)
-\frac{1}{2}\,\norm{D
  h(x)}_{c_s(X_s,x)}^2
\br) m_s(x) dx\,.
\end{multline}
Given $R>0$ and $\kappa\in(0,1)$\,, for all $r$ great enough,
\begin{multline}
  \label{eq:194}
  \sup_{Dh\in H_r}\int_{\R^l}\bl( D  h(x)^T
c_s(X_s,x)k_s(x,\lambda)
-\frac{1}{2}\,\norm{D
  h(x)}_{c_s(X_s,x)}^2\,
\br)\eta_r^2(x)\,m_s(x) dx\\
\le \sup_{Dh\in H_r}\int_{\R^l}\bl( D  h(x)^T
c_s(X_s,x)k_s(x,\lambda)\ind_{\{\abs{x}\le R\}}(x)
-\frac{1-\kappa}{2}\,\norm{D
  h(x)}_{c_s(X_s,x)}^2\,
\br)\eta_r^2(x)\,m_s(x) dx\\
+\int_{\R^l}\sup_{y\in\R^l}\bl( y^T
c_s(X_s,x)k_s(x,\lambda)\ind_{\{\abs{x}>R\}}(x)
-\frac{\kappa}{2}\,\norm{y}_{c_s(X_s,x)}^2\,
\br)\eta_r^2(x)\,m_s(x) dx\\
=\frac{1}{1-\kappa}\,
\sup_{Dh\in H_r}\int_{\R^l}\bl( D  h(x)^T
c_s(X_s,x)k_s(x,\lambda)\ind_{\{\abs{x}\le R\}}(x)
-\frac{1}{2}\,\norm{D
  h(x)}_{c_s(X_s,x)}^2\,
\br)\eta_r^2(x)m_s(x)\, dx
\\+\frac{1}{2\kappa}\,\int_{\R^l}\norm{k_s(x,\lambda)}^2_{c_s(X_s,x)}
\ind_{\{\abs{x}>R\}}(x)
\eta_r^2(x)m_s(x)\, dx\\
\le
\frac{1}{1-\kappa}\,
\sup_{Dh\in H_r}\int_{\R^l}\bl( D  h(x)^T
c_s(X_s,x)k_s(x,\lambda)\ind_{\{\abs{x}\le R\}}(x)
-\frac{1}{2}\,\norm{D
  h(x)}_{c_s(X_s,x)}^2\,
\eta_r^2(x)\br)m_s(x)\, dx
\\+\frac{1}{2\kappa}\,\int_{\R^l}\norm{k_s(x,\lambda)}^2_{c_s(X_s,x)}
\ind_{\{\abs{x}>R\}}(x)
m_s(x)\, dx\,.
\end{multline}
If $\norm{Dh(\cdot)}_{\mathbb
  L^2(\R^l,\R^l,c_s(x,X_s),\eta_r^2(x)
m_s(x)\,dx)}>2\norm{k_s(x,\lambda)}_{\mathbb
  L^2(\R^l,\R^l,c_s(x,X_s),m_s(x)\,dx)}$\,, then, by the
Cauchy--Schwarz inequality, the first integral on the rightmost side of
\eqref{eq:194} is negative.
Therefore, on
denoting by $K_r$ the subset of $H_r$ of (the equivalence classes of)
 functions $Dh$ such that
$\norm{Dh}_{\mathbb
  L^2(\R^l,\R^l,c_s(x,X_s),\eta_r^2(x)
m_s(x)\,dx)}\le 2\norm{k_s(x,\lambda)}_{\mathbb
  L^2(\R^l,\R^l,c_s(x,X_s),m_s(x)\,dx)}$\,, we have that
\begin{multline}
  \label{eq:181}
  \sup_{Dh\in H_r}\int_{\R^l}\bl( D  h(x)^T
c_s(X_s,x)k_s(x,\lambda)\ind_{\{\abs{x}\le R\}}(x)
-\frac{1}{2}\,\norm{D
  h(x)}_{c_s(X_s,x)}^2\,
\eta_r^2(x)\br)m_s(x)\, dx \\
=\sup_{Dh\in K_r}\int_{\R^l}\bl( D  h(x)^T
c_s(X_s,x)k_s(x,\lambda)\ind_{\{\abs{x}\le R\}}(x)
-\frac{1}{2}\,\norm{D
  h(x)}_{c_s(X_s,x)}^2\,
\eta_r^2(x)\br)m_s(x)\, dx \,.
\end{multline}
Being  closed subspaces of the reflexive Banach spaces
$\mathbb{L}^{2}(\R^l,\R^l,c_s(x,X_s),\eta_r^2(x)m_s(x)\,dx)$\,, the spaces 
$H_r$ are reflexive Banach spaces
in their own right. The sets $K_r$ are
 weakly compact subsets of the $H_r$\,, being convex, bounded, and
 closed.
The  $K_r$ make up an  inverse system of compact Hausdorff spaces
with   identifications of elements of $K_r$ as elements
of $K_{r'}$\,, for $r>r'$\,,
as bonding maps $\pi_{rr'}$\,. The inverse limit of the $K_r$\,, as $r\to\infty$\,,
which we denote by $K$ is a compact Hausdorff space and can be
identified with the subspace  of $
\mathbb   L^{1,2}_0(\R^l,\R^l,c_s(x,X_s),(x)m_s(x)\,dx)$
of the elements  $Dh$ such that 
$\norm{Dh}_{\mathbb
  L^2(\R^l,\R^l,c_s(x,X_s),
m_s(x)\,dx)}\le 2\norm{k_s(x,\lambda)}_{\mathbb
  L^2(\R^l,\R^l,c_s(x,X_s),m_s(x)\,dx)}$\,. The projections from $K$
to $K_r$ are identifications of the elements of $K$ as elements of $K_r$\,.
For each $r$\,,
the integral in \eqref{eq:181} is a concave and
continuous function of $Dh$ for the norm topology on $H_r$\,,
so, it is weakly upper semicontinuous on $H_r$\,, and, hence, on $K_r$\,.
Those integrals decrease as $r$ increases and converge to
$\int_{\R^l}\bl( D  h(x)^T
c_s(X_s,x)k_s(x,\lambda)\ind_{\{\abs{x}\le R\}}(x)
-\norm{D
  h(x)}_{c_s(X_s,x)}^2/2
\br)m_s(x)\, dx$\,, as $r\to\infty$\,.
By the version of Dini's lemma in Remark \ref{re:dini}, 
\begin{multline}
  \label{eq:195}
  \lim_{r\to\infty}\sup_{Dh\in K_r}\int_{\R^l}\bl( D  h(x)^T
c_s(X_s,x)k_s(x,\lambda)\ind_{\{\abs{x}\le R\}}(x)
-\frac{1}{2}\,\norm{D
  h(x)}_{c_s(X_s,x)}^2\,
\eta_r^2(x)\br)m_s(x)\, dx =\\
\sup_{Dh\in K}\int_{\R^l}\bl( D  h(x)^T
c_s(X_s,x)k_s(x,\lambda)\ind_{\{\abs{x}\le R\}}(x)
-\frac{1}{2}\,\norm{D
  h(x)}_{c_s(X_s,x)}^2\,
\br)m_s(x)\, dx \,.
\end{multline}
By \eqref{eq:194}, \eqref{eq:181}, and \eqref{eq:195},
\begin{multline}
  \label{eq:196}
  \limsup_{r\to\infty}
  \sup_{Dh\in H_r}\int_{\R^l}\bl( D  h(x)^T
c_s(X_s,x)k_s(x,\lambda)
-\frac{1}{2}\,\norm{D
  h(x)}_{c_s(X_s,x)}^2\,
\br)\eta_r^2(x)\,m_s(x) dx\\
\le
\frac{1}{1-\kappa}\,
\sup_{Dh\in \mathbb{L}_0^{1,2}(\R^l,\R^l,c_s(x,X_s),m_s(x)\,dx)}\int_{\R^l}\bl( D  h(x)^T
c_s(X_s,x)k_s(x,\lambda)\ind_{\{\abs{x}\le R\}}(x)
\\-\frac{1}{2}\,\norm{D
  h(x)}_{c_s(X_s,x)}^2
\br)m_s(x)\, dx
+\frac{1}{2\kappa}\,\int_{\R^l}\norm{k_s(x,\lambda)}^2_{c_s(X_s,x)}
\ind_{\{\abs{x}> R\}}(x)
m_s(x)\, dx\,.
\end{multline}
Since, similarly to \eqref{eq:195},
\begin{multline*}
  \sup_{Dh\in \mathbb{L}_0^{1,2}(\R^l,\R^l,c_s(x,X_s),m_s(x)\,dx)}\int_{\R^l}\bl( D  h(x)^T
c_s(X_s,x)k_s(x,\lambda)\ind_{\{\abs{x}\le R\}}(x)
-\frac{1}{2}\,\norm{D
  h(x)}_{c_s(X_s,x)}^2
\br)m_s(x)\, dx\\
\le \frac{1}{1-\kappa}\,
\sup_{Dh\in \mathbb{L}_0^{1,2}(\R^l,\R^l,c_s(x,X_s),m_s(x)\,dx)}\int_{\R^l}\bl( D  h(x)^T
c_s(X_s,x)k_s(x,\lambda)
-\frac{1}{2}\,\norm{D
  h(x)}_{c_s(X_s,x)}^2
\br)m_s(x)\, dx\\
+\frac{1}{2\kappa}\,\int_{\R^l}\norm{k_s(x,\lambda)}^2_{c_s(X_s,x)}
\ind_{\{\abs{x}> R\}}(x)
m_s(x)\, dx\,,
\end{multline*}
one obtains \eqref{eq:180} by letting $R\to\infty$ and $\kappa\to0$ on
the righthand side of \eqref{eq:196} and accounting for
$k_s(\cdot,\lambda)$ being square integrable.
The convergence in 
\eqref{eq:125} is obtained by an application of Lebesgue's dominated
convergence theorem.

By Theorem \ref{le:vid}, if $\mathbf I'(X,\mu)<\infty$\,, then
\begin{multline}
    \label{eq:172}
  \int_0^t\sup_{h\in\mathbb C_0^1(\R^l)}\int_{\R^l}
\bl(Dh(x)^T(\frac{1}{2}\,\text{div}\,c_s(X_s,x)
- a_s(X_s,x))-\frac{1}{2}\,\norm{Dh(x)}^2_{c_s(X_s,x)}\,
\br)m_s(x)\,dx\,ds<\infty\,.
\end{multline}
Suppose \eqref{eq:193} holds and let $L$ denote an upper bound for
 the lefthand side of \eqref{eq:172}. 
By \eqref{eq:121} in the statement of
Lemma \ref{le:density} and Condition \ref{con:coeff}, 
$L$ is also an upper bound on the integrals on
the left of \eqref{eq:172} for $h\in\mathbb W^{1,q}_0(S)$\,, where 
$S$ is an open ball in $\R^l$\,,
$q\ge 2$\,, and $q>l$\,.
 On taking $h(x)=
\kappa(\abs{x}^2\vee r_1^2\wedge r_2^2-r_2^2)$\,, where $\kappa>0$ and
 $0<r_1<r_2$\,, we have that 
\begin{multline*}
  \int_0^t\int_{x\in\R^l:\,r_1\le\abs{x}\le r_2}\bl(\kappa x^T\,
\bl(\frac{1}{2}\,\text{div}\,c_s(X_s,x)-a_s(X_s,x)\br)
-\kappa^2\norm{c_s(X_s,x)}\abs{x}^2\br)m_s(x)\,dx\,ds\le L\,.
\end{multline*}
If $r_1$ is great enough, there exists $\delta>0$ such that
$x^Ta_s(X_s,x)\le -\delta \abs{x}^2$ if $\abs{x}\ge r_1$\,.
Therefore, for small enough $\kappa>0$\,, great enough $r_1$\,, and
all $r_2>r_1$\,,
\begin{equation*}
\frac{\kappa\delta}{2}
  \int_0^t\int_{x\in\R^l:\,r_1\le\abs{x}\le
    r_2}\abs{x}^2m_s(x)\,dx\,ds\le L\,.
\end{equation*}
The square integrability of $a_s(X_s,x)$ now follows by Condition
\ref{con:coeff}. 

Suppose now \eqref{eq:30} holds.
We take, for given $s$\,,
\begin{equation*}
  h(x)=-
(\hat a_s(X_s,x)\vee(-\delta)\wedge \delta)
\eta_r(x)\,.
\end{equation*}
Then
\begin{equation*}
  Dh(x)=-D\hat a_s(X_s,x)\,\ind_{\{\abs{ \hat a_s(X_s,x)}\le \delta\}}(x)\eta_r(x)
-(\hat a_s(X_s,x)\vee(-\delta)\wedge \delta)\,
\frac{1}{r}\,\frac{x}{\abs{x}}\,D\eta\bl(\frac{\abs{x}}{r}\br)\,,
\end{equation*} 
  $h\in \mathbb W^{1,q}_0(S)$\,, for large enough ball $S$\,, and,
 for $\kappa\in(0,1)$\,,
\begin{multline*}
  \int_{\R^l}\bl( D  h(x)^T \bl(
\frac{1}{2}\,\text{div}\,c_s(X_s,x)-
a_s(X_s,x)\br)
-\frac{1}{2}\,
\norm{D  h(x)}_{c_s(X_s,x)}^2\br)m_s(x)\,dx\\\ge
\frac{1-\kappa}{2}\,
\int_{\R^l}\norm{D_x \hat a_s(X_s,x)}^2_{c_s(X_s,x)}\ind_{\{\abs{\hat
    a_s(X_s,x)}\le \delta\}}(x)
\eta_r(x)\,m_s(x)\,dx\\
-\int_{\R^l}(\hat a_s(X_s,x)\vee(-\delta)\wedge\delta)
\frac{1}{r}\,D\eta\bl(\frac{\abs{x}}{r}\br)\,\frac{x^T}{\abs{x}}\,
c_s(X_s,x)D_x\hat a_s(X_s,x)\,m_s(x)\,dx\\
-\frac{1}{2r^2}\,\bl(1+\frac{1}{\kappa}\br)
\int_{\R^l}(\hat a_s(X_s,x)\vee(-\delta)\wedge\delta)^2
\norm{D\eta\bl(\frac{\abs{x}}{r}\br)}^2_{c_s(X_s,x)}m_s(x)\,dx
\,.
\end{multline*}
As $r\to\infty$\,, the integrals from $0$ to $t$ of 
the latter two integrals converge to zero 
(we recall that by Theorem \ref{le:vid},
$\int_0^t\int_{\R^l}\abs{x^Ta_s(X_s,x)}/\abs{x}m_s(x)\,dx\,ds<\infty$\,,
so 
$\int_0^t\int_{\R^l}\abs{x^Tc_s(X_s,x)
D_x\hat a_s(X_s,x)}/\abs{x}m_s(x)\,dx\,ds<\infty$)\,.
Therefore,
\begin{equation*}
\frac{1}{2}\,  \int_0^t \int_{\R^l}\norm{
a_s(X_s,x)-\frac{1}{2}\,\text{div}\,c_s(X_s,x)
}^2_{c_s(X_s,x)^{-1}}\,m_s(x)\,dx\,ds\le L\,,
\end{equation*}
which implies the square integrability of $a_s(X_s,\cdot)$ thanks to
Conditions \ref{con:coeff} and \ref{con:positive}.
\end{proof}
\begin{remark}\label{re:dini}
  The version of Dini's lemma invoked in the proof is as follows. 
Let $(K_\sigma,\pi_{\sigma\tau},\Sigma)$
 be an inverse system  of compact Hausdorff topological spaces over directed set
 $\Sigma$\,. In particular, for $\tau\le\sigma$\,,
 $\pi_{\sigma\tau}$  is a continuous map from $\,K_\sigma$  to $ K_\tau$\,. Let $K$
  represent the inverse limit of the $K_\sigma$ and let
  $\pi_\sigma:\,K\to K_\sigma$ be the canonical projections. We note
  that $K$ is a compact Hausdorff topological space.
Let real--valued functions $f_\sigma$ be defined and upper semicontinuous
on the $K_\sigma$\,. Suppose that the sequence $f_\sigma$ is monotonically
decreasing  in the sense that if $x_\sigma\in K_\sigma$\,, then
$f_\sigma(x_\sigma)\le 
f_\tau(\pi_{\sigma\tau}x_\sigma)$\,, for all $\tau\le \sigma$\,. 
Let $x\in K$\,. Since, for $\tau\le\sigma$\,, $f_\tau(\pi_\tau x)=
f_\tau(\pi_{\sigma\tau}\circ \pi_\sigma x)\ge f_\sigma(\pi_\sigma
x)$\,, the net $f_\sigma(\pi_\sigma x)$ is monotonically decreasing.
Let $f(x)$ represent the  limit
of $f_\sigma(\pi_\sigma x)$\,, as $\sigma\in\Sigma$\,. Then 
$\sup_{x_\sigma\in K_\sigma}f_\sigma(x_\sigma)
\to \sup_{x\in K}f(x)$\,. For a proof, one notes
that
 the sets $\pi_\sigma^{-1}\bl(\{x_\sigma\in K_\sigma:
\,f_\sigma(x_\sigma)\ge \sup_{x\in K}f(x)+\epsilon\}\br)$ form a
nested collection of closed subsets of $K$ with an empty intersection.
(There is also a sequential version of this result.) \end{remark}
\begin{remark}
  A ''down--to--earth'' version of the proof of \eqref{eq:195} goes as
  follows.
Let the supremum on the lefthand side be attained at $Dh_r$\,. 
All of the $Dh_r$ belong to
$K_{r_0}$\,, provided $r\ge r_0$\,. By the diagonal process, we may
assume that the $Dh_r$ converge weakly to some $D\hat h$ that belongs to all
$K_{r_0}$\,, so it belongs to $K$\,. We also have that
\begin{multline*}
  \int_{\R^l}\bl( D  h_r(x)^T
c_s(X_s,x)k_s(x,\lambda)\ind_{\{\abs{x}\le R\}}(x)
-\frac{1}{2}\,\norm{D
  h_r(x)}_{c_s(X_s,x)}^2\,
\eta_r^2(x)\br)m_s(x)\, dx\\
\le \int_{\R^l}\bl( D  h_r(x)^T
c_s(X_s,x)k_s(x,\lambda)\ind_{\{\abs{x}\le R\}}(x)
-\frac{1}{2}\,\norm{D
  h_r(x)}_{c_s(X_s,x)}^2\,
\eta_{r_0}^2(x)\br)m_s(x)\, dx\,.
\end{multline*}
By the upper semicontinuity in $K_{r_0}$\,,
\begin{multline*}
  \limsup_{r\to\infty}\int_{\R^l}\bl( D  h_r(x)^T
c_s(X_s,x)k_s(x,\lambda)\ind_{\{\abs{x}\le R\}}(x)
-\frac{1}{2}\,\norm{D
  h_r(x)}_{c_s(X_s,x)}^2\,
\eta_{r_0}^2(x)\br)m_s(x)\, dx\\
\le \int_{\R^l}\bl( D  \hat h(x)^T
c_s(X_s,x)k_s(x,\lambda)\ind_{\{\abs{x}\le R\}}(x)
-\frac{1}{2}\,\norm{D
  \hat h(x)}_{c_s(X_s,x)}^2\,
\eta_{r_0}^2(x)\br)m_s(x)\, dx\,.
\end{multline*}
Letting $r_0\to\infty$ yields
\begin{multline*}
  \limsup_{r\to\infty}\int_{\R^l}\bl( D  h_r(x)^T
c_s(X_s,x)k_s(x,\lambda)\ind_{\{\abs{x}\le R\}}(x)
-\frac{1}{2}\,\norm{D
  h_r(x)}_{c_s(X_s,x)}^2\,
\eta_{r}^2(x)\br)m_s(x)\, dx\\
\le \int_{\R^l}\bl( D  \hat h(x)^T
c_s(X_s,x)k_s(x,\lambda)\ind_{\{\abs{x}\le R\}}(x)
-\frac{1}{2}\,\norm{D
  \hat h(x)}_{c_s(X_s,x)}^2\,
\br)m_s(x)\, dx\,.
\end{multline*}

\end{remark}
\begin{proof}[Proof of Lemma \ref{le:exp_tight}]
    By Ascoli's theorem,  set $\Gamma\subset
  \bC(\R_+,\mathbb{M}(\R^l))$ is relatively compact if and only if for each
  $t\in\R_+$ there exists  compact $K_t\subset
\mathbb{M}(\R^l)$ such that $\tilde\nu_t\in
  K_t$ for all $\tilde\nu=(\tilde\nu_s\,,s\in\R_+)\in\Gamma$ and 
  \begin{equation}
    \label{eq:147}
\lim_{\delta\to0}\sup_{\tilde\nu\in \Gamma}
\sup_{s_1,s_2\in[0,t]:\,\abs{s_1-s_2}\le\delta}d(\tilde\nu_{s_1},\tilde\nu_{s_2})=0\,. 
  \end{equation}
Suppose, the net $\nu_\epsilon$ is  sequentially exponentially tight
 for rate $1/\epsilon$\,.
Given arbitrary $t>0$ and $\eta>0$\,, 
let $\{\epsilon_i,\,i\in\N\}$ be a sequence converging to zero as
$i\to\infty$ such that 
\begin{align*}
\lim_{i\to\infty}
\mathbf{P}_{\epsilon_i}\bl(\nu_{{\epsilon_i},t}(x\in\R^l:\,\abs{x}>i)>\eta\br)^{{\epsilon_i}}=
  \lim_{N\to\infty}\limsup_{\epsilon\to0}
\mathbf{P}_\epsilon\bl(\nu_{\epsilon,t}(x\in\R^l:\,\abs{x}>N)>\eta\br)^{\epsilon}
\intertext{and} 
\lim_{i\to\infty}
\sup_{s_1\in[0,t]}
\mathbf{P}_{\epsilon_i}\bl(\sup_{s_2\in [s_1,s_1+1/i]}
d(\nu_{{\epsilon_i},s_1},\nu_{{\epsilon_i},s_2})
>\eta\br)^{{\epsilon_i}}=
\lim_{\delta\to0}\limsup_{\epsilon\to0}
\sup_{s_1\in[0,t]}
\mathbf{P}_\epsilon\bl(\sup_{s_2\in [s_1,s_1+\delta]}
d(\nu_{\epsilon,s_1},\nu_{\epsilon,s_2})
>\eta\br)^{\epsilon}\,.
\end{align*}
 Since 
the sequence $\{\nu_{\epsilon_i},\,i\in\N\}$ is exponentially
tight for rate $1/{\epsilon_i}$\,,  given
$\kappa>0$, there exists compact $\Gamma$ such that
$\limsup_{i\to\infty}
\mathbf{P}_{\epsilon_i}(\nu_{\epsilon_i}\notin\Gamma)^{\epsilon_i}<\kappa$\,. Hence,
for the associated $K_t$, $\limsup_{i\to\infty}
\mathbf{P}_{\epsilon_i}(\nu_{\epsilon_i,t}\notin K_t)^{\epsilon_i}<\kappa$\,.
By compactness of $K_t$,  for arbitrary $\eta>0$ there exists 
$i$ such that
$\sup_{\tilde{\nu}\in K_t}\tilde\nu\bl(x:\,\abs{x}>i\br)\le\eta$, so 
$\limsup_{i\to\infty}\mathbf{P}_{\epsilon_i}\bl(\nu_{\epsilon_i,t}
(x:\,\abs{x}>i)>\eta\br)^{\epsilon_i}
<\kappa\,.$ Similarly, choosing $i$ in such a way that 
$\sup_{\nu\in\Gamma}\sup_{s_1,s_2\in[0,t]:\,\abs{s_1-s_2}<1/i}d(\nu_{s_1},\nu_{s_2})\le
\eta$
we obtain that 
$\limsup_{i\to\infty}
\mathbf{P}_{\epsilon_i}\bl(\sup_{s_1,s_2\in[0,t]:\,\abs{s_1-s_2}<1/i}
d(\nu_{\epsilon_i,s_1},\nu_{\epsilon_i,s_2})
>\eta\br)^{\epsilon_i}<\kappa$\,, which implies that
$\limsup_{i\to\infty}
\sup_{s_1\in[0,t]}
\mathbf{P}_\epsilon\bl(\sup_{s_2\in [s_1,s_1+1/i]}
d(\nu_{\epsilon_i,s_1},\nu_{\epsilon_i,s_2})
>\eta\br)^{\epsilon_i}<\kappa$\,.

Suppose now that the convergences in the hypotheses hold. We first
show that 
\begin{equation*}
\lim_{\delta\to0}\limsup_{\epsilon\to0}
\mathbf{P}_\epsilon\bl(\sup_{s_1,s_2\in[0,t]:\,
\abs{s_1-s_2}\le\delta}
d(\nu_{\epsilon,s_1},\nu_{\epsilon,s_2})
>\eta\br)^{\epsilon}=0\,.
\end{equation*}
It follows by the bound (cf. Billingsley \cite[Chapter 2]{Bil68})
\begin{multline*}
  \mathbf{P}_\epsilon\bl(\sup_{s_1,s_2\in[0,t]:\,
\abs{s_1-s_2}\le\delta}
d(\nu_{\epsilon,s_1},\nu_{\epsilon,s_2})
>\eta\br)\le
\mathbf{P}_\epsilon\Bl(\bigcup_{i=0}^{\lfloor t/\delta\rfloor}
\{3\sup_{s\in[i\delta,(i+1)\delta]}
d(\nu_{\epsilon,i\delta},\nu_{\epsilon,s})>\eta\}
\Br)\\\le
\sum_{i=0}^{\lfloor t/\delta\rfloor}
\mathbf{P}_\epsilon\bl(3\sup_{s\in[i\delta,(i+1)\delta]}
d(\nu_{\epsilon,i\delta},\nu_{\epsilon,s})>\eta\br) \le
\bl(\frac{t}{\delta}+1\br)\sup_{s_1\in[0,t]}\mathbf{P}_\epsilon\bl(
\sup_{s_2\in[s_1,s_1+\delta]}
d(\nu_{\epsilon,s_1},\nu_{\epsilon,s_2})>\frac{\eta}{3}\br).
\end{multline*}
Let $\epsilon_i\to0$ as $i\to\infty$\,. 
Given $\kappa>0$\,, $t\in\R_+$\,, and $j\in\N$\,,
one can choose $\delta_{j,t}>0$  such that $\delta_{j,t}\to0$ as
$j\to\infty$ and
\begin{equation*}
\limsup_{i\to\infty} \mathbf{P}_{\epsilon_i}\bl(\sup_{s_1,s_2\in[0,t]:\,
\abs{s_1-s_2}\le\delta_{j,t}}
d(\nu_{\epsilon_i,s_1},\nu_{\epsilon_i,s_2})
>\frac{1}{j}\br)^{{\epsilon_i}}<\frac{\kappa}{2^j2^t}\,,
\end{equation*}
 for all $i\in\N$ and all $j\in\N$\,.
Since the space $\mathbb{M}(\R^l)$ is complete and separable, each of
the measures $\nu_\epsilon$ has a tight distribution, so we can choose
$\tilde{\delta}_{j,t}>0$ such that
$\tilde{\delta}_{j,t}\to0$ as $j\to\infty$ and  the inequality
\begin{equation}
  \label{eq:149}
    \mathbf{P}_{\epsilon_i}\bl(\sup_{s_1,s_2\in[0,t]:\,
\abs{s_1-s_2}\le\tilde\delta_{j,t}}
d(\nu_{\epsilon_i,s_1},\nu_{\epsilon_i,s_2})
>\frac{1}{j}\br)^{{\epsilon_i}}<\frac{\kappa}{2^j2^t}
\end{equation}
holds for all $i$\,.
Similarly, one can choose  $N_{j,t}\to\infty$ as $j\to\infty$
 such that the inequality
\begin{equation}
  \label{eq:148}
  \mathbf{P}_{\epsilon_i}\bl(
\max_{L=1,2,\ldots,\lfloor t/\tilde\delta_{j,t}\rfloor}
\nu_{{\epsilon_i},L\tilde \delta_{j,t}}(x\in\R^l:\,\abs{x}>
N_{j,t})>\frac{1}{j}\br)^{{\epsilon_i}}<\frac{\kappa}{2^j2^t}
\end{equation}
holds for all $i\in\N$\,.

Let 
\begin{multline*}
\Gamma=\bigcap_{M\in\N}\bigcap_{j\in\N}
\Bl(\{\tilde\nu\in \bC(\R_+,\mathbb{M}(\R^l)):\,
\sup_{s_1,s_2\in[0,M]:\,
\abs{s_1-s_2}\le\tilde\delta_{j,M}}
d(\tilde\nu_{s_1},\tilde\nu_{s_2})
\le\frac{1}{j}\}\\\bigcap\bigcap_{L=1}^{\lfloor M/\tilde\delta_{j,M}\rfloor}
\{\tilde\nu\in \bC(\R_+,\mathbb{M}(\R^l)):\,
\tilde\nu_{L\tilde\delta_{j,M}}
(\abs{x}> N_{j,M})\le\frac{1}{j}\}\Br)\,.
\end{multline*}
According to the definition, \eqref{eq:147} holds for all $t\in\R_+$\,.
In addition, if $M\ge t$, then, for $\tilde\nu\in \Gamma$, 
using the definition of $d(\cdot,\cdot)$\,,
\begin{equation*}
  \tilde\nu_t
(\abs{x}>
 N_{j,M}+1)\le
\max_{L=1,2,\ldots,\lfloor M/\tilde\delta_{j,M}\rfloor}  \tilde\nu_{L\tilde\delta_{j,M}}
(\abs{x}>
 N_{j,M})+\frac{1}{j}\le \frac{2}{j}\,.
\end{equation*}
Hence,
\begin{equation*}
  \lim_{j\to\infty}\sup_{\tilde\nu\in \Gamma}
  \tilde\nu_t
(\abs{x}>
 N_{j,M}+1)=0\,,
\end{equation*}
so by Prohorov's theorem, the $\tilde\nu_t$ belong to a compact set.
It follows that $\Gamma$ is relatively 
compact in $\bC(\R_+,\mathbb{M}(\R^l))$\,.

In addition, by  \eqref{eq:149} and \eqref{eq:148}, assuming $\epsilon_i<1$\,,
\begin{multline*}
  \mathbf{P}_{\epsilon_i}(\nu_{\epsilon_i}\notin
  \Gamma)^{\epsilon_i}\le
\sum_{M=1}^\infty\sum_{j=1}^\infty\Bl(
    \mathbf{P}_{\epsilon_i}\bl(\sup_{s_1,s_2\in[0,M]:\,
\abs{s_1-s_2}\le\tilde\delta_{j,M}}
d(\nu_{\epsilon_i,s_1},\nu_{\epsilon_i,s_2})
>\frac{1}{j}\br)^{{\epsilon_i}}\\+
  \mathbf{P}_{\epsilon_i}\bl(
\max_{L=1,2,\ldots,\lfloor M/\tilde\delta_{j,M}\rfloor}
\nu_{{\epsilon_i},L\tilde\delta_{j,M}}(x\in\R^l:\,\abs{x}>
N_{j,M})>\frac{1}{j}\br)^{{\epsilon_i}}\Br)\le
2\kappa\,.
\end{multline*}
Thus, the sequence $\{\nu_{\epsilon_i},\,i\in\N\}$  is exponentially tight for rate
$1/\epsilon_i$ as $i\to\infty$\,. This completes the proof of the
first part of the lemma. The second part is proved similarly.
\end{proof}
\begin{proof}[Proof of Lemma \ref{le:spaces}]
  
  We follow the approach of R\"ockner and Zhang
\cite[pp.204,205]{RocZha92}, \cite{RocZha94}, see also Bogachev, Krylov, and R\"ockner
\cite{BogKryRoc96}.
Let $\psi$ be a bounded function from $\mathbb W^{1,2}(O,m(x)\,dx)$\,.
Let, for $j\in\N$\,, $b_j$ represent an $\R_+$-valued
$\mathbb C_0^\infty$-function on $\R$ such that, for $y\in\R$\,,
$\ind_{[-j,j]}(y)\le b_j(y)\le
\ind_{(-j-1,j+1)}(y)$ and  $\abs{Db_j(y)}\le
2$\,, and let
 $\phi_j(x)=b_j(\ln m(x))$
 if $ m(x)>0$ and $\phi_j(x)=0$ if $m(x)=0$\,, where $x\in\R^d$\,.
It is noteworthy that $\phi_j\in\mathbb{W}^{1,2}(O)$\,.
 We have that
\begin{multline*}
  \int_{O}\abs{\psi(x)-\phi_j(x)\psi(x)}^2\,m(x)\,dx\le
  \int_{O}\abs{\psi(x)}^2\ind_{\{\abs{\ln m(x)}\ge
    j\}}(x)\,m(x)\,dx\to 0\text{ as }j\to\infty\,.
\end{multline*}
In addition,
\begin{multline*}
  \int_{O}\abs{D\psi(x)-D(\phi_j(x)\psi(x))}^2\,m(x)\,dx\le
2  \int_{O}\abs{D\psi(x)}^2
\ind_{\{\abs{\ln m(x)\ge
    j}\}}(x)\,m(x)\,dx\\
+8 \int_{O}\abs{\psi(x)}^2\ind_{\{\abs{\ln m(x)}\ge
    j\}}(x)\,\frac{\abs{Dm(x)}^2}{m(x)}
\,dx
\to 0\text{ as }j\to\infty\,.
\end{multline*}
Thus, $\phi_j\psi\to \psi$ in
$\mathbb W^{1,2}(O,m(x)\,dx)$ as $j\to\infty$\,.

We have that
\begin{equation*}
  \int_{O}\abs{\phi_j(x)\psi(x)}^2\,dx=
\int_{O}\abs{\phi_j(x)\psi(x)}^2\ind_{\{m(x)> e^{-j-1}\}}\,dx
\le e^{j+1}\int_{O}\abs{\psi(x)}^2m(x)\,dx\,.
\end{equation*}
Similarly,
\begin{equation*}
\int_{O}\abs{D(\phi_j(x)\psi(x))}^2\,dx\le
8e^{j+1}\int_{O}
\frac{\abs{Dm(x)}^2}{m(x)}\,\abs{\psi(x)}^2\,dx+
2e^{j+1}\int_{O}\abs{D\psi(x)}^2m(x)\,dx\,.
\end{equation*}
Thus, $\phi_j\psi\in \mathbb W^{1,2}(O)$\,, so there exists 
sequence $\psi_i$ of $\mathbb C^\infty(O)$-functions with bounded
$\abs{\psi_i(x)}$ and $\abs{D\psi_i(x)}$ that converges to 
$\phi_j\psi$ in $ \mathbb W^{1,2}(O)$ as $i\to\infty$\,.
The following inequalities show that
 $\phi_{j+1}\psi_i\to \phi_{j}\psi$ in
$\mathbb W^{1,2}(O,m(x)\,dx)$ as $i\to\infty$\,: since
$\phi_j=\phi_j\phi_{j+1}$\,, 
\begin{multline*}
  \int_{O}\bl(\abs{\phi_j(x)\psi(x)-\phi_{j+1}(x)\psi_i(x)}^2+
\abs{D(\phi_j(x)\psi(x))-D(\phi_{j+1}(x)\psi_i(x))}^2\br)\,m(x)\,dx\\\le
  \int_{O}\bl(9\abs{\phi_j(x)\psi(x)-\psi_i(x)}^2
+2\abs{D(\phi_j(x)\psi(x))-D\psi_i(x)}^2\br)
\,\ind_{\{\abs{\ln m(x)}\le    j+2\}}(x)\,m(x)\,dx\\\le
e^{j+2}\int_{O}\bl(9\abs{\phi_j(x)\psi(x)-\psi_i(x)}^2+2
\abs{D(\phi_j(x)\psi(x))-D\psi_i(x)}^2\br)
\,dx\to0\;\text{ as }i\to\infty\,.
\end{multline*}

It  remains to check that $\phi_{j+1}\psi_i\in
\mathbb H^{1,2}(O,m(x)\,dx)$\,, which
follows provided $\phi_{j+1}\in \mathbb H^{1,2}(O,m(x)\,dx)$\,.
If $O$ is an open ball in $\R^d$\,, we
let $\eta(x)$ represent a $\mathbb{C}_0^\infty(\R^d)$-function such that 
$\eta(x)\in[0,1]$ and $\eta(x)=1$ on $O$ and let 
$\tilde{m}(x)=m(x)\eta(x)+(1-\eta(x))/(1+\abs{x}^{d+1})$ and 
$\tilde\phi_{j+1}(x)=b_j(\ln\tilde m(x))$\,. 
Then 
 $\tilde m(x)=m(x)$ and $\tilde\phi_{j+1}(x)=\phi_{j+1}(x)$
 for $x\in O$\,. If $O=\R^d$\,, we let 
$\tilde m(x)=m(x)$ and $\tilde\phi_{j+1}(x)=\phi_{j+1}(x)$ for all $x\in\R^d$\,.
It suffices to prove that
$\tilde\phi_{j+1}\in \mathbb H^{1,2}(\R^d,\tilde m(x)\,dx)$\,.
 We have that 
 $\sqrt{\tilde m}\in \mathbb W^{1,2}(\R^d)$ for both cases where $O$
 is an open ball   and  $O=\R^d$\,. By R\"ockner and Zhang
 \cite{RocZha94}
(see pp.461--463),  $\tilde\phi_{j+1}$ belongs to the
domain  of the closure of the Dirichlet form
 $\mathcal{E}^0$  
on $\mathbb{L}^2(\R^d,\,\tilde{m}(x)\,dx)$ defined by 
$\mathcal{E}^0(f,g)=\int_{\R^d}Df(x)^T Dg(x)\,\tilde{m}(x)\,dx$\,, where
$f,g\in \mathbb{C}^\infty_0(\R^d)$\,. 
Thus, there exist $q_i\in \mathbb{C}^\infty_0(\R^d)$ such that 
$\int_{\R^d}\abs{Dq_i(x)-D\tilde\phi_{j+1}(x)}^2\,\tilde m(x)\,dx
+\int_{\R^d}\abs{q_i(x)-\tilde\phi_{j+1}(x)}^2\,\tilde m(x)\,dx\to0$ as
$i\to\infty$\, which implies the needed property.


If  $\psi\in \mathbb{H}^{1,2}(\R^d)$ is not bounded, then it is the limit of the
functions $\psi\wedge i\vee(-i)$  in $\mathbb{W}^{1,2}(\R^d)$
as $i\to\infty$\,.

\end{proof}
\begin{proof}[Proof of Lemma \ref{le:lok_kwadr}]

We follow the approach of
 Bogachev, Krylov, and
R\"ockner \cite{BogKryRoc96}, see also Metafune, Pallara, and Rhandi
 \cite{MetPalRha05}.
  Let 
\begin{equation*}
  \rho(x)=
  \begin{cases}
\displaystyle\bl(
\int_{y\in\R^d:\,\abs{y}<1}e^{1/(\abs{y}^2-1)}\,dy\br)^{-1}\,e^{1/(\abs{x}^2-1)},&\text{if }\abs{x}<1\\
0,&\text{if }\abs{x}\ge1,
      \end{cases}
\end{equation*}
let
$  \rho_\epsilon(x)=(1/\epsilon^d)\,\rho\bl(x/\epsilon\br),$
and let 
$     p_\epsilon(x)=\int_{\R^d}\rho_\epsilon(x-y)\,p(y)\,dy$\,,
where  $\epsilon>0$ and
$p\in \mathbb{C}_0^\infty(\R^d)$\,.
We note that 
$\int_{\R^d}\rho_\epsilon(x)\,dx=1$\,,
 $p_\epsilon\in \mathbb{C}_0^\infty(\R^d)$\,, and 
$\int_{\R^d}c(x) \rho_\epsilon(x-\cdot)m(x)\,dx\in
\mathbb{C}_0^\infty(\R^d,
\R^{d\times d})$\,.
One can see that, similarly to
 the calculation on p.227 of Bogachev, Krylov, and R\"ockner
\cite{BogKryRoc96}, 
\begin{multline*}
\int_{\R^d}\text{tr}\,\bl(c(x)D ^2p_\epsilon(x)\br)m(x)\,dx+
\int_{\R^d}b(x)^T D p_\epsilon(x)m(x)\,dx=
-\int_{\R^d} D p(x)^T c(x)\int_{\R^d}
D \rho_\epsilon(x-y)m(y)\,dy\,dx\\
-\int_{\R^d} D p(x)^T\int_{\R^d}(c(y)-c(x)) D \rho_\epsilon(x-y)m(y)\,dy\,dx+
\int_{\R^d}D p(x)^T  \int_{\R^d}b(y)
\rho_\epsilon(x-y)m(y)\,dy\,dx\,.
\end{multline*}
Since the lefthand side equals zero, introducing
\begin{equation}
  \label{eq:145}
  m_\epsilon(x)=\int_{\R^d}\rho_\epsilon(x-y)\,m(y)\,dy\,,
  \end{equation}
we have that
\begin{multline}
  \label{eq:57}
    \int_{\R^d} D p(x)^T c(x)
D m_\epsilon(x)\,dx=
-\int_{\R^d} D p(x)^T\int_{\R^d}(c(y)-c(x)) D \rho_\epsilon(x-y)m(y)\,dy\,dx\\+
\int_{\R^d}D p(x)^T  \int_{\R^d}b(y)
\rho_\epsilon(x-y)m(y)\,dy\,dx\,.
\end{multline}
Let $\eta\in \mathbb{C}_0^\infty(\R^d)$ be such that
$\abs{\eta(x)}\le 1$\,,
let $\kappa>0$\,, and let
$  p(x)=\eta^2(x)\ln\bl(m_\epsilon(x)+\kappa\br)$\,.
By  (\ref{eq:57}),
\begin{multline*}
  \int_{\R^d}\eta^2(x)\frac{D m_\epsilon(x)^T
    c(x)\,D m_\epsilon(x)}{m_\epsilon(x)+\kappa}\,dx=
-\int_{\R^d}\ln(m_\epsilon(x)+\kappa)\,2\eta(x)\,D \eta(x)^T
c(x)\,D m_\epsilon(x)\,dx\\
-\int_{\R^d}
\eta^2(x)\frac{D m_\epsilon(x)^T}{m_\epsilon(x)+\kappa}\int_{\R^d}(c(y)-c(x))D \rho_\epsilon(x-y)m(y)\,dy\,dx\\
-\int_{\R^d}\ln(m_\epsilon(x)+\kappa)\,2\eta(x)\,D \eta(x)^T
\int_{\R^d}(c(y)-c(x))D \rho_\epsilon(x-y)m(y)\,dy
\,dx\\
+\int_{\R^d}
\eta^2(x)\frac{D m_\epsilon(x)^T}{m_\epsilon(x)+\kappa}\int_{\R^d}
b(y)\rho_\epsilon(x-y)m(y)\,dy\,dx\\
+\int_{\R^d}\ln(m_\epsilon(x)+\kappa)\,2\eta(x)\,D \eta(x)^T\int_{\R^d}
b(y)\rho_\epsilon(x-y)m(y)\,dy\,dx\,.
\end{multline*}
The next step is to provide bounds on the terms on the righthand side.
By the Cauchy inequality, for arbitrary $\delta>0$,
  \begin{align*}
    \begin{split}
 &  \abs{\int_{\R^d}\ln(m_\epsilon(x)+\kappa)\,2\eta(x)\,D \eta(x)^T
c(x)\,D m_\epsilon(x)\,dx}\le    \delta
\int_{\R^d}\eta^2(x)\,\frac{D m_\epsilon(x)^T
c(x)\,D m_\epsilon(x)}{m_\epsilon(x)+\kappa}\,dx
\\+&\frac{1}{\delta}\,\int_{\R^d}\bl(\ln(m_\epsilon(x)+\kappa))^2(m_\epsilon(x)+\kappa)\,D \eta(x)^T c(x)D \eta(x)\,dx,
    \end{split}\\
    \begin{split}
&\abs{      \int_{\R^d}
\eta^2(x)\frac{D m_\epsilon(x)^T}{m_\epsilon(x)+\kappa}\int_{\R^d}(c(y)-c(x))D \rho_\epsilon(x-y)m(y)\,dy\,dx}\\\le&
\delta\int_{\R^d}
\frac{\eta^2(x)}{m_\epsilon(x)+\kappa}\,\abs{D m_\epsilon(x)}^2\,dx+
\frac{1}{4\delta}\,
\int_{\R^d}\frac{\eta^2(x)}{m_\epsilon(x)+\kappa}\,\abs{\int_{\R^d}(c(y)-c(x))D \rho_\epsilon(x-y)m(y)\,dy}^2\,dx,
    \end{split}
\\    \begin{split}    
    &\abs{  \int_{\R^d}\ln(m_\epsilon(x)+\kappa)\,2\eta(x)\,D \eta(x)
^T\int_{\R^d}(c(y)-c(x))D \rho_\epsilon(x-y)m(y)\,dy
\,dx}\\\le&
\int_{\R^d}\bl(\ln(m_\epsilon(x)+\kappa))^2(m_\epsilon(x)+\kappa)\,\abs{D \eta(x)}^2\,dx+
\int_{\R^d}\frac{\eta^2(x)}{m_\epsilon(x)+\kappa}\,\abs{\int_{\R^d}(c(y)-c(x))D \rho_\epsilon(x-y)m(y)\,dy}^2\,dx,
    \end{split}
\\    \begin{split}
      &\abs{\int_{\R^d}
\eta^2(x)\frac{D m_\epsilon(x)^T}{m_\epsilon(x)+\kappa}\int_{\R^d}
b(y)\rho_\epsilon(x-y)m(y)\,dy\,dx}\\\le&
\delta
\int_{\R^d}\eta^2(x)\,\frac{\abs{D
    m_\epsilon(x)}^2}{m_\epsilon(x)+\kappa}\,dx
+
\frac{1}{4\delta}\,\int_{\R^d}\frac{\eta^2(x)}{m_\epsilon(x)+\kappa}\,
\abs{\int_{\R^d}
b(y)\rho_\epsilon(x-y)m(y)\,dy}^2\,dx,    \end{split}\\
\begin{split}
&  \int_{\R^d}\ln(m_\epsilon(x)+\kappa)\,2\eta(x)\,D \eta(x)^T\int_{\R^d}
b(y)\rho_\epsilon(x-y)m(y)\,dy\,dx\\\le&
\int_{\R^d}\bl(\ln(m_\epsilon(x)+\kappa))^2(m_\epsilon(x)+\kappa)\,\abs{D \eta(x)}^2\,dx+
\int_{\R^d}\frac{\eta^2(x)}{m_\epsilon(x)+\kappa}\,
\abs{\int_{\R^d}
b(y)\rho_\epsilon(x-y)m(y)\,dy}^2\,dx\,.
\end{split}
  \end{align*}
Therefore, assuming $\vartheta>0$ is such that 
$\vartheta\le (y/\abs{y})^T c(x)(y/\abs{y})\le
\vartheta^{-1}$ for all $x,y\in\R^d$\,,
\begin{multline*}
(\vartheta-\delta\vartheta^{-1}-2\delta)    \int_{\R^d}\eta^2(x)\frac{\abs{D m_\epsilon(x)}^2}{m_\epsilon(x)+\kappa}\,dx\le
\bl(\frac{\vartheta^{-1}}{\delta}+2\br)\,\int_{\R^d}\bl(\ln(m_\epsilon(x)+\kappa))^2(m_\epsilon(x)+\kappa)\,\abs{D \eta(x)}^2\,dx\\
+\bl(\frac{1}{4\delta}+1\br)\,
\int_{\R^d}\frac{\eta^2(x)}{m_\epsilon(x)+\kappa}\,\abs{\int_{\R^d}(c(y)-c(x))D \rho_\epsilon(x-y)m(y)\,dy}^2\,dx
\\+\bl(\frac{1}{4\delta}+1\br)\,\int_{\R^d}\frac{\eta^2(x)}{m_\epsilon(x)+\kappa}\,
\abs{\int_{\R^d}
b(y)\rho_\epsilon(x-y)m(y)\,dy}^2\,dx\,.
\end{multline*}
By the Cauchy-Schwarz inequality and \eqref{eq:145},
\begin{multline}
  \label{eq:100}
(\vartheta-\delta\vartheta^{-1}-2\delta)    \int_{\R^d}\eta^2(x)\frac{\abs{D m_\epsilon(x)}^2}{m_\epsilon(x)+\kappa}\,dx\le
\bl(\frac{\vartheta^{-1}}{\delta}+2\br)\,\int_{\R^d}\bl(\ln(m_\epsilon(x)+\kappa))^2(m_\epsilon(x)+\kappa)\,\abs{D \eta(x)}^2\,dx\\
+\bl(\frac{1}{4\delta}+1\br)\,
\int_{\R^d}\eta^2(x)\,\int_{\abs{x-y}<\epsilon}\frac{\abs{(c(y)-c(x))D \rho_\epsilon(x-y)}^2}{\rho_\epsilon(x-y)}m(y)\,dy\,dx
\\+\bl(\frac{1}{4\delta}+1\br)\,\int_{\R^d}\eta^2(x)\int_{\R^d}\abs{ b(y)}^2\rho_\epsilon(x-y)m(y)\,dy\,dx\,.
\end{multline}
Let $N>0$ be such that $\eta(x)=0$ if $\abs{x}\ge N$\,.
By Jensen's inequality,
\begin{multline}
  \label{eq:101}
  \int_{\R^d}\bl(\ln(m_\epsilon(x)+\kappa))^2(m_\epsilon(x)+\kappa)\,\abs{D \eta(x)}^2\,dx
\\\le \sup_{\abs{x}\le N}\abs{D\eta(x)}^2
\int_{\abs{x}\le
  N}\int_{\R^d}\bl(\ln(m(y)+\kappa))^2(m(y)+\kappa)\rho_\epsilon(x-y)\,dy\,dx
\\\le 
\sup_{\abs{x}\le N}\abs{D\eta(x)}^2\int_{\abs{y}\le
  N+\epsilon}\bl(\ln(m(y)+\kappa))^2(m(y)+\kappa)\,dy\,.
\end{multline}

Since $c(x)$ is locally Lipschitz continuous, there exists $L>0$ which
only depends on $N$ such
that
$\norm{c(y)-c(x)}\le L\abs{x-y}$ for all $\epsilon>0$ small enough
 provided $\abs{x-y}\le \epsilon$ and
$\abs{x}\le N$\,. It follows  that
\begin{multline}
    \int_{\R^d}\eta^2(x)\,\int_{\abs{x-y}<\epsilon}\frac{\abs{(c(y)-c(x))D \rho_\epsilon(x-y)}^2}{\rho_\epsilon(x-y)}m(y)\,dy\,dx\\\le
  \int_{\R^d}\int_{\abs{x-y}<\epsilon}\eta^2(x)\,\frac{\norm{c(y)-c(x)}^2}{\epsilon^2}
\frac{4\abs{(y-x)/\epsilon}^2}{(\abs{(y-x)/\epsilon}^2-1)^4}\, \rho_\epsilon(y-x)
\,dx\, m(y)\,dy
  \label{eq:102}
\\\le 4L^2   \int_{\abs{x}<1}
\frac{\abs{x}^4}{(\abs{x}^2-1)^4}\, \rho(x)
\,dx\,.
\end{multline}
Also,
\begin{multline}
  \label{eq:103}
  \int_{\R^d}\eta^2(x)\int_{\R^d}\abs{
    b(y)}^2\rho_\epsilon(x-y)m(y)\,dy\,dx
\le
\int_{\abs{x}\le N}\int_{\abs{y}\le N+\epsilon}\abs{
  b(y)}^2\rho_\epsilon(x-y)m(y)\,dy\,dx\\\le
\int_{\abs{y}\le N+\epsilon}\abs{
  b(y)}^2m(y)\,dy\,.
\end{multline}
Combining (\ref{eq:100}), (\ref{eq:101}), (\ref{eq:102}), and
(\ref{eq:103}) and assuming that  $\delta>0$ is small enough yields
\begin{multline}
  \label{eq:68}
  \limsup_{\epsilon\to0}     \int_{\R^d}
     \eta^2(x)\frac{\abs{D m_\epsilon(x)}^2}{m_\epsilon(x)}\,dx\le
\frac{1}{\vartheta-\delta\vartheta^{-1}-2\delta }\bl(
\bl(\frac{\vartheta^{-1}}{\delta}+2\br)
\sup_{\abs{x}\le N}\abs{D\eta(x)}^2\int_{\abs{x}\le N}(\ln m(x))^2m(x)\,dx\\+
\bl(\frac{1}{4\delta}+1\br)\,4L^2   \int_{\abs{x}<1}
\frac{\abs{x}^4}{(\abs{x}^2-1)^4}\, \rho(x)
\,dx
+\bl(\frac{1}{4\delta}+1\br)\,\int_{\abs{x}\le N}\abs{ b(x)}^2m(x)\,dx\br)\,.
\end{multline}
Thus, the net $\sqrt{m_\epsilon}$ is weakly relatively compact in
$\mathbb{W}^{1,2}(S)$\,, so it
 is strongly relatively compact in $\mathbb{L}^2(S)$
for any open ball $S\subset \R^d$\,, see the Rellich-Kondrashov
theorem on p.168 of Adams and Fournier \cite{MR56:9247}.  Hence,
the net $m_\epsilon$ is strongly relatively compact in
$\mathbb{L}^1(S)$\,. Furthermore, 
any limit point of $\sqrt{m_\epsilon}$ in the
weak topology of $\mathbb{W}^{1,2}(S)$ is a limit point of
$ m_\epsilon$ in  $\mathbb{L}^1(S)$\,. On the other
hand, by \eqref{eq:145},
the net $m_\epsilon$ converges to $m$ in $\mathbb{L}^1(\R^d)$\,.
It follows that $\sqrt{m}$ is the unique weak limit of
$\sqrt{m_\epsilon}$ in $\mathbb{W}^{1,2}(S)$\,. 
Thus, $\sqrt{m}\in \mathbb{W}^{1,2}_{\text{loc}}(\R^d)$\,.
Letting $\epsilon\to0$ in \eqref{eq:68} yields \eqref{eq:108}.

\end{proof}
\bibliographystyle{plain}
\def\cprime{$'$} \def\cprime{$'$} \def\cprime{$'$} \def\cprime{$'$}
  \def\cprime{$'$} \def\polhk#1{\setbox0=\hbox{#1}{\ooalign{\hidewidth
  \lower1.5ex\hbox{`}\hidewidth\crcr\unhbox0}}} \def\cprime{$'$}
  \def\cprime{$'$} \def\cprime{$'$} \def\cprime{$'$} \def\cprime{$'$}

\end{document}